%% file: habil.tex
\pdfoutput=1
\documentclass[fontsize=11pt,chapterprefix=true,paper=a4]{scrbook}
\usepackage{amsmath,amssymb,amsthm}
\usepackage{mathtools}

\usepackage{ifxetex}
\ifxetex 
\usepackage{fontspec}
\usepackage{xunicode}
\usepackage{realscripts}
\defaultfontfeatures{Mapping=tex-text}
\usepackage{unicode-math}
\PassOptionsToPackage{unicode,psdextra}{hyperref}
\addtokomafont{disposition}{\rmfamily}
\addtokomafont{descriptionlabel}{\rmfamily}
\else
\usepackage{dutchcal}
\usepackage[T1]{fontenc}
\usepackage[utf8]{inputenc}
\let\mathscr\mathcal
\fi

\usepackage{tikz,pgfplots,tikz-cd}
\pgfplotsset{compat=1.14}
\usepackage{hyperref}
\hypersetup{
bookmarks=true,
colorlinks=true,
citecolor=[rgb]{0,0,0.5},
linkcolor=[rgb]{0,0,0.5},
urlcolor=[rgb]{0,0,0.75},
pdfpagemode=UseNone,
pdfstartview=FitH,
pdfdisplaydoctitle=true,
pdftitle={Habilitationsschrift},
pdfauthor={Pavel Zorin-Kranich},
pdflang=en-US
}
\usepackage[style=alphabetic]{biblatex}

\newcommand*{\Z}{\mathbb{Z}}

\def\C{\mathbb{C}}
\newcommand*{\N}{\mathbb{N}}
\newcommand*{\R}{\mathbb{R}}
\newcommand{\one}{\mathbf{1}}

\def\<{\left\langle}
\def\>{\right\rangle}
\newcommand*{\dif}{\mathrm{d}}
\newcommand*{\DMO}[1]{\expandafter\DeclareMathOperator\csname #1\endcsname {#1}}
\DMO{diam}
\DMO{dist}
\DMO{supp}
\DMO{Id}
\DMO{tr}

\newcommand{\ds}{\mathscr{d}} 
\newcommand{\DI}{\mathcal{I}}
\newcommand{\scale}{s}
\newcommand{\Scales}{\mathbf{S}}
\newcommand{\dm}{n} 
\newcommand{\degp}{d} 
\newcommand{\V}{\mathbb{V}}
\def\C{\mathbb{C}}
\newcommand{\size}{\mathbf{size}}
\newcommand{\parent}{\mathop{\mathrm{par}}}
\newcommand{\Leaves}{\mathcal{L}}
\newcommand{\W}{\mathbb{W}}
\newcommand{\h}{\mathrm{h}}
\newcommand{\p}{\mathrm{p}}
\newcommand{\eset}{B}
\newcommand{\wplus}{\oplus}
\newcommand{\wtimes}{\circledast}
\DeclarePairedDelimiter\abs{\lvert}{\rvert}
\DeclarePairedDelimiter\norm{\lVert}{\rVert}
\providecommand\given{}
\newcommand\SetSymbol[1][]{%
\nonscript\:#1\vert
\allowbreak
\nonscript\:
\mathopen{}}
\DeclarePairedDelimiterX\Set[1]\{\}{%
\renewcommand\given{\SetSymbol[\delimsize]}
#1
}
\DeclarePairedDelimiter\meas{\lvert}{\rvert}
\newcommand{\pvint}{\mathrm{p.v.}\int}

\newcommand{\calI}{\mathcal{I}}
\newcommand{\calD}{\mathcal{D}}
\newcommand{\calL}{\mathcal{L}}
\newcommand{\calQ}{\mathcal{Q}}
\newcommand{\Rmin}{\underline{R}}
\newcommand{\Rmax}{\overline{R}}

\renewcommand{\top}{\mathrm{top}}
\newcommand{\calE}{\mathcal{E}}
\newcommand{\calF}{\mathcal{F}}
\newcommand{\calS}{\mathcal{S}}
\newcommand{\calJ}{\mathcal{J}}
\newcommand{\ch}{\operatorname{ch}}

\newcommand{\smin}{\underline{\sigma}}
\newcommand{\smax}{\overline{\sigma}}
\newcommand{\sumin}{\scale_{\min}}
\newcommand{\sumax}{\scale_{\max}}
\newcommand{\bd}{\operatorname{bd}}
\newcommand{\CZK}{K}
\newcommand{\calN}{\mathcal{N}}

\DeclareMathOperator{\mdens}{\overline{dens}}

\def\PZdefchar#1{\expandafter\def\csname T#1\endcsname{\mathfrak{#1}}}
\def\PZdefloop#1{\ifx#1\PZdefloop\else\PZdefchar#1\expandafter\PZdefloop\fi}
\PZdefloop abcdefghijklmnopqrstuvwxyzABCDEFGHIJKLMNOPQRSTUVWXYZ\PZdefloop

\DeclarePairedDelimiterX\innerp[2]{\langle}{\rangle}{#1,#2}

\newcommand{\calP}{\mathcal{P}}
\newcommand{\calR}{\mathcal{R}}

\newcommand{\calG}{\mathcal{G}}

\newcommand{\calC}{\mathcal{C}}
\newcommand{\calH}{\mathcal{H}}

\newcommand{\bfE}{\mathbf{E}}
\newcommand{\Lip}{\mathrm{Lip}}
\newcommand{\Aem}{\mathcal{A}}
\newcommand{\Dem}{\mathcal{D}}
\newcommand{\Dini}{\mathrm{Dini}}

\numberwithin{equation}{section}
\newtheorem{theorem}[equation]{Theorem}
\newtheorem*{theorem*}{Theorem}
\newtheorem{lemma}[equation]{Lemma}
\newtheorem{proposition}[equation]{Proposition}
\newtheorem{corollary}[equation]{Corollary}
\newtheorem{claim}[equation]{Claim}
\newtheorem{conjecture}[equation]{Conjecture}
\theoremstyle{definition}
\newtheorem{definition}[equation]{Definition}
\theoremstyle{remark}
\newtheorem{remark}[equation]{Remark}

\def\CZ#{Calder\'{o}n--Zygmund}
\def\CZO#{Calder\'{o}n--Zygmund operator}

\begin{document}
\frontmatter
\setkomafont{subject}{\large}
\setkomafont{subtitle}{\large}
\setkomafont{date}{\large}
\subject{Habilitationsschrift}
\title{Modulation invariant operators}
\subtitle{\vspace{1cm}
  zur Erlangung der venia legendi\\
  \vspace{0.5cm}
  im Fach Mathematik\\
  \vspace{1cm}
  eingereicht bei der\\
  \vspace{0.5cm}
  Mathematisch-Naturwissenschaftlichen Fakult\"at\\
  \vspace{0.5cm}
  der Rheinischen Friedrich-Wilhelms-Universit\"at Bonn\\
  \vspace{1cm}
  von}
\author{Pavel Zorin-Kranich}
\date{Bonn, M\"arz 2018}

\maketitle
\tableofcontents
\mainmatter
\include{intro}
\include{cancellation-simplex}
\include{two-general}
\include{poly-carleson}
\include{single-scale}
\backmatter
\printbibliography[heading=bibintoc]
\end{document}

%% file: intro.tex
\chapter{Introduction}
\label{chap:intro}
The body of this cumulative thesis consists of five logically independent chapters.
They are all motivated by the same circle of problems in time-frequency analysis concerning modulation invariant operators.
\begin{description}
\item[Chapter~\ref{chap:cancellation-simplex}] is concerned with simplex \CZ{} forms.
These are singular variants of Brascamp--Lieb forms and enjoy the widest class of modulation invariances.
$L^{p}$ estimates for these forms would imply most of the other results that we are going to discuss, but seem to be out of reach of current techniques.
We obtain a small gain over the trivial bounds for these forms coming from H\"older's inequality.
This chapter is based on \cite{MR3685286}.
\item[Chapter~\ref{chap:two-general}] deals with a dyadic model of the triangular Hilbert form (the triangle here corresponds to a $3$-simplex in the previous chapter).
We obtain $L^{p}$ estimates assuming that one of the input functions has a special form.
Despite this restriction our result turns out to imply corresponding estimates for dyadic models of the Carleson operator and the bilinear Hilbert transform.
This chapter is joint work with Vjekoslav Kova\v{c} and Christoph Thiele \cite{arxiv:1506.00861}.
\item[Chapter~\ref{chap:poly-car}] deals with the polynomial Carleson operator.
Its boundedness has been conjectured in an article by Elias Stein ans Stephen Wainger and proved in dimension $1$ by Victor Lie using the argument for the Carleson operator due to Charles Fefferman.
We combine the ideas of these authors with a new discretization of the parameter space for this problem to obtain estimates on $L^{p}(\R^{\ds})$ for $\ds\geq 1$ and every $1< p<\infty$.
This chapter has previously appeared as \cite{arxiv:1711.03524}.
\item[Chapter~\ref{chap:bi-Lipschitz}] deals with the interaction of bi-Lipschitz transformations with Littlewood--Paley theory.
We obtain a paraproduct-type estimate that has implications for singular Radon transforms along variable curves.
This chapter is joint work with Shaoming Guo, Francesco di Plinio, and Christoph Thiele \cite{MR3841536}.
\item[Chapter~\ref{chap:single-scale}] deals with a directional square function associated to convolution with a bump function along a Lipschitz vector field.
Surprisingly, this square function is bounded on $L^{p}(\R^{2})$, $2<p<\infty$, although the original operator need not be bounded on these spaces.
This is also joint work with Shaoming Guo, Francesco di Plinio, and Christoph Thiele \cite{MR3841536}.
\end{description}
In Chapter~\ref{chap:intro} we discuss the historical background and motivation for the results listed above.

\section{Maximally modulated singular integrals}
\subsection{Pointwise convergence of Fourier series and integrals}
Let $f \in L^{p}(\R/\Z)$ and consider the partial Fourier sums
\[
\tilde S_{N}f(x) = \sum_{\abs{n}\leq N} \hat{f}_{n} e^{2\pi i n x}.
\]
The question whether the partial Fourier sums $\tilde S_{N}f$ converge pointwise almost everywhere to $f$ as $N\to\infty$ has been initially raised by Luzin \cite{zbMATH02605027} (in the case $p=2$). 
Soon afterwards, Kolmogorov found an $L^{1}$ function for which this is not the case \cite{zbMATH02598226}, casting some doubt on the conjectured convergence, until it has been proved for $p=2$ by Carleson over 40 years later \cite{MR0199631}.
Long before Carleson's work it has been known that pointwise a.e.\ convergence of partial Fourier series of $L^{2}$ functions is equivalent to the associated maximal operator $\tilde S_{*}f = \sup_{N} \abs{\tilde S_{N} f}$ having weak type $(2,2)$ \cite[Vol.\ 2, Theorem XIII.1.22]{MR0107776}.
A more general version of this result that applies for $1\leq p\leq 2$ became known as Stein's maximal principle \cite{zbMATH03168076}.
This has been taken up by Hunt \cite{MR0238019}, who has substantiated Carleson's claim that his convergence result can be extended to $L^{p}$ functions, $1<p<\infty$ (with endpoints near $p=1$ and $p=\infty$; we will not consider these endpoint issues).
Hunt brought to this subject the view point of the Calder\'on--Zygmund school that emphasizes the mapping properties of the operator $\tilde S_{*}$.

From today's point of view it is more natural to consider the corresponding question on the real line.
Let
\[
S_{N}f(x) = \int_{\abs{\xi}\leq N} \hat{f}(\xi) e(\xi x) \dif\xi
\]
denote the partial Fourier integral, where
\[
e(t) = e^{2\pi i t}
\]
is the standard character on $\R$.
Then $S_{N}$ is given by convolution with the kernel $D_{N}(t) = \sin(2\pi N t)/(\pi t)$, and the pointwise convergence $S_{N}f\to f$ as $N\to\infty$ is easy to show for Schwartz functions $f$.
The question again reduces to (weak or strong) type $(p,p)$ of the maximal operator $S_{*}f = \sup_{N} \abs{S_{N}f}$.
For $1<p<\infty$ this turns out to be equivalent to (weak or strong, respectively) type $(p,p)$ of the corresponding operator $\tilde S_{*}$ acting on functions on $\R/\Z$, see e.g.\ \cite{MR583403}, based on the ideas from \cite{zbMATH03272605}.
Using this fact one can also deduce that $S_{*}$ has weak type $(p,p)$, $1<p\leq 2$, from pointwise a.e.\ convergence $S_{N}f\to f$ for $f\in L^{p}(\R)$ by first showing $\tilde S_{N}f\to f$ for $f\in L^{p}(\R/\Z)$ (comparing the Dirichlet kernels on the torus and on the line), using Stein's maximal principle, and transferring the resulting maximal inequality back to $\R$.

\subsection{The Carleson operator and generalizations}
The maximal Fourier integral operator $S_{*}$ is pointwise bounded by the \emph{Carleson operator}
\begin{equation}
\label{eq:Car-op}
Cf(x) := \sup_{\xi\in\R} \abs[\big]{\pvint_{\R} \frac{e(\xi y) f(y)}{x-y} \dif y}.
\end{equation}
It has been introduced in \cite{MR0199631} and $L^{p}$ estimates have been established in \cite{MR0238019} following Carleson's approach.
Alternative approaches to estimating this operator are due to Fefferman \cite{MR0340926} and Lacey and Thiele \cite{MR1783613}; both latter approaches are used in this thesis.
The Carleson operator is the prototypical \emph{modulation invariant} operator in the sense that $Cf=C(f(\cdot) e(\xi \cdot))$ for any modulation by a linear phase $\xi\cdot$.

A natural question is whether a multidimensional analog of Carleson's theorem on pointwise convergence of Fourier series/integrals holds.
In the multidimensional setting there are several natural choices of summation schemes.
A surprising result of Fefferman \cite{MR0296602} tells that the ball multiplier operators
\[
\widehat{S_{N}f}(\xi) = \widehat{f}(\xi) \one_{\abs{\xi}\leq N}
\]
are not bounded on any $L^{p}(\R^{\ds})$ space unless $\ds=1$ or $p=2$.
By Stein's maximal principle the convergence $\tilde S_{N}f\to f$ therefore cannot hold pointwise almost everywhere for all $f\in L^{p}(\R^{\ds}/\Z^{\ds})$, $1<p<2$, since otherwise the associated maximal operator $\tilde S_{*}$ would have weak type $(p,p)$.
By interpolation this would imply that $\tilde S_{N}$ is bounded on $L^{q}(\R^{\ds}/\Z^{\ds})$ for $p<q<2$, and by transference \cite{zbMATH03272605} it would follow that $S_{N}$ is bounded on $L^{q}(\R^{\ds})$, a contradiction.
The corresponding problem in $L^{2}(\R^{\ds})$ is still open.

When partial Fourier integrals are taken over polygonal regions, some positive and negative results are either easy or are direct consequences of Carleson's theorem, see \cite{MR0279529,MR0435724}.
The first genuinely multidimensional extension of Carleson's theorem is due to Sj\"olin \cite{MR0336222}, who replaced the Hilbert kernel $1/t$ in \eqref{eq:Car-op} by a (sufficiently smooth) multidimensional \CZ{} kernel.
Specifically, let $k : \R^{\ds} \setminus\Set{0} \to \C$ be a function that is homogeneous of degree $-\ds$, that is, $k(\lambda x) = \lambda^{-\ds} k(x)$, has integral $0$ on the unit sphere, and satisfies the smoothness condition
\begin{equation}
\label{eq:Sjolin-smooth-cond}
\abs{k^{(\alpha)}(x)} \lesssim \abs{x}^{-\ds-\abs{\alpha}},
\quad
0\leq \abs{\alpha}\leq \ds+1.
\end{equation}
Here $\alpha\in\N^{\ds}$ is a multiindex, $k^{(\alpha)}$ denotes the $\alpha$-th partial derivative, and $\abs{\alpha} = \sum_{j} \alpha_{j}$.
Then the operator
\begin{equation}
\label{eq:Car-Sj-op}
C_{k}f(x) := \sup_{\xi\in\R^{\ds}} \abs[\big]{\pvint_{\R^{\ds}} k(x-y) e(\xi \cdot y) f(y) \dif y}
\end{equation}
is bounded on $L^{p}(\R^{\ds})$, $1<p<\infty$.

\subsection{Non-translation invariant kernels}
A (non-translation invariant,) $\tau$-H\"older continuous \CZ{} kernel on $\R^{\ds}$ is a function $\CZK:\Set{(x,y) \in \R^{\ds}\times\R^{\ds} \given x\neq y} \to \C$ such that
\begin{equation}
\label{eq:K-size}
\abs{\CZK(x,y)} \lesssim \abs{x-y}^{-\ds},
\end{equation}
\begin{equation}
\label{eq:K-reg}
\abs{\CZK(x,y)-\CZK(x',y)}+\abs{\CZK(y,x)-\CZK(y,x')} \lesssim \frac{\abs{x-x'}^{\tau}}{\abs{x-y}^{\ds+\tau}},
\quad
\abs{x-x'} \leq \frac12 \abs{x-y}.
\end{equation}
There are several equivalent ways to formalize the notion of an associated \CZ{} operator and its ($L^{2}$-)boundedness.
Perhaps the easiest condition to state is that the truncated operators
\begin{equation}
\label{eq:truncated-CZO}
T_{\Rmin}^{\Rmax}f(x) := \int_{\Rmin<\abs{x-y}<\Rmax} \CZK(x,y) f(y) \dif y
\end{equation}
are uniformly bounded on $L^{2}(\R^{\ds})$.
In this case the principal value integral
\[
Tf(x) = \lim_{\Rmin\to 0, \Rmax\to \infty} T_{\Rmin}^{\Rmax}f(x)
\]
exists almost everywhere and defines a bounded operator on $L^{2}(\R^{\ds})$.

General conditions for $L^{2}$-boundedness of an operator associated to a \CZ{} kernel are given by the $T(1)$ theorem \cite{MR763911} and its generalizations such as the $T(b)$ theorem \cite{MR850408}.
A classical example of a non-translation invariant \CZ{} kernel arises from the Cauchy integral on a Lipschitz curve.
Other examples are associated to pseudodifferential operators, see e.g.\ \cite[p.\ 294]{MR1085488} or the English translation \cite[p.\ 80]{MR1456993}.
It has been shown in \cite{MR2654142} that maximal modulations by linear phases of certain pseudodifferential operators to which the $T(1)$ theorem applies define bounded operators on $L^{p}(\R^{\ds})$, $1<p<\infty$.

\subsection{Polynomial modulations}
One of the first results in a series of papers of Ricci and Stein on singular integrals and singular Radon transforms on nilpotent Lie groups \cite{MR822187,MR890662,MR937632,MR1021141} was the boundedness on $L^{p}(\R^{\ds})$ of the polynomially modulated singular integral operator
\begin{equation}
\label{eq:RS-polymod-CZ-op}
Tf(x) = \int_{\R^{\ds}} e(P(x,y)) K(x,y) f(y) \dif y,
\end{equation}
where $P$ is a polynomial in $2\ds$ variables, with a bound that depends only on the degree of $P$ but not on its coefficients.
Motivated by this result, Sj\"olin's multidimensional Carleson theorem, and another result of Stein in the case $\ds=1$, $d=2$ \cite{MR1364908}, Stein and Wainger \cite{MR1879821} have asked whether (at least in in the translation-invariant case) the more general operator
\begin{equation}
\label{eq:poly-CS}
C_{k,d} f(x) = \sup_{P : \deg P \leq d} \abs[\Big]{\int_{\R^{\ds}} e(P(y)) k(y) f(x-y) \dif y},
\end{equation}
the supremum being taken over all real polynomials in $\ds$ variables of degree $\leq d$, could be bounded on $L^{p}(\R^{\ds})$.
They succeeded in establishing this with a supremum over polynomials without linear terms.
Their result is in some sense orthogonal to the Carleson--Sj\"olin result because forbidding linear terms in the polynomial $P$ eliminates all modulation invariances from \eqref{eq:poly-CS}.

The weak type $(2,2)$ estimate for the operator \eqref{eq:poly-CS} with an unrestricted supremum over all polynomials of a given degree $\leq d$ has been obtained by Lie, initially in the case $d=2$ \cite{MR2545246}, and subsequently for general $d$ \cite{arxiv:0805.1580}, in the one-dimensional case $\ds=1$ for the Hilbert kernel $k(t)=1/t$.
These articles followed Fefferman's approach to Carleson's theorem in \cite{MR0340926}.
Lie has subsequently refined \cite{arXiv:1105.4504} this approach in such a way that it yields the strong type $(2,2)$ estimate without an appeal to Marcinkiewicz type interpolation.
This is a remarkable development in view of possible applications to directional singular integrals on which we will comment later.

We show that the operator \eqref{eq:poly-CS}, and even its maximally truncated and non-translation invariant version, is bounded on $L^{p}(\R^{\ds})$ for $1< p<\infty$.

\begin{theorem}[{\cite{arxiv:1711.03524}}]
\label{thm:poly-car}
Let $\ds\geq 1$ and let $\CZK$ be a $\tau$-H\"older continuous Calder\'on--Zygmund kernel on $\R^{\ds}$.
We define the associated maximally polynomially modulated, maximally truncated singular integral operators by
\begin{equation}
\label{eq:Car-op-Kd}
C_{K,d} f(x):=\sup_{Q\in\calQ_{d}} \sup_{0< \Rmin \leq \Rmax < \infty}
\abs[\Big]{\int_{\Rmin \leq \abs{x-y} \leq \Rmax} \CZK(x,y) e(Q(y)) f(y) \dif y},
\end{equation}
where $\calQ_{d}$ denotes the class of all polynomials in $\ds$ variables with real coefficients and degree at most $d\in\N$.

Suppose that the truncated integral operators \eqref{eq:truncated-CZO} associated to $\CZK$ are bounded on $L^{2}(\R^{\ds})$ uniformly in $0<\Rmin<\Rmax<\infty$.
Then the operator \eqref{eq:Car-op-Kd} is bounded on $L^{p}(\R^{\ds})$ for every $1< p<\infty$ and every $d\in\N$.
\end{theorem}

Theorem~\ref{thm:poly-car} extends the previously mentioned results.
The extension to H\"older regular kernels is new even in the case $\ds=d=1$.

\section{Multilinear Hilbert transforms}
The multilinear Hilbert transforms have been originally invented as a tool to handle the Cauchy integral on Lipschitz curves.
One of the ways to write the latter operator is
\[
Cf(x) = \pvint_{\R} \frac{f(y)}{x-y+i(A(x)-A(y))} \dif y,
\]
where $A : \R \to \R$ is a Lipschitz function.
Assuming $\norm{A}_{\Lip} < 1$ and disregarding other convergence issues this can be written as
\begin{align*}
Cf(x)
&=
\int_{\R} \big( 1 + i \frac{A(x)-A(x-t)}{t} \big)^{-1} \frac{f(x-t)}{t} \dif t\\
&=
\int_{\R} \sum_{k=0}^{\infty} \big( - i \frac{A(x)-A(x-t)}{t} \big)^{k} \frac{f(x-t)}{t} \dif t\\
&=
\sum_{k=0}^{\infty} (-i)^{k} C_{k}(f,A),
\end{align*}
where
\[
C_{k}(f,A)(x) = \int_{\R} \big(\frac{A(x)-A(x-t)}{t} \big)^{k} \frac{f(x-t)}{t} \dif t
\]
is the \emph{$k$-th Calder\'on commutator}.
Writing $a=A'$ and representing the difference of $A$'s as an integral of $a$ we obtain
\begin{align*}
C_{k}(f,A)(x)
&=
\int_{\R} \big(\int_{0}^{1} a(x-\beta t) \dif \beta \big)^{k} \frac{f(x-t)}{t} \dif t\\
&=
\int_{[0,1]^{k}} \tilde\Lambda_{1,\beta_{1},\dotsc,\beta_{k}}(f,a,\dotsc,a)(x) \dif \beta_{1} \dotsm \dif \beta_{k},
\end{align*}
where
\begin{equation}
\label{eq:multilinearHT-op}
\tilde\Lambda_{\beta_{0},\dotsc,\beta_{k}}(f_{0},\dotsc,f_{k})(x)
=
\int_{\R} \prod_{j=0}^{k} f_{j}(x-\beta_{j}t) \frac{\dif t}{t}
\end{equation}
is the \emph{$(k+1)$-linear Hilbert transform}.
Estimates for the Calder\'on commutators would therefore follow from uniform (in $\beta_{0},\dots,\beta_{k}$) estimates for $\tilde\Lambda$ as a multilinear operator $L^{p} \times L^{\infty} \times \dotsb \times L^{\infty} \to L^{p}$.
For the purpose of the following discussion we will consider the dual form of the $\dm$-linear Hilbert transform:
\begin{equation}
\label{eq:multilinearHT}
\Lambda_{\beta_{0},\dotsc,\beta_{\dm}}(f_{0},\dotsc,f_{\dm})
:=
\int_{\R} \pvint_{\R} \prod_{i=0}^{\dm} f_i(x-\beta_{i}t) \frac{\dif t}{t} \dif x.
\end{equation}
It can be conjectured that this form is bounded on $L^{p_{0}}\times\dotsb\times L^{p_{\dm}}$ with $\dm<p_{i}<\infty$, say\footnote{This is a conservative choice because this is the easiest case when $\dm=2$}, provided that the necessary scaling condition
\begin{equation}
\label{eq:Lp-range}
\sum_{i=0}^{\dm} \frac{1}{p_{i}} = 1.
\end{equation}
holds.
The evidence pointing in the direction of this conjecture is quite circumstantial.
First, it has been known for a long time that Calder\'on commutators and the Cauchy integral on Lipschitz curves are in fact $L^{p}$ bounded operators \cite{MR0466568,MR672839,MR763911}, so that nowadays the above calculations should be seen as corroborating the conjecture about the multilinear Hilbert transform rather than a serious path to the former results.

Before stating the known results for $\dm=2$ let us recall the way to generalize $L^{p}$ estimates to negative exponents $p$.
An $(\dm+1)$-linear form $\Lambda$ is said to have \emph{generalized restricted type $(p_{0},\dotsc,p_{\dm})$} if for every tuple of measurable sets $E_{0},\dotsc,E_{\dm}$ and $\epsilon>0$ there exist subsets $E_{j}'\subset E_{j}$ that are \emph{major} in the sense that $\abs{E_{j}'} > (1-\epsilon) \abs{E_{j}}$ such that for any functions $f_{j}$ with $\abs{f_{j}} \leq \one_{E_{j}'}$ we have
\begin{equation}
\label{eq:gen-rest-type}
\abs{\Lambda(f_{0},\dotsc,f_{\dm})} \lesssim_{\epsilon} \prod_{j=0}^{\dm} \abs{E_{j}}^{1/p_{j}}.
\end{equation}
It is clear that the set of tuples $(1/p_{j})_{j}$ such that $\Lambda$ has generalized restricted type $(p_{j})$ is convex.
For the multilinear Hilbert transforms it is contained in the hyperplane described by \eqref{eq:Lp-range}, and \eqref{eq:gen-rest-type} interpolates to $L^{p}$ estimates at the relative interior points inside this hyperplane.

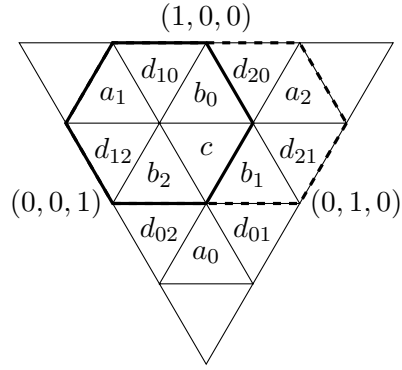
\begin{figure}
\begin{center}
\begin{tikzpicture}[x=70pt,y=70pt,cm={sin(30),-cos(30), -sin(30) ,-cos(30) ,(0,0)}]
\draw
(-1,1) -- (1,1) -- (1,-1) -- cycle
(1,-0.5) -- (0.5,-0.5) -- (0.5,1) -- (1,0.5) -- (-0.5,0.5) -- (-0.5,1) -- cycle
(1,0) node[right] {$(0,1,0)$} -- (0,0) node[above] {$(1,0,0)$} -- (0,1) node[left] {$(0,0,1)$} -- cycle
(0.33,0.33) node{$c$}
(0.16,0.16) node{$b_{0}$}
(0.16,0.66) node{$b_{2}$}
(0.66,0.16) node{$b_{1}$}
(0.33,0.83) node{$d_{02}$}
(-0.16,0.83) node{$d_{12}$}
(0.83,0.33) node{$d_{01}$}
(0.83,-0.16) node{$d_{21}$}
(0.33,-0.16) node{$d_{20}$}
(-0.16,0.33) node{$d_{10}$}
(0.66,0.66) node{$a_{0}$}
(0.66,-0.33) node{$a_{2}$}
(-0.33,0.66) node{$a_{1}$}
;
\draw[very thick] (0.5,0) -- (0.5,0.5) -- (0,1) -- (-0.5,1) -- (-0.5,0.5) -- (0,0) -- cycle;
\draw[very thick,dashed] (1,-0.5) -- (1,0) -- (0,1) -- (-0.5,1) -- (-0.5,0.5) -- (0.5,-0.5) -- cycle;
\end{tikzpicture}
\end{center}
\caption{Ranges of exponents satisfying the scaling condition \eqref{eq:Lp-range} in coordinates $(1/p_{0},1/p_{1},1/p_{2})$.}
\label{fig:exponents}
\end{figure}

For $\dm=2$ the picture of the $L^{p}$ estimates for \eqref{eq:multilinearHT}, known in this case as the bilinear Hilbert transform, seems to be nearing completion.
Figure~\ref{fig:exponents} shows the isometric projection of the set of tuples $(1/p_{0},1/p_{1},1/p_{2})$ satisfying \eqref{eq:Lp-range}.

We begin with the non-uniform estimates with all $\beta_{j}$ distinct.
Estimates in the local $L^{2}$ range (triangle $c$) have been first obtained by Lacey and Thiele in \cite{MR1491450}.
These estimates have been later extended to the outermost triangles $a_{j}$ in \cite{MR1689336}.
The Hilbert kernel has been replaced by more general symbols in \cite{MR1845098}.
It is currently not known whether these estimates can be further extended to the non-labeled triangles.

The range of exponents for which uniform estimates are possible is necessarily smaller.
Before discussing this issue let us notice that the form \eqref{eq:multilinearHT} does not change if we add the same number to all $\beta_{j}$'s, so we may assume without loss of generality $\beta_{0}=0$.
Moreover, it also does not change if we multiply all $\beta_{j}$ by the same non-zero real number, so we assume $\beta_{2}=1$.
Now, $\Lambda_{0,1,1}(f_{0},f_{1},f_{2})=\Lambda_{0,1}(f_{0},f_{1}f_{2})$.
In other words, in this degenerate case $\Lambda$ collapses to the bilinear form that is dual to the usual Hilbert transform.
Since in this degenerate case there is no $L^{p}$ estimate with $p_{0}=\infty$, the estimates on $L^{\infty} \times L^{p} \times L^{p'}$ for $\Lambda_{0,\beta_{1},1}$ cannot be uniform as $\beta_{1}\to 1$.
In Figure~\ref{fig:exponents} this corresponds to excluding $d_{02} \cup a_{0} \cup d_{01}$.

Since all other degeneracies of the bilinear Hilbert transform are equivalent up to a permutation of indices, we restrict ourselves to the case $\beta_{1}\to 1$.
The first uniform estimate in this regime has been proved in \cite{MR1933076}, it was a weak type estimate at the common vertex of $c$ and $a_{1}$.
Estimates in the interior of the triangle $c$ have been obtained in \cite{MR2113017} and in the interior of the triangle $a_{1}$ in \cite{MR2320411}.
In the regime $\beta_{1}\to 1$ the set of tuples of exponents for which uniform estimates hold is mirror symmetric across the vertical axis in Figure~\ref{fig:exponents}, so uniform estimates are also available in the triangle $a_{2}$, and hence in the convex hull of $a_{1}\cup a_{2}\cup c$.

For a dyadic model (explained in more detail in Section~\ref{sec:intro:dyadic}) uniform estimates in the triangle $d_{12}$ have been obtained by Oberlin and Thiele \cite{MR2997005}.
Hence for this dyadic model the question which of the known non-uniform estimates have uniform refinements is settled.
It is conjectured that the same estimates (that is, for all tuples of exponents inside the bold hexagon in Figure~\ref{fig:exponents}) should also hold in the real case.

For $\dm\geq 3$ not much is known.
On the negative side, for $\dm=3$ the estimate \eqref{eq:gen-rest-type} fails for the trilinear Hilbert transform if $p_{1}=p_{2}=p_{3}<1+\log_{6}2/(1+\log_{6}2)$ \cite{MR2449537}.
Consider now the multilinear form
\begin{equation}
\label{eq:multiplier-form}
\Lambda(f_{0},\dotsc,f_{\dm}) = \int \delta(\xi_{0}+\dotsb+\xi_{\dm}) m(\xi_{0},\dotsc,\xi_{\dm}) \widehat{f_{0}}(\xi_{0}) \dotsm \widehat{f_{\dm}}(\xi_{\dm}) \dif\xi_{0} \dotsm \dif\xi_{\dm},
\end{equation}
where $m$ is a function on the hyperplane $\Set{\xi_{0}+\dotsb+\xi_{\dm} = 0}$ that satisfies
\begin{equation}
\label{eq:Gamma-multiplier}
\abs{\partial^{\gamma} m(\xi)} \lesssim \dist(\xi,\Gamma)^{-\abs{\gamma}}
\end{equation}
for some linear subspace $\Gamma$.
Bounds for this form have been proved in \cite{MR1845098} for $\dm=2$ and $\dim\Gamma=1$ and in \cite{MR1887641} for $\dm\geq 2$ and $0 \leq \dim\Gamma < (\dm+1)/2$.
The bilinear Hilbert transforms can be recovered in the case when $\Gamma$ is in general position and the multiplier $m$ is translation invariant in the direction of $\Gamma$.
On the other hand, for $\dm\geq 4$ and generic subspaces $\Gamma$ with $\dim\Gamma=\dm-1$ there exist multipliers $m$ with \eqref{eq:Gamma-multiplier} such that the form \eqref{eq:multiplier-form} fails to be bounded on any $L^{p}$ spaces \cite{MR3294566}.
The multipliers constructed in \cite{MR3294566} are not translation invariant in the direction of $\Gamma$, so this result does not contradict the conjecture on the multilinear Hilbert transforms \eqref{eq:multilinearHT}.

Let $\xi_{0},\dotsc,\xi_{\dm}$ be real numbers such that $\sum_{j} \xi_{j} = 0$ and $\sum_{j} \beta_{j}\xi_{j} = 0$.
Then we have the modulation symmetry
\[
\Lambda_{\beta_{0},\dots,\beta_{\dm}}(f_{0},\dotsc,f_{\dm})
=
\Lambda_{\beta_{0},\dots,\beta_{\dm}}(f_{0}e(\xi_{0}\cdot),\dotsc,f_{\dm}e(\xi_{\dm}\cdot)),
\]
and this symmetry is non-trivial (in the sense that not all $\xi_{j}$ vanish) if $\dm\geq 2$.
In the case $\dm=2$ this shows that the bilinear Hilbert transform has the same modulation symmetry as the Carleson operator.
In fact, the method of proof used by Lacey and Thiele for the bilinear Hilbert transform turned out to extend \cite{MR1783613} to the Carleson operator \eqref{eq:Car-op} as a hybrid between methods of Carleson and of Fefferman.
Sj\"olin's multidimensional extension of the Carleson--Hunt theorem has also been reproved using the Lacey--Thiele approach in \cite{MR2007237} (weak type $(2,2)$) and \cite{MR2031458} ($1<p<\infty$).

The connection between these two objects, the Carleson operator on the one hand, and the bilinear Hilbert transform on the other hand, has been deepened in \cite{MR2597511}, where a certain common extension of these results has been obtained, and in \cite{MR3312633}, where the authors have introduced an abstraction of the Lacey--Thiele argument that has been subsequently applied as a black box to the Carleson operator \cite{arXiv:1610.07657}.
One of our results is a common extension of the uniform estimates for a dyadic model of the bilinear Hilbert transform and a dyadic model of the Carleson operator.
This extension is a special case of an open problem whose origin in ergodic theory we will explain next.


\section{Ergodic theory}
The interest in multilinear operators in ergodic theory stems from Furstenberg's ergodic proof of Szemer\'edi's theorem on arithmetic progressions in sets of positive upper density.
\begin{theorem}[{\cite{MR0369312}}]
Let $A\subset\N$ have \emph{positive upper density}
\[
\overline{d}(A) := \limsup_{N\to\infty} \abs{A \cap [1,N]}/N > 0.
\]
Then $A$ contains arithmetic progressions of arbitrary length $\dm$, that is, subsets of the form $\Set{x+t,\dotsc,x+\dm t}$, where $t>0$.
\end{theorem}
Furstenberg reformulated this result in terms of measure-preserving dynamical systems and found a new proof.
\begin{definition}
A \emph{measure-preserving dynamical system (mps)} $(X,\mu,T)$ consists of a standard probability space $(X,\mu)$ and a measure-preserving transformation $T:X\to X$, that is, a measurable map such that $\mu(T^{-1}(A))=\mu(A)$ for every measurable subset $A\subset X$.
\end{definition}
\begin{theorem}[{\cite{MR0498471}}]
Let $(X,\mu,T)$ be an mps and $A\subset X$ a measurable subset with $\mu(A)>0$.
Then for every $\dm\geq 1$ we have
\begin{equation}
\label{eq:Fu-recurrence}
\liminf_{N\to\infty} \int_{X} \frac1N \sum_{t=1}^{N} \one_{A}(T^{t}x) \dotsm \one_{A}(T^{\dm t}x) \dif \mu(x) > 0.
\end{equation}
\end{theorem}
Since then it has been of interest to know whether the integrand converges and in which sense.
More generally, replacing the characteristic functions by general functions, one asks whether
\begin{equation}
\label{eq:Fu-averages}
\lim_{N\to\infty} \frac1N \sum_{t=1}^{N} f_{1}(T^{t}x) \dotsm f_{\dm}(T^{\dm t}x)
\end{equation}
exists in some sense (usually $L^{p}$ or pointwise almost everywhere).

\subsection{Norm convergence}
The question of norm convergence (say, in $L^{2}$ if all functions $f_{1},\dots,f_{\dm}$ are bounded) has been solved via a fine structural analysis of measure-preserving systems.
A \emph{factor} of the measure-preserving system $(X,\mu,T)$ is a measure-preserving system $(Y,\nu,S)$ together with an equivariant measure-preserving map $\pi : X \to Y$, that is, with a commuting diagram
\begin{center}
\begin{tikzcd}
X \arrow{r}{T} \arrow{d}{\pi} & X \arrow{d}{\pi}\\
Y \arrow{r}{S} & Y
\end{tikzcd}
\end{center}
A factor is called \emph{characteristic} for the averages \eqref{eq:Fu-averages} if the limit \eqref{eq:Fu-averages} is $0$ whenever one of the functions $f_{j}$ is orthogonal to $L^{2}(Y)$ (identified with a subspace of $L^{2}(X)$).
In particular, for the purpose of studying convergence of \eqref{eq:Fu-averages} we may replace $(X,\mu,T)$ by a characteristic factor.

Assume that the measure-preserving system $(X,\mu,T)$ is \emph{ergodic}, that is, all measurable $T$-invariant subsets of $X$ have measure either $0$ or $1$ (a general measure-preserving system can be represented as a direct integral of erdogic systems).
For $\dm=1$ it is a classical fact and one of the possible formulations of von Neumann's mean ergodic theorem that the minimal characteristic factor for \eqref{eq:Fu-averages} is the \emph{invariant factor}, which for ergodic measure preserving systems consists of one point with the identity transformation.
For $\dm=2$ it is an almost equally classical fact that the minimal characteristic factor for \eqref{eq:Fu-averages} is the \emph{Kronecker factor} that is also the maximal factor $(Y,\nu,S)$ such that $Y$ is a compact commutative group, $\nu$ is the Haar measure, and $Sy=y+s$ for some $s\in Y$.

A \emph{$d$-step nilsystem} is a compact quotient $G/\Gamma$ of a $d$-step nilpotent Lie group $G$ by a discrete subgroup $\Gamma$ with the Haar measure and a measure-preserving map of the form $S(g\Gamma) = g_{0}g\Gamma$ with some $g_{0}\in G$.
The appropriate extension of the result about the Kronecker factor to the case $\dm\geq 3$ tells that there is a characteristic factor for \eqref{eq:Fu-averages} that is an inverse limit of $(\dm-1)$-step nilsystems.
This has been proved for $\dm=3$ in \cite{MR788966,MR939438,MR1827115} and for general $\dm$ in \cite{MR2150389,MR2257397}.
Here an inverse limit is taken in the usual categorical sense as the minimal (up to isomorphy) measure-preserving system $(Y,\nu,S)$ in the commutative diagram
\[
Y \to \dotsb \to G_{n}/\Gamma_{n} \to \dotsb \to G_{0}/\Gamma_{0},
\]
where each $G_{n}/\Gamma_{n}$ is a $(\dm-1)$-step nilsystem and arrows are factor maps (that can be assumed to be continuous \cite{MR2600993}).
The analysis of the averages \eqref{eq:Fu-averages} on nilsystems and their inverse limits is then relatively easy \cite{MR2122919}.

The construction of characteristic nilfactors in \cite{MR2150389} uses a version of \emph{uniformity \mbox{(semi-)}\-norms} introduced in Gowers's effective proof of Szemer\'edi's theorem \cite{MR1631259,MR1844079}.
Green and Tao became interested in transferring this construction to the integers in connection with their result that the primes contain arbitrarily long arithmetic progressions \cite{MR2415379} (the motivating special case of Erd\H{o}s's conjecture on arithmetic progressions).
Together with Ziegler they have proved the \emph{inverse theorem for Gowers uniformity norms} \cite{MR2950773} that tells that every sequence whose Gowers uniformity norm is bounded from below correlates with a \emph{nilsequence}, that is, a sequence of the form $F(g^{n}\Gamma)$, where $(G/\Gamma,g)$ is a nilsystem and $F$ is a Lipschitz function on $G/\Gamma$ (with quantitative control on some structural constants of $G,\Gamma$).

\subsection{Pointwise convergence}
Birkhoff's pointwise ergodic theorem tells that in the case $\dm=1$ the averages \eqref{eq:Fu-averages} converge pointwise almost everywhere for every $f_{1}\in L^{1}(X,\mu)$.
The easy way to see this uses the decomposition $L^{2}(X) = \operatorname{Fix} T \oplus \overline{(I-T)(L^{\infty}(X))}$ for the unitary operator $Tf=f\circ T$ (this decomposition follows e.g.\ from the spectral theorem).
Pointwise convergence of the averages \eqref{eq:Fu-averages} is easy to show for functions in the spaces $\operatorname{Fix} T$ and $(I-T)(L^{\infty}(X))$, and one can pass to the closure using boundedness of the Hardy--Littlewood maximal operator.

For $\dm>1$ convergence on a dense subclass is not easy to prove and is only known in the bilinear case $\dm=2$ (due to Bourgain \cite{MR1037434}).
A distinctive feature of Bourgain's approach is that convergence has to be quantified in order for real analysis methods to be applicable (since only local estimates can be transferred from the real line to general mps by the Calder\'on transference principle \cite{MR0227354}).
The quantitative device used by Bourgain were so-called oscillation inequalities introduced in \cite{MR950982}.
A conceptually clearer approach to Bourgain's result using the framework of Lacey and Thiele has been later given by Demeter \cite{MR2417419} (for truncations of the Hilbert kernel).

More recently a different device for quantifying convergence entered service.
Let $(a_{T})_{T}$ be a sequence of complex numbers indexed by a totally ordered set and $0<r<\infty$.
The \emph{$r$-variation seminorm} of $(a)$ is defined by
\[
V^{r}(a_{T}|T) := \sup_{T_{0}<T_{1}<\dotsb < T_{J}} \big( \sum_{j} \abs{a_{T_{j+1}}-a_{T_{j}}}^{r} \big)^{1/r}.
\]
This family of $r$-variation seminorms is monotonically decreasing in $r$.
In the limiting case $r=\infty$ we obtain the $\ell^{\infty}$ norm modulo addition of constants.
\begin{theorem}[{\cite{MR3623404}}]
Let $\CZK:\R\to\C$ be a function such that $\abs{\hat{\CZK}(\xi)} \lesssim \min (1,\abs{\xi}^{-1})$ and $\abs{\partial^{n} \hat{\CZK}(\xi)} \lesssim \min(\abs{\xi}^{-n+1},\abs{\xi}^{-n-1})$ for $n\geq 1$.
Let $1<p_{1},p_{2}\leq\infty$ and $0 < 1/q = 1/p_{1} + 1/p_{2} < 3/2$.
Then for every $r>2$ we have
\[
\norm{ V^{r}( \int f_{1}(x+y) f_{2}(x-y) t^{-1} \CZK(t^{-1}y) \dif y | t) }_{L^{q}_{x}}
\lesssim
\norm{f_{1}}_{p_{1}} \norm{f_{2}}_{p_{2}}.
\]
\end{theorem}
One can approximate the characteristic function of the interval $[0,1]$ by smooth functions $K$ as in this theorem (losing control on the exponent $r$ in the process), replace continuous averages by discrete averages, and transfer to measure-preserving systems to recover Bourgain's result.

\subsubsection{Wiener--Wintner theorems}
Applying Birkhoff's pointwise ergodic theorem to the product space $X\times \R/\Z$ with the product measure, the transformation $S(x,y) = (Tx, y+\alpha)$, and the function $g(x,y)=f(x)e(y)$ one sees that also the modulated averages
\begin{equation}
\label{eq:WW-ave}
\frac1N \sum_{n=1}^{N} e(n\alpha) f(T^{n}x)
\end{equation}
converge for almost every $x$ as $N\to\infty$.
The Wiener--Wintner theorem tells that there is a full measure subset $X'\subset X$ such that for $x\in X'$ the averages \eqref{eq:WW-ave} converge for every $\alpha\in\R/\Z$.
This result is relatively easy to prove, partly because the associated maximal operator is dominated by the usual Hardy--Littlewood maximal operator.
The corresponding result for singular integrals is that the truncated singular integrals
\begin{equation}
\label{eq:WW-sing}
\sum_{n\neq 0, \abs{n}\leq N} \frac1n e(n\alpha) f(T^{n}x)
\end{equation}
converge for almost every $x$ as $N\to\infty$.
This result is necessarily more subtle because the associated maximal operator is the Carleson operator.
Due to the lack of a natural dense subset of $L^{\infty}(X)$ on which convergence in \eqref{eq:WW-sing} would be easy to show, the estimate for the Carleson operator has to be refined in order to establish convergence.
An oscillation inequality for the Carleson operator has been established in \cite{MR2430729}.
A more precise variation norm estimate follows from the main result of \cite{MR2881301} (as explained in Appendix D of that article).

The linear phase $n\alpha$ in the Wiener--Wintner theorem can be replaced by a polynomial \cite{MR1257033,MR2246591}.
Is this also the case for its singular version?
More specifically, it seems reasonable to propose the following problem.
\begin{conjecture}
\label{conj:var-poly-Car}
Let $r<\infty$ be sufficiently large.
Then for every Schwartz function $\phi$ we have
\[
\norm[\big]{ \sup_{Q\in\calQ_{d}} V^{r}( \int e(Q(y)) f(y) \frac{\phi(2^{k}(x-y)) \dif y}{x-y} | k\in\Z)  }_{2,\infty}
\lesssim
\norm{f}_{2},
\]
where the supremum is taken over all polynomials of degree at most $d$ and for simplicity we consider only smooth dyadic truncations.
\end{conjecture}
As outlined before, this is known in the cases $d=1$ \cite{MR2881301} and $r=\infty$ \cite{arXiv:1105.4504}.
However, these two extensions of Carleson's theorem are proved using different methods: the first one is based on the Lacey--Thiele approach and the second on Fefferman's approach.
It would be useful to further improve our understanding of the relation between these approaches.

More generally, the polynomial phases $e(Q(\alpha))$ in the Wiener--Wintner theorem can be replaced by \emph{nilsequences} \cite{MR2544760,arxiv:1208.3977}.
This again suggests possible extensions of Carleson's theorem with maximal modulation by nilsequences.

\subsubsection{Return times}
Yet another refinement of the Wiener--Wintner theorem is the return times theorem.
\begin{theorem}
\label{thm:RTT}
Let $(X,\mu,T)$ be an mps and $f\in L^{\infty}(X)$.
Then there exists a full measure set $X'\subset X$ such that for every $x\in X$, for every other mps $(Y,\nu,S)$ and $g\in L^{\infty}(Y)$, for almost every $y\in Y$ the limit
\[
\lim_{N\to\infty} \frac1N \sum_{n=1}^{N} f(T^{n}x) g(S^{n}y)
\]
exists.
\end{theorem}
This result is originally due to Bourgain \cite{MR939436}, but the full original proof seems to have remained unpublished following the discovery of the short argument in \cite{MR1557098}.
To see that Theorem~\ref{thm:RTT} contains the Wiener--Wintner theorem it again suffices to consider $Y=\R/\Z$ and $S(y)=y+\alpha$.
Similarly to the Wiener--Wintner theorem, more refined results have been obtained using time-frequency analysis \cite{MR2420509,MR2653686,MR2994676}, in particular a singular integral version.

There are also multilinear versions of the return times theorem \cite{MR1489899,arxiv:1210.5202,arxiv:1506.05748}.
It would be interesting to treat some of these extensions analytically, although it seems unlikely that currently available techniques can be applied.

\subsubsection{Bilinear maximal function}
Let us now consider the real variable version of \eqref{eq:Fu-averages}.
Replacing the probability space $(X,\mu)$ by $\R$, the transformation $T$ by translation, and the integer parameter $t$ by a real parameter we arrive at the multilinear averages
\[
\frac{1}{T} \int_{0}^{T} f_{1}(x-t) \dotsm f_{\dm}(x-\dm t) \dif t.
\]
It is also natural to consider general coefficients in place of $1,\dotsc,\dm$.
A natural object related to these averages is the maximal operator
\begin{equation}
\label{eq:multilinearMO}
\sup_{T>0} \abs[\Big]{ \frac{1}{T} \int_{0}^{T} \prod_{j=1}^{\dm} f_{j}(x-\beta_{j} t) \dif t }.
\end{equation}
This differs from \eqref{eq:multilinearHT-op} in that the Hilbert kernel has been replaced by a maximal average.
In the case $\dm=1$ the operator \eqref{eq:multilinearMO} is the usual Hardy--Littlewood maximal operator, and in particular it is bounded on $L^{p}$, $1<p<\infty$.
By positivity it follows that for general $\dm$ the operator \eqref{eq:multilinearMO} is bounded on the spaces
\[
L^{p} \times L^{\infty} \times \dotsb \times L^{\infty} \to L^{p},
\quad
1<p\leq\infty,
\]
and similarly for any permutation of the spaces on the left-hand side.
By multilinear interpolation it follows that it is bounded on the spaces
\[
L^{p_{1}} \dotsb \times L^{p_{\dm}} \to L^{p},
\quad
1<p_{j}\leq\infty,
\quad
\frac{1}{p} = \sum_{j} \frac{1}{p_{j}},
\]
as long as $p>1$.
By the Calder\'on transference principle \cite{MR0227354} this result can be transferred to measure-preserving dynamical systems.

In the case $\dm=2$ this result has been improved by Lacey \cite{MR1745019} using the methods developed for the bilinear Hilbert transform.
Specifically, Lacey has shown that the restriction on $p$ can be relaxed to $p>2/3$ (this improvement propagates to higher values of $\dm$ by positivity and interpolation).

\subsection{Commuting transformations}

In order to extend Szemer\'edi's theorem on arithmetic progressions to subsets of $\N^{\dm}$, Furstenberg and Katznelson \cite{MR531279} have generalized \eqref{eq:Fu-recurrence} by showing that if $(X,\mu)$ is a standard probability space, $T_{1},\dotsc,T_{k}$ are \emph{commuting} measure-preserving transformations, and $A\subset X$ is a measurable subset with $\mu(A)>0$, then
\begin{equation}
\label{eq:FK-recurrence}
\liminf_{N\to\infty} \int_{X} \frac1N \sum_{t=1}^{N} \one_{A}(T_{1}^{t}x) \dotsm \one_{A}(T_{\dm}^{t}x) \dif x > 0.
\end{equation}
Furstenberg's previous result \eqref{eq:Fu-recurrence} is the special case $T_{j}=T^{j}$.

The study of convergence properties of the corresponding multiple ergodic averages
\begin{equation}
\label{eq:FK-averages}
\frac1N \sum_{t=1}^{N} f_{1}(T_{1}^{t}x) \dotsm f_{\dm}(T_{\dm}^{t}x)
\end{equation}
has so far proceeded by softer methods than those initially developed in the special case \eqref{eq:Fu-averages}, some work of Austin \cite{MR3358029} on certain polynomial version of this problem notwithstanding.
Norm convergence of the averages \eqref{eq:FK-averages} has been proved by Tao \cite{MR2408398}, and the most general version of this result is due to Walsh \cite{MR2912715} (see also \cite{arxiv:1111.7292}).
Walsh's argument relies on a Hilbert space version of the Szemer\'edi regularity lemma due to Gowers \cite{MR2669681} that will also be used in Chapter~\ref{chap:cancellation-simplex}.

The real variable version of the averages \eqref{eq:FK-averages} is
\[
\frac{1}{T} \int_{0}^{T} f_{1}(x-e_{1}t) \dotsm f_{\dm}(x-e_{\dm} t) \dif t,
\]
where the functions $f_{1},\dotsc,f_{\dm}$ are defined on $\R^{\dm}$ and $e_{1},\dotsc,e_{\dm}\in\R^{\dm}$ are the standard unit vectors.

This connection with ergodic theory has motivated the study of certain modulation invariant operators in several dimensions.
The first of them \cite{MR2597511} is the analogue of the bilinear Hilbert transform with the Hilbert kernel replaced by a two-dimensional \CZ{} kernel
\[
\int_{\R^{2}} f_{1}(x-A_{1}t) f_{2}(x-A_{2}t) \CZK(t) \dif t,
\]
where $x\in\R^{2}$, analogously to Sj\"olin's extension of the Carleson--Hunt theorem.
The pair of matrices $A_{1},A_{2}$ can now happen to be degenerate in several ways, not all of which allow a reduction to a linear operator.
One of these cases led to the introduction of a new type of paraproduct, called the \emph{twisted paraproduct}.
Estimates for the twisted paraproduct have been obtained in \cite{MR2990138}.
An interesting recent development using the ideas from the latter paper is the optimal quantitative version \cite{arxiv:1603.00631} of the $L^{2}$ norm convergence of the averages \eqref{eq:FK-averages} with $\dm=2$ for $f_{1},f_{2}\in L^{4}$.

\section{Simplex \CZ{} forms}
\label{sec:intro:simplex}
The most optimistic conjecture regarding multidimensional modulation invariant operators that seems to be consistent with the negative results on \eqref{eq:multiplier-form} is the following.
\begin{conjecture}
\label{conj:simplex}
Let $\V=\R^{\ds}$ and consider the $(\dm+1)$-linear form defined on functions of $\dm$ variables
\begin{equation}
\label{eq:simplex-form}
\Lambda_{K}(F_{0},\dotsc,F_{\dm})
:=
\int_{\V^{\dm+1}} \prod_{i=0}^{\dm}F_{i}(x_{(i)}) K(\sum_{i=0}^{\dm} x_{i}) \dif x,
\end{equation}
where $x_{(i)} = (x_{0},\dotsc,x_{i-1},x_{i+1},\dotsc,x_{\dm})$ denotes the omission of the $i$-th coordinate.
Suppose that $\CZK$ is a (sufficiently smooth) \CZ{} kernel.
Then
\[
\abs{ \Lambda_{K}(F_{0},\dotsc,F_{\dm}) }
\lesssim
\prod_{i=0}^{\dm} \norm{F_{i}}_{p_{i}},
\quad
\sum_{i=0}^{\dm} \frac1{p_{i}}=1,
\quad
\dm < p_{i} < \infty.
\]
\end{conjecture}
We call \eqref{eq:simplex-form} a \emph{simplex \CZ{} form}.
The eponymous simplex is spanned by the set $\{0,\dotsc,\dm\}$.
Each function $F_{i}$ is associated to a side of the simplex and accepts the variables whose indices span that side.
If the \CZ{} kernel $\CZK$ is replaced by the Dirac delta distribution, then the multilinear form \eqref{eq:simplex-form} can be interpreted as a Brascamp--Lieb form of the type studied in \cite{MR2377493}.
More general singular multilinear forms of similar flavor suggest themselves, but we concentrate on \eqref{eq:simplex-form}.

The form \eqref{eq:simplex-form} (with $\dm\geq 2$) has an even wider family of symmetries than seen before: any functions $F_{i},F_{i'}$ can be multiplied by an arbitrary function of modulus $1$ (and its complex conjugate, respectively) that depends on the $\dm-1$ variables shared by $F_{i},F_{i'}$.
Forms with such modulation invariance are called \emph{entangled}.
Already in the case $\ds=1$, $\dm=2$ Conjecture~\ref{conj:simplex} would unify some of the central, known or hypothetical, results in time-frequency analysis:
\begin{enumerate}
\item Estimates for the Carleson maximal operator.
\item Uniform estimates for the one-dimensional bilinear Hilbert transform in \cite{MR2113017}.
\item Uniform estimates for the two-dimensional version of the bilinear Hilbert transform studied in \cite{MR2597511}, at least for odd homogeneous kernels.
\end{enumerate}
In the case of higher degree of multilinearity $\dm$, in addition to obvious consequences for the multilinear Hilbert transforms \eqref{eq:multilinearHT}, Conjecture~\ref{conj:simplex} also contains estimates for the polynomial Carleson operator (translation invariant case of Theorem~\ref{thm:poly-car}).

A reason for cautious optimism regarding Conjecture~\ref{conj:simplex} is that a better estimate than that coming from taking absolute value inside the integral is available.
Let $k$ be a \CZ{} kernel on $\V$ that defines an $L^{2}(\R^{\ds})$ bounded operator and satisfies the smoothness condition \eqref{eq:Sjolin-smooth-cond}.
Consider the truncated kernels
\[
\psi_{\Scales} = \sum_{\scale\in \Scales} \psi_{\scale},
\quad\text{where}\quad
\psi_{\scale}(x)=\phi(2^{-\scale}\abs{x})k(x),
\]
$\Scales\subset\Z$ is an interval, and $\phi$ is an even, smooth function supported on $\pm[1,4]$ such that $\sum_{\scale\in\Z}\phi(2^{-\scale}t)=1$ for all $t\neq 0$.
We call the form $\Lambda_{\Scales} := \Lambda_{\psi_{\Scales}}$ a \emph{truncated simplex \CZ{} form}.

Since $\norm{\psi_{\scale}}_{1}=O(1)$ and by H\"older's inequality the estimate
\begin{equation}
\label{eq:trivial-estimate}
\abs{\Lambda_{\Scales}(F_{0},\dotsc,F_{\dm})}
\lesssim_{\dm} \abs{\Scales} \prod_{i=0}^{\dm} \norm{ F_{i} }_{p_{i}}
\end{equation}
is immediate for any Hölder tuple of exponents $1\leq p_{i}\leq\infty$.
In Chapter~\ref{chap:cancellation-simplex}, based on \cite{MR3685286}, the following qualitative improvement over this bound is proved.
\begin{theorem}
\label{thm:tiny-gain}
Let $\dm\geq 1$.
Then for any $1<p_{i}<\infty$ with $\sum_{i=0}^{\dm} p_{i}^{-1}=1$ we have
\[
\abs{\Lambda_{\Scales}(F_{0},\dotsc,F_{\dm})}
\leq
o_{\dm,p_{0},\dotsc,p_{\dm}}(\abs{\Scales}) \prod_{i=0}^{\dm} \norm{ F_{i} }_{p_{i}}.
\]
\end{theorem}
The corresponding result for the multilinear Hilbert transform has been proved by Tao \cite{MR3484017} using the inverse theorem for Gowers uniformity norms.
Our proof is similar but uses Gowers's Hilbert space regularity lemma, consistently with ergodic theoretical motivation.
Our result has been later improved in \cite{arxiv:1608.00156}, where $o(\abs{\Scales})$ is replaced by $O(\abs{\Scales}^{1-\epsilon})$ with an explicit $\epsilon=\epsilon(p_{0},\dotsc,p_{\dm})$, at least in the case $\ds=1$.

Another reason for optimism regarding Conjecture~\ref{conj:simplex} is that it holds in a dyadic model for $\dm=2$ in a particular case when one of the functions $F_0,F_1,F_2$ takes a special form, see Chapter~\ref{chap:two-general}, based on \cite{arxiv:1506.00861}.
This case still turns out to be general enough to imply (dyadic versions of) both the $L^{p}$ bounds for the Carleson operator and uniform bounds for the bilinear Hilbert transform.
In this sense our result has stronger one-dimensional consequences than the dyadic version of the argument from \cite{MR2597511} which appears in \cite{MR3334208}: the latter does not contain the uniform bounds for the bilinear Hilbert transform.
In the remaining part of Section~\ref{sec:intro:simplex} we explain the connections between Conjecture~\ref{conj:simplex} and the previously introduced objects in more detail.

\subsection{Uniform estimates}
We begin with the connection to uniform estimates made in \cite{arxiv:1506.00861} and \cite{MR3685286}.
Consider the family of multilinear forms
\begin{equation}
\label{eq:simplex-def-beta}
\Lambda_{\beta_{0},\dotsc,\beta_{\dm}}(F_{0},\dotsc,F_{\dm})
:=
\int_{\V^{\dm}} \int_{\V} \prod_{i=0}^{\dm}F_i( x-\beta_{i}t) K(t)\dif t \dif  x,
\end{equation}
where $\beta_{i}\in L(\V)^{\dm}$ are in general position.
The main observation is that
\begin{equation}
\label{eq:simplex-beta-hom}
\norm{\Lambda_{K}}_{L^{p_{0}} \times \dotsb \times L^{p_{\dm}}}
=
\abs{\det B}^{1/p_{0}+\dotsb+1/p_{\dm}-1}
\norm{\Lambda_{\beta_{0},\dotsc,\beta_{\dm}}}_{L^{p_{0}} \times \dotsm \times L^{p_{\dm}}},
\end{equation}
where
\[
B :=
\begin{pmatrix}
\Id_{\V} & \dots & \Id_{\V}\\
\beta_{0} & \dots & \beta_{\dm}
\end{pmatrix}.
\]
In particular, the norm of \eqref{eq:simplex-def-beta} does not depend on the $\beta_{i}$'s for H\"older tuples of exponents \eqref{eq:Lp-range}.

\begin{proof}[Proof of \eqref{eq:simplex-beta-hom}]
Consider the change of variables
\[
\begin{pmatrix}
t\\  u
\end{pmatrix}
=
B  x,
\quad
\quad
t\in\V,
 u\in\V^{\dm},
 x\in\V^{\dm+1}.
\]
If $\pi_i:\V^{\dm+1}\to\V^{\dm}$ denotes omission of the $i$-th coordinate, then for arbitrary functions $F_0,\dotsc,F_\dm$ we have
\begin{align*}
\Lambda_{K}(F_{0},\dotsc,F_{\dm})
&=
\int_{\V^{\dm+1}} \prod_{i=0}^{\dm}F_{i}(\pi_i( x)) K(\sum_{j=0}^{\dm} x_{j}) \dif  x\\
&=
\abs{\det B}^{-1}
\int_{\V\times\V^{\dm}}
\prod_{i=0}^{\dm} F_{i}(\pi_i B^{-1}(t, u))
K(t) \dif (t, u)\\
&=
\abs{\det B}^{-1}
\int_{\V\times\V^{\dm}}
\prod_{i=0}^{\dm} \tilde F_{i}( u-{\beta_{i}}t)
K(t) \dif (t, u)\\
&=
\abs{\det B}^{-1} \Lambda_{\beta_{0},\dotsc,\beta_{\dm}}(\tilde F_{0}, \dotsc, \tilde F_{\dm}),
\end{align*}
where
\begin{equation}\label{eq:fn-substitution}
\tilde F_{i}( u) := F_{i}(\pi_i B^{-1}(0, u)).
\end{equation}
Here we have used the fact that
\[
\pi_i B^{-1}
\begin{pmatrix}
t \\ \beta_{i}t
\end{pmatrix}
=
0.
\]
It remains to observe that
\[
\norm{\tilde F_{i}}_{p_{i}} = \abs{\det B}^{1/p_{i}} \norm{F_{i}}_{p_{i}}.
\]
Indeed, the inverse of the map $ u \mapsto \pi_{i}B^{-1}(0, u)$ is given by the operator matrix $(\beta_{j}-\beta_{i})_{j\neq i}$.
By multilinearity and antisymmetry of the determinant we have
\begin{align*}
\MoveEqLeft
\abs*{\det \begin{pmatrix}
\beta_{0}-\beta_{i} & \dots & \beta_{i-1}-\beta_{i} & \beta_{i+1}-\beta_{i} & \dots & \beta_{\dm}-\beta_{i}
\end{pmatrix}}\\
&=
\abs*{\det \begin{pmatrix}
0 & \dots & 0 & \Id_{\V} & 0 & \dots & 0\\
\beta_{0}-\beta_{i} & \dots & \beta_{i-1}-\beta_{i} & \beta_{i} & \beta_{i+1}-\beta_{i} & \dots & \beta_{\dm}-\beta_{i}
\end{pmatrix}}\\
&=
\abs*{\det \begin{pmatrix}
\Id_{\V} & \dots & \Id_{\V} & \Id_{\V} & \Id_{\V} & \dots & \Id_{\V}\\
\beta_{0} & \dots & \beta_{i-1} & \beta_{i} & \beta_{i+1} & \dots & \beta_{\dm}
\end{pmatrix}}\\
&=
\abs{\det B}.
\qedhere
\end{align*}
\end{proof}

Next we will see how to encode the multilinear Hilbert transform \eqref{eq:multilinearHT} in \eqref{eq:simplex-def-beta}.
We consider the case $\ds=1$, so that $\beta_{i} \in \R^{\dm}$.
The first component of $\beta_{i}$ will be given by the corresponding number $\beta_{i}$ from \eqref{eq:multilinearHT}, and the remaining components can be chosen freely to ensure that the matrix $B$ is invertible.
For simplicity let also $K$ be a truncated \CZ{} kernel and $f_{i}$ compactly supported smooth functions in order to ensure that all integrals converge absolutely.
Let then
\begin{equation}
\label{eq:simplex-tensor-function}
F_{i,L}(x_{1},\dots,x_{m}) := f_{i}(x_{1}) \prod_{j=2}^{\dm} D_{L}^{p_{i}}\phi(x_{j}),
\end{equation}
where $\phi$ is a smooth positive function with compact support and $D_{L}^{p}\phi(x) = L^{-1/p}\phi(x/L)$ for a large number $L$.
Then $\norm{F_{i,L}}_{p_{i}} \sim \norm{f_{i}}_{p_{i}}$ uniformly in $L$ and
\begin{multline*}
\Lambda_{(\beta_{j})_{j}}(F_{0,L},\dots,F_{\dm,L})
=
\int_{\R^{\dm}} \int_{\R} \prod_{i=0}^{\dm}F_i( x-\beta_{i}t) K(t)\dif t \dif  x\\
=
\int_{\R^{\dm}} \int_{\R} \prod_{i=0}^{\dm} \Big( f_i(x_{1}-\beta_{i}t) \prod_{j=2}^{\dm} D_{L}^{p_{i}}\phi(x_{j}-\beta_{i,j}t) \Big) K(t)\dif t \dif  x\\
=
\int_{\R\times\R} \Big(\prod_{i=0}^{\dm} f_i(x_{1}-\beta_{i}t) \Big) \cdot \prod_{j=2}^{\dm} \underbrace{\Big( L^{-1} \int_{\R} \prod_{i=0}^{\dm} \phi((x-\beta_{i,j}t)/L) \dif x \Big)}_{\to C \text{ as } L\to\infty} K(t)\dif t \dif x_{1}.
\end{multline*}
Assuming Conjecture~\ref{conj:simplex}, the left-hand side is bounded uniformly in $\beta$ and $L$, and it follows that also \eqref{eq:simplex-form} is bounded uniformly in $\beta$.

\subsection{Maximally modulated operators}
Recall that a \emph{multiindex} is a vector $\gamma\in\N^{\ds}$ and $\abs{\gamma} = \sum_{j}\gamma_{j}$.
Let $\Gamma = \Set{\gamma\in\N^{\ds} \given \abs{\gamma}\leq\degp}$ and let $N_{\gamma} : \V \to \R$, $\gamma\in\Gamma$, be measurable linearizing functions.
We will see that an appropriate choice of the functions $F_{i}$ allows us to encode in \eqref{eq:simplex-def-beta} (with $\dm$ replaced by $\dm+\degp$) the maximally $\Gamma$-polynomially modulated $\dm$-linear entangled \CZO{}
\begin{equation}
\label{eq:max-mod-MHT}
C_{N}(f_{0},\dotsc,f_{\dm-1})(x)
=
\operatorname{p.v.} \int_{\V} e(\sum_{\gamma\in\Gamma} N_{\gamma}(x)t^{\gamma}) \prod_{j=0}^{\dm-1} f_{j}(x-A_{j}t) \CZK(t) \dif t,
\end{equation}
where $A_{0},\dotsc,A_{\dm-1} \in L(\V)$ are generic linear maps.

The encoding is made possible by the following algebraic observation.
\begin{lemma}
\label{lem:entangle-monomial}
Let $t,x_{0},\dotsc,x_{\degp}$ denote $\ds$-vectors of formal variables and let $\gamma\in\N^{\ds}$ be a multiindex with $\abs{\gamma}\leq \degp$.
Then there exist polynomials with integer coefficients $p_{0,\gamma},\dotsc,p_{\degp,\gamma}$ such that
\begin{equation}
\label{eq:entangle-monomial}
t^{\gamma} = p_{0,\gamma}(x_{0},\dotsc,x_{\degp}) + p_{1,\gamma}(x_{0},x_{1}+t,x_{2}\dotsc,x_{\degp}) + \dotsb + p_{\degp,\gamma}(x_{0},\dotsc,x_{\degp-1},x_{\degp}+t).
\end{equation}
\end{lemma}
\begin{proof}
By induction on $\degp$.
In the case $\degp=0$ we have $t^{(0,\dotsc,0)}=1=:p_{0,(0,\dotsc,0)}$.
Suppose that the conclusion is known for some $\degp\in\N$ and consider a multiindex $\gamma$ with $\abs{\gamma}=\degp+1$.
Then by the binomial formula
\begin{align*}
(x_{\degp}+t)^{\gamma}
&=
t^{\gamma} + \sum_{\gamma'<\gamma} \binom{\gamma}{\gamma'} t^{\gamma'} x_{\degp}^{\gamma-\gamma'}\\
&=
t^{\gamma} + \sum_{\gamma'<\gamma} \binom{\gamma}{\gamma'} x_{\degp}^{\gamma-\gamma'} \sum_{j=0}^{\degp} p_{j,\gamma'}(\dots),
\end{align*}
where each $p_{j,\gamma'}$ takes the arguments indicated in \eqref{eq:entangle-monomial}.
Rearranging we obtain the claim.
\end{proof}

Let
\begin{enumerate}
\item $\beta_{0}=e_{0} \otimes A_{0}$,
\item $\beta_{j}=e_{0} \otimes A_{j}+\epsilon e_{j} \otimes \Id_{\V}$ for $j=1,\dotsc,\dm-1$,
\item $\beta_{\dm}=0$,
\item $\beta_{j}=e_{j-1}\otimes \Id_{\V}$ for $j=\dm+1,\dotsc,\dm+\degp$.
\end{enumerate}
Then with
\[
F_{j}(x_1,\ldots,x_\dm) = f_{j}(x_{0}),
\quad j=0,\dotsc,\dm-1,
\]
and
\[
F_j(x_0,\dotsc,x_{\dm-1},y_{1},\dotsc,y_{\degp})
=
g_j(x_{0}) \prod_{\gamma : j \leq \abs{\gamma} \leq \degp} e(N_{\gamma}(x_{0})p_{j-n,\gamma}(x_0,y_{1},\dotsc,y_{\abs{\gamma}}))
\]
for $j=\dm,\dotsc,\dm+\degp$ the form \eqref{eq:simplex-def-beta} formally becomes
\[
\int g_{\dm} \dotsm g_{\dm+\degp} C_{N_{\Gamma}}(f_{0},\dotsc,f_{\dm-1}).
\]
To be precise we should use cut-off functions as in \eqref{eq:simplex-tensor-function}.

In the case $\ds=1$, $\degp=0$, $\dm=2$, the operator \eqref{eq:max-mod-MHT} is a bilinear Hilbert transform.
The case $\ds=2$, $\degp=0$, $\dm=2$ has been considered in \cite{MR2597511}.
The case $\ds=1$, $\degp=0$, $\dm>2$ corresponds to the multilinear Hilbert transform, this is the case considered in \cite{MR3484017}.
The case $\degp>0$, $\dm=1$ is the polynomial Carleson operator from Theorem~\ref{thm:poly-car}.

\subsection{Two-dimensional analog of the bilinear Hilbert transform}
The trilinear forms introduced in \cite{MR2597511} can be written as
\begin{equation}
\label{eq:DT-2D-BHT}
\Lambda^{K}_{B_{0},B_{1},B_{2}}(F_{0},F_{1},F_{2})
:=
\int_{\R^{2}} \mathrm{p.v.}\int_{\R^2} \prod_{i=0}^{2} F_i\big(\vec{x}-B_{i}\vec{t}\big) K(\vec{t}) \dif\vec{t} \dif\vec{x},
\end{equation}
where $B_0,B_1,B_2$ are now $2\times 2$ real matrices (interpreted as linear operators on $\R^2$) and $K$ is a two-dimensional Calder\'on--Zygmund kernel.
If $K$ is odd and homogeneous of degree $-2$, then it takes the form
\[
K(r v) = \frac{\Omega(v)}{r^2},\ \ \Omega(-v)=-\Omega(v),\ \text{for}\ 0<r<\infty,\ v\in\mathbb{S}^{1}.
\]
Observe that
\begin{align*}
& \mathrm{p.v.}\int_{\R^2} \prod_{i=0}^{2} F_i\big(\vec{x}-B_{i}\vec{t}\big) K(\vec{t}) \dif\vec{t}\\
& = \mathrm{p.v.}\int_{\mathbb{S}^{1}} \int_{0}^{\infty} \prod_{i=0}^{2} F_i\big(\vec{x}-B_{i}(r v)\big) \frac{\Omega(v)}{r^2} r \dif r \dif v\\
& = \int_{\mathbb{S}^{1'}} \Omega(v) \,\mathrm{p.v.}\int_{\R} \prod_{i=0}^{2} F_i\big(\vec{x}-r B_{i} v\big) \frac{\dif r}{r} \dif v,
\end{align*}
where $\mathbb{S}^{1'}$ is the upper half of the unit circle, so that
\[
\Lambda^{K}_{B_{0},B_{1},B_{2}} = \int_{\mathbb{S}^{1'}} \Omega(v) \,\Lambda_{B_{0} v, B_{1} v, B_{2} v} \dif v,
\]
i.e.\ $\Lambda^{K}_{B_{0},B_{1},B_{2}}$ is a superposition of the forms \eqref{eq:simplex-def-beta}.
Consequently, $L^p$ estimates for all cases of the matrices studied in \cite{MR2597511} and the remaining case from \cite{MR2990138} would follow from Conjecture~\ref{conj:simplex} with $\ds=1$, $\dm=2$, $\CZK(t)=1/t$, even uniformly over all choices of $B_0,B_1,B_2$.

As the author has learned from Micha\l{} Warchalski, the opposite implication also holds.
Let $\tilde K$ be a symmetric truncation of the Hilbert kernel and
\[
K(t,s)
=
\begin{cases}
\tilde K(t) \abs{t}^{-1} \phi(s/t), & t\neq 0,\\
0, & t=0, s\neq 0,
\end{cases}
\]
where $\phi$ is a smooth positive compactly supported function.
Then $K$ is a truncated odd homogeneous \CZ{} kernel on $\R^{2}$.
Moreover, as the second columns of the matrices $B_{i}$ converge to $0$, the form \eqref{eq:DT-2D-BHT} converges to a constant times the form \eqref{eq:simplex-def-beta} with $\dm=2$, $\ds=1$, kernel $\tilde K$, and $\vec\beta_{i}$ being the first column of $B_{i}$, at least if the functions $F_{i}$ are smooth and compactly supported.

\section{Dyadic models}
\label{sec:intro:dyadic}
Many of the problems discussed so far have been also studied in a discrete setting.
In this setting the real line is replaced by the (Walsh) field $\W=\mathbb{F}_{2}((1/t))$ of one-sidedly infinite power series with coefficients in the two-element field $\mathbb{F}_{2}$.
The field $\W$ is traditionally identified with $[0,\infty)$ via the map $\sum_{k}a_{k}t^{k}\mapsto \sum_{k}a_{k}2^{-k}$, where $\mathbb{F}_{2}$ is identified with $\{0,1\}$.
This map is one-to-one on a conull set, and we normalize the Haar measure on $\W$ in such a way that this map becomes measure-preserving.
Under this identification the addition $\wplus$ and the multiplication $\wtimes$ on $W$ correspond to addition and multiplication of binary numbers without carrying over digits.

Similarly to the situation in $\R$, the locally compact commutative group $\W$ can be identified with its Pontryagin dual by associating to $w\in\W$ the character $x\mapsto e(w\wtimes x)$, where $e$ is the standard character $e(\sum_{k} a_{k} t^{k}) = (-1)^{a_{0}}$.
The difference from the real case is that wave packets with compact support both in space and in frequency are available on $\W$.
Indeed, under the identification with $[0,\infty)$ one can use the Haar wavelets.

This provides a rigorous framework for ignoring tails that allows to develop combinatorial ideas in a simplified setting.
While dyadic models have been initially used mostly for expository purposes, many of the more recent results have been first developed in a dyadic setting as a step towards the desired result in the real case.
We list some problems in and around time-frequency analysis for which dyadic models have been considered in the literature.
\begin{enumerate}
\item Carleson operator \cite{MR0217510,MR2692998,MR1695038,MR3334208}
\item Maximally truncated bilinear Hilbert transform \cite{MR1846092}
\item Uniform estimates for the bilinear Hilbert transform \cite{MR1924689,MR2997005}
\item Bi-Carleson operator \cite{MR2127984}
\item Twisted paraproduct \cite{MR2990138}
\item Entangled $T(1)$ theorem \cite{MR3275738}
\item Pointwise convergence of bilinear ergodic averages \cite{MR3224563}
\item Norm convergence of bilinear ergodic averages \cite{MR3480347}
\item Return times theorem \cite{MR2490160,MR3090139}
\item Multilinear $T(b)$ theorem \cite{MR3653076}
\item Special case of the triangular Hilbert transform \cite{arxiv:1506.00861}
\end{enumerate}
Let us now state the latter result (that is proved in Chapter~\ref{chap:two-general})

Under the identification of $\W$ with $[0,\infty)$ the sets
\[
A_{k} = [0,2^{k}),
\quad k\in\Z
\]
become additive subgroups and their cosets are simply \emph{dyadic intervals} of length $2^k$, the collection of which will be denoted by $\mathbf{I}_{k}$.
Some dyadic intervals (typically denoted by Latin letters, such as $I$) will be interpreted as time intervals and they will always be subsets of the unit interval $[0,1)$.
Other dyadic intervals will be interpreted as frequency intervals (typically denoted by Greek letters, such as $\omega$) and they will have integer endpoints.
For a dyadic interval $I$ we write $I^1$ for its left half and $I^{-1}$ for its right half.
The unique dyadic parent of $I$ will be denoted $\parent I$.
When we mention a \emph{dyadic square} we will always mean a dyadic square contained in $[0,1)^2$.

We work with real-valued functions, which is no restriction since all systemic functions under consideration, most notably the Haar functions, are real valued.
Let us then reserve the letter $i$ to denote an index $i\in \{0,1,2\}$.
It is convenient to regard $i$ as an element of $\Z/3\Z$ and interpret $i+1$ and $i-1$ correspondingly.
We shall also consider the set $\DI_{k}$ of all triples $\vec I = (I_0,I_1,I_2)$ of dyadic intervals contained in $[0,1)$ such that
\[
\abs{I_0}=\abs{I_1}=\abs{I_2}=2^{k}, \quad 0\in I_0\wplus I_1 \wplus I_2
\]
and the set $\DI = \cup_{k\leq 0} \DI_{k}$.
We write
\[
(I_0,I_1,I_2)\subset (J_0,J_1,J_2)
\]
if $I_{i}\subset J_{i}$ for $i=0,1,2$.

Any function $F$ on the unit square shall be interpreted as the integral operator
\[
(F\varphi)(x):=\int_0^1 F(x,y)\varphi(y)\, \dif y
\]
on $L^2([0,1))$, denoted by the same letter.
For any dyadic interval $I$ we normalize the \emph{Haar function} $\h_{I}$ in $L^\infty$, so that $\h_{I} = \sum_{j\in\{\pm 1\}} j 1_{I^{j}}$.
We shall also write $\h_I$ for the spatial multiplier operator acting on $L^2([0,1))$ and defined by
\[
(\h_I \varphi)(x):=\h_I(x)\varphi(x) .
\]

The \emph{dyadic triangular Hilbert transform} can be written as
\begin{equation}
\label{eq:THT-def}
\Lambda^{\epsilon}(F_0,F_1,F_2)
:=
\sum_{\vec I \in \DI} \epsilon_{\vec I}
\abs{I_i}^{-1}\tr(\h_{I_i} F_{i-1}  \h_{I_{i+1}} F_i \h_{I_{i-1}} F_{i+1}),
\end{equation}
where $(\epsilon_{\vec I})_{\vec I\in\DI}$ is an arbitrary sequence of scalars bounded in magnitude by $1$ and $i\in\{0,1,2\}$ is a fixed index.
The expression does not depend on the specific choice of $i$ by cyclicity of the trace.
An explicit integral representation of \eqref{eq:THT-def} is
\begin{align}
& \Lambda^{\epsilon}(F_0,F_1,F_2) \label{eq:THT-def2} \\
& = \sum_{\vec I \in \DI} \epsilon_{\vec I}\, \abs{I_0}^{-1}
\iiint_{\W^{3}} \h_{I_1}(x)F_0(x,y)\h_{I_2}(y)F_1(y,z)\h_{I_0}(z)F_2(z,x) \dif x \dif y \dif z. \nonumber
\end{align}
We note that \eqref{eq:THT-def2} is a perfect Calder\'on--Zygmund kernel analogue of \eqref{eq:simplex-form} in the case $\dm=2$, i.e.
\[
\sum_{\vec I \in \DI} \epsilon_{\vec I} \abs{I_0}^{-1} \h_{I_1}(x)\h_{I_2}(y)\h_{I_0}(z)
\]
replaces $1/(x+y+z)$.
It is necessary to insert the coefficients $\epsilon_{\vec I}$, as otherwise the above kernel would telescope to the Dirac mass $\delta_0$ evaluated at $x+y+z$, and the form would become the integral of a pointwise product of $F_{0},F_{1},F_{2}$, which is bounded by H\"older's inequality.
Informally speaking, the Walsh model cannot distinguish between $\mathrm{p.v.}\frac{1}{t}$ and $\delta_0(t)$, so it becomes faithful only after breaking the form into scales.
We obtain the following strong type estimates.
\begin{theorem}
\label{thm:2general}
Let $F_{0},F_{1},F_{2} : \W\to\R$ be functions supported on $A_{0}^{2}$.
Suppose that either
\begin{equation}
\label{eq:F0-diagonal}
F_{0}(x_{1},x_{2})=f\big(x_{2}\wplus (a\wtimes x_{1})\big)
\text{ for all }x_{1},x_{2}\in A_{0}
\end{equation}
holds with some $a\in \W\setminus A_{0}$ and some measurable $f:\W\to\R$ or
\begin{equation}
\label{eq:F0-fiberwise-character}
F_{0}(x_{1},x_{2})=f(x_{2}) e(N_{x_{2}} \wtimes x_{1})
\text{ for all }x_{1},x_{2}\in A_{0}
\end{equation}
holds with some measurable $N : \W\to\W$ and $f:\W\to\R$.
Then
\begin{equation}
\label{eq:lp-estimate}
\abs{\Lambda^{\epsilon}(F_0,F_1,F_2)}
\lesssim
\norm{F_{0}}_{p_{0}} \norm{F_{1}}_{p_{1}} \norm{F_{2}}_{p_{2}}
\end{equation}
for any $1<p_{2}<\infty$ and $2<p_{0},p_{1}<\infty$ with \eqref{eq:Lp-range}.
The implicit constant does not depend on $a$, $N$, or the scalars $\abs{\epsilon_{\vec I}}\leq 1$ with $\epsilon_{\vec I}=0$ whenever some $I_{i} \not\subseteq A_{0}$.
In case \eqref{eq:F0-fiberwise-character} we can relax the restriction on $p_{0}$ to $1<p_{0}<\infty$.
In case \eqref{eq:F0-diagonal}, $a\in A_{1}\setminus A_{0}$, we can relax the restrictions on both $p_{0}$ and $p_{1}$ to $1<p_{0},p_{1}<\infty$.
\end{theorem}

The cases \eqref{eq:F0-diagonal} and \eqref{eq:F0-fiberwise-character} are treated in a unified way and cover all types of functions used to recover algebraically defined dyadic models for the Carleson operator and uniform estimates for the bilinear Hilbert transform, see Section~\ref{appendix:dyadic}.
However, note that already one type, namely the case \eqref{eq:F0-diagonal}, $a\in A_{1}\setminus A_{0}$, suffices to recover the bounds for both these operators.
In particular we recover the full range of uniform estimates for the dyadic bilinear Hilbert transform.

\section{Weighted and vector-valued estimates}
\label{sec:weighted}
A \emph{weight} $w$ is a non-negative measurable function on $\R^{\ds}$.
We will write $w(A) = \int_{A} w$ for measurable subsets $A \subset \R^{\ds}$.

Vector-valued and weighted estimates for the Hardy--Littlewood maximal operator $M$ on $\R^{\ds}$ have been introduced by Fefferman and Stein \cite{MR0284802}.
Their weighted estimate reads
\[
w(\Set{ x \given Mf(x) > \lambda }) \lesssim \lambda^{-1} \int \abs{f} Mw,
\]
and it has been used to prove the vector-valued inequalities
\[
\norm{ (\sum_{k} \abs{Mf_{k}}^{q} )^{1/q} }_{L^{p}}
\lesssim
\norm{ (\sum_{k} \abs{f_{k}}^{q} )^{1/q} }_{L^{p}},
\quad 1<p,q<\infty.
\]

The weights $w$ for which the Hardy--Littlewood maximal function is bounded on the weighted space $L^{p}(w)$ for a given $1<p<\infty$ have been characterized by Muckenhoupt \cite{MR0293384} as those for which the characteristic
\[
[w]_{A_{p}} = \sup_{B \text{ ball}} \abs{B}^{-p} w(B) (w^{1-p'}(B))^{p-1}
\]
is finite.
The class of such weights is called the Muckenhoupt $A_{p}$ class.

The most stunning feature of $A_{p}$ weights is the Rubio de Francia extrapolation theorem \cite{MR745140}.
In its quantitative form \cite{MR2754896} it tells that if for a pair of functions $f,g$ the estimate
\begin{equation}
\label{eq:Ap-weghted-bound}
\norm{g}_{L^{p}(w)} \leq C_{p}([w]_{A_{p}}) \norm{f}_{L^{p}(w)}
\end{equation}
holds for some $p=p_{0}$, $1<p_{0}<\infty$, and all weights $w$, then the same estimate holds for all $1<p<\infty$ with a function $C_{p}$ that depends only on $C_{p_{0}}$.

Substituting $g=Tf$ with any operator $T$ it follows that $A_{p_{0}}$-weighted $L^{p_{0}}(w)$ boundedness of $T$ for some $1<p_{0}<\infty$ implies $A_{p}$-weighted $L^{p}(w)$ boundedness of $T$ for all $1<p<\infty$.
More in general, suppose that a sequence of operators $T_{k}$ satisfies
\[
\norm{T_{k} f}_{L^{q}(w)} \leq C_{q}([w]_{A_{q}}) \norm{f}_{L^{q}(w)}
\]
uniformly in $k$ for some $1<q<\infty$.
Then for every $1<p<\infty$ we have the vector-valued inequality
\[
\norm{ (\sum_{k} \abs{T_{k} f_{k}}^{q} )^{1/q} }_{L^{p}(w)}
\leq C_{p}([w]_{A_{p}})
\norm{ (\sum_{k} \abs{f_{k}}^{q} )^{1/q} }_{L^{p}(w)}.
\]
In order to see this observe that the functions inside $L^{p}$ norms satisfy \eqref{eq:Ap-weghted-bound} with $p$ replaced by $q$.

A very effective way to summarize localization properties of operators has been introduced by Lerner.
Let $\calD$ denote a dyadic grid (e.g.\ the collection of all standard dyadic cubes).
A subcollection $\calS\subset\calD$ is called \emph{sparse} if it satisfies the Carleson condition $\sum_{Q\in\calS : Q \subset Q_{0}} \abs{Q} \lesssim \abs{Q_{0}}$.
The corresponding \emph{sparse operator} (with exponent $p$) is given by
\[
A_{\calS,p} f(x) = \sum_{Q\in\calS} \one_{Q}(x) \big( \abs{Q}^{-1} \int_{Q} \abs{f}^{p} \big)^{1/p}.
\]
\CZ{} operators can be dominated by sparse operators in the sense that for every \CZ{} operator $T$ on $\R^{\ds}$ and every function $f$ there exist $3^{\ds}$ sparse collections $\calS_{i}$ such that $\abs{Tf} \lesssim \sum_{i} A_{\calS_{i},1}f$.
This result emerged in a long sequence of articles by many authors, but the final argument is very short and can be found in \cite{MR3484688} along with historical references.
In the last few years similar results have been obtained in numerous other situations most of which are not directly relevant to us.

The effectiveness of sparse domination comes from the fact that $A_{p}$ weighted estimates are easy to show for sparse operators.
In addition to being effective sparse domination is also efficient in the sense that it recovers all known optimal weighted estimates for \CZ{} operators even near the difficult weak $(1,1)$ endpoint, see \cite{MR2480568,MR3455749,arxiv:1710.07700}.

Weighted estimates for the Carleson operator have been initially proved using unweighted estimates as a black box \cite{MR0338655,MR2115460,MR3291794,arxiv:1410.6085,arxiv:1611.03808}.
Let us recall the short argument in \cite[Theorem 4.3.2]{beltran-phdthesis}.
For a sublinear operator $T$ define the non-tangentially maximally truncated operator
\[
N_{T}f(x) := \sup_{Q \ni x} \norm{T(f\chi_{\R^{\ds}\setminus 3Q})}_{L^{\infty}(Q)},
\]
where the supremum is taken over all dyadic cubes.
If $T$ is a maximally modulated \CZ{} operator, then following the proof of Cotlar's inequality and using only kernel estimates one can show
\[
N_{T}f(x) \lesssim Mf(x) + M(Tf)(x) + \norm{T}_{p\to p} M_{p}f(x),
\]
where $M$ is the Hardy--Littlewood maximal function and $M_{p}f = (M(\abs{f}^{p}))^{1/p}$.
In particular, it follows that for the Carleson operator \eqref{eq:Car-op} the operator $N_{C}$ has weak type $(p,p)$ for any $1<p<\infty$.
A general stopping type argument \cite{MR3484688} then allows to find $3$ sparse collections $\calS_{i}$ (in different dyadic grids) such that
\[
Cf \lesssim \sum_{i} A_{\calS_{i},p} f
\]
pointwise almost everywhere.
Applying this reasoning to the polynomial Carleson operator \eqref{eq:poly-CS} we obtain the following result.
\begin{corollary}
\label{cor:poly-car-sparse}
Under the hypotheses of Theorem~\ref{thm:poly-car} for every function $f\in L^{1}(\R^{\ds})$ and $p>1$ there exist sparse collections $\calS_{i}$, $i=1,\dotsc,3^{\ds}$ such that
\[
C_{K,d}f \lesssim_{p} \sum_{i} A_{\calS_{i},p} f.
\]
\end{corollary}

For the variational Carleson operator (with variation norm taken in the modulation parameter $\xi$, as treated in \cite{MR2881301}) it does not seem possible to use this method.
In this case a proof of weighted estimates not using unweighted bounds as a black box appeared in \cite{MR2997585}.
A conceptually clearer (in the author's opinion) approach to weighted estimates for modulation invariant operators has been developed in \cite{MR3873113,arXiv:1610.07657,MR3829751} using the language and machinery of outer measure spaces introduced in \cite{MR3312633}.
This approach seems to be a logical continuation of the localization procedure introduced in \cite[Proposition 4]{MR1689336} as a tool for proving $L^{p}$ estimates with $p<2$.
A different treatment of localized estimates for the bilinear Hilbert transform and related operators has been given by Benea and Muscalu \cite{MR3599522,MR3661402,arxiv:1707.05484}.

For the Carleson operator vector-valued estimates can be deduced as a consequence of the weighted estimates.
A hypothetical vector-valued estimate for certain related operators seems to be a promising tool to tackle a problem of independent interest concerning directional operators.
This will be discussed in the next section.

\section{Directional operators along Lipschitz vector fields}
Let $v : \R^{2}\to \mathbb{S}^{1}$ be a vector field.
Under which conditions on $v$ and $p$ can we expect that for every $f\in L^{p}(\R^{2})$ the directional analogue of Lebesgue's differentiation theorem
\[
f(x) = \lim_{R\to 0} \frac{1}{R} \int_{r=0}^{R} f(x+v(x)r) \dif r
\]
holds pointwise almost everywhere?
A related question is whether the local maximal function
\[
M_{v}f(x) = \sup_{0<R<1} \frac{1}{R} \int_{r=0}^{R} \abs{f(x+v(x)r)} \dif r
\]
is bounded on $L^{p}(\R^{2})$ for any $p<\infty$.

Some structural or regularity hypothesis on the vector field is necessary.
To see this consider a Perron tree (an arrangement of triangles used to construct Kakeya sets, see e.g.\ \cite[Section X.1]{MR1232192} for the details).
The measure of this tree can be made arbitrarily small, while the continuations of the unit length segments contained in it (with length $1.5$, say) cover a set with measure bounded below by a constant.
The vector fields sketched in the picture below then witness the unboundedness of $M_{v}$ on $L^{p}(\R^{2})$, $p<\infty$, for general vector fields $v$.
\begin{center}
\begin{tikzpicture}
\foreach \i in {-2,...,2}
{
  \draw[yslant=(0.2*\i),yshift={0.15cm*\i}] (0,0) -- (-1,0.1) -- (-1,-0.1) -- cycle;
  \draw[yslant=(0.2*\i),yshift={0.15cm*\i},->] (1,0) -- (0.5,0);
}
\end{tikzpicture}
\end{center}

The vector fields in the Perron tree example can be chosen to be H\"older continuous with any exponent strictly less than $1$ but not Lipschitz with an arbitrarily small Lipschitz constant.
This motivates Zygmund's conjecture dating back to the 1920's that $M_{v}$ is bounded (say on $L^{2}(\R^{2})$) provided that $\norm{v}_{\Lip}$ is sufficiently small.
A general result in this direction is due to Bourgain, who found a condition ensuring boundedness of $v$ that is satisfied by real analytic vector fields \cite{MR1009171}, so there is a huge gap between the known and the conjectured regularity conditions.

If the vector field takes only finitely many values, then a logarithmic bound in terms of the number of values is available for $2\leq p<\infty$ \cite{MR1681088}.

A related question of Stein is whether the directional singular integral
\[
H_{v}f(x) = \pvint_{\abs{r}\leq 1} f(x+v(x)r) \frac{\dif r}{r}
\]
is bounded on any $L^{p}(\R^{2})$ under similar assumptions on the vector field $v$.
In the real analytic case this has been shown by Stein and Street \cite{MR2784671}.

We will only consider directional singular integrals.
Suppose now that $v : \R^{2} \to \mathbb{S}^{1}$ is $1/100$-Lipschitz.
Covering the circle by small closed arcs, considering for each arc the set of points on which $v$ points into the direction of that arc, and extending $v$ from that set to $\R^{2}$ as a Lipschitz function with values in the arc, we may assume that $v$ itself takes values in a small arc.
Assuming that this small arc is close to the horizontal direction, representing $v(x,y)$ as a multiple of $(1,u(x,y))$, and multiplying the singular integral by a bounded factor at each point we reduce this way to the operator
\begin{equation}
\label{eq:Hu}
H_u f(x,y):= \pvint  f(x+r,y+u(x,y)r) k(r) \dif r,
\end{equation}
where $k$ is a one-dimensional \CZ{} kernel supported on $[-1,1]$.
Notice that the single scale operator
\[
A_u f(x,y):= \int_{-1}^{+1}  \abs{f(x+r,y+u(x,y)r)} \dif r
\]
is bounded on $L^{p}(\R^{2})$ for any $p$ because the map $(x,y)\mapsto (x+r,y+u(x,y)r)$ is bi-Lipschitz for any $-1<r<1$, so we can adjust the truncation in \eqref{eq:Hu}.

For measurable functions $u(x,y)=u(x)$ that depend only on the first variable, an $L^{2}$ estimate for \eqref{eq:Hu} is equivalent to an $L^{2}$ estimate for the maximally modulated operator associated to the kernel $k$ (this observation is attributed to Coifman and El-Kohen in \cite{MR1757580}).
Indeed, formally taking the Fourier transform in the second variable inside the integral we obtain
\begin{equation}
\label{eq:Hu-Car-Fourier}
\calF_{2} H_{u}f(x,\eta) = \pvint_{\R} \calF_{2} f(x+r,\eta) e(\eta u(x) r) k(r) \dif r.
\end{equation}
By Plancherel's identity we need an $L^{2}(\R)$ estimate for each fixed $\eta$, and a uniform estimate in all measurable $u$'s is equivalent to Carleson's theorem.

For general functions $u$ the operator $H_{u}$ is no longer diagonalized by the Fourier transform in the vertical direction.
However, it turns out that it is almost diagonalized assuming that $u$ has some regularity in the second variable.
For $u\in C^{1+\epsilon}$ this has been observed by Lacey and Li \cite{MR2654385}.
Jointly with Guo, di Plinio, and Thiele \cite{MR3841536} we have extended this observation to the endpoint case of Lipschitz functions $u$.
To this end we use the following one-dimensional result.

\begin{theorem}
\label{thm:Lip-LP-diag}
Let $A : \R\to\R$ be a Lipschitz function with $\norm{A}_{\Lip} \leq 1/100$ and consider the change of variable $T_{A}f(x):=f(x+A(x))$.

Let $\psi$ be a Schwartz function on $\R$ such that $\widehat{\psi}$ identically equals $1$ on $\pm [99/100,103/100]$ and vanishes outside $\pm [98/100,104/100]$.
Let $\Psi$ be another Schwartz function on $\R$ such that $\widehat{\Psi}$ is supported on $\pm[1,101/100]$.
Let $P_{t}f:=\psi_{t}*f$ be the Littlewood--Paley operators associated to $\psi$, where $\psi_t(x)=t^{-1}\psi(t^{-1}x)$.
Then
\[
\norm[\Big]{ \sum_{t \in 2^{\Z}} \abs{(1-P_{t})T_{A} (\Psi_{t} * f)} }_{p}
\lesssim_{p,\psi,\Psi}
\norm{A}_{\Lip} \norm{f}_{p},
\quad
1<p<\infty.
\]
\end{theorem}

If the sum over $t$ inside the $L^{p}$ norm is replaced by an $\ell^{2}$ norm, then the estimate follows from Littlewood--Paley theory and the Fefferman--Stein maximal inequality.
The main point is that we obtain additional cancellation due to the factors $(1-P_{t})$.
This cancellation is perfect if the map $A$ is linear in the sense that the left-hand side of the desired estimate vanishes in that case.
In order to locally compare $A$ to linear functions we use a version of Jones beta numbers \cite{MR1013815} (this idea goes back to Dorronsoro \cite{MR796440}).

If the Lipschitz norm of $A$ is too large, then in general $T_A$ fails to be a bijection, and the estimate of Theorem~\ref{thm:Lip-LP-diag} breaks down.

Let us denote by $*_{2}$ the convolution of a one-variable function with a two-variable function in the second variable :
\[
\Psi *_{2} f(x,y) = \int_{\R} \Psi(z) f(x,y-z).
\]
Applying Theorem~\ref{thm:Lip-LP-diag} in the second variable for each fixed $x,r$ we obtain the following result on the directional Hilbert transform in the plane.
\begin{corollary}
\label{cor:LP-diag}
Assume that $u(x,\cdot)$ has Lipschitz constant $\leq 1/100$ for almost every $x\in\R$.
With notation as in Theorem~\ref{thm:Lip-LP-diag}, we have
\[
\norm[\Big]{ \sum_{t \in 2^{\Z}} \abs{(1-P_{t})H_{u} (\Psi_{t}*_{2} f)} }_{p}
\lesssim_{p,\phi,\Psi}
\norm{f}_{p},
\quad
1<p<\infty,
\]
where $P_t$ acts in the second variable.
\end{corollary}
\begin{proof}
By Minkowski's integral inequality we obtain
\begin{align*}
\MoveEqLeft
\norm[\Big]{ \sum_{t \in 2^{\Z}} \abs{(1-P_{t})T_{u} (\Psi_{t}*_{2} f)} }_{L^{p}(\R^{2})}\\
&\leq
\int_{r=-1}^{1} \Big(\int_{\R} \norm[\Big]{ \sum_{t \in 2^{\Z}} \abs{(1-P_{t})\big(\Psi_{t} * f(x+r,\cdot+r u(, \cdot))\big)} }_{L^{p}(\R)}^{p} \dif x \Big)^{1/p} \frac{\dif r}{\abs{r}}\\
\intertext{By Theorem~\ref{thm:Lip-LP-diag} this is bounded by}
&\lesssim
\int_{r=-1}^{1} \Big(\int_{\R} \norm{ru(x,\cdot)}_{\Lip}^{p} \norm{ f(x+r,\cdot) }_{L^{p}(\R)}^{p} \dif x \Big)^{1/p} \frac{\dif r}{\abs{r}}\\
&\lesssim
\int_{r=-1}^{1} \Big(\int_{\R} \norm{ f(x+r,\cdot) }_{L^{p}(\R)}^{p} \dif x \Big)^{1/p} \dif r \lesssim
\int_{r=-1}^{1} \norm{ f }_{L^{p}(\R^{2})} \dif r \lesssim
\norm{ f }_{L^{p}(\R^{2})}.
\qedhere
\end{align*}
\end{proof}

The importance of Corollary~\ref{cor:LP-diag} is that it reduces bounds for $H_u$ to bounds for a square function.
Indeed,
\begin{align*}
\norm[\Big]{ H_{u} (\sum_{t \in 2^{\Z}} \Psi_{t}*_{2} f) }_{p}
&\leq
\norm[\Big]{ \sum_{t \in 2^{\Z}} P_{t} H_{u} (\Psi_{t}*_{2} f) }_{p}
+
\norm[\Big]{ \sum_{t \in 2^{\Z}} \abs{(1-P_{t})H_{u} (\Psi_{t}*_{2} f)} }_{p}\\
&\lesssim
\norm[\Big]{ \big( \sum_{t \in 2^{\Z}} \abs{ H_{u} (\Psi_{t}*_{2} f) }^{2} \big)^{1/2} }_{p}
+
\norm{f}_{p}
\end{align*}
for any $1<p<\infty$ by Corollary~\ref{cor:LP-diag} and Littewood--Paley square function inequality.
Since $f$ can be written as an average of functions of the form $\sum_{t \in 2^{\Z}} \Psi_{t}*_{2} f$ (with varying $\Psi$), $L^{p}$ bounds for the operator $H_{u}$ reduce to bounds for the square function
\begin{equation}
\label{eq:LP-diag-square-fct}
\big( \sum_{t \in 2^{\Z}} \abs{ H_{u} (\Psi_{t}*_{2} f) }^{2} \big)^{1/2}.
\end{equation}
The $L^2$ part of the following corollary is then immediate.

\begin{corollary}
\label{cor:LL-single-band}
Let $u:\R^{2}\to\R$ be a measurable function such that $u(x,\cdot)$ has Lipschitz constant $\leq 1/100$ for almost every $x\in\R$.
Assume further with notation as in Theorem \ref{thm:Lip-LP-diag} that
\begin{equation}\label{passumption}
\sup_{0<t<t_{0}}\norm{H_{u} (\Psi_{t} *_2 f)}_{p_{0}}
\lesssim
\norm{f}_{p_{0}}
\end{equation}
for some $1<p_{0}\leq 2$ and $t_{0}>0$.
If $p_{0}=2$, then
\begin{equation}\label{2conclusion}
\norm{H_{u} f}_2\lesssim \norm{f}_2.
\end{equation}
If $1<p_{0}<2$, then
\begin{equation}\label{pconclusion}
\norm{H_{u} f}_p\lesssim \norm{f}_p,
\quad 1 + \frac{1}{3-p_{0}} < p < \infty.
\end{equation}
\end{corollary}

The estimate \eqref{passumption} has been proved for all $2<p_{0}<\infty$ (including a weak type $(2,2)$ endpoint) by Lacey and Li \cite{MR2219012,MR2654385}.
This result comes very close to containing an estimate for the Carleson operator (that would follow from a strong type $(2,2)$ estimate for $f \mapsto H_{u}(\Psi_{t}*_{2}f)$ with $u$ being a measurable function of the first variable by the argument in \eqref{eq:Hu-Car-Fourier}), and in fact their proof is based on the Lacey--Thiele argument for the Carleson operator.

This time the underlying modulation invariance is relatively hidden.
The class of operators $H_{u}$ is invariant under shearings leaving vertical lines invariant.
The adjoint linear transformations, acting on the Fourier space, are also shearings, but now leaving the horizontal lines invariant.
Hence for each fixed vertical frequency $\eta$ the family of symmetries are precisely modulations in the horizontal direction.
This modulation invariance connects the directional singular integrals to the other problems considered in this thesis.
However, one has to consider bands of frequencies $\eta$ (the supports of $\widehat{\Psi_{t}}$), which complicates the picture slightly.

The difficulty in extending the estimate \eqref{passumption} to $p_{0}\leq 2$ lies in the maximal estimates required in the Lacey--Thiele approach to the Carleson operator in this range.
One such possible estimate has been proved by Lacey and Li \cite{MR2219012}, unfortunately only with exponent $2$.
They conjectured \cite[Conjecture 1.14]{MR2654385} that it is possible to lower the exponent in that estimate below $2$, and this would also allow to push $p_{0}$ in \eqref{passumption} below $2$.
In particular, assuming \cite[Conjecture 1.14]{MR2654385} for a Lipschitz function $u$ we obtain $L^{p}$ estimates on $H_{u}$ for all $2\leq p < \infty$, giving a positive answer to \cite[Conjecture 1.21]{MR2654385}.
Lacey and Li have shown that their conjectured maximal estimate holds for analytic vector fields, and more generally for the class of vector fields previously considered by Bourgain \cite{MR1009171}.
The author does not have a strong opinion on whether it can be proved in full generality, but it might be possible to sidestep this difficulty at least for $p_{0}=2$ using V.~Lie's refinement of Fefferman's argument that directly proves $L^{2}$ estimates for the Carleson operator.

The required maximal estimate is known in a few specific situations.
In the case that $u$ depends only on the first variable a suitable maximal estimate has been proved by Bateman \cite{MR2457435,MR3055689}.
In this case \eqref{passumption} has been proved for all $1<p_{0}<\infty$ in \cite{MR3090145}.
The conclusion \eqref{pconclusion} of Corollary~\ref{cor:LL-single-band} in this case has been proved by Bateman and Thiele \cite{MR3148061}.
Similar results have been obtained under the hypothesis that $u$ is constant on Lipschitz curves that are close to being vertical by Guo \cite{MR3393679,MR3592519}.
The proof of Corollary~\ref{cor:LL-single-band} follows the arguments in these articles in a simplified form, in particular we use \eqref{passumption} as a black box, whereas in \cite{MR3148061} elements of the proof of this estimate for one-parameter vector fields have been used.

\subsection{Radon transforms}
The subject of singular Radon transforms is vast, see e.g.\ \cite{MR1726701}, and we will only comment on a few recent works that are most directly related to our topic.
Consider the curved analog of the operator \eqref{eq:Hu} given by
\begin{equation}
\label{eq:Stein-directional-op:curved}
H^{(\alpha)}_{u} f (x) := \int_{-1}^{1} f(x+r,y+u(x,y)r^{\alpha}) \frac{\dif r}{r},
\end{equation}
where $r^\alpha$ may be interpreted either as $\abs{r}^\alpha$ or $\operatorname{sgn}(r)\abs{r}^\alpha$.
If the function $u$ is Lipschitz in the second variable with a sufficiently small Lipschitz norm, then as in the case $\alpha=1$ Theorem~\ref{thm:Lip-LP-diag} reduces $L^{p}$ estimates for $H^{(\alpha)}_{u}$ to $L^{p}$ estimates for the square function
\[
\big( \sum_{t \in 2^{\Z}} \abs{ H_{u}^{(\alpha)} (\Psi_{t}*_{2} f) }^{2} \big)^{1/2}.
\]
Unlike in the case of straight lines, estimates for this square function have been essentially obtained in \cite{MR3669936} (see \cite{MR3841536} for some additional details).

Another kind of modulation invariant singular Radon transform has been introduced in \cite{arxiv:1505.03882}.
Let $\CZK$ be a sufficiently nice \CZ{} kernel and $\calQ$ be a linear space of polynomials in $\ds$ variables.
Consider the maximally modulated operator
\begin{equation}
\label{eq:max-mod-sing-Radon}
\sup_{Q\in\calQ} \abs[\Big]{ \int_{\R^{\ds}} f(x-y,t-\abs{y}^{2}) e(Q(y)) \CZK(y) \dif y },
\end{equation}
Substituting various spaces $\calQ$ and functions of the form $f(x,t) = \tilde f(x) \phi(t)$ one can encode in \eqref{eq:max-mod-sing-Radon} various maximally modulated integrals.
For instance, with $\calQ$ being the space of all poynomials of a given degree we obtain the polynomial Carleson operator from Theorem~\ref{thm:poly-car}.

The authors of \cite{arxiv:1505.03882} considered $\ds\geq 2$ and $\calQ = \operatorname{span}(p_{2},\dotsc,p_{d})$, where each $p_{j}$ is a real polynomial in $\ds$ variables homogeneous of degree $j$ and $p_{2}(y) \not\equiv C \abs{y}^{2}$.
The exclusion of linear terms avoids modulation invariance in the $x$ variable, while the exclusion of multiples of $\abs{y}^{2}$ avoids modulation invariance in the $t$ variable.
In this setting \cite{arxiv:1505.03882} laboriously combines the $TT^{*}$ method of \cite{MR1879821} with the smoothing methods for singular Radon transforms that go back to \cite{MR508453}.

In the case $\ds=1$ a few partial results on \eqref{eq:max-mod-sing-Radon} have been obtained in \cite{MR3708001} (also replacing the parabola $(y,\abs{y}^{2})$ by other monomial curves $(y,\abs{y}^{m})$, although we focus on $m=2$) with one-dimensional spaces $\calQ = \Set{ c y^{d} \given c \in \R }$.
Similarly to \cite{MR3090145,MR3148061} the linearizing functions for the supremum in \eqref{eq:max-mod-sing-Radon} are allowed to depend only on one variable.
In the case that the linearizing function depends on $x$, $L^{p}$ estimates have been proved for $d\geq 3$, while if the linearizing function depends on $t$ estimates have been proved for $d\geq 2$.
The latter case contains a modulation invariant case $d=2$ (the invariance is under modulations of $f$ by linear phases in the direction $t$), and in this case the Carleson theorem has been used as a black box.

\section{Notation}
The characteristic function of a set $I$ is denoted by $\one_{I}$.

The letter $C$ denotes an unspecified positive constant that can change from line to line.
The constant $C$ typically does not depend on functions $f,g$ but may depend on \CZ{} kernels, exponents $p$, and so on.
We write $C_{\beta}$ if we want to emphasize the dependence on a specific parameter $\beta$.
We write $A\lesssim B$ if $A\leq CB$ and $A\lesssim_{\beta}B$ if $A\leq C_{\beta}B$.



%% file: cancellation-simplex.tex
\chapter{Cancellation for simplex CZ forms}
\label{chap:cancellation-simplex}
In this chapter we prove Theorem~\ref{thm:tiny-gain} that is restated below for convenience.
\begin{theorem*}
Let $\dm\geq 1$.
Then for any $1<p_{i}<\infty$ with $\sum_{i=0}^{\dm} p_{i}^{-1}=1$ we have
\[
\abs{\Lambda_{\Scales}(F_{0},\dotsc,F_{\dm})}
\leq
o_{\dm,p_{0},\dotsc,p_{\dm}}(\abs{\Scales}) \prod_{i=0}^{\dm} \norm{ F_{i} }_{p_{i}}.
\]
\end{theorem*}
The proof is by induction on $\dm$.
The case $\dm=1$ follows from the standard theory of truncated Calder\'on--Zygmund operators, see e.g.\ \cite[\textsection I.7]{MR1232192}.
In the inductive step we assume that the theorem holds with $\dm>1$ replaced by $\dm-1$.
Multilinear interpolation with the trivial estimate \eqref{eq:trivial-estimate} shows that it suffices to consider $p_{0},\dotsc,p_{\dm}>\dm$ and indicator functions $F_{i}=1_{E_{i}}$.
We make these assumptions throughout Section~\ref{sec:tree}, which contains a single tree estimate, and Section~\ref{sec:selection}, which describes a tree selection algorithm.

\section{The regularity lemma}
\label{sec:regularity}
The material in this section is almost identical to Gowers's original exposition in \cite{MR2669681}.
The only difference from the finite-dimensional case is that it turns out convenient to work with \emph{extended seminorms}, that is, functions $\norm{\cdot}$ on a vector space $H$ taking values in the extended positive reals $[0,+\infty]$ that are subadditive, homogeneous, and map $0$ to $0$ (this observation has peen previously used to further streamline \cite{arxiv:1111.7292} Walsh's proof of the multilinear mean ergodic theorem \cite{MR2912715}).
The reason is that the atomic seminorms $\norm{\cdot}_{\Sigma}$, defined below, are typically extended.
\begin{lemma}
\label{lem:Sigma-ext-seminorm}
Let $H$ be a Hilbert space and $\Sigma\subset H$.
Then the formula
\[
\norm{f}_{\Sigma} := \inf\Big\{ \sum_{t}\abs{\lambda_{t}} : f=\sum_{t}\lambda_{t}\sigma_{t}, \sigma_{t} \in \Sigma \Big\},
\]
where sums are finite (possibly empty), and the infimum of an empty set is by convention $+\infty$, defines an extended seminorm on $H$ whose dual extended seminorm is given by
\[
\norm{f}_{\Sigma}^{*} := \sup_{\phi\in H: \norm{\phi}_{\Sigma}\leq 1} \abs{\<f,\phi\>} = \sup_{\sigma\in\Sigma}\abs{\<f,\sigma\>}.
\]
\end{lemma}

Gowers's Hilbert space regularity lemma reads as follows.
\begin{theorem}
\label{thm:structure}
Let $\delta>0$ and $\eta \colon\R_+\to\R_+$ be any function.
Let $H$ be a Hilbert space with norm $\norm{\cdot}_{H}$ and let $\norm{\cdot}$ be an arbitrary further extended seminorm on $H$.
Then for every $f\in H$ with $\norm{f}_{H} \leq 1$ there exists $C=O_{\delta,\eta}(1)$ and a decomposition
\begin{equation}
\label{eq:decomposition}
f = \sigma + u + v
\end{equation}
such that
\begin{equation}
\norm{\sigma} < C, \quad
\norm{u}^* < \eta(C), \quad\text{and}\quad
\norm{v}_{H} < \delta.
\end{equation}
\end{theorem}
The proof uses the following separation lemma.
\begin{lemma}
\label{lem:sep}
Let $V_{i}$, $i=1,\dotsc,k$, be convex subsets of a Hilbert space $H$, at least one of which is open, and each of which contains $0$.
Let $V:=c_{1}V_{1}+\dotsb+c_{k}V_{k}$ with $c_{i}>0$ and take $f\not\in V$.
Then there exists a vector $\phi\in H$ such that $\<f,\phi\> \geq 1$ and $\Re\<v,\phi\> < c_{i}^{-1}$ for every $v\in V_{i}$ and every $i$.
\end{lemma}
\begin{proof}
By the assumption the set $V$ is open, convex and does not contain $f$.
By the Hahn--Banach theorem there exists a $\phi\in H$ such that $\<f,\phi\> \geq 1$ and $\Re\<v,\phi\> < 1$ for every $v\in V$.
The claim follows.
\end{proof}
There is also a constructive version of Lemma~\ref{lem:sep} with an $\epsilon$ loss, in the sense that the conclusion changes to $\Re\<v,\phi\> < (1+\epsilon)c_{i}^{-1}$ (this version still suffices for our purpose).
Indeed, since $V\ni 0$ is open and $f\not\in V$, we have $f\not\in(1+\epsilon)^{-1}\overline{V}$.
Let $g\in(1+\epsilon)^{-1}\overline{V}$ be the element that minimizes the distance from $f$ (such $g$ is unique).
One can then take $\phi = (f-g)/\<f-g,f\>$.

\begin{proof}[Proof of Theorem~\ref{thm:structure}]
Let $r$ be chosen later (depending only on $\delta$) and define
\begin{equation}
\label{eq:C}
C_{r} = 1,
\qquad
C_{i-1} = \max \Big\{ C_{i}, \frac{2}{\eta(C_{i})} \Big\}.
\end{equation}
Let $V_{1},V_{2},V_{3}$ be the open unit balls of $\norm{\cdot}$, $\norm{\cdot}^{*}$, and $\norm{\cdot}_{H}$, respectively.
Suppose that the conclusion fails, then for every $i \in \{1,\dotsc,r\}$ we have
\[
f \not\in C_{i}V_{1} + \eta(C_{i})V_{2} + \delta V_{3}.
\]
Since $V_{3}$ is open in $H$, Lemma~\ref{lem:sep} applies, and we obtain vectors $\phi_{i} \in H$ such that
\[
\<\phi_{i},f\> \geq 1,
\quad \norm{\phi_{i}}^{*} \leq (C_{i})^{-1},
\quad \norm{\phi_{i}}^{**} \leq \eta(C_{i})^{-1},
\quad \norm{\phi_{i}} \leq \delta^{-1}.
\]
For every pair $i<j$ by \eqref{eq:C} we have
\[
\abs{\<\phi_{i},\phi_{j}\>}
\leq \norm{\phi_{i}}^{*} \norm{\phi_{j}}^{**}\\
\leq (C_{i})^{-1} \eta(C_{j})^{-1}
\leq (C_{j-1})^{-1} \eta(C_{j})^{-1}
\leq \frac12,
\]
so that
\[
r^{2} \leq \<\phi_{1}+\dotsb+\phi_{r},f\>^{2}
\leq \norm{\phi_{1}+\dotsb+\phi_{r}}^{2}
\leq r \delta^{-2} + \frac{r^{2}-r}{2},
\]
which is a contradiction if $r \geq 2 \delta^{-2}$.
\end{proof}

\section{The tree estimate}
\label{sec:tree}
For each $\scale\in\Z$ let $\DI_{\scale}$ be the collection of the dyadic cubes $I\subset\V^{\dm+1}$ of the form
\[
I=I_{0}\times\dotsm\times I_{\dm}=2^{\scale}(m_{0},\dotsc,m_{\dm}) + [0,2^{\scale}]^{\ds \times \dm},
\qquad
\sum_{i}m_{i}=0,
\quad
m_{i}\in\Z^{\ds}.
\]
The \emph{scale} of a dyadic cube $I\in\DI_{\scale}$ is defined as $s(I):=\scale$.
Let also $\DI_{\Scales} := \cup_{\scale\in \Scales} \DI_{\scale}$ and $\DI:=\DI_{\Z}$.
This gives the splitting
\begin{equation}
\label{eq:tile-decomposition}
\Lambda_{\Scales}(F_{0},\dotsc,F_{\dm})
=
\sum_{I\in\DI_{\Scales}} \Lambda_{I}(F_{0},\dotsc,F_{\dm}),
\end{equation}
where for each $I\in\DI_{\scale}$ we have set
\[
\Lambda_{I}(F_{0},\dotsc,F_{\dm})
:=
\int_{\V^{\dm+1}} \prod_{i=0}^{\dm}F_{i}(x_{(i)}) \psi_{\scale}(\sum x) \prod_{i=0}^{\dm-1} 1_{I_{i}}(x_{i}) \dif x.
\]
Contrary to what could be expected, our argument would not benefit from using smoother versions of the cutoffs $1_{I_{i}}$.
However, this appears to be a limitation rather than a strength of our approach.

We write elements of $\V^{\dm+1}$ as $x=(x',x_{\dm}) \in \V^{\dm}\times\V$ and dyadic cubes $I\in\DI$ as $I'\times I_{\dm}$, where $I'$ is a dyadic cube in $\V^{\dm}$ and $I_{\dm}$ is a dyadic cube in $\V$.
A \emph{tree} with top $J\in\DI$ is a collection of boxes $I\in\DI$ such that $I'\subset J'$.
In this section we obtain a gain over the trivial bound (coming from Fubini's theorem) for the restriction of the sum \eqref{eq:tile-decomposition} to a tree.
\begin{proposition}
\label{prop:single-tree}
For every $\delta>0$ there exists $S_{\delta,\dm}\in\N$ such that
for any functions $F_{0},\dotsc,F_{\dm}:\V^{\dm}\to [0,1]$ and for every dyadic cube $J\in\DI$ there exists an interval of scales $\Scales'\subset\Z$ with $\abs{\Scales'} \leq S_{\delta,\dm}$ and $\max \Scales'=s(J)$ such that
\[
\abs[\big]{ \sum_{\scale\in \Scales'} \sum_{I\in\DI_{\scale}: I'\subset J'} \Lambda_{I}(F_{0},\dotsc,F_{\dm}) }
\lesssim
\delta 2^{\dm s(J) \ds} \abs{\Scales'}.
\]
\end{proposition}
Note that $\Scales'$ depends both on the (bounded) functions $F_{i}$ and the dyadic square $J$, but $S_{\delta,\dm}$ does not.

\begin{proof}[Proof of Proposition~\ref{prop:single-tree}]
By scaling we may assume $s(J)=0$.
Note that
\begin{equation}
\label{eq:tree-single-scale}
\sum_{I\in\DI_{\scale} : I'\subset J'}\Lambda_{I}(F_{0},\dotsc,F_{\dm})
=
\int_{\V^{\dm+1}} \prod_{i=0}^{\dm} F_{i}(x_{(i)}) \psi_{\scale}(\sum x) \prod_{i=0}^{\dm-1} 1_{J_{i}}(x_{i}) \dif x
\end{equation}
for every $\scale$ and the integrand is supported on $10J$, say.

A \emph{dual function} is a function from $X := \prod_{i=0}^{\dm-1} 10J_{i}$ to $\C$ of the form
\[
x\mapsto\prod_{A\subsetneq\{0,\dotsc,\dm-1\}}f_{A}(x\abs{_{A}),
\]
where $f_{A} : \prod_{i\in A} 10J_{i}\to\C$ are functions bounded by $1$.
Denote the set of dual functions by $\Sigma$ and apply Theorem~\ref{thm:structure} with $H=L^{2}(X)$, $f=F_{\dm}}_{X}$, the extended seminorm given by Lemma~\ref{lem:Sigma-ext-seminorm} and a function $\eta$ to be chosen later.

To dispose of the $L^{2}$ error term note that at each scale $\scale\leq 0$ the right-hand side of \eqref{eq:tree-single-scale} is bounded by
\[
\int_{10 J} \abs{F_{\dm}(x_{(\dm)})} \abs{\psi_{\scale}(\sum x)} \dif\vec x\\
\lesssim
\norm{F_{\dm}}_{L^{1}(10 J_{(\dm)})} \norm{\psi_{\scale}}_{1}
\lesssim
\norm{F_{\dm}}_{L^{2}(10 J_{(\dm)})}.
\]
The contribution of the uniform term (bounded in $\norm{\cdot}_{\Sigma}^{*}$) is estimated by
\[
\sum_{\scale\in \Scales'} \norm{\hat\psi_{\scale}}_{1} \abs[\Big]{ \int_{10 J} \prod_{i}F_{i}(x_{(i)}) e(\xi_{\scale}\sum x) \prod_{i=0}^{\dm-1} 1_{J_{i}}(x_{i}) \dif x }
\]
for some choice of frequencies $\xi_{\scale}\in\hat\V$.
Note that the derivative bounds on $K$ imply $\norm{\hat\psi_{\scale}}_{1} \lesssim 2^{-\scale \ds}$.
Inside the absolute value, the character splits into a product of functions depending on one variable each.
Since $\dm>1$, each function that depends on only one coordinate $x_{i}$ can be absorbed into one of the functions $F_{i}$, $i<\dm$.
Thus for each fixed $x_{\dm}$ the integral above is a pairing of $F_{\dm}$ with a dual function, and we obtain the estimate
\[
\sum_{\scale\in \Scales'} 2^{-\scale \ds} \norm{ F_{\dm} }_{\Sigma}^{*}
\lesssim
2^{\abs{\Scales'} \ds} \norm{F_{\dm}}_{\Sigma}^{*}.
\]

It remains to treat the structured term.
Suppose $F_{\dm}\in\Sigma$, so that
\[
F_{\dm} = \prod_{i=0}^{\dm-1} f_{i},
\]
where each function $f_{i}$ is bounded by $1$ and does not depend on the $i$-th coordinate.
Substituting this into \eqref{eq:tree-single-scale} we obtain
\[
\int_{10 J} \prod_{i=0}^{\dm-1} (F_{i}f_{i})(x_{(i)}) \psi_{\Scales'}(\sum x) \prod_{i=0}^{\dm-1}1_{J_{i}}(x_{i}) \dif x.
\]
This can be written as
\[
\int_{10 J_{\dm}}\int_{\V^{\dm}} \prod_{i=0}^{\dm-1} (1_{10 J}F_{i}1_{A_{i}} \prod_{j\neq i,\dm} 1_{J_{j}})(x'_{(i)},x_{\dm}) \psi_{\Scales'}(\sum x' + x_{\dm}) \dif x' \dif x_{\dm}.
\]
Changing variable in the inner integral and applying the inductive hypothesis (Theorem~\ref{thm:tiny-gain} with $\dm-1$ in place of $\dm$ and $p_{0}=\dotsb=p_{\dm-1}=\dm$) we can bound this by
\[
c_{\dm-1}(\abs{\Scales'}) \abs{\Scales'}
\]
with a monotonically decreasing function $c_{\dm-1}$ such that $\lim_{\abs{\Scales'}\to\infty}c_{\dm-1}(\abs{\Scales'}) = 0$.
Summing the contributions of the three terms given by Theorem~\ref{thm:structure} we obtain
\[
\delta \abs{\Scales'} + 2^{\abs{\Scales'} \ds} \eta(C) + C c_{\dm-1}(\abs{\Scales'}) \abs{\Scales'},
\]
where $C=O_{\delta,\eta}(1)$.
Choose a monotonically increasing function $\tilde S_{\delta,\dm} : \V_{+}\to\N$ such that $a c_{\dm-1}(\tilde S_{\delta,\dm}(a)) \leq \delta$ for all $a$.
Let $\eta(a) := \delta \tilde S_{\delta,\dm}(a)2^{-\tilde S_{\delta,\dm}(a) \ds}$.
Then the claim follows with $\abs{\Scales'}=\tilde S_{\delta,\dm}(C)$.
\end{proof}

\begin{corollary}
\label{cor:single-tree}
Let $\delta>0$ and $S_{\delta,\dm}$ be the number from Proposition~\ref{prop:single-tree}.
Then for every $J\in\DI$ and every interval $\Scales'\subset\Z$ with $\max \Scales'=s(J)$, we have
\[
\abs[\big]{ \sum_{\scale\in \Scales'} \sum_{I\in\DI_{\scale} : I'\subset J'} \Lambda_{I}(F_{0},\dotsc,F_{\dm}) }
\lesssim_{\dm}
2^{\dm s(J) \ds} (\min(\abs{\Scales'}, S_{\delta,\dm}) + \delta \max(\abs{\Scales'}-S_{\delta,\dm},0))
\]
for any functions $F_{0},\dotsc,F_{\dm}$ bounded by $1$.
\end{corollary}
\begin{proof}
By induction on $\abs{\Scales'}$.
For $\abs{\Scales'}\leq S_{\delta,\dm}$ the estimate follows from $\abs{F_{i}}\leq 1$ and $\norm{\psi_{\scale}}_{1}\lesssim 1$.

If $\abs{\Scales'}> S_{\delta,\dm}$, then by Proposition~\ref{prop:single-tree} we can find a final interval $\Scales'' \subset \Scales'$ such that the sum over $\Scales''$ can be estimated by $2^{\dm s(J) \ds}\delta \abs{\Scales''}$.
The remaining part of the sum splits into sums over subintervals of scale $s(J)-\abs{\Scales''}$, and to those we apply the Corollary with $\Scales'\setminus \Scales''$ in place of $\Scales'$.
\end{proof}

\section{Tree selection}
\label{sec:selection}
For cubes $I\in\DI_{\scale}$ write
\[
a_{I} := 2^{-\scale \dm \ds}\Lambda_{I}(F_{0},\dotsc,F_{\dm})
\]
The integrand in the definition of $\Lambda_{I}$ vanishes outside $10 I$, say, and by the Loomis--Whitney inequality we can estimate
\[
\abs{a_{I}} \lesssim \prod_{i=0}^{\dm} \min_{\pi_{\Delta} I} M_{\dm}F_{i},
\]
where $\pi_{\Delta} I$ is the subset of the diagonal $\Delta = \{ x\in\V^{\dm+1} : \sum x = 0\}$ consisting of the points whose first $\dm$ coordinates lie in $I'$ and $M_{\dm}$ is the maximal function $M_{\dm}F(x) = \sup_{Q\ni x} (\abs{Q}^{-1}\int_{Q} \abs{F}^{\dm})^{1/\dm}$.
Raising this to a power $\alpha$ and summing over the squares $I$ of a given size we obtain
\begin{equation}
\label{eq:small-tiles:1}
\sum_{I\in\DI_{\scale}} \abs{a_{I}}^{\alpha} 2^{\scale \dm \ds}
\lesssim
\int_{\Delta} \prod_{i=0}^{\dm} M_{\dm}F_{i}(x_{(i)})^{\alpha}.
\end{equation}
By Hölder's inequality this is bounded by
\begin{equation}
\label{eq:small-tiles:2}
\prod_{i=0}^{\dm} \norm{ (M_{\dm}F_{i})^{\alpha} }_{p_{i}}
=
\prod_{i=0}^{\dm} \norm{ M_{\dm}F_{i} }_{\alpha p_{i}}^{\alpha}
\lesssim
\prod_{i} \abs{E_{i}}^{1/p_{i}}
\end{equation}
provided $\alpha p_{i}>\dm$ for all $i$.
It follows from $p_{i}>\dm$ that 
\begin{equation}
\label{eq:small-tiles}
\sum_{\scale\in \Scales} \sum_{I : \abs{a_{I}} < \delta, s(I)=\scale} \abs{a_{I}} 2^{\dm\scale\ds}
\leq
\sum_{\scale\in \Scales} \sum_{I : s(I)=\scale} \abs{a_{I}}^{\alpha} \delta^{1-\alpha} 2^{\dm\scale \ds}
\lesssim_{\alpha}
\delta^{1-\alpha} \abs{\Scales} \prod_{i} \abs{E_{i}}^{1/p_{i}}
\end{equation}
for every $\max_{i}(n/p_{i}) < \alpha \leq 1$ and every $\delta>0$.

Let now $\mathcal{J}\subset\DI_{\Scales}$ be the collection of maximal cubes $J$ with $\abs{a_{J}}>\delta$.
The union of these cubes cannot be too large.
Indeed, we have
\[
\pi_{\Delta} \bigcup\{J : \abs{a_{J}}>\delta\}
\subset
\Delta \cap \{ \prod_{i=0}^{\dm} M_{\dm}1_{E_{i}} \gtrsim \delta \}.
\]
The measure of the latter set is bounded by
\[
\delta^{-1} \norm{\prod_{i=0}^{\dm} M_{\dm}1_{E_{i}}}_{L^{1}(\Delta)}
\leq
\delta^{-1} \prod_{i=0}^{\dm} \norm{ M_{\dm}1_{E_{i}}}_{L^{p_{i}}(\Delta)}
\lesssim
\delta^{-1} \prod_{i=0}^{\dm} \abs{ E_{i} }^{1/p_{i}},
\]
where we have again used $p_{i}>\dm$.
Let $S = S_{\delta^{2},\dm}$ be the number given by Proposition~\ref{prop:single-tree} with $\delta^{2}$ in place of $\delta$.
For those $J\in\mathcal{J}$ with $s(J)> \min \Scales + \delta^{-2} S$ by Corollary~\ref{cor:single-tree} we have
\[
\abs[\big]{ \sum_{I\in\DI_{\Scales} : I'\subset J'} \Lambda_{I}(F_{0},\dotsc,F_{\dm}) }
\lesssim_{\dm}
2^{\dm s(J) \ds} \delta^{2} \abs{\Scales}
\sim
\abs{J} \delta^{2} \abs{\Scales}.
\]
In particular,
\begin{equation}
\label{eq:fat-tiles:1}
\begin{split}
\sum_{J\in\mathcal{J} : s(J)> \min \Scales + \delta^{-2} S}
\abs[\big]{ \sum_{I\in\DI_{\Scales} : I'\subset J'} \Lambda_{I}(F_{0},\dotsc,F_{\dm}) }
&\lesssim_{\dm}
\delta^{2} \abs{\Scales} \abs[\Big]{ \pi_{\Delta} \bigcup \Set{J \given \abs{a_{J}}>\delta} }\\
&\lesssim
\delta \abs{\Scales} \prod_{i} \abs{E_{i}}^{1/p_{i}}.
\end{split}
\end{equation}
On the other hand, by \eqref{eq:small-tiles:1} and \eqref{eq:small-tiles:2} with $\alpha=1$ we have
\begin{equation}
\label{eq:fat-tiles:2}
\sum_{J\in\mathcal{J} : s(J) \leq \min \Scales + \delta^{-2} S} \sum_{I\in\DI_{\Scales} : I'\subset J'} \abs{\Lambda_{I}(F_{0},\dotsc,F_{\dm})}
\lesssim
\delta^{-2} S \prod_{i} \abs{E_{i}}^{1/p_{i}}.
\end{equation}
Summing the contributions of \eqref{eq:small-tiles}, \eqref{eq:fat-tiles:1}, and \eqref{eq:fat-tiles:2} we obtain the claim of Theorem~\ref{thm:tiny-gain} (in the case of characteristic functions).


%% file: two-general.tex
\chapter{Dyadic triangular Hilbert transform (special case)}
\label{chap:two-general}

In this chapter we prove Theorem~\ref{thm:2general}.
Let us recall its statement.
\begin{theorem*}
Let $F_{0},F_{1},F_{2}$ be functions supported on $A_{0}^{2}=[0,1]^{2}$.
Suppose that either
\begin{equation}
\tag{\ref{eq:F0-diagonal}}
F_{0}(x_{1},x_{2})=f\big(x_{2}\wplus (a\wtimes x_{1})\big)
\text{ for all }x_{1},x_{2}\in A_{0}
\end{equation}
holds with some $a\in \W\setminus A_{0}$ and some measurable $f:\W\to\R$ or
\begin{equation}
\tag{\ref{eq:F0-fiberwise-character}}
F_{0}(x_{1},x_{2})=f(x_{2}) e(N_{x_{2}} \wtimes x_{1})
\text{ for all }x_{1},x_{2}\in A_{0}
\end{equation}
holds with some measurable $N : \W\to\W$ and $f:\W\to\R$.
Then
\begin{equation}
\tag{\ref{eq:lp-estimate}}
\abs{\Lambda^{\epsilon}(F_0,F_1,F_2)}
\lesssim
\norm{F_{0}}_{p_{0}} \norm{F_{1}}_{p_{1}} \norm{F_{2}}_{p_{2}}
\end{equation}
for any $1<p_{2}<\infty$ and $2<p_{0},p_{1}<\infty$ with \eqref{eq:Lp-range}.
The implicit constant does not depend on $a$, $N$, or the scalars $\abs{\epsilon_{\vec I}}\leq 1$ with $\epsilon_{\vec I}=0$ whenever some $I_{i} \not\subseteq A_{0}$.
In case \eqref{eq:F0-fiberwise-character} we can relax the restriction on $p_{0}$ to $1<p_{0}<\infty$.
In case \eqref{eq:F0-diagonal}, $a\in A_{1}\setminus A_{0}$, we can relax the restrictions on both $p_{0}$ and $p_{1}$ to $1<p_{0},p_{1}<\infty$.
\end{theorem*}

Since the conditions \eqref{eq:F0-diagonal} and \eqref{eq:F0-fiberwise-character} (with $a$ and $N$ fixed) describe subspaces of $L^{p_{0}}(\W^{2})$ that are themselves $L^{p_{0}}$ spaces, Theorem~\ref{thm:2general} follows by real interpolation from (generalized) restricted weak type estimates.
Such estimates also hold for certain negative values of $p_{i}$, the precise range of which we summarize with the aid of Figure~\ref{fig:exponents}.
Theorem~\ref{thm:2general} is the restriction of our results to the Banach triangle $c\cup b_{0}\cup b_{1}\cup b_{2}$ in Figure~\ref{fig:exponents}.

The local $L^{2}$ case of Theorem~\ref{thm:2general} (triangle $c$ in Figure~\ref{fig:exponents}) is covered by Proposition~\ref{prop:restricted-type}.
In this case the localization $\vec I \in \DI_{k}$, $k\leq 0$, in definition \eqref{eq:THT-def} can be removed using the Loomis--Whitney inequality
\[
\abs[\Big]{ \iiint_{\R^{3}} F_{0}(x,y) F_{1}(y,z) F_{2}(z,x) \dif(x,y,z)}
\leq
\norm{F_{0}}_{2} \norm{F_{1}}_{2} \norm{F_{2}}_{2}
\]
to estimate contributions of scales $k>0$.

Triangle $d_{12}$ is covered by Theorem~\ref{thm:2general-restricted-type}; this gives the lower half of the solid hexagon in Figure~\ref{fig:exponents}.
Triangle $d_{10}$ in cases \eqref{eq:F0-fiberwise-character} and \eqref{eq:F0-diagonal}, $a\in A_{1}\setminus A_{0}$, is covered by Theorem~\ref{thm:2general-restricted-type:F0}; together with the previous result this gives the full solid hexagon in Figure~\ref{fig:exponents}.
Finally, the case \eqref{eq:F0-diagonal}, $a\in A_{1}\setminus A_{0}$, is symmetric in indices $0,2$; in this case we obtain estimates in the dashed extension of the solid hexagon in Figure~\ref{fig:exponents}.

\section{Tile decomposition}
In this section we describe a time-frequency decomposition for the form \eqref{eq:THT-def} that is well adapted both to diagonal functions \eqref{eq:F0-diagonal} and to fiberwise characters \eqref{eq:F0-fiberwise-character}.
While the decomposition of the \emph{form} is the same in both cases, the time-frequency projections of (one of) the \emph{functions} differ.
However, in both cases the time-frequency projections satisfy the same localization and scale compatibility properties, summarized in Definition~\ref{tile-proj}.
The proof of the local $L^{2}$ bounds uses only these properties and a single tree estimate.
We will have to come back to the definition of time-frequency projections in the multi-frequency Calder\'on--Zygmund decomposition in Section~\ref{sec:mfcz}.

\subsection{Wave packets}
\label{sec:wave-packets}
The characters on the Walsh field $\W$ are the \emph{Walsh functions}
\[
w_{N}(x) := e(N \wtimes x),
\]
where $N\in\W$ and $e\colon\W\to\R$ is simply the periodization of $\h_{[0,1)}$.
Their particular cases are the \emph{Rademacher functions} $r_k:=w_{2^{-k}}$, $k\in\Z$.
The \emph{Walsh wave packet} associated with a dyadic rectangle $I\times\omega$ of area $1$ is
\[
w_{I\times\omega}(x) := \abs{I}^{-1/2} 1_{I}(x) e(l(\omega) \wtimes x),
\]
where $l(\omega)$ is the left endpoint of $\omega$.
Note that replacing $l(\omega)$ by any other member of $\omega$ only multiplies $w_{I\times\omega}$ by a constant factor $\pm 1$.
This definition satisfies the usual recursive relations
\[
w_{P_{\mathrm{up}}} = (w_{P_{\mathrm{left}}} - w_{P_{\mathrm{right}}})/\sqrt{2},
\quad
w_{P_{\mathrm{down}}} = (w_{P_{\mathrm{left}}} + w_{P_{\mathrm{right}}})/\sqrt{2}
\]
on every dyadic rectangle $P$ of area $2$ and therefore coincides with the usual definition; see \cite[\S 1]{MR2692998}.

\subsection{Tile decomposition}
Our time-frequency analysis is $1\frac12$-dimensional in the sense of \cite{MR2597511}.
We define \emph{tiles} as dyadic boxes
\[
\p = I_{\p,0}\times I_{\p,2} \times \omega_{\p,1},
\quad\text{where}\quad
\abs{I_{\p,0}} = \abs{I_{\p,2}} = \abs{\omega_{\p,1}}^{-1}.
\]
A \emph{bitile} is then any dyadic box of the form
\[
P=I_{P,0}\times I_{P,2} \times \omega_{P,1},
\quad\text{where}\quad
\abs{I_{P,0}} = \abs{I_{P,2}} = 2\abs{\omega_{P,1}}^{-1}.
\]
We will omit the subscripts $\p,P$ if no confusion seems possible.
For notational convenience we will throughout write $I_{1}=I_{0}\wplus I_{2}$.

Dyadic boxes are partially ordered by
\[
P \leq P' :\iff I_{i} \subseteq I_{i}',\ \omega_{i} \supseteq \omega_{i}'.
\]
Writing one of the Haar functions in \eqref{eq:THT-def} as a difference of two characteristic functions we arrive at
\[
\Lambda^{\epsilon}(F_{0},F_{1},F_{2})
=
\sum_{\vec I} \epsilon_{\vec I} \sum_{j \in \{\pm 1\}} j \abs{I_{1}}^{-1}
\tr(1_{I_{1}^{j}}1_{I_{1}^{j}}F_{0} \h_{I_{2}} F_{1} \h_{I_{0}} F_{2}),
\]
where $1_I$ denotes, along with the characteristic function of the interval $I$, also the projection operator
\[
(1_I \varphi)(x)=1_I(x)\varphi(x).
\]
Inserting identity operators (expanded in the Walsh basis) between characteristic functions we obtain
\[
\sum_{\vec I} \epsilon_{\vec I} \sum_{j \in \{\pm 1\}} j
\sum_{\omega_{1} : \abs{\omega_{1}} = 2\abs{I_{1}}^{-1}} 2\abs{I_{1}}^{-2}\\
\tr\big(
w_{I_{1}^{j} \times \omega_{1}}\otimes w_{I_{1}^{j} \times \omega_{1}} F_{0}
\h_{I_{2}} F_{1}
\h_{I_{0}} F_{2}\big).
\]
Changing the order of summation we obtain
\[
\Lambda^{\epsilon}(F_{0},F_{1},F_{2})
=
\sum_{P \text{ bitile}} \epsilon_{\vec I_{P}}
\Lambda_{P}(F_{0},F_{1},F_{2}),
\]
where
\[
\Lambda_{\vec I\times \vec\omega}(F_{0},F_{1},F_{2})
:=
\sum_{j \in \{\pm 1\}} j 2\abs{I_{1}}^{-2}
\tr(w_{I_{1}^{j} \times \omega_{1}}\otimes w_{I_{1}^{j} \times \omega_{1}}F_{0} \h_{I_{2}} F_{1} \h_{I_{0}} F_{2}).
\]

\subsection{Time-frequency projections}
We begin by collecting desirable properties of time-frequency projections.
\begin{definition}
\label{tile-proj}
We call orthogonal projections $\Pi^{(i)}_{\p}$, acting on $L^{2}(x_{i-1},x_{i+1})$ and indexed by tiles $\p$, \emph{time-frequency projections} if they satisfy the following conditions.
\begin{enumerate}
\item\label{tile-proj:orth} (Orthogonality) The projections $\Pi^{(i)}_{\p}$ corresponding to disjoint tiles are orthogonal.
\item\label{tile-proj:compatible} (Scale compatibility) Bitile projections $\Pi^{(i)}_{P}$ are well-defined (there are two ways to write a bitile as a disjoint union of tiles, and the corresponding sums of tile projections are equal).
\item\label{tile-proj:support} (Support) $\supp \Pi_{\p}^{(i)}F_{i} \subset I_{i-1}\times I_{i+1}$.
\end{enumerate}
\end{definition}
A collection of bitiles $\mathbf{P}$ is called \emph{convex} if $P,P''\in\mathbf{P}$, $P\leq P'\leq P''$ implies $P'\in\mathbf{P}$.
The union of any finite convex collection of bitiles $\mathbf{P}$ can be written as the union of a collection of disjoint tiles $\mathbf{p}$ (this is proved by induction on the number of bitiles, cf.\ \cite[Lemma 1.7]{MR2692998}).
Given time-frequency projections, this allows us to consider the projections
\[
\Pi_{\mathbf{P}}^{(i)}F_{i} := \sum_{\p\in\mathbf{p}} \Pi_{\p}^{(i)}F_{i}.
\]
The property \eqref{tile-proj:compatible} implies that these projections do not depend on the choice of $\mathbf{p}$, cf.\ \cite[Corollary 1.9]{MR2692998}.

\begin{definition}
We call time-frequency projections \emph{adapted} to $F_{0}$ if for every choice of $F_{1},F_{2}$, every bitile $P$, and any convex collection of bitiles $\mathbf{P}\ni P$ we have
\begin{equation}
\label{eq:Lambda-tile-proj}
\Lambda_{P}(F_{0},F_{1},F_{2}) = \Lambda_{P}(\Pi_{\mathbf{P}}^{(0)}F_{0},\Pi_{\mathbf{P}}^{(1)}F_{1},\Pi_{\mathbf{P}}^{(2)}F_{2}).
\end{equation}
\end{definition}

The existence of adapted time-frequency projections suffices to establish restricted type bounds on the dyadic triangular Hilbert transform in the local $L^{2}$ range.
\begin{proposition}
\label{prop:restricted-type}
Let $E_{i} \subset A_{0}^{2}$, $i\in\{0,1,2\}$, be measurable sets and $\abs{F_{i}} \leq 1_{E_{i}}$ be functions for which there exist time-frequency projections adapted to $F_{0}$.
Then
\[
\abs{\Lambda^{\epsilon}(F_{0},F_{1},F_{2})}
\lesssim
a_{1}^{1/2} a_{2}^{1/2} (1+\log\frac{a_{0}}{a_{1}}),
\]
where $a_{i}=\abs{E_{\sigma(i)}}$ is a decreasing rearrangement, that is, $\sigma$ is a permutation of $\{0,1,2\}$ and $a_{0}\geq a_{1}\geq a_{2}$.
The implicit constant is independent of the choices of the scalars $\abs{\epsilon_{\vec I}}\leq 1$.
\end{proposition}

We finish this section with the construction of time-frequency projections adapted to \eqref{eq:F0-diagonal} and \eqref{eq:F0-fiberwise-character}.
For indices $0$ and $2$ we use the projections
\begin{equation}
\label{eq:Pi2}
\Pi^{(2)}_{\p} F_{2}(x_{0},x_{1}) := 1_{I_{0}}(x_{0}) \< F_{2}(x_{0},\cdot), w_{I_{1}\times \omega_{1}} \> w_{I_{1}\times \omega_{1}}(x_{1})
\end{equation}
and
\begin{equation}
\label{eq:Pi0}
\Pi^{(0)}_{\p} F_{0}(x_{2},x_{1}) := 1_{I_{2}}(x_{2}) \< F_{0}(x_{2},\cdot), w_{I_{1}\times \omega_{1}} \> w_{I_{1}\times \omega_{1}}(x_{1}).
\end{equation}
The structural information given by \eqref{eq:F0-diagonal} and \eqref{eq:F0-fiberwise-character} is encoded in the projections $\Pi^{(1)}$.

\subsubsection{One-dimensional functions}
\label{sec:1dimfct}
Suppose \eqref{eq:F0-diagonal}.
Then we have
\[
\Pi^{(0)}_{\p}F_{0}(x_{1},x_{2}) = 1_{I_{1}}(x_{1}) (\Pi_{I_{2}\times a\wtimes\omega_{1}}F_{0}(\cdot,x_{1}))(x_{2}),
\]
where the projection on the right-hand side is a one-dimensional time-frequency projection (as defined e.g.\ in \cite{MR2997005}) with a possibly multidimensional range.
In this case we define
\[
\Pi^{(1)}_{P}F_{1}(x_{2},x_{0}) := 1_{I_{0}}(x_{0}) (\Pi_{I_{2}\times a\wtimes\omega_{1}}F_{1}(\cdot,x_{0}))(x_{2}).
\]

\subsubsection{Fiberwise characters}
\label{sec:fiberwisechar}
Suppose \eqref{eq:F0-fiberwise-character}.
Then we have
\[
\Pi^{(0)}_{\p}F_{0}(x_{1},x_{2}) = 1_{I_{1}}(x_{1}) 1_{I_{2}}(x_{2}) 1_{\omega_{1}}(N_{x_{2}}) F_{0}(x_{1},x_{2}).
\]
In this case we define
\[
\Pi^{(1)}_{\p}F_{1}(x_{2},x_{0}) := 1_{I_{0}}(x_{0}) 1_{I_{2}}(x_{2}) 1_{\omega_{1}}(N_{x_{2}}) F_{1}(x_{2},x_{0}).
\]
The projections $\Pi^{(1)}$ constructed above satisfy \eqref{eq:Lambda-tile-proj} only for bitiles with $I_{i}\subseteq A_{0}$, which explains the truncation in Theorem~\ref{thm:2general}.

\section{Single tree estimate}
\label{sec:single-tree}
A \emph{tree} $T$ is a convex set of bitiles that contains a maximal element
\[
P_{T} = \vec I_{T} \times \vec \omega_{T} = I_{T,0} \times I_{T,2} \times \omega_{T,1}.
\]
Equivalently, a tree can be described by a top frequency $\xi_{T,1}$ and a convex collection of space boxes $\DI_{T}$.
The corresponding tree $T$ then consists of all bitiles $P = \vec I\times\vec\omega$ with $\vec I\in \DI_{T}$ and $\xi_{T,1}\in\omega_{1}$.

For a convex collection $\mathbf{P}$ of bitiles define
\begin{equation}
\label{eq:size}
\size^{(i)}(\mathbf{P},F_{i})
:=
\sup_{T\subset \mathbf{P} \text{ tree}} \abs{\vec I_{T}}^{-1/2} \norm{\Pi_{T}^{(i)}F_{i}}_{2}.
\end{equation}
For a collection $\mathbf{P}$ of bitiles write
\[
\Lambda^{\epsilon}_{\mathbf{P}}(F_{0},F_{1},F_{2})
:=
\sum_{P\in\mathbf{P}} \epsilon_{\vec I_{P}}
\Lambda_{P}(F_{0},F_{1},F_{2}).
\]
The objective of this section is to show that Definition~\ref{tile-proj} implies
\begin{equation}
\label{eq:single-tree-estimate:symmetric}
\abs{\Lambda^{\epsilon}_{T}(F_{0},F_{1},F_{2})}
\lesssim
\abs{\vec I_{T}} \prod_{i=0}^{2} \size^{(i)}(T,F_{i}),
\end{equation}
where $T$ is a tree and the implied constant is absolute.
It follows from Definition~\ref{tile-proj} that
\[
\abs{\vec I_{P}}^{-1/2} \norm{\Pi^{(i)}_{P} F_{i}}_{L^{2}(I_{i-1,P}\times I_{i+1,P})} \lesssim \size^{(i)}(T,F_{i})
\quad
\text{ for all } P\in T.
\]
Thus in view of \eqref{eq:Lambda-tile-proj} it suffices to show
\begin{equation}
\label{eq:single-tree-estimate:symmetric2}
\abs{\Lambda^{\epsilon}_{T}(F_{0},F_{1},F_{2})}
\lesssim
\abs{\vec I_{T}} \prod_{i=0}^{2} \sup_{\vec I\in \DI_{T} \cup \Leaves_{T}} \abs{\vec I}^{-1/2} \norm{F_{i}}_{L^{2}(I_{i-1}\times I_{i+1})},
\end{equation}
where $\Leaves_{T}$ denotes the collection of leaves of a tree, that is, maximal elements of $\DI$ contained in a member of $T$ that are not themselves members of $\DI_{T}$.
By modulation we may assume $\xi_{T,1}=0$.
The tree operator can be written as
\[
\sum_{\vec I\in \DI_{T}} \epsilon_{\vec I} \abs{I_{1}}^{-2}
(\tr((1\otimes 1)1_{I_{1}} F_{0} \h_{I_{2}} F_{1} \h_{I_{0}} F_{2} \h_{I_{1}})
+
\tr((1\otimes 1)\h_{I_{1}} F_{0} \h_{I_{2}} F_{1} \h_{I_{0}} F_{2} 1_{I_{1}})).
\]
The two summands are symmetric (under permuting the indices $0$ and $2$) and we consider only the first of them.
With the convention that the domain of integration is $x_{i},y_{i}\in I_{i}$ and the dyadic intervals have size $\abs{I_{i}}=2^{k}$ we have
\begin{multline*}
\tr((1\otimes 1)1_{I_{1}} F_{0} \h_{I_{2}} F_{1} \h_{I_{0}} F_{2} \h_{I_{1}})\\
=
\int F_{0}(x_{1},x_{2}) r_{k}(x_{2}) F_{1}(x_{2},x_{0}) r_{k}(x_{0}) F_{2}(x_{0},y_{1}) r_{k}(y_{1})
\dif x_{1} \dif x_{2} \dif x_{0} \dif y_{1}.
\end{multline*}
The change of variables $x_{1}=x_{2}+y_{0}$, $y_{1}=x_{0}+y_{2}$ gives
\begin{align*}
&\int F_{0}(x_{2}+y_{0},x_{2}) r_{k}(x_{2}) F_{1}(x_{2},x_{0}) r_{k}(x_{0}) F_{2}(x_{0},x_{0}+y_{2}) r_{k}(x_{0}+y_{2})
\dif y_{0} \dif x_{2} \dif x_{0} \dif y_{2}\\
&=\int \tilde F_{0}(y_{0},x_{2}) r_{k}(x_{2}) F_{1}(x_{2},x_{0}) \tilde F_{2}(x_{0},y_{2}) r_{k}(y_{2})
\dif y_{0} \dif x_{2} \dif x_{0} \dif y_{2},
\end{align*}
where $\tilde F_{0}(y_{0},x_{2}):=F_{0}(x_{2}+y_{0},x_{2})$ and $\tilde F_{2}(x_{0},y_{2}) := F_{2}(x_{0},x_{0}+y_{2})$.

Thus the first half of the tree operator can be written as a single tree operator from \cite[\textsection 3]{MR2990138} with square-dependent coefficients.
The first step in the proof of \cite[Proposition 4]{MR2990138} is an application of the Cauchy--Schwarz inequality in the sum over squares, so it still works in our situation.
This, together with \cite[(2.2)]{MR2990138}, gives the required estimate.

\section{Tree selection and local $L^{2}$ bounds}
\subsection{The tree selection algorithm}
We organize bitiles into trees closely following the argument in \cite[Lemma 2.2]{MR2997005}.
Here and later we use coordinate projections $\pi_{(i)}: \W^{3}\to\W^{2}, (x_{i-1},x_{i},x_{i+1}) \mapsto (x_{i-1},x_{i+1})$.
\begin{proposition}
\label{prop:tree-selection}
Let $n\in\Z$, $i\in\{0,1,2\}$, a function $F_{i}$, and a system of (not necessarily adapted) time-frequency projections $\Pi^{(i)}$ be given.
Then every finite convex collection of bitiles $\mathbf{P}$ can be partitioned into a convex collection of bitiles $\mathbf{P}'$ with
\[
\size_{i}(\mathbf{P}',F_{i}) \leq 2^{-n}
\]
and a further convex collection of bitiles that is the disjoint union of a collection of convex trees $\mathbf{T}$ with
\begin{equation}
\label{eq:tree-selection:tree-counting-bound}
\sum_{T\in\mathbf{T}, \vec I_{T} \subset \vec J} \abs{\vec I_{T}}
\leq
9 \cdot 2^{2n} \norm{1_{\pi_{(i)}\vec J} F_{i}}_{2}^{2},
\quad \vec J\in\DI.
\end{equation}
\end{proposition}
The latter bound includes both an $L^{1}$ estimate (taking $\vec J$ large enough to contain all time intervals in $\mathbf{P}$) and a $\mathrm{BMO}$ estimate (noting $\norm{1_{\pi_{(i)}\vec J} F_{i}}_{2}^{2} \leq \abs{\vec J} \norm{F_{i}}_{\infty}^{2}$) for the counting function $\sum_{T\in\mathbf{T}} 1_{\vec I_{T}}$.
\begin{proof}
We will remove three collections of trees, each of which satisfies \eqref{eq:tree-selection:tree-counting-bound} with a smaller constant.
At each step we remove a tree that is also a down-set, thus ensuring that both the remaining collection $\mathbf{P}'$ and the collection of all removed tiles are convex.

Replacing $F_{i}$ by $2^{n}F_{i}$ we may assume $n=0$.
We write every bitile $P$ as $P^{+1}\cup P^{-1}$, where the tiles $P^{j}$, $j=\pm 1$, are given by $\vec I_{P} \times \omega_{P,1}^{j}$.

For a tree $T$ write
\[
T_{j} := \{ P\in T : P^{j} \leq P_{T}\},
\quad
j=\pm 1.
\]
Then
\[
\Pi_{T}^{(i)} F_{i}
=
\Pi_{P_{T}}^{(i)} F_{i}
+
\sum_{j=\pm 1}\sum_{P\in T_{j}} \Pi_{P^{-j}}^{(i)} F_{i},
\]
and this sum is orthogonal by Definition~\ref{tile-proj}~\ref{tile-proj:orth}.

Let $\{P_{1},\dots,P_{n}\}$ be the collection of maximal bitiles in $\mathbf{P}$ that satisfy
\[
\norm{ \Pi_{P_{k}}^{(i)} F_{i} }_{2}^{2} > 3^{-1} \abs{ \vec I_{P_{k}} }.
\]
These bitiles are necessarily pairwise disjoint, so we have
\[
\sum_{k : \vec I_{P_{k}} \subset \vec J} \abs{ \vec I_{P_{k}} }
<
3 \sum_{k : \vec I_{P_{k}} \subset \vec J} \norm{ \Pi_{P_{k}}^{(i)} F_{i} }_{2}^{2}
\leq
3 \norm{ 1_{\pi_{(i)}\vec J} F_{i} }_{2}^{2}
\]
for every $\vec J\in\DI$, where the last inequality follows from parts \ref{tile-proj:orth} and \ref{tile-proj:support} of Definition~\ref{tile-proj}.
Thus, removing the bitiles $P\leq P_{k}$ from $\mathbf{P}$, we may assume
\[
\norm{ \Pi_{P}^{(i)} F_{i} }_{2}^{2} \leq 3^{-1} \abs{ \vec I_{P} },
\quad
P\in\mathbf{P}.
\]
The next step will be done twice, for $j=\pm 1$.
In each case we remove a collection of trees $\mathbf{T}_{j}$ such that for every remaining tree $T$ we have
\begin{equation}
\label{eq:tree-selection:up}
\sum_{P\in T_{j}} \norm{ \Pi_{P^{-j}}^{(i)} F_{i} }^{2} \leq 3^{-1} \abs{\vec I_{T}}^{2}.
\end{equation}
The collection $\mathbf{T}_{j}=\{T_{1},T_{2},\dots\}$ is selected iteratively.
Suppose that $T_{1},\dots,T_{k}$ have been selected and suppose that \eqref{eq:tree-selection:up} is violated for some remaining tree $T \subset \mathbf{P}\setminus T_{1}\cup\dots\cup T_{k}$.
Choose one such tree for which either the left endpoint of $\omega_{T,1}$ is minimal (if $j=-1$) or the right endpoint is maximal (for $j=+1$) and let $T_{k+1}\subset\mathbf{P}$ be the down-set spanned by the chosen tree.

We claim that the tiles of the form $P_{m}^{-j}$, $P_{m}\in (T_{m})_{j}$, are pairwise disjoint.
This is clear within each tree, so assume for contradiction $P_{k}^{-j} < P_{l}^{-j}$, $k\neq l$.
In particular, we have $P_{k} < P_{l}$, and this implies $k<l$, since otherwise $P_{k}$ should have been included in $T_{l}$.
On the other hand, $\omega_{P_{k},1}^{-j} \supsetneq \omega_{P_{l},1}^{-j}$ implies $\omega_{P_{k},1}^{-j} \supseteq \omega_{P_{l},1} \supsetneq \omega_{P_{l},1}^{j} \supseteq \omega_{T_{l},1}$, whereas $\omega_{T_{k},1} \subseteq \omega_{P_{k},1}^{j}$.
Thus $\omega_{T_{k},1}$ is either to the right (if $j=-1$) or to the left (if $j=+1$) of $\omega_{T_{l},1}$, in both cases contradicting the choice of $T_{k}$.

Violation of \eqref{eq:tree-selection:up} for $T_{k}\in\mathbf{T}_{j}$ and parts \ref{tile-proj:orth} and \ref{tile-proj:support} of Definition~\ref{tile-proj} give
\[
\sum_{k : \vec I_{T_{k}} \subset \vec J} \abs{\vec I_{T_{k}}}
<
\sum_{k : \vec I_{T_{k}} \subset \vec J} 3 \sum_{P\in (T_{k})_{j}} \norm{ \Pi_{P^{-j}}^{(i)} F_{i} }_{2}^{2}
\leq
3 \norm{ 1_{\pi_{(i)}\vec J} F_{i} }_{2}^{2},
\]
as required.
For each remaining tree we will have
\[
\norm{ \Pi_{P_{T}}^{(i)} F_{i} }^{2}
+
\sum_{j=\pm 1}\sum_{P\in T_{j}} \norm{ \Pi_{P^{-j}}^{(i)} F_{i} }^{2}
\leq
(3^{-1}+3^{-1}+3^{-1}) \abs{\vec I_{T}},
\]
and this gives the required estimate for $\size_{i}(T,F_{i})$.
\end{proof}

\subsection{Local $L^{2}$ bounds (triangle $c$)}
\begin{proof}[Proof of Proposition~\ref{prop:restricted-type}]
Normalizing $\tilde F_{i} = F_{i}/\abs{E_{i}}^{1/2}$ we have to show
\begin{equation}
\label{eq:restricted-type-rescaled}
\abs{\Lambda_{\mathbf{P}}^{\epsilon}(\tilde F_{0},\tilde F_{1},\tilde F_{2})}
\lesssim
a_{0}^{-1/2} (1+\log\frac{a_{0}}{a_{1}})
\end{equation}
with a constant independent of the (finite) convex collection of bitiles $\mathbf{P}$.
We have $\size^{(i)}(\tilde F_{i}) \leq \norm{\tilde F_{i}}_{\infty} \leq \abs{E_{i}}^{-1/2} = a_{\sigma^{-1}(i)}^{-1/2}$ and $\norm{\tilde F_{i}}_{2}\leq 1$.
Fix integers $n_{i}$ such that $2^{n_{i}-1} < a_{i}^{-1/2} \leq 2^{n_{i}}$; note that in particular $n_{0}\leq n_{1}\leq n_{2}$.
Running the tree selection algorithm (Proposition~\ref{prop:tree-selection}) iteratively at each scale $n\leq n_{2}$ for each $i\in\{0,1,2\}$ we obtain collections of trees $\mathbf{T}_{n}$ with
\[
\sum_{T\in\mathbf{T}_{n}} \abs{I_{T,i}}^{2} \lesssim 2^{-2n}
\]
and
\[
\size^{(i)}(T,\tilde F_{i}) \leq \min (2^{n}, 2^{n_{\sigma^{-1}(i)}}),
\quad T\in \mathbf{T}_{n}.
\]
Summing the single tree estimate \eqref{eq:single-tree-estimate:symmetric} over all trees we obtain
\[
\abs{\Lambda^{\epsilon}(\tilde F_{0},\tilde F_{1},\tilde F_{2})}
\lesssim
\sum_{n\leq n_{2}} 2^{-2n}
\prod_{i=0}^{2} \min (2^{n}, 2^{n_{i}}).
\]
The sum over $n$ is an increasing geometric series for $n<n_{0}$ and a decreasing geometric series for $n>n_{1}$.
In particular, the sum is dominated by the terms $n_{0}\leq n\leq n_{1}$, that is, we have the estimate
\[
2^{n_{0}}(1+n_{1}-n_{0})
\lesssim
a_{0}^{-1/2} (1+\log\frac{a_{0}}{a_{1}})
\]
as required.
\end{proof}

\section{Fiberwise multi-frequency Calder\'on--Zygmund decomposition}
\label{sec:mfcz}
In order to extend the range of exponents in our main result we perform a fiberwise multi-frequency Calder\'on--Zygmund decomposition.
Here, in contrast to the local $L^{2}$ range, we have to use the special form of the time-frequency projections $\Pi^{(0)}$ and $\Pi^{(2)}$.

Our decomposition unites the main features of the one-dimensional multi-frequency Calder\'on--Zygmund decomposition in \cite{MR2997005} and the fiberwise single-frequency Calder\'on--Zygmund decomposition in \cite{MR3161332,MR2990138}.
A useful simplification with respect to \cite{MR2997005} is that we do not attempt to control the size of the good function, this corresponds to the observation that the argument on page 1709 of \cite{MR2997005} works directly for $a$ in place of $a_{m}$.

\subsection{Triangles $b_{2}$ and $d_{12}$}
\begin{theorem}
\label{thm:2general-restricted-type}
Let $0< \alpha_{0} \leq 1/2 \leq \alpha_{2} < 1$ and $-1/2<\alpha_{1}<1/2$ satisfy \eqref{eq:Lp-range}.
Then for any measurable sets $E_{i} \subset A_{0}^{2}$, $i\in\{0,1,2\}$ there exists a major subset $E_{1}'\subset E_{1}$ (which can be taken equal to $E_{1}$ if $\alpha_{1}>0$)
such that for any dyadic test functions $\abs{F_{i}} \leq 1_{E_{i}}$, $\abs{F_{1}} \leq 1_{E_{1}'}$ with \eqref{eq:F0-diagonal} or \eqref{eq:F0-fiberwise-character} we have
\[
\abs{\Lambda^{\epsilon}(F_0,F_1,F_2)}
\lesssim_{\alpha_{0},\alpha_{1},\alpha_{2}}
\prod_{i=0}^{2} \abs{E_{i}}^{\alpha_{i}},
\]
where the implied constant is independent of the choices of the scalars $\abs{\epsilon_{\vec I}}\leq 1$ with $\epsilon_{\vec I}=0$ whenever $I_{i} \not\subset A_{0}$.
\end{theorem}
\begin{proof}
The required estimate is invariant under rescaling by powers of $2$, so we may normalize $\abs{E_{1}} \approx 1$.
The localization changes to $E_{i} \subset A_{k}^{2}$ for some $k\in\Z$, but all previous results still apply by scale invariance.
In the case $\abs{E_{2}} \gtrsim \abs{E_{1}}$ the estimate with $E_{1}'=E_{1}$ follows from the local $L^{2}$ case $0<\alpha_{0},\alpha_{1},\alpha_{2}\leq 1/2$, which is given by Proposition~\ref{prop:restricted-type}.
Thus we may assume $\abs{E_{2}} < 2^{-20}$.

Define the exceptional sets
\[
\eset_{0} := \{M_{p_{0}}(\abs{E_{0}}^{-1/p_{0}} 1_{E_{0}}) > 2^{10}\}
\]
and
\[
\eset_{2} := \{\tilde M_{p_{2}}(\abs{E_{2}}^{-1/p_{2}} 1_{E_{2}}) > 2^{10}\},
\]
where $\tilde M_{p_{2}}$ is the directional maximal function (in the direction $x_{1}$).
The set
\[
\eset_{1}:=\pi_{(1)}((\pi_{(0)}^{-1}\eset_{0}\cup \pi_{(2)}^{-1}\eset_{2})\cap\Delta),
\quad
\Delta := \{x_{0}\wplus x_{1}\wplus x_{2} = 0\} \subset \W^{3},
\]
has measure $<1/2$ by the Hardy--Littlewood maximal inequality.
Consider the major subset $E_{1}':=E_{1} \setminus \eset_{1}$.

Define normalized functions
\[
\tilde F_{i} := \abs{E_{i}}^{-1/p_{i}} F_{i}.
\]
By construction of the major subset only the bitiles $P$ with
\[
\pi_{(1)}\vec I_{P}\not\subset \eset_{1}
\]
contribute to the trilinear form $\Lambda$, so consider a finite convex collection $\mathbf{P}$ of such bitiles.
Since the $M_{p_{0}}$ maximal function dominates the $M_{2}$ maximal function pointwise and by Definition~\ref{tile-proj}~\ref{tile-proj:support} we have
\[
\size^{(0)}(\mathbf{P},\tilde F_{0}) \lesssim 1,
\quad
\size^{(1)}(\mathbf{P},\tilde F_{1}) \lesssim 1.
\]
By the tree selection algorithm in Proposition~\ref{prop:tree-selection} we partition $\mathbf{P}$ into a sequence of pairwise disjoint convex unions of pairwise disjoint trees $\mathbf{P}_{k} = \cup_{T\in\mathbf{T}_{k}} T$ and a remainder set with zero contribution to $\Lambda$ in such a way that
\[
\size^{(0)}(\mathbf{P}_{k},\tilde F_{0}) \lesssim 2^{-k}
\]
and
\[
\norm{ N_{k} }_{p} \lesssim_{p} 2^{2k} \norm{\tilde F_{0}}_{2}^{2/p} \norm{\tilde F_{0}}_{\infty}^{2-2/p},
\quad N_{k} := \sum_{T\in T_{k}} 1_{\vec I_{T}},
\quad 1\leq p<\infty.
\]
Choosing $p=p_{0}/2$ we obtain the bound
\[
\norm{ N_{k} }_{p} \lesssim_{p} 2^{2k}.
\]
For a fixed $k$ we will show
\[
\abs{\Lambda^{\epsilon}_{\mathbf{P}_{k}}(\tilde F_{0},\tilde F_{1},\tilde F_{2})} \lesssim 2^{-\delta k}
\]
for some $\delta>0$, depending only on the $p_{i}$'s, to be determined later.

Let $\DI_{\eset}$ denote the collection of the maximal one-dimensional dyadic intervals of the form $\{x_{0}\}\times J_{1} \subset \eset_{2}$.
For each one-dimensional interval $J=\{x_{0}\}\times J_{1}\in\DI_{\eset}$ let
\[
\Omega_{J} := \{ \omega : \abs{\omega}\abs{J}=1, \exists T\in\mathbf{T}_{k} : \vec I_{T}\supseteq J, \omega\supseteq \omega_{T}\}.
\]
Let
\[
G:= \sum_{J\in\DI_{\eset}} G_{J},
\quad
G_{J}(x_{0},x_{1}) := 1_{J}(x_{0},x_{1}) \sum_{\omega\in\Omega_{J}} (\Pi_{J_{1}\times\omega} \tilde F_{2}(x_{0},\cdot))(x_{1}).
\]
The sum defining the function $G$ is pointwise finite, and $G$ is measurable since $\tilde F_{2}$ is a dyadic test function.

We claim that for every $P = \vec I \times \omega_{1} \in\mathbf{P}_{k}$ we have
\[
\Lambda_{P}(\tilde F_{0},\tilde F_{1},\tilde F_{2})
=
\Lambda_{P}(\tilde F_{0},\tilde F_{1},G).
\]
Since $E_{2}\subset \eset_{2}$ by construction and the collection $\DI_{\eset}$ covers $\eset_{2}$, it suffices to show
\begin{multline*}
\int_{J_{1}} \tilde F_{1}(x_{0},x_{2}) \h_{I_{0}}(x_{0}) \tilde F_{2}(x_{0},x_{1}) w_{I_{1}^{j}\times \omega_{1}}(x_{1}) \dif x_{1}\\
=
\int_{J_{1}} \tilde F_{1}(x_{0},x_{2}) \h_{I_{0}}(x_{0}) G_{J}(x_{0},x_{1}) w_{I_{1}^{j}\times \omega_{1}}(x_{1}) \dif x_{1}
\end{multline*}
for every $J=\{x_{0}\}\times J_{1}\in\DI_{\eset}$, every $x_{2}\in I_{2}$, and every $j\in\{\pm 1\}$.
If $I_{0}\times I_{1}\cap J=\emptyset$, then both sides vanish identically.
Otherwise we must have $x_{0}\in I_{0}$.
If now $I_{1}\subseteq J_{1}$, then by construction $\tilde F_{1}$ vanishes on $\{x_{0}\} \times I_{2}$, so both sides again vanish identically.
On the other hand, if $J_{1}\subsetneq I_{1}$, then by construction $\Omega_{J}$ contains an ancestor of $\omega_{1}$, so the integrals coincide again.
This finishes the proof of the claim.

Now we estimate $\norm{G}_{2}$.
By H\"older and Hausdorff--Young inequalities we get
\begin{align*}
\norm{G_{J}}_{L^{2}(J)}^{2}
&=
\sum_{\omega\in\Omega_{J}} \abs{\<\tilde F_{2}(x_{0},\cdot), w_{J_{1}\times\omega}\>}^{2}\\
&\leq
\abs{\Omega_{J}}^{1-2/p_{2}'}(\sum_{\omega\in\Omega_{J}} \abs{\<\tilde F_{2}(x_{0},\cdot), w_{J_{1}\times\omega}\>}^{p_{2}'})^{2/p_{2}'}\\
&\leq
\abs{\Omega_{J}}^{1-2/p_{2}'} \norm{ \tilde F_{2} }_{L^{p_{2}}(J)}^{2} \abs{J_{1}}^{1-2/p_{2}}.
\end{align*}
Maximality of $J\subset \eset_{2}$ gives an upper bound on the above $L^{p_{2}}(J)$ norm, and we obtain
\[
\norm{G_{J}}_{L^{2}(J)}^{2}
\lesssim
\abs{\Omega_{J}}^{1-2/p_{2}'} \abs{J_{1}}
\leq
\int_{J} N_{k}^{1-2/p_{2}'}.
\]
Integrating these bounds and using monotonicity of $L^{p}$ norms (recall $\abs{\eset_{2}}\lesssim 1$) we get
\begin{align*}
\norm{G}_{2}^{2}
&\lesssim
\int_{\eset_{2}} N_{k}^{1-2/p_{2}'}
\lesssim
(\int_{\eset_{2}} N_{k}^{p} )^{(1-2/p_{2}')/p}\\
&\leq
\norm{ N_{k} }_{p}^{1-2/p_{2}'}
\lesssim
2^{2k(1-2/p_{2}')}.
\end{align*}
Normalize
\[
\tilde G:=2^{-k(1-2/p_{2}')} G,
\]
so that $\norm{\tilde G}_{2}\lesssim 1$.
We claim
\[
\abs{\Lambda^{\epsilon}_{\mathbf{P}_{k}}(\tilde F_{0},\tilde F_{1},\tilde G)} \lesssim 2^{-k}(1+pk),
\]
which would finish the proof.
By the tree selection algorithm in Proposition~\ref{prop:tree-selection} (beginning at some scale $l_{0}\leq 0$ with $\size^{(2)}(\mathbf{P}_{k},\tilde G) \leq 2^{-l_{0}}$) we partition
\[
\mathbf{P}_{k} = \cup_{l=l_{0}}^{\lceil pk \rceil} \cup_{T\in \mathbf{T}_{k,l}} T \cup \mathbf{P}_{k}',
\]
where
\[
\size^{(2)}(T,\tilde G) \lesssim 2^{-l},
\quad
\size^{(1)}(T,\tilde F_{1}) \lesssim \min(1,2^{-l})
\]
for $T\in \mathbf{T}_{k,l}$,
\[
\sum_{T\in \mathbf{T}_{k,l}} \abs{\vec I_{T}} \lesssim 2^{2l},
\]
and
\[
\size^{(2)}(\mathbf{P}_{k}',\tilde G), \size^{(1)}(\mathbf{P}_{k}',\tilde F_{1}) \lesssim 2^{-pk}.
\]
By the single tree estimate \eqref{eq:single-tree-estimate:symmetric} we obtain
\[
\abs{\sum_{l}\sum_{T\in\mathbf{T}_{k,l}} \Lambda^{\epsilon}_{T}(\tilde F_{0},\tilde F_{1}, \tilde G)}
\lesssim
\sum_{l=-\infty}^{\lceil pk \rceil} 2^{2l} 2^{-k} \min(1,2^{-l}) 2^{-l}
\lesssim
2^{-k}(1+pk).
\]
The remaining term can be written as
\[
\abs{\Lambda^{\epsilon}_{\mathbf{P}_{k}'}(\tilde F_{0},\tilde F_{1},\tilde G)}
=
\abs{\sum_{T\in\mathbf{T}_{k}}\Lambda^{\epsilon}_{T\cap \mathbf{P}_{k}'}(\tilde F_{0},\tilde F_{1},\tilde G)}.
\]
Each $T\cap \mathbf{P}_{k}'$ is the disjoint union of a set of trees the union of whose top squares has measure bounded by $\abs{\vec I_{T}}$.
We have
\[
\sum_{T\in\mathbf{T}_{k}} \abs{\vec I_{T}}
\leq
\norm{ \sum_{T\in\mathbf{T}_{k}} 1_{\vec I_{T}} }_{p}^{p}
\lesssim
2^{2pk},
\]
so, again by the single tree estimate \eqref{eq:single-tree-estimate:symmetric},
\[
\abs{\Lambda^{\epsilon}_{\mathbf{P}_{k}'}(\tilde F_{0},\tilde F_{1},\tilde G)}
\lesssim
2^{2pk} 2^{-k} 2^{-pk} 2^{-pk}
=
2^{-k},
\]
finishing the proof of the claim.
\end{proof}

\subsection{Triangles $b_{0}$ and $d_{10}$}
\begin{theorem}
\label{thm:2general-restricted-type:F0}
Let $0< \alpha_{2} \leq 1/2 \leq \alpha_{0} < 1$ and $-1/2<\alpha_{1}<1/2$ satisfy \eqref{eq:Lp-range}.
Then for any measurable sets $E_{i} \subset A_{0}^{2}$, $i\in\{0,1,2\}$ there exists a major subset $E_{1}'\subset E_{1}$ (which can be taken equal to $E_{1}$ if $\alpha_{1}>0$)
such that for any dyadic test functions $\abs{F_{i}} \leq 1_{E_{i}}$ satisfying $\abs{F_{1}} \leq 1_{E_{1}'}$ and either \eqref{eq:F0-diagonal} with $a\in A_{1}\setminus A_{0}$ or \eqref{eq:F0-fiberwise-character} we have
\[
\abs{\Lambda^{\epsilon}(F_0,F_1,F_2)}
\lesssim_{\alpha_{0},\alpha_{1},\alpha_{2}}
\prod_{i=0}^{2} \abs{E_{i}}^{\alpha_{i}},
\]
where the implied constant is independent of the choices of the scalars $\abs{\epsilon_{\vec I}}\leq 1$ with $\epsilon_{\vec I}=0$ whenever $I_{i} \not\subset A_{0}$.
\end{theorem}
\begin{proof}
We can assume $\abs{E_{0}}\leq 2^{-20} \abs{E_{1}}$, since otherwise the conclusion follows from the local $L^{2}$ case with $E_{1}'=E_{1}$.

In case \eqref{eq:F0-fiberwise-character} we can also without loss of generality assume $E_{0} = A_{0} \times \tilde E_{0}$.
Setting $E_{1}' = E_{1} \setminus \tilde E_{0} \times A_{0}$ we get that the left-hand side of the conclusion vanishes identically.

In case \eqref{eq:F0-diagonal} we argue as in the proof of Theorem~\ref{thm:2general-restricted-type} with the roles of indices $0$ and $2$ interchanged.
The main difference from the previous case is that the time-frequency projections in general need not be adapted to the good function $G$.
However, under the additional condition $a\in A_{1}\setminus A_{0}$ we may assume
\[
1_{B_{0}}(x_{1},x_{2}) = 1_{\tilde B_{0}}(x_{2} \wplus (a\wtimes x_{1})),
\]
and then the directional maximal function $\tilde M_{p_{0}}1_{B_{0}}$ coincides with the two-dimensional maximal function $M_{p_{0}}1_{B_{0}}$.
It follows that for every $J\in\DI_{\eset}$ and every bitile $P=\vec I\times\omega_{1}\in\mathbf{P}$ we have either $J\cap I_{1}\times I_{2} = \emptyset$ or $J_{1}\subsetneq I_{1}$, which in turn implies
\[
\Pi^{(0)}_{\mathbf{P}_{k}} \tilde F_{0} = \Pi^{(0)}_{\mathbf{P}_{k}} G.
\]
Thus we may replace $\tilde F_{0}$ by $G$ in the single tree estimates.
\end{proof}

\section{Previously known special cases}
\label{appendix:dyadic}

Let us discuss briefly how our main result specializes to some cases that have already appeared in the literature in a very similar form.

\subsection{Maximally modulated Haar multiplier}
\label{appendix:dyadic:max-mod-haar}
Since the ordinary Haar multipliers
\[
(H^{\epsilon}f)(x) := \sum_{I} \epsilon_I \abs{I}^{-1} \<f,\h_I\> \h_I(x),
\]
where $\abs{\epsilon_I}\leq 1$ for each dyadic interval $I$, constitute a good dyadic model for the Hilbert transform, the maximally modulated Haar multipliers
\begin{equation}
\label{eq:maximal-def}
(H^{\epsilon}_{\star}f)(x) := \sup_{N} \abs{(H^{\epsilon}M_N f)(x)}
\end{equation}
provide a reasonable algebraic model for the Carleson operator, albeit different from the model of truncated Walsh--Fourier series considered e.g.\ in \cite{MR0217510}.
Here $M_N$ simply represents the Walsh modulation operator,
\[
(M_N f)(x) := w_{N}(x) f(x).
\]

Let $\epsilon_{\vec I}=\epsilon_{(I_0,I_1,I_2)}$ depend only on the interval $I_0$ and take two functions $f$ and $g$ on $A_{0}$.
Suppose that $N\colon A_{0}\to\{0,1,2,\ldots\}$ is a choice function that linearizes the supremum in \eqref{eq:maximal-def}.
If we substitute
\begin{align*}
F_0(x_{1},x_{2}) &:= f(x_{1}\wplus x_{2}),\\
F_1(x_{2},x_{0}) &:= \mathop{\mathrm{sgn}}g(x_{0})\sqrt{\abs{g(x_{0})}} \,w_{N(x_{0})}(x_{2}),\\
F_2(x_{0},x_{1}) &:= \sqrt{\abs{g(x_{0})}} \,w_{N(x_{0})}(x_{1}\wplus x_{0})
\end{align*}
into \eqref{eq:THT-def2}, we will obtain for $\Lambda^{\epsilon}(F_0,F_1,F_2)$ the equal expression
\[
\sum_{\vec I\in\DI} \frac{\epsilon_{I_0}}{\abs{I_0}} \iiint f(x_{1}\wplus x_{2}) g(x_{0}) w_{N(x_{0})}(x_{1}\wplus x_{0}) w_{N(x_{0})}(x_{2}) \h_{I_1}(x_{1}) \h_{I_2}(x_{2}) \h_{I_0}(x_{0}) \dif \vec x.
\]
Here and later in this appendix we use the convention $x_{i},y_{i}\in I_{i}$ for integration domains, unless specified otherwise.
By the character property of the Walsh functions and the fact that the Haar functions are simply restrictions of the Rademacher functions to the corresponding intervals this equals
\[
\sum_{\vec I\in\DI} \frac{\epsilon_{I_0}}{\abs{I_0}} \iiint f(x_{1}\wplus x_{2}) g(x_{0}) w_{N(x_{0})}(x_{1}\wplus x_{2}\wplus x_{0}) r_{k}(x_{1}\wplus x_{2}\wplus x_{0}) \dif \vec x.
\]
By changing the variables $y_{0}=x_{1}\wplus x_{2}$ (for fixed $x_{1}$) and observing $y_{0}\in I_1\wplus I_2=I_0$, the above equals
\[
\sum_{\vec I\in\DI} \frac{\epsilon_{I_0}}{\abs{I_0}} \iiint f(y_{0}) g(x_{0}) w_{N(x_{0})}(y_{0}\wplus x_{0}) r_{k}(y_{0}\wplus x_{0}) \dif y_{0} \dif x_{1} \dif x_{0}.
\]
Observe that at each scale $k$ the integral $\sum_{I\in\mathbf{I}_{k}}\int_{x_{1}\in I_{1}}$ can be disregarded as it simply integrates over the union of intervals $I_1$, which is $A_{0}$.
Using the character property once again we obtain
\begin{align*}
&\sum_{I_{0}} \frac{\epsilon_{I_0}}{\abs{I_0}} \iint f(y_{0}) g(x_{0}) w_{N(x_{0})}(y_{0}) w_{N(x_{0})}(x_{0}) \h_{I_0}(y_{0}) \h_{I_0}(x_{0}) \dif y_{0} \dif x_{0}\\
&=
\int_{\W} \sum_{I_0} \frac{\epsilon_{I_0}}{\abs{I_0}} \<w_{N(x_{0})}f,\h_{I_0}\> \h_{I_0}(x_{0}) w_{N(x_{0})}(x_{0}) g(x_{0}) \dif x_{0}\\
&=
\int (M_{N(x_{0})} H^{\epsilon} M_{N(x_{0})}f)(x_{0}) g(x_{0}) \dif x_{0}.
\end{align*}
From the established bound for $\Lambda^{\epsilon}$ in Theorem~\ref{thm:2general} using duality we deduce
\[
\norm{H^{\epsilon}_{\star}f}_{p}\lesssim\norm{f}_p
\quad\text{for any}\quad
1<p<\infty.
\]

\subsection{Walsh model of uniform bilinear Hilbert transform}
\label{sec:walsh-ubht}
Theorem~\ref{thm:2general} implies a bound for the trilinear form
\[
\Lambda^{\epsilon,L}_{\mathrm{BHT}}(f,g,h)
:=
\int \sum_{k} \sum_{\substack{I\in\mathbf{I}_{k}\\ \omega:\abs{\omega}=2^{-k}}}\! \epsilon_{I}
\big(\Pi_{I\times (\omega\wplus 2^{-k})}f\big) \big(\Pi_{I\times (2^{L}\omega)}g\big)
\big(\Pi_{I\times (2^{L}\omega \wplus \omega \wplus 2^{-k})}h\big),
\]
where $\epsilon=(\epsilon_I)_I$ is a sequence of coefficients indexed by dyadic intervals and satisfying $\abs{\epsilon_I}\leq 1$, while $L$ is an arbitrary positive integer.
This observation is interesting because a single estimate for the triangular Hilbert transform implies bounds for a sequence of one-dimensional trilinear forms $\Lambda^{\epsilon,L}_{\mathrm{BHT}}$ with constants independent of $\epsilon$ and $L$.

This form is similar to, but different from the trilinear form studied in \cite{MR2997005}.
As in Section~\ref{appendix:dyadic:max-mod-haar}, the discrepancy is due to the fact that our model is based on the algebraic structure of the Walsh field rather than on the order structure.

In order to apply Theorem~\ref{thm:2general} substitute
\begin{align*}
F_0(x_{1},x_{2}) &:= f(x_{1} \wplus x_{2} \wplus 2^{-L}x_{2}),\\
F_1(x_{2},x_{0}) &:= h(2^{-L}x_{2} \wplus x_{0}),\\
F_2(x_{0},x_{1}) &:= g(x_{0}\wplus 2^{-L}x_{0} \wplus 2^{-L}x_{1})
\end{align*}
into \eqref{eq:THT-def2} to obtain
\begin{align*}
\Lambda^{\epsilon}(F_0,F_1,F_2)
=\sum_{k} 2^{-k} \sum_{\vec I\in\DI_{k}} \epsilon_{\vec I}
\iiint f(x_{1}\wplus x_{2}\wplus 2^{-L}x_{2}) g(x_{0}\wplus 2^{-L}x_{0}\wplus 2^{-L}x_{1}) & \\[-2mm]
h(2^{-L}x_{2}\wplus x_{0}) r_k(x_{1}\wplus x_{2}\wplus x_{0}) \dif x_{1} \dif x_{2} \dif x_{0} & .
\end{align*}
Observe that $x_{i}\in I_i$, $i=0,1,2$, implies
\begin{align*}
& x_{1} \wplus x_{2} \wplus 2^{-L}x_{2}\in I_1\wplus I_2\wplus 2^{-L}I_2 = I_0\wplus 2^{-L}I_2,\\
& x_{0}\wplus 2^{-L}x_{0} \wplus 2^{-L}x_{1}\in I_0\wplus 2^{-L}(I_0\wplus I_1) = I_0\wplus 2^{-L}I_2,\\
& 2^{-L}x_{2} \wplus x_{0}\in I_0\wplus 2^{-L}I_2,
\end{align*}
so we should expand $f,g,h$ into the Walsh-Fourier series on the dyadic interval $I=I_0\wplus 2^{-L}I_2$ of length $2^k$, i.e.\ into the wave packets with fixed eccentricity:
\begin{align*}
f(x_{1}\wplus x_{2}\wplus 2^{-L}x_{2}) &= 2^{-k} \sum_{m_{0}=0}^{\infty} \<f,1_I w_{m_{0} 2^{-k}}\> w_{m_{0} 2^{-k}}(x_{1}\wplus x_{2}\wplus 2^{-L}x_{2}),\\
g(x_{0}\wplus 2^{-L}x_{0}\wplus 2^{-L}x_{1}) &= 2^{-k} \sum_{m_{2}=0}^{\infty} \<g,1_I w_{m_{2} 2^{-k}}\> w_{m_{2} 2^{-k}}(x_{0}\wplus 2^{-L}x_{0}\wplus 2^{-L}x_{1}),\\
h(2^{-L}x_{2}\wplus x_{0}) &= 2^{-k} \sum_{m_{1}=0}^{\infty} \<h,1_I w_{m_{1} 2^{-k}}\> w_{m_{1} 2^{-k}}(2^{-L}x_{2}\wplus x_{0}).
\end{align*}
Inserting these into the previous expression for $\Lambda^{\epsilon}(F_0,F_1,F_2)$ we obtain
\begin{align*}
\sum_{k} 2^{-4k} \sum_{\vec I\in\DI_{k}} \epsilon_{\vec I} \sum_{m_{0},m_{2},m_{1}}
& \<f,1_I w_{m_{0} 2^{-k}}\> \<g,1_I w_{m_{2} 2^{-k}}\> \<h,1_I w_{m_{1} 2^{-k}}\> \\[-2mm]
& \Big(\int_{I_1} e\big((m_{0}\wplus m_{2} 2^{-L}\wplus 1)2^{-k}\wtimes x_{1}\big) dx_{1}\Big)\\
& \Big(\int_{I_2} e\big((m_{0}\wplus m_{0} 2^{-L}\wplus m_{1} 2^{-L}\wplus 1)2^{-k}\wtimes x_{2}\big) dx_{2}\Big)\\
& \Big(\int_{I_0} e\big((m_{2}\wplus m_{2} 2^{-L}\wplus m_{1}\wplus 1)2^{-k}\wtimes x_{0}\big) dx_{0}\Big).
\end{align*}
Since we are integrating over intervals of length $2^k$, the above summands vanish unless
\[
m_{0}\wplus m_{2} 2^{-L}\wplus 1,\ m_{0}\wplus m_{0} 2^{-L}\wplus m_{1} 2^{-L}\wplus 1,\ \text{and}\ m_{2}\wplus m_{2} 2^{-L}\wplus m_{1}\wplus 1
\]
all belong to $A_{0}$, which is easily seen to be equivalent to the conditions
\[
m_{0}\wplus m_{2}\wplus m_{1}=0\quad \text{and}\quad m_{0}\wplus m_{2} 2^{-L}\wplus 1\in A_{0}.
\]
Moreover, in that case the three functions under the integrals over $I_1,I_2,I_0$ are precisely the constants
\[
2^{k} e\big((m_{0}\wplus m_{2} 2^{-L}\wplus 1)2^{-k}\wtimes l(I_i)\big),\ \ i=1,2,0,
\]
where $l(I_i)$ is the left endpoint of $I_i$.
Because of $0\in I_0\wplus I_1\wplus I_2$ they multiply to $2^{3k}$.
Allow the coefficients $\epsilon_{\vec I}$ to depend on $I=I_0\wplus 2^{-L}I_2$ only and observe that each interval $I\in\mathbf{I}_{k}$ appears for exactly $2^{-k}$ choices of $\vec I$ as they range over $\DI_{k}$.
(Indeed, $I_2$ is arbitrary and $I_0,I_1$ are then uniquely determined.)
We end up with
\[
\sum_{k} 2^{-2k} \sum_{I\in\mathbf{I}_{k}} \epsilon_{I} \sum_{\substack{m_{0},m_{2},m_{1}\\ m_{0}\wplus m_{2}\wplus m_{1}=0\\ m_{0}\wplus m_{2} 2^{-L}\wplus 1\in A_{0}}}
\<f,1_I w_{m_{0} 2^{-k}}\> \<g,1_I w_{m_{2} 2^{-k}}\> \<h,1_I w_{m_{1} 2^{-k}}\>,
\]
i.e., by substituting $m=m_{0}\wplus 1$ and $n=m_{2}\wplus(m_{0}\wplus 1)2^L$,
\begin{align}
\label{eq:BHTcase-expanded}
\sum_{k} 2^{-2k} \sum_{I\in\mathbf{I}_{k}} \epsilon_{I} \sum_{\substack{m,n\\ 0\leq n<2^L}}
\<f,1_I w_{(m\wplus 1)2^{-k}}\> \<g,1_I w_{(m 2^L\wplus n)2^{-k}}\> & \\[-3mm]
\<h,1_I w_{(m 2^L\wplus n\wplus m\wplus 1)2^{-k}}\> & . \nonumber
\end{align}

On the other hand, we can start from $\Lambda^{\epsilon,L}_{\mathrm{BHT}}$ and write the dyadic interval $\omega$ explicitly as $\omega=[m 2^{-k},(m+1)2^{-k})$.
The three time-frequency projections appearing in the definition can be expanded using vertical decompositions into tiles as:
\begin{align*}
\Pi_{I\times (\omega\wplus 2^{-k})}f &= 2^{-k}
\<f,1_I w_{(m\wplus 1)2^{-k}}\> 1_I w_{(m\wplus 1)2^{-k}},\\
\Pi_{I\times (2^{L}\omega)}g &= 2^{-k} \sum_{n=0}^{2^L-1}
\<g,1_I w_{(m 2^L\wplus n)2^{-k}}\> 1_I w_{(m 2^L\wplus n)2^{-k}},\\
\Pi_{I\times (2^{L}\omega \wplus \omega \wplus 2^{-k})}h &= 2^{-k} \sum_{n'=0}^{2^L-1}
\<h,1_I w_{(m 2^L\wplus n'\wplus m\wplus 1)2^{-k}}\> 1_I w_{(m 2^L\wplus n'\wplus m\wplus 1)2^{-k}}.
\end{align*}
Observe that the integral
\[
\int \big(\Pi_{I\times (\omega\wplus 2^{-k})}f\big) \big(\Pi_{I\times (2^{L}\omega)}g\big)
\big(\Pi_{I\times (2^{L}\omega \wplus \omega \wplus 2^{-k})}h\big)
\]
is equal to
\[
2^{-2k} \sum_{n=0}^{2^L-1} \<f,1_I w_{(m\wplus 1)2^{-k}}\> \<g,1_I w_{(m 2^L\wplus n)2^{-k}}\>
\<h,1_I w_{(m 2^L\wplus n\wplus m\wplus 1)2^{-k}}\>,
\]
since the terms with $n\neq n'$ disappear.
That way we arrive at \eqref{eq:BHTcase-expanded} once again, completing the proof of
$\Lambda^{\epsilon}(F_0,F_1,F_2)=\Lambda^{\epsilon,L}_{\mathrm{BHT}}(f,g,h)$.

\subsection{Endpoint counterexample}
The observation from the previous section is also useful to explain the failure of some estimates at the boundary of the Banach triangle.
By formally taking $L\to\infty$ we are motivated to substitute
\[
F_0(x_{1},x_{2}) := f(x_{1} \wplus x_{2}),\ 
F_1(x_{2},x_{0}) := h(x_{0}),\ 
F_2(x_{0},x_{1}) := g(x_{0}),
\]
in which case \eqref{eq:THT-def2} becomes
\begin{align*}
& \sum_{\vec I\in\DI} \epsilon_{I_0} \abs{I_0}^{-1} \Big(\iint f(x_{1}\wplus x_{2}) \h_{I_1}(x_{1})\h_{I_2}(x_{2}) \dif x_{1} \dif x_{2}\Big)
\Big(\int g(x_{0})h(x_{0}) \h_{I_0}(x_{0}) \dif x_{0}\Big)\\
& = \sum_{I_0} \epsilon_{I_0} \abs{I_0}^{-1} \<f,\h_{I_0}\> \<gh,\h_{I_0}\>
= \int f(x) H^{\epsilon}(gh)(x) \dif x.
\end{align*}
Since Haar multipliers are generally not bounded on $L^1$, we see that Estimate \eqref{eq:lp-estimate} cannot hold when $p_0=\infty$.

The positive results in this limiting case do not reveal the true structural complexity of $\Lambda^{\epsilon}$.
Indeed, when one of the functions depends on a single variable alone (such as $F_0(x,y)=x$), then the triangle ``breaks'' immediately.
No techniques from time-frequency analysis are required to bound such degenerate cases, even though they correspond both to the limiting case $a\to\infty$ and to the special case $N\equiv 0$ in Theorem~\ref{thm:2general}.


%% file: poly-carleson.tex
\chapter{Maximal polynomial modulations of singular integrals}
\label{chap:poly-car}
In this chapter we prove Theorem~\ref{thm:poly-car}.
Hence we assume throughout that $\CZK$ is a $\tau$-H\"older continuous Calder\'on--Zygmund kernel on $\R^{\ds}$ whose truncations define $L^{2}$ bounded operators.

By Theorem~\ref{thm:BT-VV}, Theorem~\ref{thm:poly-car} is a consequence of the following localized $L^{2}$ estimates.
\begin{theorem}
\label{thm:loc}
Let $0 \leq \alpha < 1/2$ and $0 < \nu,\kappa \leq 1$.
Let $F,G \subset\R^{\ds}$ be measurable subsets and $\tilde{F} := \Set{ M\one_{F} > \kappa }$, $\tilde{G} := \Set{ M\one_{G} > \nu }$.
Then
\begin{align}
\norm{T}_{2 \to 2} &\lesssim 1, \label{eq:loc:full}\\
\norm{\one_{G} T \one_{\R^{\ds}\setminus \tilde{G}}}_{2 \to 2} &\lesssim_{\alpha} \nu^{\alpha}, \label{eq:loc:G}\\
\norm{\one_{\R^{\ds} \setminus \tilde{F}} T \one_{F} }_{2 \to 2} &\lesssim_{\alpha} \kappa^{\alpha}. \label{eq:loc:F}
\end{align}
\end{theorem}
The estimate \eqref{eq:loc:full} is a special case of both \eqref{eq:loc:G} and \eqref{eq:loc:F}, but we formulate and prove it separately because it is the easiest case.

\section{Discretization}
Modifying the notation used in the introduction, we denote by $\calQ$ the vector space of all real polynomials in $\ds$ variables of degree at most $d$ modulo $+\R$.
That is, we identify two polynomials if and only if their difference is constant.
This identification is justified by the fact that the absolute value of the integral in \eqref{eq:Car-op-Kd} does not depend on the constant term of $Q$.
Notice that $Q(x)-Q(x')\in\R$ is well-defined for $Q\in\calQ$ and $x,x'\in\R^{\ds}$.

Let $D=D(d,\ds)$ be a large integer to be chosen later.
Let $\psi$ be a smooth function supported on the interval $[1/(4D),1/2]$ such that $\sum_{s\in\Z} \psi(D^{-s}\cdot) \equiv 1$ on $(0,\infty)$.
Then the kernel can be decomposed as
\[
\CZK(x,y) = \sum_{s\in\Z} \CZK_s(x,y)
\text{ with }
\CZK_{s}(x,y) := \CZK(x,y) \psi(D^{-s} \abs{x-y}).
\]
The functions $\CZK_{s}$ are supported on the sets $\Set{(x,y)\in\R^{\ds}\times\R^{\ds} \given D^{s-1}/4<\abs{x-y}<D^{s}/2}$ and satisfy
\begin{equation}
\label{eq:Ks-size}
\abs{\CZK_{s}(x,y)} \lesssim D^{-\ds s}
\text{ for all } x, y \in \R^{\ds},
\end{equation}
\begin{equation}
\label{eq:Ks-reg}
\abs{\CZK_{s}(x,y)-\CZK_{s}(x',y)}+\abs{\CZK_{s}(y,x)-\CZK_{s}(y,x')} \lesssim \frac{\abs{x-x'}^{\tau}}{D^{(\ds+\tau)s}}
\text{ for all } x, x', y \in \R^{\ds}.
\end{equation}
We can replace the maximal operator \eqref{eq:Car-op-Kd} by the smoothly truncated operator
\begin{equation}
\label{eq:T-smooth-trunc}
Tf(x):=\sup_{Q\in\calQ_{d}} \sup_{\smin \leq \smax \in \Z}
\abs[\Big]{ \sum_{s=\smin(x)}^{\smax(x)} \int \CZK_{s}(x,y) e(Q(y)) f(y) \dif y},
\end{equation}
where $e(t)=e^{2\pi i t}$ denotes the standard character on $\R$,
at the cost of an error term that is controlled by the Hardy--Littlewood maximal operator $M$ (see Appendix~\ref{sec:HL-loc-est} for the required localized estimates for $M$).

Since the absolute value of the integral in \eqref{eq:T-smooth-trunc} is a continuous function of $Q$, we may restrict $\smin,\smax,Q$ to a finite set as long as we prove estimates that do not depend on this finite set.
After these preliminary reductions we can linearize the supremum in \eqref{eq:T-smooth-trunc} and replace that operator by
\begin{equation}
\label{eq:T-linearized}
Tf(x):=\sum_{s=\smin(x)}^{\smax(x)} \int \CZK_{s}(x,y) e(Q_{x}(x)-Q_{x}(y)) f(y) \dif y,
\end{equation}
where $\smin,\smax : \R^{\ds}\to\Z$, $Q_{\cdot}:\R^{\ds}\to\calQ$ are measurable functions with finite range.
Let $\sumin := \min_{x\in\R^{\ds}} \smin(x) > -\infty$ and $\sumax := \max_{x\in\R^{\ds}} \smax(x) < +\infty$.
All stopping time constructions will start at the largest scale $\sumax$ and terminate after finitely many steps at the smallest scale $\sumin$.

\subsection{Tiles}
\label{sec:tiles}
The grid of $D$-adic cubes in $\R^{\ds}$ will be denoted by
\[
\calD := \bigcup_{s \in \Z} \calD_{s},
\quad
\calD_{s} := \Set[\big]{ \prod_{i=1}^{\ds}[D^{s}a_{i},D^{s}(a_{i}+1)) \given a_{1},\dotsc,a_{\ds}\in\Z}.
\]
We denote elements of $\calD$ by the letters $I$, $J$ and call them \emph{grid cubes}.
The unique integer $s=\scale(I)$ such that $I\in\calD_{s}$ will be called the \emph{scale} of a grid cube.
The \emph{parent} of a grid cube $I$ is the unique grid cube $\hat{I} \supset I$ with $\scale(\hat{I}) = \scale(I)+1$.
The side length of a cube $I$ is denoted by $\ell(I)$.
If $I$ is a cube and $a>0$, then $aI$ denotes the concentric cube with side length $a\ell(I)$.

For every bounded subset $I\subset\R^{\ds}$ we define a norm on $\calQ$ by
\begin{equation}
\label{eq:normI}
\norm{Q}_{I} := \sup_{x,x'\in I} \abs{Q(x)-Q(x')},
\quad
Q\in\calQ.
\end{equation}

\begin{lemma}
\label{lem:normQ}
If $Q\in\calQ$ and $B(x,r) \subset B(x,R) \subset \R^{\ds}$, then
\begin{align}
\norm{Q}_{B(x,R)}
&\lesssim_{d} \label{eq:normQ:up}
(R/r)^{d} \norm{Q}_{B(x,r)},\\
\norm{Q}_{B(x,r)}
&\lesssim_{d} \label{eq:normQ:low}
(r/R) \norm{Q}_{B(x,R)}.
\end{align}
\end{lemma}

\begin{proof}
By translation we may assume $x=0$, and we choose a representative for the congruence class modulo $+\R$ with $Q(0)=0$.
Fixing $y\in\R^{\ds}$ with $\norm{y}=1$ and considering the one-variable polynomial $Q(\cdot y)$ we may also assume $\ds=1$.

To show \eqref{eq:normQ:up} suppose by scaling that $r=1$ and $\norm{Q}_{B(x,r)}=1$.
The coefficients of $Q$ can now be recovered from its values on the unit ball using the Lagrange interpolation formula.
In particular these coefficients are bounded by a ($d$-dependent) constant, and the conclusion follows.

Similarly, to show \eqref{eq:normQ:low} suppose by scaling that $R=1$ and $\norm{Q}_{B(x,R)}=1$.
Then the coefficients of $Q$ are $O(1)$ and the conclusion follows.
\end{proof}

\begin{corollary}
\label{cor:normQ}
If $D$ is sufficiently large, then for every $I \in \calD$ and $Q \in \calQ$ we have
\begin{equation}
\label{eq:normQ:parent}
\norm{Q}_{\hat{I}} \geq 10^{4} \norm{Q}_{I}.
\end{equation}
\end{corollary}

We choose $D$ so large that \eqref{eq:normQ:parent} holds.

\begin{definition}
\label{def:pair}
A \emph{pair} $\Tp$ consists of a \emph{spatial cube} $I_{\Tp}\in\calD$ and a Borel measurable subset $\calQ(\Tp) \subset \calQ$ that will be called the associated \emph{uncertainty region}.
Abusing the notation we will say that $Q\in \Tp$ if and only if $Q\in \calQ(\Tp)$.
Also, $\scale(\Tp):=\scale(I_{\Tp})$.

\begin{lemma}
\label{lem:tile}
There exist collections of pairs $\TP_{I}$ indexed by the grid cubes $I\in\calD$ with $\sumin \leq \scale(I) \leq \sumax$ such that
\begin{enumerate}
\item\label{def:tile:ball} To each $\Tp \in \TP_{I}$ is associated a \emph{central polynomial} $Q_{\Tp}\in\calQ$ such that
\begin{equation}
\label{eq:tile:ball}
B_{I}(Q_{\Tp},0.2) \subset \calQ(\Tp) \subset B_{I}(Q_{\Tp},1),
\end{equation}
where $B_{I}(Q,r)$ denotes the ball with center $Q$ and radius $r$ with respect to the norm \eqref{eq:normI},
\item\label{def:tile:cover} for each grid cube $I \in \calD$ the uncertainty regions $\Set{\calQ(\Tp) \given \Tp\in\TP_{I}}$ form a disjoint cover of $\calQ$, and
\item\label{def:tile:nested} if $I \subseteq I'$, $\Tp \in \TP_{I}$, $\Tp' \in \TP_{I'}$, then either $\calQ(\Tp) \cap \calQ(\Tp') = \emptyset$ or $\calQ(\Tp) \supseteq \calQ(\Tp')$.
\end{enumerate}
\end{lemma}
This is similar to the construction of Christ grid cubes but easier because we can start at a smallest scale and we do not need a small boundary property.

The requirement \eqref{eq:tile:ball} on the uncertainty regions $\calQ(\Tp)$ is dictated by Lemma~\ref{lem:osc-int}.
The uncertainty regions used in \cite{MR2545246,arXiv:1105.4504} in the case $\ds=1$ also satisfy \eqref{eq:tile:ball} up to multiplicative constants.
However, it seems to be convenient not to prescribe the exact shape of the uncertainty regions in order to obtain the nestedness property \eqref{def:tile:nested}.
\begin{proof}
For each $I\in\calD$ choose a maximal $0.7$-separated subset $\calQ_{I} \subset \calQ$ with respect to the $I$-norm.

We start with the cubes $I \in \calD_{\sumax}$.
Fix $I\in \calD_{\sumax}$.
Then the balls $B_{I}(Q,0.3)$, $Q\in \calQ_{I}$, are disjoint, and the balls $B_{I}(Q,0.7)$, $Q\in \calQ_{I}$, cover $\calQ$.
Hence there exists a disjoint cover $\calQ = \cup_{Q\in \calQ_{I}} \calQ(I,Q)$ such that $B_{I}(Q,0.3) \subset \calQ(I,Q) \subset B_{I}(Q,0.7)$.
We use the cells of this partition as uncertainty regions of the pairs that we set out to construct.

Suppose now that $\TP_{I'}$ has been constructed for some $I'\in\calD$ and let $I\in\calD$ be a grid cube contained in $I'$ with $\scale(I) = \scale(I') - 1$.
Using \eqref{eq:normQ:parent} we construct a partition $\calQ_{I'} = \cup_{Q\in \calQ_{I}} \ch(I,Q)$ such that for each $Q\in \calQ_{I}$ and $Q' \in \calQ_{I'}$ we have
\[
Q' \in B_{I}(Q,0.3)
\implies
Q' \in \ch(I,Q)
\implies
Q' \in B_{I}(Q,0.7).
\]
Then the cells $\calQ(I,Q) := \cup_{Q' \in \ch(I,Q)} \calQ(I',Q')$ partition $\calQ$ and we use these cells as uncertainty regions of the pairs in $\TP_{I}$.
\end{proof}

\begin{definition}
\label{def:tile}
We write
\[
\TP := \bigcup_{s=\sumin}^{\sumax} \bigcup_{I\in\calD_{s}} \TP_{I}
\]
and call members of $\TP$ \emph{tiles}.
\end{definition}

For a pair $\Tp$ let
\[
E(\Tp) := \Set{x\in I_{\Tp} \given Q_x\in \calQ(\Tp) \land \smin(x) \leq \scale(\Tp) \leq \smax(x)}.
\]
For every tile $\Tp\in\TP$ we define the corresponding operator
\begin{equation}
\label{eq:Ttile}
T_{\Tp}f(x) := \one_{E(\Tp)}(x) \int e(Q_{x}(x)-Q_{x}(y)) \CZK_{\scale(\Tp)}(x,y) f(y) \dif y.
\end{equation}
\end{definition}
The tile operators and their adjoints
\begin{equation}
\label{eq:Ttile*}
T_{\Tp}^{*}g(y) = \int e(-Q_{x}(x)+Q_{x}(y)) \overline{\CZK_{\scale(\Tp)}(x,y)} (\one_{E(\Tp)}g)(x) \dif x.
\end{equation}
have the support properties
\begin{equation}
\label{eq:Ttile-supp}
\supp T_{\Tp}f \subseteq I_{\Tp},
\qquad
\supp T_{\Tp}^{*}g \subseteq I_{\Tp}^{*} := 2 I_{\Tp}
\end{equation}
for any $f,g \in L^{2}(\R^{\ds})$.
For a collection of tiles $\TC \subset \TP$ we write $T_{\TC} := \sum_{\Tp\in\TC} T_{\Tp}$.
Then the linearized operator \eqref{eq:T-linearized} can be written as $T_{\TP}$.

\subsection{General notation}
The characteristic function of a set $I$, as well as the corresponding multiplication operator, is denoted by $\one_{I}$.
The \emph{Hardy--Littlewood maximal operator} is given by
\[
Mf(x) := \sup_{x\in I}\frac{1}{\abs{I}}\int_{I}\abs{f},
\]
the latter supremum being taken over all (not necessarily grid) cubes containing $x$.
For $1<q<\infty$ the \emph{$q$-maximal operator} is given by
\begin{equation}
\label{eq:HL-q-max-op}
M_{q}f := (M \abs{f}^{q})^{1/q}.
\end{equation}

Parameters $\epsilon,\eta$ (standing for small numbers) and $C$ (standing for large numbers) are allowed to change from line to line, but may only depend on $d,\ds,\tau$ and the implicit constants related to $\CZK$ unless an additional dependence is indicated by a subscript.

For $A,B>0$ we write $A\lesssim B$ (resp. $A\gtrsim B$) in place of $A<CB$ (resp. $A>CB$).
If the constant $C=C_{\delta}$ depends on some quantity $\delta$, then we may write $A\lesssim_{\delta}B$.

The operator norm on $L^{2}(\R^{\ds})$ is denoted by $\norm{T}_{2\to 2} := \sup_{\norm{f}_{2} \leq 1} \norm{Tf}_{2}$.

\section{Tree selection algorithm}
\label{sec:dens-selection}

\subsection{Spatial decomposition}
\label{sec:spatial-decomposition}
We begin with a simplified version of Lie's stopping time construction from \cite{arXiv:1105.4504}.

\begin{definition}\label{def:ord}
Let $\Tp,\Tp'$ be pairs.
We say that
\begin{align*}
\Tp< \Tp' &:\iff I_{\Tp}\subsetneq I_{\Tp'} \text{ and } \calQ(\Tp') \subseteq \calQ(\Tp),\\
\Tp\leq \Tp' &:\iff I_{\Tp}\subseteq I_{\Tp'} \text{ and } \calQ(\Tp') \subseteq \calQ(\Tp).
\end{align*}
\end{definition}
The relations $<$ and $\leq$ are transitive, similarly to \cite{MR0340926} and differently from \cite{arXiv:1105.4504}.

\begin{definition}
A \emph{stopping collection} is a subset $\calF \subset\calD$ of the form $\calF=\cup_{k\geq 0}\calF_{k}$, where each $\calF_{k}$ is a collection of pairwise disjoint cubes such that for each $F\in\calF_{k+1}$ there exists $F'\in\calF_{k}$ with $F'\supsetneq F$ ($F'$ is called the \emph{stopping parent of $F$}).
The collection of \emph{stopping children} of $F\in\calF_{k}$ is $\ch_{\calF}(F) := \Set{ F'\in\calF_{k+1} \given F'\subset F}$.
More generally, the collection of \emph{stopping children} of $I\in\calD$ is $\ch_{\calF}(I) := \Set{ F\in\calF \text{ maximal} \given F\subsetneq I }$.
We denote by $\ch^{m}$ the set of children of $m$-th generation, that is, $\ch^{0}(I) := \Set{I}$, $\ch^{m+1}(I) := \cup_{I'\in\ch^{m}(I)} \ch(I')$.
\end{definition}
\begin{lemma}
\label{lem:spatial-decomposition}
There exists a stopping collection $\calF$ with the following properties.
\begin{enumerate}
\item $\calF_{0} = \calD_{\sumax}$.
\item For each $F\in\calF$ we have
\begin{equation}
\label{eq:Lie-support-decay}
\sum_{F'\in\ch(F)} \abs{F'} \leq D^{-10 \ds} \abs{F}.
\end{equation}
\item\label{it:stopping-neighbors} For each $k$ and $F\in\calF_{k}$ and $F'\in\calD$ such that $\scale(F') < \scale(F)$ and $F\cap 5F' \neq\emptyset$ there exists $F''\in\calF_{k}$ such that $\scale(F'') \geq \scale(F)-1$ and $F'\subseteq F''$.
\item For $k\geq 0$ consider the set of grid cubes
\begin{equation}
\label{eq:Ck}
\calC_{k} := \tilde\calC_{k} \setminus \tilde\calC_{k+1},
\quad
\tilde\calC_{k} := \Set{ I\in\calD \given \exists F\in\calF_{k} : I \subseteq F}
\end{equation}
and the corresponding set of tiles
\begin{equation}
\label{eq:Pnk}
\TP_{k} := \Set{ \Tp\in\TP \given I_{\Tp} \in \calC_{k}}.
\end{equation}
Then for every $n\geq 1$ the set of tiles
\begin{equation}
\label{eq:Mnk}
\TM_{n,k} := \Set{ \Tp\in\TP_{k} \text{ maximal w.r.t.\ ``$<$'' } \given \abs{E(\Tp)}/\abs{I_{\Tp}} \geq 2^{-n} }
\end{equation}
satisfies
\begin{equation}
\label{eq:Mnk-counting}
\norm[\big]{\sum_{\Tp\in\TM_{n,k}} \one_{I_{\Tp}}}_{\infty} \lesssim 2^{n} \log (n+1).
\end{equation}
\end{enumerate}
\end{lemma}
The stopping property \eqref{it:stopping-neighbors} can be informally stated by saying that each stopping cube is completely surrounded by stopping cubes of the same generation $k$ and similar (up to $\pm 1$) scale.
This is very useful for handling tail estimates.
\begin{proof}
We start with $\calF_{0} := \calD_{\sumax}$ being the set of all cubes of the maximal spatial scale.
Let now $k\geq 0$ and suppose that $\calF_{k}$ has been constructed already.
Let $\tilde\TM_{n,k}$ be the collection of the $<$-maximal tiles $\Tp\in\TP$ with $\frac{\abs{E(\Tp)}}{\abs{I_\Tp}}\geq 2^{-n}$ and $I_{\Tp}\in \tilde\calC_{k}$.
Since the sets $E(\Tp)$ corresponding to $\Tp\in\tilde\TM_{n,k}$ are pairwise disjoint, we have the Carleson packing condition
\[
\sum_{\Tp\in\tilde\TM_{n,k} : I_{\Tp}\subseteq J} \abs{I_{\Tp}}
\leq
2^{n}
\sum_{\Tp\in\tilde\TM_{n,k} : I_{\Tp}\subseteq J} \abs{E(\Tp)}
\leq
2^{n} \abs{J}
\text{ for every }
J\in\calD.
\]
Let $C$ be a large constant to be chosen later and for $F\in\calF_{k}$ let
\[
B(F) :=
\bigcup_{n\geq 1}\Set[\Big]{\sum_{\Tp\in\tilde\TM_{n,k} : I_{\Tp}\subseteq F} \one_{I_{\Tp}} \geq C2^{n}\log(n+1) }.
\]
By the John--Nirenberg inequality we obtain
\[
\abs{B(F)}
\lesssim
\sum_{n\geq 1} e^{-c \frac{C2^{n}\log(n+1)}{2^{n}}} \abs{F}
\lesssim
\Big(\sum_{n\geq 1} (n+1)^{-cC}\Big) \abs{F}.
\]
The numerical constant on the right-hand side can be made arbitrarily small by taking $C$ sufficiently large.
Let $\calJ(F) \subset \calD$ be the set of grid cubes contained in $B(F)$ and $\calJ' := \cup_{F\in\calF_{k}} \calJ(F)$.
Let $\calJ'' \subset \calD$ be the minimal collection such that $\calJ'\subseteq\calJ''$ and $\calJ''$ satisfies part \ref{it:stopping-neighbors} of the conclusion of this lemma.
Let $\calF_{k+1}$ consist of the maximal cubes in $\calJ''$.
The claimed properties can now be routinely verified.
\end{proof}

\subsection{Fefferman forest selection}
\label{sec:tree-selection}

A set of tiles $\TA\subset\TP$ is called an \emph{antichain} if no two tiles in $\TA$ are related by ``$<$'' (this is the standard order theoretic term for a concept already used in \cite{MR0340926} under a different name).
A set of tiles $\TC\subset\TP$ is called \emph{convex} if
\[
\Tp_1,\Tp_2\in\TC, \Tp\in\TP, \Tp_1< \Tp < \Tp_2 \implies \Tp\in\TC.
\]
We call a subset $\TD\subset\TC$ of a convex set $\TC\subset\TP$ a \emph{down subset} if $\Tp < \Tp'$ with $\Tp\in\TC$ and $\Tp'\in\TD$ implies $\Tp\in\TD$.
Unions of down subsets are again down subsets.
Both down subsets and their relative complements are convex.

For $a\geq 1$ and a tile $\Tp$ we will write $a\Tp$ for the pair $(I_{\Tp},B_{I_{\Tp}}(Q_{\Tp},a))$.
Counterintuitively, for $a' \geq a \geq 1$ and a tile $\Tp$ we have $a'\Tp \leq a\Tp$; this notational inconsistency cannot be avoided without breaking the convention used in all time-frequency analysis literature starting with \cite{MR0340926}.

\begin{definition}\label{def:tree}
A \emph{tree} (of generation $k$) is a convex collection of tiles $\TT\subset\TP_{k}$ together with a \emph{top tile} $\Tp_0 = \top\TT \in \TP_{k}$ such that for all $\Tp\in\TT$ we have $4\Tp<\Tp_0$.
To each tree we associate the \emph{central polynomial} $Q_{\TT} = Q_{\top\TT}$ and the spatial cube $I_{\TT} = I_{\top \TT}$.
\end{definition}

\begin{definition}
\label{def:dist}
For $\Tp\in\TP$ and $Q\in\calQ$ we write
\[
\Delta(\Tp,Q) := \norm{Q_{\Tp}-Q}_{I_{\Tp}} + 1.
\]
\end{definition}

\begin{definition}\label{def:sep}
Two trees $\TT_1$ and $\TT_2$ are called \emph{$\Delta$-separated} if
\begin{enumerate}
\item $\Tp\in\TT_1\ \land\ I_{\Tp}\subseteq I_{\TT_{2}} \implies \Delta(\Tp,Q_{\TT_2})>\Delta$ and
\item $\Tp\in\TT_2\ \land\ I_{\Tp}\subseteq I_{\TT_{1}} \implies \Delta(\Tp,Q_{\TT_1})>\Delta$.
\end{enumerate}
\end{definition}

\begin{remark}
If $I_{\TT_{1}}\cap I_{\TT_{2}}=\emptyset$, then $\TT_{1}$ and $\TT_{2}$ are $\Delta$-separated for any $\Delta$.
\end{remark}

\begin{definition}\label{def:F-forest}
Let $n,k\in\N$.
A \emph{Fefferman forest of level $n$ and generation $k$} is a disjoint union $\TF=\cup_{j}\TT_j$ of $2^{C n}$-separated trees $\TT_{j} \subset \TP_{k}$ (with a large constant $C$ to be chosen later) such that
\begin{equation}\label{eq:F-forest-counting}
\norm[\big]{\sum_{j}\one_{I_{\TT_{j}}}}_{\infty} \leq C2^n \log(n+1)
\end{equation}
with the absolute constant $C$ in \eqref{eq:F-forest-counting} is the same as in \eqref{eq:Mnk-counting}.
\end{definition}

\begin{definition}\label{def:mass}
We define the \emph{maximal density} of a tile $\Tp\in\TP$ by
\begin{equation}\label{eq:densk}
\mdens_{k}(\Tp):=\sup_{\lambda\geq 2} \lambda^{-\dim\calQ} \sup_{\Tp'\in\TP_{k} : \lambda\Tp \leq \lambda\Tp'} \frac{\abs{E(\lambda \Tp')}}{\abs{I_{\Tp'}}}.
\end{equation}
We also write $\mdens_{k}(\TS) = \sup_{\Tp\in\TS} \mdens_{k}(\Tp)$ for sets of tiles $\TS\subset\TP_{k}$.
The subset of ``heavy'' tiles is defined by
\begin{equation}
\label{eq:heavy}
\TH_{n,k} := \Set{ \Tp\in\TP_{k} \given \mdens_{k}(\Tp) > C_{0} 2^{-n}},
\end{equation}
where $C_{0} = C_{0}(d,\ds) > 1$ is a sufficiently large constant to be chosen later.
\end{definition}
The maximal density is monotonic in the sense that if $\Tp_{1}\leq\Tp_{2}$ are in $\TP_{k}$, then $\mdens_{k}(\Tp_{1}) \geq \mdens_{k}(\Tp_{2})$.
Indeed, in this case by \eqref{eq:normQ:parent} we have $\lambda\Tp_{1}\leq\lambda\Tp_{2}$ for every $\lambda\geq 2$, and the claim follows by transitivity of $\leq$.
It follows that each set $\TH_{n,k} \subset \TP_{k}$ is a down subset, and in particular convex.

\begin{proposition}
\label{prop:fef-forest}
For every $n \geq 1$ and every $k\geq 0$ the set $\TH_{n,k}$ can be represented as the disjoint union of $O(n^{2})$ antichains and $O(n)$ Fefferman forests of level $n$ and generation $k$.
\end{proposition}
\begin{proof}
We would like to avoid the $\lambda$-dilates in Definition~\ref{def:mass}.
To this end we consider the down subset of $\TP_{k}$
\[
\TC_{n,k} := \Set{ \Tp\in\TP_{k} \given \exists \Tm\in\TM_{n,k} : 2\Tp < 100\Tm }.
\]
We claim that the remaining set of tiles $\TH_{n,k}\setminus\TC_{n,k}$ can be partitioned into at most $n$ antichains.
Indeed, otherwise there exists a chain $\Tp_{0}<\dotsb<\Tp_{n}$ inside $\TH_{n,k}\setminus\TC_{n,k}$.
By definition \eqref{eq:densk} there exists $\lambda\geq 2$ and a tile $\Tp'\in\TP_{k}$ such that $\lambda\Tp_{n} \leq \lambda\Tp'$ and
\begin{equation}
\label{eq:largeElambdaTp'}
\abs{E(\lambda \Tp')}/\abs{I_{\Tp'}} > C_{0} 2^{-n} \lambda^{\dim\calQ}.
\end{equation}
It follows e.g.\ from the existence of the John ellipsoid associated to the unit ball of the norm $\norm{\cdot}_{I_{\Tp'}}$ that the set $\calQ(\lambda \Tp')$ can be covered by $O(\lambda^{\dim\calQ})$ uncertainty regions of the form $\calQ(\Tp'')$, where $\Tp'' \in \TP_{k}$ are tiles with $I_{\Tp''}=I_{\Tp'}$ and $\norm{Q_{\Tp'}-Q_{\Tp''}}_{I_{\Tp'}} \leq \lambda+1$.
It follows that for at least one such tile we have $\abs{E(\Tp'')} \gtrsim C_{0} 2^{-n} \abs{I_{\Tp''}}$, so that $\abs{E(\Tp'')} > 2^{-n} \abs{I_{\Tp''}}$ provided that $C_{0}$ in \eqref{eq:heavy} is sufficiently large.
By definition \eqref{eq:Mnk} there exists $\Tm \in \TM_{n,k}$ with $\Tp'' \leq \Tm$.

From \eqref{eq:largeElambdaTp'} we obtain
\[
\lambda \leq \lambda^{\dim\calQ} < 2^{n} \abs{E(\lambda \Tp')}/\abs{I_{\Tp'}} \leq 2^{n},
\]
and it follows from \eqref{eq:normQ:parent} that for all $Q\in \calQ(100 \Tm)$ we have
\begin{align*}
\norm{Q_{\Tp_{0}}-Q}_{I_{\Tp_{0}}}
&\leq
\norm{Q_{\Tp_{0}}-Q_{\Tp_{n}}}_{I_{\Tp_{0}}}
+
\norm{Q_{\Tp_{n}}-Q_{\Tp'}}_{I_{\Tp_{0}}}
+
\norm{Q_{\Tp'}-Q_{\Tp''}}_{I_{\Tp_{0}}}\\
&\quad+
\norm{Q_{\Tp''}-Q_{\Tm}}_{I_{\Tp_{0}}}
+
\norm{Q_{\Tm}-Q}_{I_{\Tp_{0}}}\\
&\leq
1
+
10^{-4n}( \norm{Q_{\Tp_{n}}-Q_{\Tp'}}_{I_{\Tp_{n}}}
+
\norm{Q_{\Tp'}-Q_{\Tp''}}_{I_{\Tp'}}\\
&\quad+
\norm{Q_{\Tp''}-Q_{\Tm}}_{I_{\Tp''}}
+
\norm{Q_{\Tm}-Q}_{I_{\Tm}} )\\
&\leq
1 + 10^{-4n}( \lambda + (\lambda+1) + 1 + 100 )
\leq
2.
\end{align*}
Hence $2 \Tp_{0} \leq 100 \Tm$, contradicting the choice $\Tp_{0}\not\in\TC_{n,k}$.

We want to show that $\TC_{n,k}$ can be decomposed into $O(n)$ Fefferman forests and $O(n^{2})$ antichains; then since $\TH_{n,k}$ is convex the same will hold for $\TH_{n,k}\cap\TC_{n,k}$.
Let
\[
\TB(\Tp):=\Set{ \Tm \in \TM_{n,k} \given 100 \Tp\leq \Tm},
\qquad \Tp\in\TC_{n,k}.
\]
In view of \eqref{eq:Mnk-counting} we have $1\leq \abs{B(\Tp)} \lesssim 2^{n} \log(n+1)$ for every $\Tp\in\TC_{n,k}$.
Let
\[
\TC_{n,k,j}:=
\Set{ \Tp\in\TC_{n,k} \given 2^j\leq \abs{\TB(\Tp)} < 2^{j+1} }.
\]
For the remaining part of the proof fix $j\geq 0$ such that $2^{j} \lesssim 2^{n} \log(n+1)$.
It suffices to show that $\TC_{n,k,j}$ can be written as the union of a Fefferman forest and $O(n)$ antichains.

First we verify that the set $\TC_{n,k,j}$ is convex.
Indeed, if $\Tp_{1}<\Tp<\Tp_{2}$ with $\Tp_{1},\Tp_{2}\in\TC_{n,k,j}$ and $\Tp\in\TC_{n,k}$, then $100 \Tp_{1} < 100 \Tp < 100 \Tp_{2}$, so that $\TB(\Tp_{1}) \supseteq \TB(\Tp) \supseteq \TB(\Tp_{2})$, so that $\Tp\in\TC_{n,k,j}$.

Let $\TU\subseteq\TC_{n,k,j}$ be the set of tiles $\Tu$ such that there is no $\Tp\in\TC_{n,k,j}$ with $I_{\Tu} \subsetneq I_{\Tp}$ and $\calQ(100 \Tu) \cap \calQ(100 \Tp) \neq\emptyset$.
These are our candidates for being tree tops.

In order to verify the counting function estimate \eqref{eq:F-forest-counting} we will show that for every $x\in\R^{\ds}$ the set $\TU(x) := \Set{ \Tu\in\TU \given x\in I_{\Tu}}$ has cardinality $O(2^{-j}2^{n}\log (n+1))$.
The family $\TU(x)$ can be subdivided into $O(1)$ families, denoted by $\TU'(x)$, in each of which the sets $\calQ(100\Tu)$, $\Tu\in\TU'(x)$, are disjoint (just make this decomposition at each scale independently).
In particular, the sets $\TB(\Tu)$, $\Tu\in\TU'(x)$, are pairwise disjoint.
These sets have cardinality at least $2^{j}$, and their union has cardinality at most $2^{n} \log(n+1)$ by \eqref{eq:Mnk-counting}.
This implies $\abs{\TU'(x)} \lesssim 2^{-j} 2^{n} \log(n+1)$.

Let
\[
\TD(\Tu) := \Set{\Tp\in \TC_{n,k,j} \given 2 \Tp < \Tu},
\quad
\Tu\in\TU.
\]
We will show that
\[
\TA'_{j} := \TC_{n,k,j} \setminus \cup_{\Tu\in\TU} \TD(\Tu)
\]
is an antichain.
Suppose that, on the contrary, there exist $\Tp, \Tp_1\in \TA'_{j}$ with $\Tp< \Tp_1$.
We claim that in this case for every $l=1,2,\dotsc$ there exists a sequence of tiles $\Tp_{1},\dotsc,\Tp_{l}\in \TC_{n,k,j}$ with
\[
2\Tp < 200\Tp_{1} < \dotsb < 200\Tp_{l}.
\]
This will produce a contradiction because the spatial cubes of these tiles are in $\calC_{k}$ and therefore have bounded scale.
For $l=1$ the claim follows from \eqref{eq:normQ:parent}.
Suppose now that the claim is known for some $l\geq 1$.
If $\Tp_{l} \in \TU$, then $\Tp\in\TD(\Tp_{l})$, and this is a contradiction.
Otherwise by definition of $\TU$ there exists a tile $\Tp_{l+1}\in\TC_{n,k,j}$ such that $I_{\Tp_{l}} \subsetneq I_{\Tp_{l+1}}$ and $\calQ(100 \Tp_{l}) \cap \calQ(100 \Tp_{l+1}) \neq \emptyset$.
It follows from \eqref{eq:normQ:parent} that $\calQ(200 \Tp_{l}) \supseteq \calQ(200 \Tp_{l+1})$, hence $200 \Tp_{l} < 200 \Tp_{l+1}$.
This finishes the proof of the claim and of the fact that $\TA'_{j}$ is an antichain.

Let $\TU' := \Set{ \Tu\in\TU \given \TD(\Tu) \neq\emptyset}$ and introduce on this set the relation
\begin{equation}
\label{eq:propto:def}
\Tu \propto \Tu'
:\iff
\exists \Tp\in \TD(\Tu) \text{ with } 10 \Tp\leq \Tu'.
\end{equation}
We claim that
\begin{equation}
\label{eq:propto-Fef-trick}
\Tu\propto\Tu' \implies I_{\Tu} = I_{\Tu'} \text{ and } \calQ(100 \Tu)\cap \calQ(100 \Tu')\neq\emptyset.
\end{equation}
\begin{proof}[Proof of the claim \eqref{eq:propto-Fef-trick}]
Let $\Tu,\Tu'\in\TU'$ with $\Tu\propto \Tu'$.
By definition there exists $\Tp\in \TC_{n,k,j}$ with $2\Tp < \Tu$ and $10 \Tp\leq \Tu'$.

First we notice that it suffices to show that
\begin{equation}
\label{eq:propto-Fef-trick:100}
\calQ(100 \Tu)\cap \calQ(100 \Tu')\neq\emptyset.
\end{equation}
Indeed, the spatial cubes $I_{\Tu},I_{\Tu'}$ both contain $I_{\Tp}$, so unless they coincide they are strictly nested, contradicting $\Tu,\Tu'\in\TU$.

Now we make a case distinction.
If $I_{\Tp} = I_{\Tu'}$, then $100 \Tu' \leq 2\Tp < \Tu$, and \eqref{eq:propto-Fef-trick:100} follows.

In the case $I_{\Tp} \subsetneq I_{\Tu'}$ we deduce from \eqref{eq:normQ:parent} that $100 \Tp < 100 \Tu'$ and $100 \Tp < 100 \Tu$.
If \eqref{eq:propto-Fef-trick:100} does not hold, then the sets $\TB(\Tu)$ and $\TB(\Tu')$ are disjoint.
On the other hand, $\TB(\Tp) \supseteq \TB(\Tu) \cup \TB(\Tu')$, so that $\abs{\TB(\Tp)} \geq \abs{\TB(\Tu)} + \abs{\TB(\Tu')} \geq 2\cdot 2^{j}$, a contradiction to $\Tp\in\TC_{n,k,j}$ (this is the Fefferman trick \cite[p.\ 569]{MR0340926}).
This establishes \eqref{eq:propto-Fef-trick:100}.
\end{proof}

Next we verify that ``$\propto$'' is an equivalence relation.
Let $\Tu,\Tu',\Tu'' \in \TU'$ be such that $I_{\Tu}=I_{\Tu'}=I_{\Tu''}$, $\calQ(100\Tu) \cap \calQ(100\Tu') \neq \emptyset$, and $\calQ(100\Tu')\cap\calQ(100\Tu'') \neq \emptyset$.
For all, and since $\TD(\Tu)\neq\emptyset$ in particular for some, $\Tp\in \TD(\Tu)$ we have $2 \Tp < \Tu$.
By \eqref{eq:normQ:parent} this implies $4 \Tp < 1000 \Tu$, and it follows that
\begin{equation}
\label{eq:propto-equiv-rel}
4\Tp < \Tu''.
\end{equation}
Using \eqref{eq:propto-Fef-trick} and the fact that \eqref{eq:propto-equiv-rel} implies $\Tu \propto \Tu''$ we deduce transitivity, symmetry, and reflexivity of the relation ``$\propto$''.

Let $\TV\subseteq\TU'$ be a set of representatives for equivalence classes modulo $\propto$ and let
\[
\TT(\Tv) := \cup_{\Tu \propto \Tv} \TD(\Tu),
\quad
\Tv\in\TV.
\]
Each $\TT(\Tv)$ is a union of down subsets $\TD(\Tu) \subset \TC_{n,k,j}$ and therefore convex.
It follows from \eqref{eq:propto-equiv-rel} that each $\TT(\Tv)$ is a tree with top $\Tv$.
It follows from \eqref{eq:propto:def} that these trees satisfy the separation condition
\begin{equation}\label{eq:10sep}
\forall \Tv\neq \Tv' \quad \forall \Tp\in\TT(\Tv) \qquad 10\Tp\not\leq \Tv'.
\end{equation}

In order to upgrade the condition \eqref{eq:10sep} to $2^{C n}$-separateness it suffices to remove the bottom $O(n)$ layers of tiles.
More precisely, for $l=1,\dotsc,Cn$ let $\TA_{n,k,j,l}$ be the set of minimal tiles in $\cup_{\Tv\in\TV} \TT(\Tv) \setminus \cup_{l'<l} \TA_{n,k,j,l'}$.
Then each $\TA_{n,k,j,l}$ is an antichain and each $\TT'(\Tv) := \TT(\Tv) \setminus \cup_{l} \TA_{n,k,j,l}$ is still a convex set, hence a tree with top $\Tv$.
Moreover, it follows from \eqref{eq:10sep} that tiles in distinct trees $\TT(\Tv)$ are not comparable.
Therefore for every $\Tp\in\TT'(\Tv)$ there exist tiles $\Tp_{1}<\dotsb<\Tp_{Cn}<\Tp$ in $\TT(\Tv)$.
If $I_{\Tp} \subseteq I_{\Tv'}$ for some $\Tv'\neq\Tv$, then using \eqref{eq:normQ:parent} and \eqref{eq:10sep} for the tile $\Tp_{1}$ we obtain
\[
\norm{Q_{\Tp}-Q_{\Tv'}}_{I_{\Tp}}
\geq
(10^{4})^{Cn} \norm{Q_{\Tp}-Q_{\Tv'}}_{I_{\Tp_{1}}}
\geq
(10^{4})^{Cn} \cdot 9,
\]
and this implies $10^{4Cn}$-separateness.
\end{proof}

The trees supplied by Proposition~\ref{prop:fef-forest} at different levels $n$ need not be disjoint.
We will now make them disjoint.
Let $\TT_{n,k,j,l}'$ be the trees and $\TA_{n,k,j}'$ the antichains provided by Proposition~\ref{prop:fef-forest} at level $n\geq 1$ and generation $k$.
For $n=1$ define
\[
\TT_{n,k,j,l} := \TT_{n,k,j,l}',
\quad
\TA_{n,k,j} := \TA_{n,k,j}'.
\]
For $n>1$ define
\[
\TT_{n,k,j,l} := \TT_{n,k,j,l}' \setminus \TH_{n-1,k},
\quad
\TA_{n,k,j} := \TA_{n,k,j}' \setminus \TH_{n-1,k}.
\]
Since we remove down subsets, the sets $\TT_{n,k,j,l}$ are still (convex) trees.

These sets have the following properties.
\begin{enumerate}
\item The set of all tiles can be decomposed as the disjoint union
\begin{equation}\label{eq:tree-dec}
\TP = \bigcup_{n=1}^{\infty} \bigcup_{k\in\N} \big( \bigcup_{j\lesssim n} \bigcup_{l} \TT_{n,k,j,l} \cup \bigcup_{j\lesssim n^{2}} \TA_{n,k,j}\big).
\end{equation}
\item Each $\TA_{n,k,j}$ is an antichain.
\item\label{it:TT-convex} Each $\TT_{n,k,j,l}$ is a tree.
\item Each $\TF_{n,k,j} := \cup_{l} \TT_{n,k,j,l}$ is a Fefferman forest of level $n$ and generation $k$.
\item $\mdens_{k}(\TF_{n,k,j}) \lesssim 2^{-n}$.
\item $\mdens_{k}(\TA_{n,k,j}) \lesssim 2^{-n}$.
\end{enumerate}

\section{Estimates for error terms}
In this section we consider error terms coming from antichains and boundary parts of trees.
These terms are morally easier to handle than the main terms in the sense that they are controlled by positive operators (after a suitable $TT^{*}$ argument).

\subsection{The basic $TT^{*}$ argument}
\begin{lemma}\label{lem:sep-tiles}
Let $\Tp_1, \Tp_2\in\TP$ with $\meas{I_{\Tp_{1}}} \leq \meas{I_{\Tp_{2}}}$.
Then
\begin{equation}\label{eq:sep-tiles}
\abs[\Big]{\int T_{\Tp_1}^{*}g_{1} \overline{T_{\Tp_2}^{*}g_{2}}}
\lesssim
\frac{\Delta(\Tp_1,Q_{\Tp_2})^{-\frac{\tau}{d}}}{\abs{I_{\Tp_{2}}}}
\int_{E(\Tp_1)}\abs{g_{1}}\int_{E(\Tp_2)}\abs{g_{2}}.
\end{equation}
\end{lemma}

\begin{proof}
We may assume $I_{\Tp_{1}}^{*} \cap I_{\Tp_{2}}^{*} \neq \emptyset$, since otherwise the left-hand side of the conclusion vanishes.
Expanding the left-hand side of \eqref{eq:sep-tiles} we obtain
\begin{multline*}
\abs[\Big]{\int
\int e(-Q_{x_{1}}(x_{1})+Q_{x_{1}}(y)) \overline{\CZK_{\scale(\Tp_{1})}(x_{1},y)} (\one_{E(\Tp_{1})}g_{1})(x_{1}) \dif x_{1}\\
\cdot \overline{\int e(-Q_{x_{2}}(x_{2})+Q_{x_{2}}(y)) \overline{\CZK_{\scale(\Tp_{2})}(x_{2},y)} (\one_{E(\Tp_{2})}g_{2})(x_{2}) \dif x_{2}} \dif y}\\
\leq
\int_{E(\Tp_{1})} \int_{E(\Tp_{2})} \abs[\Big]{\int
  e((Q_{x_{1}}-Q_{x_{2}})(y)-Q_{x_{1}}(x_{1})+Q_{x_{2}}(x_{2}))\\
  \cdot \overline{\CZK_{\scale(\Tp_{1})}(x_{1},y)} \CZK_{\scale(\Tp_{2})}(x_{2},y) \dif y}
\abs{g_{1}(x_{1}) g_{2}(x_{2})} \dif x_{2} \dif x_{1}.
\end{multline*}
By Lemma~\ref{lem:osc-int} applied to the cube $I_{\Tp_{1}}^{*}$ the integral inside the absolute value is bounded by
\[
(\norm{Q_{x_{1}}-Q_{x_{2}}}_{I_{\Tp_{1}}^{*}}+1)^{-\tau/d}/\abs{I_{\Tp_{2}}},
\]
and the conclusion follows since
\begin{align*}
\norm{Q_{x_{1}}-Q_{x_{2}}}_{I_{\Tp_{1}}^{*}}
&\geq
\norm{Q_{\Tp_{1}}-Q_{\Tp_{2}}}_{I_{\Tp_{1}}^{*}}
-
\norm{Q_{\Tp_{1}}-Q_{x_{1}}}_{I_{\Tp_{1}}^{*}}
-
\norm{Q_{\Tp_{2}}-Q_{x_{2}}}_{I_{\Tp_{1}}^{*}}\\
&\geq
\norm{Q_{\Tp_{1}}-Q_{\Tp_{2}}}_{I_{\Tp_{1}}}
-
C\norm{Q_{\Tp_{1}}-Q_{x_{1}}}_{I_{\Tp_{1}}}
-
\norm{Q_{\Tp_{2}}-Q_{x_{2}}}_{C I_{\Tp_{2}}}\\
&\geq
\Delta(\Tp_{1},Q_{\Tp_{2}})-1-C
-
C \norm{Q_{\Tp_{2}}-Q_{x_{2}}}_{I_{\Tp_{2}}}\\
&\geq
\Delta(\Tp_{1},Q_{\Tp_{2}})-C.
\qedhere
\end{align*}
\end{proof}

\subsection{Antichains and boundary parts of trees}
\begin{lemma}
\label{lem:antichain-supp}
There exists $\epsilon=\epsilon(d,\ds)>0$ such that for every $0\leq\eta\leq 1$, every $1\leq\rho\leq\infty$, every antichain $\TA\subseteq \TP_{k}$, and every $Q\in\calQ$ we have
\begin{equation}\label{eq:antichain-supp}
\norm{ \sum_{\Tp\in \TA} \Delta(\Tp,Q)^{-\eta} \one_{E(\Tp)} }_{\rho}
\lesssim
\mdens_{k}(\TA)^{\epsilon \eta/\rho} \abs[\big]{\cup_{\Tp\in\TA} I_{\Tp}}^{1/\rho}.
\end{equation}
\end{lemma}
\begin{proof}
Since the sets $E(\Tp)$, $\Tp\in\TA$, are disjoint, the claimed estimate clearly holds for $\rho=\infty$.
Hence by H\"older's inequality it suffices to consider $\rho=1$.
Let also $\delta=\mdens_{k}(\TA)$.
We have to show
\[
\sum_{\Tp\in\TA} \Delta(\Tp,Q)^{-\eta} \abs{E(\Tp)}
\lesssim
\delta^{\eta\epsilon} \abs{S},
\quad
S=\cup_{\Tp\in\TA} I_{\Tp}.
\]
Let $\epsilon>0$ be a small number to be chosen later and split the summation in two parts.
For those $\Tp\in\TA$ with $\Delta(\Tp,Q) \geq \delta^{-\epsilon}$ the estimate is clear because the sets $E(\Tp)$ are pairwise disjoint.

Let $\TA' = \Set{\Tp\in\TA \given \Delta(\Tp,Q) < \delta^{-\epsilon}}$ and consider the collection $\calL$ of the maximal grid cubes $L\in\calD$ such that $L\subsetneq I_{\Tp}$ for some $\Tp\in\TA'$ and $I_{\Tp} \not\subseteq L$ for all $\Tp\in\TA'$.
The collection $\calL$ is a disjoint cover of the set $\cup_{\Tp\in\TA'} I_{\Tp}$.
Fix $L\in\calL$; we will show that
\[
\sum_{\Tp\in\TA'} \abs{E(\Tp) \cap L}
\lesssim
\delta^{1-\epsilon \dim\calQ} \abs{L}.
\]
The conclusion will follow with $\epsilon = 1/(\dim\calQ+1)$.

By construction $\hat{L} \in \calC_{k}$ and there exists a tile $\Tp_{L}\in\TA'$ with $I_{\Tp_{L}}\subseteq\hat{L}$.
If $I_{\Tp_{L}}=\hat{L}$ let $\Tp_{L}':=\Tp_{L}$, otherwise let $\Tp_{L}'$ be the unique tile with $I_{\Tp_{L}'} = \hat{L}$ and $Q\in\calQ(\Tp_{L}')$.
In both cases with $\lambda = C \delta^{-\epsilon}$ for a sufficiently large constant $C$ the tile $\Tp_{L}'$ satisfies
\begin{enumerate}
\item $\lambda\Tp_{L} \leq \lambda\Tp_{L}'$ and
\item for every $\Tp\in\TA'$ with $L\cap I_{\Tp} \neq \emptyset$ we have $\lambda\Tp_{L}' \leq \Tp$.
\end{enumerate}
In view of disjointness of $E(\Tp)$'s this implies
\[
\sum_{\Tp\in\TA'} \abs{E(\Tp) \cap L}
\leq
\abs{E(\lambda\Tp_{L}')}
\leq
\lambda^{\dim\calQ} \abs{I_{\Tp_{L}'}} \mdens_{k}(\Tp_{L})
\lesssim
\delta^{1-\epsilon \dim\calQ} \abs{L}.
\qedhere
\]
\end{proof}

For a tree $\TT$ the \emph{boundary component} is defined by
\begin{equation}
\label{eq:T-boundary}
\bd(\TT) := \Set{ \Tp\in\TT \given I_{\Tp}^{*} \not\subseteq I_{\TT}}.
\end{equation}
Notice that $\bd(\TT)$ is an up-set: if $\Tp\in\bd(\TT)$, $\Tp'\in\TT$, $\Tp \leq \Tp'$, then $I_{\Tp'}^{*} \supseteq I_{\Tp}^{*}$, so that also $\Tp'\in\bd(\TT)$.
In particular, $\TT \setminus \bd(\TT)$ is still a (convex) tree.

\begin{proposition}\label{prop:sf}
Fix $n,j$ and let either $\TS = \cup_{k}\cup_{l} \bd(\TT_{n,k,j,l})$ or $\TS = \cup_{k}\TA_{n,k,j}$.
Then
\begin{equation}
\label{eq:sf}
\norm{T_{\TS}}_{2 \to 2}
\lesssim
2^{-\epsilon n}.
\end{equation}
\end{proposition}

\begin{proof}
We start by creating additional scale separation by restricting $k$ to a fixed congruence class modulo $2$.

For $\Tp'\in\TS$ let $\TD(\Tp') := \Set{ \Tp\in\TS \given \scale(\Tp) \leq \scale(\Tp') \land I_{\Tp}^{*}\cap I_{\Tp'}^{*}\neq\emptyset}$.
Then $I_{\Tp} \subset 5 I_{\Tp'}$ for $\Tp\in\TD(\Tp')$.
By Lemma~\ref{lem:sep-tiles} we have
\begin{align*}
\int \abs[\big]{ T_{\TS}^{*}g }^2
&\leq
2 \sum_{\Tp'\in\TS} \sum_{\Tp\in\TD(\Tp')}
\abs[\Big]{\int T_{\Tp'}^{*}g \overline{T^{*}_{\Tp}g} }\\
&\lesssim
\sum_{\Tp'\in\TS}\int_{E(\Tp')}\abs{g} \sum_{\Tp\in \TD(\Tp')} \Delta(\Tp,Q_{\Tp'})^{-\tau/d} \frac{\int_{E(\Tp)}\abs{g}}{\abs{I_{\Tp'}}}.
\end{align*}
By H\"older's inequality with exponent $1<q<2$ this is
\begin{equation}
\label{eq:Vf}
\leq
\sum_{\Tp'\in\TS}\int_{E(\Tp')} \abs{g} \left(\frac{\int_{5I_{\Tp'}} \abs{g}^q}{\abs{I_{\Tp'}}}\right)^{\frac{1}{q}}
\frac{\norm{ \sum_{\Tp\in \TD(\Tp')} \Delta(\Tp,Q_{\Tp'})^{-\tau/d} \one_{E(\Tp)} }_{q'}}{\abs{I_{\Tp'}}^{\frac{1}{q'}}}.
\end{equation}

First we will show that the last fraction is $O(2^{-\epsilon n})$ uniformly in $\Tp'$.
Let $k'$ be the integer for which $\Tp' \in \TP_{k'}$.
Let $\Tp \in \TD(\Tp') \cap \TP_{k}$ and suppose that $k<k'-1$.
There is a unique grid cube $I$ with $I_{\Tp}\subseteq I\in\calD_{s(\Tp')}$ and a unique stopping cube $F'$ with $I_{\Tp'} \subsetneq F' \in \calF_{k'-1}$.
Then in particular $s(I)<s(F')$ and $3I\cap F' \neq \emptyset$.
Therefore by part \eqref{it:stopping-neighbors} of Lemma~\ref{lem:spatial-decomposition} the cube $I_{\Tp}\subseteq I$ is contained in a stopping cube of generation $k'-1$, a contradiction.
Since we have restricted $k$ to a fixed congruence class modulo $2$ it follows that $\TD(\Tp') \subset \cup_{k\geq k'} \TP_{k}$.

Now we estimate the spatial support of $\TD(\Tp') \cap \TP_{k}$.
If $F\in\calF_{k'+1}$ and $F\cap 5I_{\Tp'} \neq \emptyset$, then $s(F) \leq s(\Tp')$, since otherwise an ancestor of $I_{\Tp'}$ would have been included in $\calF_{k'+1}$ by part \eqref{it:stopping-neighbors} of Lemma~\ref{lem:spatial-decomposition}.
Therefore by \eqref{eq:Lie-support-decay} for $k>k'$ we have
\begin{multline}
\label{eq:Dp'-support}
\meas[\big]{\bigcup_{F\in\calF_{k}} F \cap 5 I_{\Tp'}}
\lesssim
\meas[\big]{\bigcup_{F\in\calF_{k} : F\cap 5 I_{\Tp'} \neq \emptyset} F}\\
\lesssim
e^{k'-k}
\meas[\big]{\bigcup_{F\in\calF_{k'+1} : F\cap 5 I_{\Tp'} \neq \emptyset} F}
\lesssim
e^{k'-k} \meas{I_{\Tp'}},
\end{multline}
and the same estimate also clearly holds for $k=k'$.

Next we decompose $\TD(\Tp')$ into antichains.
Consider first the case $\TS = \cup_{k,l} \bd(\TT_{n,k,j,l})$.
For $k\geq k'$ and $m\geq 0$ let
\[
\TA_{k,m} := \bigcup_{l} \Set{ \Tp \in \TD(\Tp') \cap \bd(\TT_{n,k,j,l}) \given \scale(\Tp) = s(k,l) - m },
\]
where
\[
s(k,l) :=
\begin{cases}
\min(s(\top \TT_{n,k',j,l}), s(\Tp')) & \text{if } k=k',\\
s(\top \TT_{n,k,j,l}) & \text{if } k>k'.
\end{cases}
\]
The sets $\TA_{k,m}$ are pairwise disjoint antichains and partition $\TD(\Tp') = \cup_{k\geq k', m\geq 0}\TA_{k,m}$.
We have
\begin{align*}
\abs[\Big]{\bigcup_{\Tp\in \TA_{k',m}} I_{\Tp}}
&\leq
\sum_{l} \meas[\Big]{5 I_{\Tp'} \cap \bigcup_{\substack{\Tp\in \bd(\TT_{n,k',j,l}) :\\ s(\Tp) = s(k',l) - m}} I_{\Tp}}\\
&\leq
\sum_{l} \meas[\Big]{5 I_{\Tp'} \cap \Set{x\in I_{\TT_{n,k',j,l}} \given \dist(x,\R^{\ds}\setminus I_{\TT_{n,k',j,l}}) < C D^{s(k',l) - m}} }\\
&\lesssim
D^{-m} \sum_{l} \meas[\Big]{5 I_{\Tp'} \cap I_{\TT_{n,k',j,l}}}\\
&\lesssim
D^{-m} 2^{n} \log(n+1) \meas{I_{\Tp'}},
\end{align*}
where we have used \eqref{eq:F-forest-counting} in the last step.
Analogously, using \eqref{eq:Dp'-support} for $k>k'$ we obtain
\begin{align*}
\abs[\Big]{\bigcup_{\Tp\in \TA_{k,m}} I_{\Tp}}
&\leq
\sum_{l : I_{\TT_{n,k,j,l}} \cap 5I_{\Tp'} \neq \emptyset} \meas[\Big]{\bigcup_{\substack{\Tp\in \bd(\TT_{n,k,j,l}) :\\ s(\Tp) = s(k,l) - m}} I_{\Tp}}\\
&\leq
\sum_{l : I_{\TT_{n,k,j,l}} \cap 5I_{\Tp'} \neq \emptyset}
\meas[\Big]{\Set{x\in I_{\TT_{n,k,j,l}} \given \dist(x,\R^{\ds}\setminus I_{\TT_{n,k,j,l}}) < C D^{s(k,l) - m}} }\\
&\lesssim
\sum_{l : I_{\TT_{n,k,j,l}} \cap 5I_{\Tp'} \neq \emptyset}
D^{-m} \meas[\Big]{I_{\TT_{n,k,j,l}}}\\
&\lesssim
D^{-m} \sum_{F\in\calF_{k} : F\cap 5 I_{\Tp'} \neq \emptyset} \meas{F} 2^{n} \log(n+1)\\
&\lesssim
e^{k'-k} D^{-m} 2^{n} \log(n+1) \meas{I_{\Tp'}}.
\end{align*}
Combining this with a trivial estimate coming from \eqref{eq:Dp'-support} we obtain
\begin{equation}
\label{eq:Akm-support}
\abs[\Big]{\bigcup_{\Tp\in \TA_{k,m}} I_{\Tp}} \lesssim e^{k'-k} \min(1,C2^{n}\log(n+1) D^{-m}) \abs{I_{\Tp'}}.
\end{equation}
In the case $\TS = \cup_{k} \TA_{n,k,j}$ we define $\TA_{k,0} := \TA_{n,k,j} \cap \TD(\Tp')$ and $\TA_{k,m} := \emptyset$ for $m>0$.
The estimate \eqref{eq:Akm-support} also holds in this case.

Using Lemma~\ref{lem:antichain-supp} with $\rho=q'$ and $0\leq\eta\leq 1$ and \eqref{eq:Akm-support} it follows that
\begin{multline}\label{eq:sparse-Lie-tree-Carl-Delta}
\frac{\norm{ \sum_{\Tp\in \TD(\Tp')} \Delta(\Tp,Q_{\Tp'})^{-\eta} \one_{E(\Tp)} }_{\rho}}{\abs{I_{\Tp'}}^{1/\rho}}
\leq
\sum_{k\geq k', m\geq 0} \frac{\norm{ \sum_{\Tp\in \TA_{k,m}} \Delta(\Tp,Q_{\Tp'})^{-\eta} \one_{E(\Tp)} }_{\rho}}{\abs{I_{\Tp'}}^{1/\rho}}\\
\lesssim
2^{-\epsilon\eta n/\rho} \sum_{k\geq k', m\geq 0} \frac{\abs{\cup_{\Tp\in \TA_{k,m}} I_{\Tp}}^{1/\rho}}{\abs{I_{\Tp'}}^{1/\rho}}\\
\lesssim
2^{-\epsilon\eta n/\rho} \sum_{k\geq k', m\geq 0} e^{(k'-k)/\rho} \min(1,C2^{n}\log(n+1) D^{-m})^{1/\rho}\\
\lesssim_{\rho}
2^{-\epsilon\eta n/\rho} n \sum_{k\geq k'} e^{(k'-k) / \rho}
\lesssim_{\rho}
2^{-\epsilon\eta n/\rho} n.
\end{multline}
Using the estimate \eqref{eq:sparse-Lie-tree-Carl-Delta} with $\eta=\tau/d$ in the last factor of \eqref{eq:Vf} we obtain the claimed exponential decay in $n$.

In order to conclude it now suffices to show
\[
\sum_{\Tp\in\TS} \int_{E(\Tp)} \abs{g} (g)_{5I_{\Tp},q}
\lesssim
n \norm{g}_2^{2},
\text{ where }
(g)_{5I,q} := (\abs{I}^{-1} \int_{5I} \abs{g}^{q} )^{1/q}.
\]
Similarly to the estimate \eqref{eq:sparse-Lie-tree-Carl-Delta} with $\eta=0$ we obtain the Carleson packing condition
\begin{equation}
\label{eq:sparse-Lie-tree-Carl}
\norm{ \sum_{\Tp\in\TS : I_{\Tp}\subseteq J}\one_{E(\Tp)} }_{\rho} \lesssim_{\rho} n \abs{J}^{1/\rho},
\quad 1\leq \rho<\infty.
\end{equation}
Let $\calS\subset\calD$ be the stopping time associated to the average $(g)_{5I,q}$, that is, $\ch_{\calS}(I)$ are the maximal cubes $J \subset I$ with $(g)_{5J,q} > C (g)_{5I,q}$ for some large constant $C$.
Since the $q$-maximal operator \eqref{eq:HL-q-max-op} has weak type $(q,q)$, the family $\calS$ is \emph{sparse} in the sense that there exist pairwise disjoint subsets $\calE(I) \subseteq I\in\calS$ with $\abs{\calE(I)} \gtrsim \abs{I}$ (one can take $\calE(I) = I \setminus \cup_{J\in\ch_{\calS}(I)} J$).
Then
\begin{align*}
\sum_{\Tp\in\TS} \int_{E(\Tp)} \abs{g}  (g)_{5I_{\Tp},q}
&\lesssim
\sum_{I\in\calS} (g)_{5I,q} \int \abs{g} \sum_{\Tp\in\TS, I_{\Tp}\subseteq I} \one_{E(\Tp)}\\
\text{by H\"older}&\leq
\sum_{I\in\calS} (g)_{5I,q} \abs{I} (g)_{I,q} \Big( \abs{I}^{-1} \int_{I} \big( \sum_{\Tp\in\TS, I_{\Tp}\subseteq I} \one_{E(\Tp)} \big)^{q'} \Big)^{1/q'}\\
\text{by \eqref{eq:sparse-Lie-tree-Carl} and sparseness}&\lesssim
n \sum_{I\in\calS} (g)_{5I,q} \abs{\calE(I)} (g)_{I,q}\\
\text{by disjointness}&\lesssim
n \int (M_{q}g)^{2}
\lesssim
n \norm{g}_{2}^{2},
\end{align*}
where we have used the strong type $(2,2)$ inequality for $M_{q}$, $q<2$, in the last step.
\end{proof}

\subsection{Localization}
In order to handle exponents $p\neq 2$ we localize the operator $T_{\TS}$.
\begin{proposition}
\label{prop:sf-loc}
Let $\TS$ be as in Proposition~\ref{prop:sf}.
Let $F,G \subseteq \R^{\ds}$ be such that
\begin{equation}
\label{eq:sf-loc:assume}
\meas{I_{\Tp} \cap G} \lesssim \nu \meas{I_{\Tp}}
\text{ and }
\meas{5I_{\Tp} \cap F} \lesssim \kappa \meas{I_{\Tp}}
\text{ for every }
\Tp\in\TS.
\end{equation}
Then for every $0\leq \alpha < 1/2$ we have
\begin{equation}
\label{eq:sf-loc}
\norm{\one_{G} T_{\TS} \one_{F}}_{2 \to 2}
\lesssim_{\alpha}
\nu^{\alpha} \kappa^{\alpha} 2^{-\epsilon n}.
\end{equation}
\end{proposition}
\begin{proof}
Taking a geometric average with \eqref{eq:sf} it suffices to show
\[
\norm{\one_{G} T_{\TS} \one_{F}}_{2 \to 2}
\lesssim
n \nu^{\alpha} \kappa^{\alpha}.
\]
To this end we replace \eqref{eq:antichain-supp} by the estimate
\[
\norm{ \sum_{\Tp\in \TA} \one_{E(\Tp) \cap G} }_{\rho}
\lesssim
\abs[\big]{\cup_{\Tp\in\TA} I_{\Tp} \cap G}^{1/\rho}
\lesssim
\nu^{1/\rho} \abs[\big]{\cup_{\Tp\in\TA} I_{\Tp}}^{1/\rho}
\]
for all antichains $\TA\subset\TS$.
Following the proof of the Carleson packing condition \eqref{eq:sparse-Lie-tree-Carl} we obtain
\begin{equation}
\label{eq:sparse-Lie-tree-Carl-loc}
\norm{ \sum_{\Tp\in\TS : I_{\Tp}\subseteq J}\one_{E(\Tp) \cap G} }_{\rho}
\lesssim_{\rho}
n \nu^{1/\rho} \abs{J}^{1/\rho},
\quad 1\leq \rho<\infty.
\end{equation}
Fix functions $f,g$ with $\supp f \subset F$ and $\supp g\subset G$.
Consider the stopping time $\calS \subset \Set{I_{\Tp} \given \Tp\in\TS }$ associated to the average $(f)_{5I,1}$ and let $\calE(I) \subset I\in\calS$ be pairwise disjoint subsets with $\abs{\calE(I)} \gtrsim \abs{I}$.
With $\alpha=1/q'$ we obtain
\begin{align*}
\int \abs{g T_{\TS} f}
&\lesssim
\sum_{\Tp\in\TS} (f)_{5 I_{\Tp}, 1} \int_{E(\Tp)} \abs{g}\\
&\lesssim
\sum_{I\in\calS} (f)_{5I,1} \int \sum_{\Tp\in\TS, I_{\Tp}\subset I} \one_{E(\Tp)} \abs{g}\\
&\lesssim
\sum_{I\in\calS} (f)_{5I,q} (\one_{F})_{5I,q'} \meas{I} (g)_{I,q} \Big( \meas{I}^{-1} \int \big(\sum_{\Tp\in\TS, I_{\Tp}\subset I} \one_{E(\Tp) \cap G} \big)^{q'} \Big)^{1/q'}\\
&\lesssim
n \kappa^{1/q'} \nu^{1/q'} \sum_{I\in\calS} (f)_{5I,q} \meas{\calE(I)} (g)_{I,q}\\
&\lesssim
n \kappa^{1/q'} \nu^{1/q'} \int (M_{q}f) (M_{q}g)\\
&\lesssim
n \kappa^{1/q'} \nu^{1/q'} \norm{M_{q}f}_{2} \norm{M_{q}g}_{2}\\
&\lesssim
n \kappa^{1/q'} \nu^{1/q'} \norm{f}_{2} \norm{g}_{2}.
\qedhere
\end{align*}
\end{proof}

\section{Estimates for trees and forests}
In this section we consider the bulk of tiles that are organized into trees.
The contribution of each tree will be estimated by a maximally truncated operator associated to the kernel $\CZK$.

\subsection{Cotlar's inequality}
We call a subset $\sigma\subset\Z$ \emph{convex} if it is order convex, that is, $s_{1}<s<s_{2}$ and $s_{1},s_{2}\in\sigma$ implies $s\in\sigma$.
For a measurable function $\sigma$ that maps $\R^{\ds}$ to the set of finite convex subsets of $\Z$ we consider the associated truncated singular integral operator
\begin{equation}
\label{eq:T-trunc}
T_{\sigma}f(x) := \sum_{s\in\sigma(x)} \int \CZK_{s}(x,y) f(y) \dif y.
\end{equation}
An inspection of the proof of Cotlar's inequality, see e.g.\ \cite[Section I.7.3]{MR1232192}, shows that the non-tangentially maximally truncated operator
\begin{equation}
\label{eq:nontang-max-trunc}
T_{\calN}f (x) := \sup_{\sigma}\sup_{\abs{x-x'} \leq C D^{\min\sigma(x)}} \abs{T_{\sigma}f(x')},
\end{equation}
is bounded on $L^{p}(\R^{\ds})$, $1<p<\infty$ (more precisely, the proof of Cotlar's inequality shows that this holds if the constant $C$ is sufficiently small, see also \cite[Lemma 3.2]{MR3484688}; one can subsequently pass to larger values of $C$, see e.g.\ \cite[Section II.2.5.1]{MR1232192}).
We refer to this fact as the non-tangential Cotlar inequality.

We will use truncated singular integral operators with sets of scales given by trees.
\begin{definition}
\label{def:T-scales}
For a tree $\TT$ we define
\begin{align*}
\sigma (\TT,x) &:=
\Set{\scale(\Tp) \given \Tp\in\TT, x\in E(\Tp) },\\
\smax (\TT, x) &:= \max \sigma(x),\\
\smin (\TT,x) &:= \min \sigma(x).
\end{align*}
\end{definition}
We will omit the argument $\TT$ if it is clear from the context.
By construction of the set of all tiles $\TP$ the set $\sigma (\TT,x)$ is convex in $\Z$ for every tree $\TT$ and every $x\in\R^{\ds}$.

\subsection{Tree estimate}
\begin{definition}
\label{def:leaves}
For a non-empty finite collection of tiles $\TS\subset\TP$
\begin{enumerate}
\item let $\calJ(\TS) \subset \calD$ be the collection of the maximal grid cubes $J$ such that $100 D J$ does not contain $I_{\Tp}$ for any $\Tp\in\TS$ and
\item let $\calL(\TS) \subset \calD$ be the collection of the maximal grid cubes $L$ such that $L\subsetneq I_{\Tp}$ for some $\Tp\in\TS$ and $I_{\Tp} \not\subseteq L$ for all $\Tp\in\TS$.
\end{enumerate}

For a collection of pairwise disjoint grid cubes $\calJ\subset\calD$ we define the projection operator
\begin{equation}
\label{eq:positive-tree-proj}
P_{\calJ}f := \sum_{J\in\calJ} \one_{J} \meas{J}^{-1} \int_{J} f.
\end{equation}
\end{definition}
For later use we note the scales of adjacent cubes in $\calJ(\TS)$ differ at most by $1$ in the sense that if $J,J'\in\calJ$ and $\dist(J,J')\leq 10\max(\ell(J),\ell(J'))$, then $\abs{s(J)-s(J')} \leq 1$.
Indeed, if $J,J'\in\calJ$, $s(J) \leq s(J')-2$, and $\dist(J,J') \leq 10 \ell(J')$, then $100D\hat{J} \subset 100DJ'$ does not contain any $I_{\Tp}$, $\Tp\in\TS$, contradicting maximality of $J$.
\begin{lemma}[Tree estimate]\label{lem:tree}
Let $\TT\subseteq\TP$ be a tree, $\calJ := \calJ(\TT)$, and $\calL := \calL(\TT)$.
Then for every $1 < p < \infty$, $f\in L^{p}(\R^{\ds})$, and $g\in L^{p'}(\R^{\ds})$ we have
\begin{equation}\label{eq:tree:proj}
\abs[\Big]{\int_{\R^{\ds}} g T_{\TT} f}
\lesssim
\norm{P_{\calJ} \abs{f}}_{p} \norm{P_{\calL} \abs{g}}_{p'}.
\end{equation}
\end{lemma}
\begin{proof}
The conclusion \eqref{eq:tree:proj} will follow from the estimate
\begin{equation}
\label{eq:tree-2loc}
\sup_{x\in L} \abs{\overline{e(Q_{\TT})} T_{\TT} e(Q_{\TT}) f}(x)
\leq
C \inf_{x\in L} (M+S) P_{\calJ} \abs{f} (x)
+
\inf_{x\in L} \abs{T_{\calN} P_{\calJ} f (x)},
\end{equation}
where
\begin{enumerate}
\item $L\in\calL$ is arbitrary,
\item $Q_{\TT}$ denotes the central polynomial of $\TT$ (notice that the left-hand side is well-defined in the sense that it does not depend on the choice of the constant term of $Q_{\TT}$),
\item the operator $S$, while depending on $\TT$, is bounded on $L^{p}(\R^{\ds})$ for $1<p<\infty$ with constants independent of $\TT$, and
\item the non-tangentially maximally truncated singular integral $T_{\calN}$, defined in \eqref{eq:nontang-max-trunc}, is bounded on $L^{p}(\R^{\ds})$ by Cotlar's inequality.
\end{enumerate}

Let $\sigma=\sigma(\TT)$ be as in Definition~\ref{def:T-scales} and fix $x\in L\in\calL$.
By definition
\begin{multline*}
\abs{\overline{e(Q_{\TT})} T_{\TT} e(Q_{\TT}) f(x)}\\
=
\abs[\Big]{ \sum_{s \in \sigma(x)}\int e(-Q_{\TT}(x)+Q_x(x)-Q_{x}(y)+Q_{\TT}(y))\CZK_s(x,y) f (y) \dif y }\\
\leq
\sum_{s \in \sigma(x)} \int \abs{e(Q_{\TT}(y)-Q_{x}(y)-Q_{\TT}(x)+Q_{x}(x))-1} \abs{\CZK_s(x,y)} \abs{f(y)} \dif y\\
+\abs[\big]{ T_{\sigma} P_{\calJ} f(x) }
+\abs[\big]{ T_{\sigma} (1-P_{\calJ}) f(x) }
=: A(x)+B(x)+C(x).
\end{multline*}
The term $B(x)$ is a truncated singular integral and is dominated by $\inf_{L} T_{\calN} P_{\calJ} f$.

We turn to $A(x)$.
If $\CZK_{s}(x,y)\neq 0$, then $\abs{x-y}\lesssim D^{s}$, and in this case
\begin{multline*}
\abs{e(Q_{\TT}(y)-Q_{x}(y)-Q_{\TT}(x)+Q_{x}(x))-1}\\
\leq
\norm{Q_{x}-Q_{\TT}}_{B(x,CD^{s})}
\lesssim
D^{s-\smax(x)} \norm{Q_{x}-Q_{\TT}}_{B(x,CD^{\smax (x)})}
\lesssim
D^{s-\smax(x)},
\end{multline*}
where we have used Lemma~\ref{lem:normQ}.
For $x\in L\in\calL$ we have $\scale(L) \leq \smin(x)-1$, and it follows that
\[
A(x)
\lesssim
D^{-\smax(x)} \sum_{s\in\sigma(x)} D^{s(1-\ds)} \int_{B(x,0.5 D^{s})} \abs{f}(y) \dif y.
\]
Since the collection $\calJ$ is a partition of $\R^{\ds}$ this can be estimated by
\[
A(x)
\lesssim
D^{-\smax(x)} \sum_{s\in\sigma(x)} D^{s(1-\ds)} \sum_{J\in\calJ : J \cap B(x,0.5 D^{s}) \neq \emptyset} \int_{J} \abs{f}(y) \dif y.
\]
The expression on the right hand side does not change upon replacing $\abs{f}$ by $P_{\calJ}\abs{f}$.
Moreover
\begin{equation}
\label{eq:calJ-smallness}
I_{\Tp}^{*} \cap J \neq \emptyset \text{ with } \Tp\in\TT \text{ and } J\in\calJ \implies J \subset 3I_{\Tp}.
\end{equation}
Hence the sum over $J\in\calJ$ is in fact restricted to cubes contained in $B(x,CD^{s})$, so that
\begin{multline*}
A(x)
\lesssim
D^{-\smax(x)} \sum_{s\in\sigma(x)} D^{s(1-\ds)} \int_{B(x,CD^{s})} P_{\calJ} \abs{f}(y) \dif y\\
\lesssim
D^{-\smax(x)} \sum_{s\in\sigma(x)} D^{s} \inf_{L} M P_{\calJ} \abs{f}
\lesssim
\inf_{L} M P_{\calJ} \abs{f}.
\end{multline*}

It remains to treat $C(x)$.
Using \eqref{eq:calJ-smallness} we estimate
\begin{multline*}
\abs{T_{\sigma} (1-P_{\calJ}) f(x)}
=
\abs[\big]{\sum_{\Tp\in\TT} \one_{E(\Tp)}(x) \int \CZK_{s}(x,y) ((1-P_{\calJ}) f)(y) \dif y}\\
\leq
\sum_{\Tp\in\TT} \one_{E(\Tp)}(x) \sum_{J\in\calJ : J \subseteq 3 I_{\Tp}} \sup_{y,y'\in J} \abs{\CZK_{s}(x,y)-\CZK_{s}(x,y')} \int_{J} \abs{f}\\
\lesssim
\sum_{I\in\calH} \one_{I}(x) \sum_{J\in\calJ : J \subseteq 3 I}
D^{-(\ds+\tau) \scale(I)} \diam(J)^{\tau} \int_{J} P_{\calJ} \abs{f},
\end{multline*}
where $\calH = \Set{ I_{\Tp} \given \Tp\in\TT}$.
The right-hand side of this inequality is constant on each $L\in\calL$.
Hence we obtain \eqref{eq:tree-2loc} with
\[
S f(x) := \sum_{I\in\calD} \one_{I}(x) \sum_{J\in\calJ : J \subseteq 3 I}
D^{-(\ds+\tau) \scale(I)} \diam(J)^{\tau} \int_{J} f.
\]
It remains to obtain an $L^{p}$ estimate for the operator $S$.
We have
\begin{align*}
\abs[\big]{\int g S f}
&\lesssim
\sum_{I\in\calD,J\in\calJ : J \subseteq 3 I} (g)_{I} D^{\tau (\scale(J)-\scale(I))} \int_{J} \abs{f}\\
&\lesssim
\sum_{J\in\calJ} \int_{J} \abs{f} Mg \sum_{I\in\calD : J \subseteq 3 I} D^{\tau (\scale(J)-\scale(I))}\\
&\lesssim
\sum_{J\in\calJ} \int_{J} \abs{f} Mg\\
&\leq
\norm{f}_{p} \norm{Mg}_{p'}.
\end{align*}
By the Hardy--Littlewood maximal inequality and duality it follows that $\norm{S}_{p\to p} \lesssim 1$ for $1<p<\infty$.
\end{proof}

\begin{corollary}
\label{cor:tree}
Let $\TT\subseteq\TP_{k}$ be a tree.
Let also $F\subseteq\R^{\ds}$ and $\kappa>0$ be such that
\begin{equation}
\label{eq:tree-loc}
I_{\Tp} \not\subseteq \Set{ M\one_{F} > \kappa }
\text{ for all }
\Tp\in\TT.
\end{equation}
Then for every $1 < p < \infty$ and $f\in L^{p}(\R^{\ds})$ we have
\begin{equation}\label{eq:tree:local}
\norm{T_{\TT} \one_{F} f}_p
\lesssim
\kappa^{1/p'} \mdens_{k}(\TT)^{1/p} \norm{f}_{p}.
\end{equation}
\end{corollary}
Notice that the hypothesis \eqref{eq:tree-loc} holds with $\kappa=1$ and $F=\R^{\ds}$ for every tree $\TT$.
\begin{proof}
Fix $L\in\calL:=\calL(\TT)$.
By construction $\hat{L} \in \calC_{k}$ and there exists a tile $\Tp_{L}\in\TT$ with $I_{\Tp_{L}}\subseteq\hat{L}$.
If $I_{\Tp_{L}}=\hat{L}$ let $\Tp_{L}':=\Tp_{L}$, otherwise let $\Tp_{L}'\in\TP_{k}$ be the unique tile with $I_{\Tp_{L}'} = \hat{L}$ and $Q_{\TT}\in\calQ(\Tp_{L}')$.
In both cases the tile $\Tp_{L}'$ satisfies
\begin{enumerate}
\item $10\Tp_{L} \leq 10\Tp_{L}'$ and
\item for every $\Tp\in\TT$ with $L\cap I_{\Tp} \neq \emptyset$ we have $10\Tp_{L}' \leq \Tp$.
\end{enumerate}
It follows that the spatial support
\[
E(L) :=
L\cap \bigcup_{\Tp\in\TT} E(\Tp)
=
L\cap \bigcup_{\Tp\in\TT : I_{\Tp} \supset L} E(\Tp)
\]
satisfies
\begin{equation}
\label{eq:leaf-mass}
\abs{E(L)}
\leq
\abs{E(10\Tp_{L}')}
\leq
10^{\dim\calQ} \abs{I_{\Tp_{L}'}} \mdens_{k}(\Tp_{L})
\lesssim
\mdens_{k}(\TT) \abs{L}.
\end{equation}

It also follows from the hypothesis \eqref{eq:tree-loc} that
\begin{equation}
\label{eq:dens-kappa}
\meas{F \cap J} \lesssim \kappa \meas{J}
\end{equation}
for all $J\in\calJ:=\calJ(\TT)$.
Using Lemma~\ref{lem:tree}, H\"older's inequality, and the estimates \eqref{eq:leaf-mass} and \eqref{eq:dens-kappa} we obtain
\begin{align*}
\abs[\Big]{\int_{\R^{\ds}} g T_{\TT} \one_{F} f}
&=
\abs[\Big]{\int_{\R^{\ds}} \sum_{L\in\calL}\one_{E(L)} g T_{\TT} \one_{F} f}\\
&\lesssim
\norm{P_{\calL} \abs[\big]{\sum_{L\in\calL}\one_{E(L)} g} }_{p'} \norm{P_{\calJ} \abs{\one_{F} f} }_{p}\\
&=
\Bigl( \sum_{L\in\calL} \meas{L} \bigl( \meas{L}^{-1} \int_{L} \one_{E(L)} \abs{g} \bigr)^{p'} \Bigr)^{1/p'}
\Bigl( \sum_{J\in\calJ} \meas{J} \bigl( \meas{J}^{-1} \int_{J} \one_{F} \abs{f} \bigr)^{p} \Bigr)^{1/p}\\
&\leq
\Bigl( \sum_{L\in\calL} \meas{L} \bigl( \meas{L}^{-1} \int_{L} \abs{g}^{p'} \bigr)
\bigl( \meas{L}^{-1} \int_{L} \one_{E(L)}^{p} \bigr)^{p'/p} \Bigr)^{1/p'}\\
&\quad \cdot
\Bigl( \sum_{J\in\calJ} \meas{J} \bigl( \meas{J}^{-1} \int_{J} \abs{f}^{p} \bigr)
\bigl( \meas{J}^{-1} \int_{J} \one_{F}^{p'} \bigr)^{p/p'} \Bigr)^{1/p}\\
&\lesssim
\mdens_{k}(\TT)^{1/p} \kappa^{1/p'}
\Bigl( \sum_{L\in\calL} \int_{L} \abs{g}^{p'} \Bigr)^{1/p'}
\Bigl( \sum_{J\in\calJ} \int_{J} \abs{f}^{p} \Bigr)^{1/p}\\
&\leq
\mdens_{k}(\TT)^{1/p} \kappa^{1/p'}
\norm{g}_{p'} \norm{f}_{p}.
\qedhere
\end{align*}
\end{proof}

\subsection{Separated trees}
\begin{definition}
\label{def:normal}
A tree $\TT$ is called \emph{normal} if for every $\Tp\in\TT$ we have $I_{\Tp}^{*} \subset I_{\TT}$.
\end{definition}
For a normal tree $\TT$ we have $\supp T_{\TT}^{*}g \subseteq I_{\TT}$ for every function $g$.
\begin{lemma}\label{lem:sep-trees}
There exists $\epsilon=\epsilon(d,\tau,\ds)>0$ such that for any two $\Delta$-separated normal trees $\TT_1, \TT_{2}$ we have
\begin{equation}
\abs[\Big]{\int_{\R^{\ds}} T_{\TT_{1}}^* g_{1} \overline{ T_{\TT_{2}}^* g_{2}} }
\lesssim
\Delta^{-\epsilon} \prod_{j=1,2} \norm{\abs{T_{\TT_{j}}^{*} g_{j}} + Mg_{j}}_{L^{2}(I_{\TT_{1}} \cap I_{\TT_{2}})}.
\end{equation}
\end{lemma}

\begin{proof}
The estimate clearly holds without decay in $\Delta$, so it suffices to consider $\Delta \gg 1$.
Without loss of generality assume $I_{0} := I_{\TT_{1}} \subseteq I_{\TT_{2}}$ and $\TT_{1} \neq \emptyset$.
Neither the left-hand side nor the right-hand side of the conclusion changes upon replacing $\TT_{2}$ by the convex set
\[
\Set{\Tp \in \TT_{2} \given I_{\Tp}^{*} \cap I_{\TT_{1}} \neq \emptyset},
\]
which we do.
Let $\TS := \Set{\Tp \in \TT_{1}\cup\TT_{2} \given I_{\Tp}\subseteq I_{0}}$.
Let $\eta>0$ be chosen later and
\[
\TT_{2}' := \Set{\Tp \in \TT_{2} \given I_{\Tp} \cap I_{0} = \emptyset, \not\exists \Tp'\in\TS \text{ with }  I_{\Tp'} \subset \Delta^{\eta} I_{\Tp}}.
\]

Let $\calJ := \Set{ J \in \calJ(\TT_{1}\cup(\TT_{2}\setminus\TT_{2}')) \given J \subseteq I_{0}}$.
This is a partition of $I_{0}$.
Since the scales of adjacent cubes in this partition differ at most by $1$, there exists an adapted partition of unity $\one = \sum_{J\in\calJ} \chi_{J}$ (on $I_{0}$), where each $\chi_{J}$ is a smooth function supported on $(1+1/D) J$ with $\abs{\nabla\chi_{J}} \lesssim \ell(J)^{-1}$.

Recall that $Q_{\TT}$ denotes the central polynomial of a tree $\TT$ and let $Q := Q_{\TT_{1}}-Q_{\TT_{2}}$.
We claim that
\begin{equation}
\label{eq:leaf-separation}
\Delta_{J} := \norm{Q}_{J} \gtrsim \Delta^{1-\eta d}
\text{ for all } J\in\calJ.
\end{equation}
\begin{proof}[Proof of Claim \eqref{eq:leaf-separation}]
By definition there exists $\Tp \in \TT_{1}\cup (\TT_{2}\setminus\TT_{2}')$ with $100D\hat{J} \supseteq I_{\Tp}$.
We distinguish the following cases.
\begin{enumerate}
\item If $I_{\Tp}\subseteq I_{0}$, then $\Tp\in\TS$, and by definition of $\Delta$-separation we obtain
\[
\norm{Q_{\TT_{1}}-Q_{\TT_{2}}}_{I_{\Tp}}
\geq
\norm{Q_{\Tp}-Q_{\TT_{j}}}_{I_{\Tp}} - 4
\geq
\Delta - 5,
\]
where $j\in\Set{1,2}$ is such that $\Tp\not\in\TT_{j}$.
The claim follows using Lemma~\ref{lem:normQ}.
\item If $I_{\Tp} \supset I_{0}$, then for an arbitrary $\Tp'\in\TT_{1}$ we have $I_{\Tp'} \subseteq I_{0} \subset I_{\Tp} \subseteq 100D \hat{J}$, reducing to the previously handled case.
\item If $I_{\Tp} \cap I_{0} = \emptyset$, then $\Tp\in\TT_{2}\setminus\TT_{2}'$, and by definition there exists $\Tp'\in\TS$ with $I_{\Tp'} \subset \Delta^{\eta} I_{\Tp}$.
Since $I_{\Tp'} \subseteq I_{0}$ and by definition of $\Delta$-separation we obtain, similarly as before,
\[
\norm{Q_{\TT_{1}} - Q_{\TT_{2}}}_{I_{\Tp'}}
\geq
\norm{Q_{\Tp'} - Q_{\TT_{2}}}_{I_{\Tp'}} - 4
\geq
\Delta-5.
\]
This time we conclude by a more subtle application of Lemma~\ref{lem:normQ}:
\[
\norm{Q}_{J}
\gtrsim
\norm{Q}_{100D\hat{J}}
\geq
\norm{Q}_{I_{\Tp}}
\gtrsim
\Delta^{-\eta d} \norm{Q}_{\Delta^{\eta} I_{\Tp}}
\geq
\Delta^{-\eta d} \norm{Q}_{I_{\Tp'}}
\gtrsim
\Delta^{1-\eta d}.
\]
\end{enumerate}
This finishes the proof of Claim \eqref{eq:leaf-separation}.
\end{proof}

In order to prepare the application of Lemma~\ref{lem:osc-int} we need to estimate local moduli of continuity of $T_{\TT}^{*}g$ for a tree $\TT$.
For every $\Tp\in\TT$ and $y,y'\in I_{\Tp}^{*}$ using \eqref{eq:Ks-size}, \eqref{eq:Ks-reg}, and Lemma~\ref{lem:normQ} we obtain
\begin{align*}
\MoveEqLeft
\abs[\big]{e(Q_{\TT}(0)-Q_{\TT}(y))T_{\Tp}^{*}g(y) - e(Q_{\TT}(0)-Q_{\TT}(y'))T_{\Tp}^{*}g(y')}\\
&=
\abs[\Big]{\int \bigl(e(-Q_{x}(x)+Q_{x}(y)-Q_{\TT}(y)+Q_{\TT}(0)) \overline{\CZK_{\scale(\Tp)}(x,y)}\\
  & \quad
  - e(-Q_{x}(x)+Q_{x}(y')-Q_{\TT}(y')+Q_{\TT}(0)) \overline{\CZK_{\scale(\Tp)}(x,y')}\bigr) (\one_{E(\Tp)}g)(x) \dif x}\\
&\leq
\int_{E(\Tp)} \abs{g}(x) \abs[\Big]{e(-Q_{x}(y')+Q_{x}(y)-Q_{\TT}(y)+Q_{\TT}(y')) \overline{\CZK_{\scale(\Tp)}(x,y)} - \overline{\CZK_{\scale(\Tp)}(x,y')}}  \dif x\\
&\leq
\int_{E(\Tp)} \abs{g}(x) \Bigr( \abs{e(-Q_{x}(y')+Q_{x}(y)-Q_{\TT}(y)+Q_{\TT}(y'))-1} \abs{\CZK_{\scale(\Tp)}(x,y)}\\
&\quad
+ \abs{\overline{\CZK_{\scale(\Tp)}(x,y)} - \overline{\CZK_{\scale(\Tp)}(x,y')}} \Bigr)  \dif x\\
&\lesssim
\int_{E(\Tp)} \abs{g}(x) \Bigr( \norm{Q_{x}-Q_{\TT}}_{I_{\Tp}^{*}} \frac{\abs{y-y'}}{D^{s(\Tp)}} D^{-s(\Tp)\ds} + D^{-s(\Tp)\ds} \bigl(\frac{\abs{y-y'}}{D^{s(\Tp)}}\bigr)^{\tau} \Bigr)  \dif x\\
&\lesssim
\bigl(\frac{\abs{y-y'}}{D^{s(\Tp)}}\bigr)^{\tau} D^{-s(\Tp)\ds}\int_{E(\Tp)} \abs{g}(x) \dif x.
\end{align*}
Let $J\in\calD$ be such that for every $\Tp\in\TT$ we have $I_{\Tp}^{*} \cap (1+1/D) J \neq \emptyset \implies s(\Tp) \geq s(J)$.
Then for every $y,y' \in (1+1/D) J$ we obtain
\begin{align*}
\MoveEqLeft
\abs[\big]{e(Q_{\TT}(0)-Q_{\TT}(y))T_{\TT}^{*}g(y) - e(Q_{\TT}(0)-Q_{\TT}(y'))T_{\TT}^{*}g(y')}\\
&\leq
\sum_{\Tp\in\TT : I_{\Tp}^{*} \cap (1+1/D) J \neq \emptyset}
\abs[\big]{e(Q_{\TT}(0)-Q_{\TT}(y))T_{\Tp}^{*}g(y) - e(Q_{\TT}(0)-Q_{\TT}(y'))T_{\Tp}^{*}g(y')}\\
&\lesssim
\sum_{s\geq s(J)} \sum_{\Tp\in\TT : I_{\Tp}^{*} \cap (1+1/D) J \neq \emptyset, s(\Tp)=s}
\bigl(\frac{\abs{y-y'}}{D^{s}}\bigr)^{\tau} D^{-s \ds}\int_{E(\Tp)} \abs{g}(x) \dif x\\
&\lesssim
\sum_{s\geq s(J)} \sum_{\Tp\in\TT : I_{\Tp}^{*} \cap (1+1/D) J \neq \emptyset, s(\Tp)=s}
\bigl(\frac{\abs{y-y'}}{D^{s}}\bigr)^{\tau} \inf_{J} Mg\\
&\lesssim
\bigl(\frac{\abs{y-y'}}{D^{s(J)}}\bigr)^{\tau} \inf_{J} Mg.
\end{align*}
This implies in particular
\[
\sup_{y\in (1+1/D) J} \abs{e(Q_{\TT}(0)-Q_{\TT}(y)) T_{\TT}^{*} g(y)}
\leq
\inf_{y\in \frac12 J} \abs{T_{\TT}^{*}g(y)}
+
C \inf_{y\in J} Mg(y).
\]
We claim that for an absolute constant $s_{0}$ we have
\begin{equation}
\label{eq:separated-smalltiles}
\Tp\in\TT_{2}', J\in\calJ, I_{\Tp}^{*} \cap J \neq \emptyset
\implies
s(\Tp) \leq s(J) + s_{0}.
\end{equation}
\begin{proof}[Proof of Claim \eqref{eq:separated-smalltiles}]
Let $s_{0}$ be chosen later and suppose $s(\Tp)>s(J)+s_{0}$.
By definition there exists $\Tp'\in\TT_{1}\cup(\TT_{2}\setminus\TT_{2}')$ with $I_{\Tp'}\subseteq 100D\hat{J}$.
If $I_{\Tp'}\cap I_{0} \neq \emptyset$, then replacing $\Tp'$ by an element of $\TT_{1}$ we may without loss of generality assume $I_{\Tp'} \subseteq I_{0}$, so that $\Tp'\in\TS$.
If $\Delta$ is sufficiently large, then it follows that $\Delta^{\eta} I_{\Tp} \supset I_{\Tp'}$, contradicting $\Tp\in\TT_{2}'$.

If on the other hand $I_{\Tp'}\cap I_{0} = \emptyset$, then since $\Tp'\not\in\TT_{2}'$ there exists $\Tp''\in\TS$ with $I_{\Tp''} \subset \Delta^{\eta}I_{\Tp'}$.
It follows that
\[
CD^{-s_{0}}\Delta^{\eta} I_{\Tp} \supset \Delta^{\eta}I_{\Tp'} \supset I_{\Tp''}.
\]
If $s_{0}$ is sufficiently large, then $CD^{-s_{0}}\leq 1$, again contradicting $\Tp\in\TT_{2}'$.
\end{proof}
In particular for every $J\in\calJ$ we obtain
\[
\sup_{y\in \frac12 J} \abs{T_{\TT_{2}'}^{*}g_{2}(y)}
\leq
\sup_{y\in \frac12 J} \sum_{s=s(J)}^{s(J)+s_{0}} \sum_{\Tp\in\TP : s(\Tp)=s} \abs{T_{\Tp}^{*}g_{2}(y)}
\lesssim
(s_{0}+1) \inf_{J} M g_{2}.
\]
Using these facts we obtain
\begin{align*}
\sup_{y\in (1+1/D) J} \abs{T_{\TT_{2}\setminus\TT_{2}'}^{*} g_{2}(y)}
&\leq
\inf_{y\in \frac12 J} \abs{T_{\TT_{2}\setminus\TT_{2}'}^{*}g_{2}(y)}
+
C \inf_{y\in J} Mg_{2}(y)\\
&\leq
\inf_{y\in \frac12 J} \abs{T_{\TT_{2}}^{*}g_{2}(y)}
+
\sup_{y\in \frac12 J} \abs{T_{\TT_{2}'}^{*}g_{2}(y)}
+
C \inf_{y\in J} Mg_{2}(y)\\
&\leq
\inf_{y\in \frac12 J} \abs{T_{\TT_{2}}^{*}g_{2}(y)}
+
C \inf_{y\in J} Mg_{2}(y)
\end{align*}
for $J\in\calJ$, and it follows that
\[
\abs{h_{J}(y)-h_{J}(y')}
\lesssim
\bigl(\frac{\abs{y-y'}}{\ell(J)}\bigr)^{\tau}
\prod_{j=1,2} \Bigl( \inf_{\frac12 J} \abs{T_{\TT_{j}}^{*} g_{j}} + \inf_{J} Mg_{j} \Bigr)
\]
for the functions
\[
h_{J}(y) :=
\chi_{J}(y)
\bigl( e(Q_{\TT_{1}}(0)-Q_{\TT_{1}}(y)) T_{\TT_{1}}^{*} g_{1}(y) \bigr)\\
\cdot
\overline{ \bigl( e(Q_{\TT_{2}}(0)-Q_{\TT_{2}}(y)) T_{\TT_{2}\setminus\TT_{2}'}^{*} g_{2}(y) \bigr) }.
\]
Using \eqref{eq:leaf-separation} and Lemma~\ref{lem:osc-int} this allows us to estimate
\begin{align*}
\abs[\Big]{\int_{\R^{\ds}} T_{\TT_{1}}^* g_{1} \overline{ T_{\TT_{2}\setminus\TT_{2}'}^* g_{2}} }
&\leq
\sum_{J} \abs[\Big]{\int e(Q(y)-Q(0)) h_{J}(y) \dif y}\\
&\lesssim
\sum_{J} \Delta_{J}^{-\tau/d} \meas{J} \prod_{j=1,2} \inf_{\frac12 J} \Bigl( \abs{T_{\TT_{j}}^{*} g_{j}} + Mg_{j} \Bigr)\\
&\lesssim
\Delta^{-(1-\eta d)\tau/d}\int_{I_{0}} \prod_{j=1,2} \Bigl( \abs{T_{\TT_{j}}^{*} g_{j}} + Mg_{j} \Bigr)\\
&\leq
\Delta^{-(1-\eta d)\tau/d}\prod_{j=1,2} \norm{ \abs{T_{\TT_{j}}^{*} g_{j}} + Mg_{j} }_{L^{2}(I_{0})}.
\end{align*}
It remains to consider the contribution of $\TT_{2}'$.
Let $\calJ' := \Set{J\in\calJ(\TT_{1}) \given J\subset I_{0}}$.
Then
\[
\Tp\in\TT_{2}', J\in\calJ', I_{\Tp}^{*} \cap J \neq \emptyset
\implies
s(\Tp) \leq s(J) - s_{\Delta}, \text{ where } D^{s_{\Delta}} \sim \Delta^{\eta},
\]
since otherwise $\Delta^{\eta}I_{\Tp} \supset 100D\hat{J} \supset I_{\Tp'}$ for some $\Tp'\in\TT_{1} \subseteq \TS$, contradicting $\Tp\in\TT_{2}'$.
Using Lemma~\ref{lem:tree} we obtain
\begin{align*}
\abs[\Big]{\int_{\R^{\ds}} T_{\TT_{1}}^* g_{1} \overline{ T_{\TT_{2}'}^* g_{2}} }
&\lesssim
\norm{g_{1}}_{2} \norm{P_{\calJ'} \abs{T_{\TT_{2}'}^* g_{2}} }_{2}\\
&\leq
\norm{g_{1}}_{2} \sum_{s\geq s_{\Delta}} \Bigl( \sum_{J\in\calJ'} \meas{J}^{-1} \abs[\Big]{\int_{J} \sum_{\Tp\in\TT_{2}' : s(\Tp) = s(J)-s, I_{\Tp}^{*} \cap J \neq \emptyset} T_{\Tp}^{*}g_{2} }^{2} \Bigr)^{1/2}\\
&\lesssim
\norm{g_{1}}_{2} \sum_{s\geq s_{\Delta}} \Bigl( \sum_{J\in\calJ'} \Bigl(\int_{J} (M g_{2})^{2} \Bigr) \frac{\int_{J} \bigl( \sum_{\Tp\in\TT_{2}' : s(\Tp) = s(J)-s, I_{\Tp}^{*} \cap J \neq \emptyset} \one_{I_{\Tp}^{*}}\bigr)^{2}}{\meas{J}} \Bigr)^{1/2}\\
&\lesssim
\norm{g_{1}}_{2} \sum_{s\geq s_{\Delta}} \Bigl( \sum_{J\in\calJ'} \Bigl(\int_{J} (M g_{2})^{2} \Bigr) \frac{D^{s(J)-s+s(J)(\ds-1)}}{D^{s(J) \ds}} \Bigr)^{1/2}\\
&\leq
\norm{g_{1}}_{2} \sum_{s\geq s_{\Delta}} D^{-s/2} \norm{M g_{2}}_{L^{2}(I_{0})}\\
&\lesssim
\Delta^{-\eta/2} \norm{g_{1}}_{2} \norm{M g_{2}}_{L^{2}(I_{0})}.
\end{align*}
Choosing $\eta$ sufficiently small and observing that $\norm{g_{1}}_{2} \leq \norm{Mg_{1}}_{L^{2}(I_{0})}$ we obtain the claim.
\end{proof}

\subsection{Rows}
\begin{definition}
\label{def:row}
A \emph{row} is a union of normal trees with tops that have pairwise disjoint spatial cubes.
\end{definition}

\begin{lemma}[Row estimate]\label{lem:row}
Let $\TR_{1}$, $\TR_{2}$ be rows such that the trees in $\TR_{1}$ are $\Delta$-separated from the trees in $\TR_{2}$.
Then for any $g_{1},g_{2}\in L^{2}(\R^{\ds})$ we have
\[
\abs[\Big]{ \int T_{\TR_{1}}^* g_{1} \overline{T_{\TR_{2}}^* g_{2}} }
\lesssim
\Delta^{-\epsilon} \norm{g_{1}}_{2} \norm{g_{2}}_{2}.
\]
\end{lemma}

\begin{proof}
The operators $S_{\TT}g := \abs{T_{\TT}^{*}g}+Mg$ are bounded on $L^{2}(\R^{\ds})$ uniformly in $\TT$ by Lemma~\ref{lem:tree} and the Hardy--Littlewood maximal inequality.
Using Lemma~\ref{lem:sep-trees} we estimate
\begin{align*}
\abs[\Big]{ \int T_{\TR_{1}}^* g_{1} \overline{T_{\TR_{2}}^* g_{2}} }
&\leq
\sum_{\TT_{1}\in\TR_{1}, \TT_{2}\in\TR_{2}} \abs[\Big]{ \int T_{\TT_{1}}^* g_{1} \overline{T_{\TT_{2}}^* g_{2}} }\\
&=
\sum_{\TT_{1}\in\TR_{1}, \TT_{2}\in\TR_{2}} \abs[\Big]{ \int T_{\TT_{1}}^* (\one_{I_{\TT_{1}}} g_{1}) \overline{T_{\TT_{2}}^* (\one_{I_{\TT_{2}}} g_{2})} }\\
&\lesssim
\Delta^{-\epsilon} \sum_{\TT_{1}\in\TR_{1}, \TT_{2}\in\TR_{2}} \prod_{j=1,2} \norm{S_{\TT_{j}} \one_{I_{\TT_{j}}} g_{j}}_{L^{2}(I_{\TT_{1}} \cap I_{\TT_{2}})}\\
&\leq
\Delta^{-\epsilon}
\prod_{j=1,2} \Big(\sum_{\TT_{1}\in\TR_{1}, \TT_{2}\in\TR_{2}} \norm{S_{\TT_{j}} \one_{I_{\TT_{j}}} g_{j}}_{L^{2}(I_{\TT_{1}} \cap I_{\TT_{2}})}^{2} \Big)^{1/2}\\
&\leq
\Delta^{-\epsilon}
\prod_{j=1,2} \Big(\sum_{\TT_{j}\in\TR_{j}} \norm{S_{\TT_{j}} \one_{I_{\TT_{j}}} g_{j}}_{L^{2}(I_{\TT_{j}})}^{2} \Big)^{1/2}\\
&\lesssim
\Delta^{-\epsilon}
\prod_{j=1,2} \Big(\sum_{\TT_{j}\in\TR_{j}} \norm{\one_{I_{\TT_{j}}} g_{j}}_{L^{2}(\R^{\ds})}^{2} \Big)^{1/2}\\
&\leq
\Delta^{-\epsilon}
\norm{g_{1}}_{2}
\norm{g_{2}}_{2}.
\qedhere
\end{align*}
\end{proof}

\subsection{Forest estimate}
Recall our decomposition \eqref{eq:tree-dec} of the set of all tiles.
In view of Proposition~\ref{prop:sf} it remains to estimate the contribution of the normal trees
\[
\TN_{n,k,j,l} := \TT_{n,k,j,l} \setminus \bd(\TT_{n,k,j,l})
\]
These sets are indeed (convex) trees since $\bd(\TT)$ are up-sets (recall the definition \eqref{eq:T-boundary}).
\begin{proposition}
\label{prop:Fef-forest}
Let $\TF_{n,k,j}' := \cup_{l} \TN_{n,k,j,l}$.
Then
\begin{equation}
\label{eq:Fef-forest-global}
\norm{T_{\TF_{n,k,j}'}}_{2\to 2} \lesssim 2^{-n/2}.
\end{equation}
Assuming in addition \eqref{eq:tree-loc} for all $\Tp\in\TF_{n,k,j}'$ we obtain
\begin{equation}
\label{eq:Fef-forest-local}
\norm{T_{\TF_{n,k,j}'} \one_{F}}_{2\to 2} \lesssim \kappa^{\alpha} 2^{-n \epsilon}
\end{equation}
for any $0\leq \alpha < 1/2$.
\end{proposition}
\begin{proof}
We subdivide $\TF_{n,k,j}'$ into rows by the following procedure: for each $m\geq 0$ let inductively $\TR_{n,k,m} = \cup_{l\in L(k,m)}\TN_{n,k,j,l}$ be the union of a maximal set of trees whose spatial cubes are disjoint and maximal among those that have not been selected yet.
This procedure terminates after $O(2^{n} \log(n+1))$ steps because the tree top cubes have overlap bounded by $O(2^{n} \log(n+1))$.
Applying Corollary~\ref{cor:tree} with the set $F$ and with the set $F$ replaced by $\R^{\ds}$ to each tree we obtain
\[
\norm{T_{\TN_{n,k,j,l}} \one_{F}}_{2\to 2} \lesssim \kappa^{1/2} 2^{-n/2},
\quad
\norm{T_{\TN_{n,k,j,l}}}_{2\to 2} \lesssim 2^{-n/2}.
\]
Using normality of the trees and disjointness of their top cubes we obtain
\begin{equation}
\label{eq:row-est}
\norm{T_{\TR_{n,k,m}} \one_{F}}_{2\to 2} \lesssim \kappa^{1/2} 2^{-n/2},
\quad
\norm{T_{\TR_{n,k,m}}}_{2\to 2} \lesssim 2^{-n/2}.
\end{equation}
Using the fact that
\begin{equation}
\label{eq:rows-disjoint}
T_{\TR_{n,k,m}}^{*} T_{\TR_{n,k,m'}} = 0
\text{ for } m \neq m'
\end{equation}
due to disjointness of $E(\Tp)$ for tiles that belong to separated trees as well as Lemma~\ref{lem:row} and an orthogonality argument we obtain \eqref{eq:Fef-forest-global}.

Using \eqref{eq:rows-disjoint} and \eqref{eq:row-est} gives
\begin{align*}
\norm{T_{\TF'_{n,k,j}} \one_{F} f}_{2}
&=
\big( \sum_{m \lesssim 2^{n} \log(n+1)} \norm{T_{\TR_{n,k,m}} \one_{F} f}_{2}^{2} \big)^{1/2}\\
&\lesssim
\big( \sum_{m \lesssim 2^{n} \log(n+1)} (\kappa^{1/2} 2^{-n/2} \norm{f}_{2})^{2} \big)^{1/2}\\
&\lesssim
\kappa^{1/2} 2^{-n/2} \norm{f}_{2} (2^{n} \log(n+1))^{1/2}\\
&\lesssim
\kappa^{1/2} (\log(n+1))^{1/2} \norm{f}_{2}.
\end{align*}
Taking a geometric average with \eqref{eq:Fef-forest-global} we obtain \eqref{eq:Fef-forest-local}.
\end{proof}

\subsection{Orthogonality between stopping generations}

\begin{lemma}\label{lem:T*-stopping-cut}
Let $\TT \subset \TP_{k}$ be a tree and $k'>k$.
Then
\[
\norm{T_{\TT} \one_{F_{k'}}}_{2\to 2}
\lesssim
e^{-(k'-k)},
\]
where $F_{k'} = \cup_{F\in\calF_{k'}} F$.
\end{lemma}

\begin{proof}
Let $\calJ := \calJ(\TT)$ and $J\in\calJ$, so that $100D \hat{J} \supseteq I_{\Tp}$ for some $\Tp\in\TT$.

Let $F'\in \calF_{k+1}$ be such that $J\cap F' \neq\emptyset$.
Suppose that $\scale(F') \geq \scale(J) + 4$.
Then $(1+\frac1D) F' \supset 100D \hat{J} \supseteq I_{\Tp}$ and $\scale(F') > \scale(\Tp)$.
By part \ref{it:stopping-neighbors} of Lemma~\ref{lem:spatial-decomposition} this implies $I\in\calF_{k+1}$ for some $I\supseteq I_{\Tp}$, contradicting $I_{\Tp}\in \calC_{k}$.

Therefore we must have $\scale(F') \leq \scale(J) + 3$, and it follows that
\[
\sum_{F'\in\calF_{k+1} : J\cap F' \neq\emptyset} \abs{F'} \lesssim \abs{J}.
\]
Hence
\begin{multline*}
\abs{J \cap F_{k'}}
\leq
\sum_{F'\in\calF_{k+1} : J\cap F' \neq\emptyset} \abs{F' \cap F_{k'}}\\
\lesssim
\sum_{F'\in\calF_{k+1} : J\cap F' \neq\emptyset} e^{-2(k'-k-1)} \abs{F'}
\lesssim
e^{-2(k'-k)} \abs{J}.
\end{multline*}
This implies $\norm{P_{\calJ} \one_{F_{k'}}}_{2\to 2} \lesssim e^{-(k'-k)}$, and the claim follows from Lemma~\ref{lem:tree}.
\end{proof}

\begin{proposition}
\label{prop:Fef-forest:small-support}
For any measurable subset $F'\subset \R^{\ds}$ we have
\begin{align}
\norm{T_{\TF_{n,k,j}'}^{*} T_{\TF_{n,k',j'}'}}_{2 \to 2} &\lesssim 10^{n} e^{-\abs{k-k'}}, \label{eq:gen-T*T}\\
\norm{T_{\TF_{n,k,j}'} \one_{F'} T_{\TF_{n,k',j'}'}^{*}}_{2 \to 2} &\lesssim 10^{n} e^{-\abs{k-k'}}. \label{eq:gen-TT*}
\end{align}
\end{proposition}

\begin{proof}
Let $\TR_{n,k,m}$ be the rows defined in the proof of Proposition~\ref{prop:Fef-forest}.
It suffices to show
\begin{align}
\norm{T_{\TR_{n,k,m}}^{*} T_{\TR_{n,k',m'}}}_{2 \to 2} &\lesssim e^{-\abs{k-k'}}, \label{eq:gen-T*T:row}\\
\norm{T_{\TR_{n,k,m}} \one_{F'} T_{\TR_{n,k',m'}}^{*}}_{2 \to 2} &\lesssim e^{-\abs{k-k'}}. \label{eq:gen-TT*:row}
\end{align}

Without loss of generality we may assume $k'\geq k$.
We will use the fact that
\[
T_{\TR_{n,k',m'}}
=
\one_{F_{k'}} T_{\TR_{n,k',m'}}
=
T_{\TR_{n,k',m'}} \one_{F_{k'}}
\]
with $F_{k'} = \cup_{F\in\calF_{k'}}F$ (the last equality uses normality of the trees).

Using \eqref{eq:row-est} we estimate
\begin{align*}
LHS\eqref{eq:gen-T*T:row}
&=
\norm{ T_{\TR_{n,k,m}}^{*} \one_{F_{k'}} T_{\TR_{n,k',m'}} }_{2\to 2}\\
&\leq
\norm{ T_{\TR_{n,k,m}}^{*} \one_{F_{k'}} }_{2\to 2} \norm{ T_{\TR_{n,k',m'}} }_{2\to 2}\\
&\lesssim
\norm{ \one_{F_{k'}} T_{\TR_{n,k,m}} }_{2\to 2}.
\end{align*}
As a consequence of \eqref{eq:Lie-support-decay} we have
\[
\norm{P_{\calL(\TN_{n,k,j,l})} \one_{F_{k'}}}_{2\to 2} \lesssim e^{-\abs{k-k'}},
\]
and \eqref{eq:gen-T*T:row} follows from Lemma~\ref{lem:tree}.
Similarly,
\begin{align*}
LHS\eqref{eq:gen-TT*:row}
&=
\norm{ T_{\TR_{n,k,m}} \one_{F'} \one_{F_{k'}} T_{\TR_{n,k',m'}}^{*} }_{2\to 2}\\
&\leq
\norm{ T_{\TR_{n,k,m}} \one_{F'\cap F_{k'}} }_{2\to 2} \norm{ T_{\TR_{n,k',m'}}^{*} }_{2\to 2}\\
&\lesssim
\norm{ T_{\TR_{n,k,m}} \one_{F_{k'}} }_{2\to 2}\\
&\lesssim
e^{-\abs{k-k'}}
\end{align*}
by Lemma~\ref{lem:T*-stopping-cut}.
\end{proof}

\section{Proof of Theorem~\ref{thm:loc}}
\label{sec:proof-loc}
\begin{proof}[Proof of \eqref{eq:loc:full}]
Using the decomposition \eqref{eq:tree-dec} we split
\begin{multline*}
\norm{T_{\TP}}_{2\to 2}
\leq
\sum_{n=1}^{\infty} \sum_{j=1}^{C n^{2}} \Big( \norm[\big]{ \sum_{k\in\N} T_{\TF_{n,k,j}'}}_{2\to 2}
+ \norm[\big]{ \sum_{k\in\N} T_{\TA_{n,k,j}}}_{2\to 2}\\
+ \norm[\big]{ \sum_{k\in\N} \sum_{l} T_{\bd(\TT_{n,k,j,l})}}_{2\to 2} \Big).
\end{multline*}
The contribution of the last two summands is estimated by Proposition~\ref{prop:sf}.
In the first summand we split the summation over $k$ in congruence classes modulo $C n$ and use Propositions \ref{prop:Fef-forest}, \ref{prop:Fef-forest:small-support}, and the Cotlar--Stein Lemma (see e.g.\ \cite[Section VII.2]{MR1232192}).
\end{proof}

\begin{proof}[Proof of \eqref{eq:loc:G}]
Let $\TP_{\tilde G} := \Set{ \Tp\in\TP \given I_{\Tp}^{*} \subseteq \tilde{G}}$, then $T_{\Tp}\one_{\R^{\ds}\setminus \tilde{G}}=0$ if $\Tp\in\TP_{\tilde G}$.
Hence
\[
\norm{ \one_{G} T_{\TP} \one_{\R^{\ds} \setminus \tilde{G}}}_{2 \to 2}
=
\norm{ \one_{G} T_{\TP \setminus \TP_{\tilde G}} \one_{\R^{\ds} \setminus \tilde{G}}}_{2 \to 2}
\leq
\norm{ \one_{G} T_{\TP \setminus \TP_{\tilde G}}}_{2 \to 2}.
\]
In order to estimate the latter quantity we run the proof of \eqref{eq:loc:full} with $\TP$ replaced by $\TP\setminus\TP_{\tilde{G}}$.
The main changes are that all tiles now have density $2^{n} \lesssim \nu$.
This yields the required improvement in the estimate for the main term.
In the error terms we use Proposition~\ref{prop:sf-loc} with $F=\R^{\ds}$.
The hypothesis \eqref{eq:sf-loc:assume} is satisfied because we have removed all tiles whose spatial cubes are contained in $\tilde{G}$.
\end{proof}

\begin{proof}[Proof of \eqref{eq:loc:F}]
Let $\TP_{\tilde F} := \Set{ \Tp\in\TP \given I_{\Tp} \subseteq \tilde{F}}$, then $\one_{\R^{\ds}\setminus \tilde{F}} T_{\Tp}=0$ if $\Tp\in\TP_{\tilde F}$.
Hence
\[
\norm{ \one_{\R^{\ds} \setminus \tilde{F}} T_{\TP} \one_{F}}_{2 \to 2}
=
\norm{ \one_{\R^{\ds} \setminus \tilde{F}} T_{\TP \setminus \TP_{\tilde F}} \one_{F}}_{2 \to 2}
\leq
\norm{ T_{\TP \setminus \TP_{\tilde F}} \one_{F}}_{2 \to 2}.
\]
In order to estimate the latter term we again run the proof of \eqref{eq:loc:full} with $\TP$ replaced by $\TP \setminus \TP_{\tilde F}$.
In particular, we split
\begin{multline*}
\norm{T_{\TP} \one_{F}}_{2\to 2}
\leq
\sum_{n=1}^{\infty} \sum_{j=1}^{C n^{2}} \Big( \norm[\big]{ \sum_{k\in\N} T_{\TF_{n,k,j}'} \one_{F}}_{2\to 2}
+ \norm[\big]{ \sum_{k\in\N} T_{\TA_{n,k,j}} \one_{F}}_{2\to 2}\\
+ \norm[\big]{ \sum_{k\in\N} \sum_{l} T_{\bd(\TT_{n,k,j,l})} \one_{F}}_{2\to 2} \Big).
\end{multline*}
The contribution of the last two terms is taken care of by Proposition~\ref{prop:sf-loc} with $G=\R^{\ds}$.
In the estimate for the main term we use \eqref{eq:Fef-forest-local} in place of \eqref{eq:Fef-forest-global} and split the summation over $k$ in congruence classes modulo $\lceil C n (\abs{\log \kappa} + 1) \rceil$.
\end{proof}

\section{A van der Corput type oscillatory integral estimate}
\label{sec:vdC}
We use the following van der Corput type estimate for oscillatory integrals in $\R^{\ds}$ that refines \cite[Proposition 2.1]{MR1879821}.

\begin{lemma}
\label{lem:osc-int}
Let $\psi : \R^{\ds}\to\C$ be a measurable function with $\supp\psi\subset J$ for a cube $J$.
Then for every $Q\in\calQ_{d}$ we have
\[
\abs[\big]{\int_{\R^{\ds}} e(Q(x)) \psi(x) \dif x}
\lesssim
\sup_{\abs{y}<\Delta^{-1/d} \ell(J)} \int_{\R^{\ds}} \abs{\psi(x)-\psi(x-y)} \dif x,
\quad
\Delta = \norm{Q + \R}_{J}+1.
\]
\end{lemma}
\begin{proof}
By scaling and translation we may assume $\ell(J)\sim 1$ and $J\subset B(0,1/2)$.
Let $\beta$ denote the right-hand side of the conclusion.
If $\Delta \lesssim 1$, then $\norm{\psi}_{1} \lesssim \beta$, so the result is only non-trivial if $\Delta \gg 1$.
In this case we replace $\psi$ on the left-hand side by $\tilde\psi:=\phi*\psi$, where $\phi = \Delta^{\ds/d} \phi_{0}(\Delta^{1/d} \cdot)$ and $\phi_{0}$ is a smooth positive bump function with integral $1$ supported on the unit ball.
The error term is controlled by
\begin{align*}
\int \abs{\psi-\tilde\psi}(x) \dif x
&=
\int \abs[\big]{\int (\psi(x)-\psi(x-y))\phi(y) \dif y} \dif x\\
&\leq
\int \phi(y) \int \abs{\psi(x)-\psi(x-y)} \dif x \dif y\\
&\leq
\beta.
\end{align*}
Moreover, $\supp \tilde\psi \subseteq B(0,1)$ and
\begin{align*}
\int \abs{\partial_{i} \tilde\psi(x)} \dif x
&=
\int \abs{\int \psi(x-y) \partial_{i} \phi(y) \dif y} \dif x\\
&=
\int \abs{\int (\psi(x)-\psi(x-y)) \partial_{i} \phi(y) \dif y} \dif x\\
&\leq
\int \int \abs{\psi(x)-\psi(x-y)} \abs{\partial_{i} \phi(y)} \dif y \dif x\\
&\lesssim
\Delta^{\ds/d+1/d}\int \int_{B(0,\Delta^{-1/d})} \abs{\psi(x)-\psi(x-y)} \dif y \dif x\\
&\lesssim
\Delta^{1/d} \beta
\end{align*}
for every $i=1,\dotsc,\ds$.
The result now follows from the proof of \cite[Proposition 2.1]{MR1879821} applied to $\tilde\psi$.
Notice that the one-dimensional van der Corput estimate (Corollary on p.\ 334 of \cite{MR1232192}) used in that proof only uses an estimate on $\nabla \tilde\psi$.
\end{proof}

\section{Localized estimates for the Hardy--Littlewood maximal operator}
\label{sec:HL-loc-est}
In Section~\ref{sec:proof-loc} we have proved Theorem~\ref{thm:loc} for the linearized smoothly truncated operators \eqref{eq:T-linearized}.
By Bateman's extrapolation we could now deduce $L^{p}$ estimates for these operators and obtain Theorem~\ref{thm:poly-car} using $L^{p}$ estimates for the Hardy--Littlewood maximal operator.
In order to show Theorem~\ref{thm:loc} for the sharply truncated operator \eqref{eq:Car-op-Kd} we need a corresponding localized estimate for the Hardy--Littlewood maximal operator.
\begin{lemma}
Let $0\leq \alpha < 1/2$ and $0 < \nu \leq 1$.
Let $G\subset\R^{\ds}$ be a measurable subset and $\tilde G := \Set{ M \one_{G} > \nu }$.
Then
\[
\norm{\one_{G} M \one_{\R^{\ds} \setminus \tilde{G}}}_{2\to 2} \lesssim_{\alpha} \nu^{\alpha},
\quad
\norm{\one_{\R^{\ds} \setminus \tilde{G}} M \one_{G}}_{2\to 2} \lesssim_{\alpha} \nu^{\alpha}.
\]
\end{lemma}
We have not attempted to optimize the conclusion of this lemma.
\begin{proof}
By the Fefferman--Stein maximal inequality \cite{MR0284802} we have
\[
\norm{\one_{G} M \one_{\R^{\ds} \setminus \tilde{G}} f}_{1,\infty}
=
\norm{M \one_{\R^{\ds} \setminus \tilde{G}} f}_{L^{1,\infty}(\one_{G})}
\lesssim
\norm{\one_{\R^{\ds} \setminus \tilde{G}} f}_{L^{1}(M \one_{G})}
\leq
\nu \norm{f}_{1}.
\]
Interpolating with the trivial $L^{\infty}$ estimate we obtain the first claim.
Let now $q=1/\alpha$.
Then by H\"older's inequality
\[
M \one_{G} f
\leq
(M_{q} \one_{G}) (M_{q'} f)
=
(M \one_{G})^{\alpha} (M_{q'} f).
\]
Hence
\[
\norm{\one_{\R^{\ds} \setminus \tilde{G}} M \one_{\tilde{G}} f}_{2}
\leq
\norm{\one_{\R^{\ds} \setminus \tilde{G}} (M\one_{G})^{\alpha} M_{q'} f}_{2}
\leq
\nu_{\alpha} \norm{M_{q'} f}_{2}
\lesssim
\nu_{\alpha} \norm{f}_{2},
\]
where we have used the fact that $M_{q'}$ is bounded on $L^{2}$ provided that $q'<2$.
\end{proof}


%% file: single-scale.tex
\chapter{Lipschitz vector fields}
\label{chap:bi-Lipschitz}
The main objective of this chapter is Theorem~\ref{thm:Lip-LP-diag} that we now recall.
\begin{theorem*}
Let $A : \R\to\R$ be a Lipschitz function with $\norm{A}_{\Lip} \leq 1/100$ and consider the change of variable $T_{A}f(x):=f(x+A(x))$.

Let $\psi$ be a Schwartz function on $\R$ such that $\widehat{\psi}$ identically equals $1$ on $\pm [99/100,103/100]$ and vanishes outside $\pm [98/100,104/100]$.
Let $\Psi$ be another Schwartz function on $\R$ such that $\widehat{\Psi}$ is supported on $\pm[1,101/100]$.
Let $P_{t}f:=\psi_{t}*f$ be the Littlewood--Paley operators associated to $\psi$, where $\psi_t(x)=t^{-1}\psi(t^{-1}x)$.
Then
\[
\norm[\Big]{ \sum_{t \in 2^{\Z}} \abs{(1-P_{t})T_{A} (\Psi_{t} * f)} }_{p}
\lesssim_{p,\psi,\Psi}
\norm{A}_{\Lip} \norm{f}_{p},
\quad
1<p<\infty.
\]
\end{theorem*}

\section{Carleson embeddings with compactly supported test functions}
\label{sec:Carlseon-embeddings-compact-support}
We refer to \cite[Section 2 and 3]{MR3312633} for the general theory of outer measure spaces.
In this section we use the outer measure space $X=\R^{d}\times (0,\infty)$ with the collection of distinguished sets $\bfE$ consisting of the \emph{tents}
\[
T(x,s) = \Set{(y,t) : \norm{x-y}+t \leq s}
\]
and an outer measure $\mu$ generated by $\sigma(T(x,s))=s^{d}$.

Let $\omega$ be a \emph{Dini modulus of continuity}, that is, $\omega : [0,\infty) \to [0,\infty)$ is a function that is subadditive in the sense
\[
u\leq s+t \implies \omega(u) \leq \omega(s) + \omega(t)
\]
and has finite Dini norm $\norm{\omega}_{\Dini} = \int_{0}^{1} \omega(t) \frac{\dif t}{t}$.
Let $\calC$ be the class of testing functions $\phi : \R^{d}\to\C$ that satisfy
\begin{align}
\label{C:cancel}
\textstyle{\int} \phi(z) \dif z &= 0,\\
\label{C:decay}
\supp \phi
&\subset
B(0,1)\\
\label{C:cont}
\abs{\phi(z)-\phi(z')}
&\leq
\omega(\norm{z-z'})
\quad\text{for all } z,z'\in\R^{d}.
\end{align}

For locally integrable functions $f$ we define the embeddings
\begin{align*}
\Aem_{c} f(x,t) &:= t^{-d} \int_{B(x,t)} \abs{f},\\
\Dem_{c} f(x,t) &:= \sup_{\phi\in\calC} \abs[\big]{t^{-d} \int f(z) \phi(t^{-1}(y-z)) dz}.
\end{align*}

\begin{theorem}[{cf.~\cite[Theorem 4.1]{MR3312633}}]
\label{thm:Carlseon-embeddings-compact-support}
For every $1<p\leq\infty$ we have
\begin{align*}
\norm{\Aem_{c} f}_{L^{p}(S^{\infty})} &\lesssim \norm{f}_{L^{p}(\R^{d})},\\
\norm{\Dem_{c} f}_{L^{p}(S^{2})} &\lesssim \norm{f}_{L^{p}(\R^{d})}.
\end{align*}
Moreover, we have the endpoint estimates
\begin{align*}
\norm{\Aem_{c} f}_{L^{1,\infty}(S^{\infty})} &\lesssim \norm{f}_{L^{1}(\R^{d})},\\
\norm{\Dem_{c} f}_{L^{1,\infty}(S^{2})} &\lesssim \norm{f}_{L^{1}(\R^{d})}.
\end{align*}
\end{theorem}
The main difference from \cite[Theorem 4.1]{MR3312633} is the supremum over $\phi\in\calC$ in the definition of $\Dem_{c}$, whereas \cite[Theorem 4.1]{MR3312633} uses a fixed $\phi$.
This supremum does not affect the proof strongly, but is important for our application.
The precise choice of the class of test functions $\calC$ is not important for this application (e.g.\ Lipschitz functions would work equally well), but the Dini regularity condition appears naturally in the proof.

We linearize the supremum in the definition of $\Dem_{c} f$ by choosing for each pair $(y,t)$ a function $\phi\in\mathcal{C}$ for which the supremum is almost attained.
Denote then $\phi_{y,t}(z) = t^{-d}\phi(t^{-1}(y-z))$.
This is an $L^{1}$ normalized wave packet at scale $t$.
The almost orthogonality of these wave packets is captured by the following estimate.

\begin{lemma}
If $t\leq t'$ then
\label{lem:wave-packet-correlation-decay}
\[
\abs{\innerp{ \phi_{y,t}}{\phi_{y',t'}}}
\lesssim
(t')^{-d} \omega(t/t')
\]
\end{lemma}

\begin{proof}
Using the cancellation condition \eqref{C:cancel} and the support condition we write
\begin{multline*}
\abs[\big]{\int_{\R^{d}} \phi_{y,t} (z) \phi_{y',t'}(z) \dif z}
=
\abs[\big]{\int_{B(y,t)} \phi_{y,t} (z) (\phi_{y',t'}(z) - \phi_{y',t'}(y)) \dif z}\\
\leq
\int_{B(y,t)} \abs{\phi_{y,t} (z)} (t')^{-d} \omega(t/t') \dif z
\lesssim
(t')^{-d} \omega(t/t').
\qedhere
\end{multline*}
\end{proof}

We use the almost orthogonality statement in Lemma~\ref{lem:wave-packet-correlation-decay} to deduce a square function estimate for $p=2$.
\begin{lemma}
\label{lem:L2-square-fct}
\begin{equation}
\label{0114e2.7}
\int_{\R^d \times \R_{>0} } \abs{\Dem_{c} f(y,t)}^2 \dif y \frac{\dif t}{t}
\lesssim
\norm{f}_{2}^{2}.
\end{equation}
\end{lemma}
\begin{proof}
We begin with a measurable selection of functions $\phi_{y,t}$ that almost extremize $\Dem_{c} f(y,t)$.
Expand the square of the left hand side of \eqref{0114e2.7}
\begin{align*}
\left( \int \abs{\innerp{f}{\phi_{y,t}}}^{2} \dif y \frac{\dif t}{t} \right)^2
&=
\left( \int_{\R^d} \left( \int_{\R^d \times \R_{>0}} \innerp{f}{\phi_{y,t}} \phi_{y,t}(z) \dif y \frac{\dif t}{t} \right) \overline{f(z)} \dif z\right)^2 \\
&\leq
\norm[\big]{\int_{\R^d \times \R_{>0}} \innerp{f}{\phi_{y,t}} \phi_{y,t}(z) \dif y \frac{\dif t}{t} }_{2}^{2} \norm{f}_2^2 \\
\intertext{We further expand the square from the former term }
&=
\iint \innerp{f}{\phi_{y,t}} \innerp{\phi_{y,t}}{\phi_{y',t'}} \innerp{\phi_{y',t'}}{f} \dif y \frac{\dif t}{t} \dif y' \frac{\dif t'}{t'} \norm{f}_2^2 \\
&\leq
\int \abs{\innerp{f}{\phi_{y,t}}}^2 \int \abs{\innerp{\phi_{y,t}}{\phi_{y',t'}}} \dif y' \frac{\dif t'}{t'} \dif y \frac{\dif t}{t} \norm{f}_2^2,
\end{align*}
using the estimate
\[
2\abs{\innerp{f}{\phi_{y,t}} \innerp{\phi_{y',t'}}{f}} \leq
\abs{\innerp{f}{\phi_{y,t}}}^2 + \abs{\innerp{f}{\phi_{y',t'}}}^2
\]
in the last inequality.
It suffices to verify
\[
\sup_{y,t} \int \abs{\innerp{\phi_{y,t}}{\phi_{y',t'}}} \dif y' \frac{\dif t'}{t'} < \infty.
\]
By Lemma~\ref{lem:wave-packet-correlation-decay} and using bounded support of the $\phi_{y,t}$'s we have
\begin{align*}
\int \abs{\innerp{\phi_{y,t}}{\phi_{y',t'}}} \dif y' \frac{\dif t'}{t'}
&\lesssim
\int_{t\leq t'} \int_{\norm{y-y'}\leq t+t'} (t')^{-d} \omega(t/t') \dif y' \frac{\dif t'}{t'}\\
\qquad + \int_{t>t'} \int_{\norm{y-y'}\leq t+t'} t^{-d} \omega(t'/t) \dif y' \frac{\dif t'}{t'} \\
&\lesssim
\int_{t\leq t'} \omega(t/t') \frac{\dif t'}{t'}
+ \int_{t>t'} \omega(t'/t) \frac{\dif t'}{t'}
\lesssim
\norm{\omega}_{\Dini}.
\end{align*}
This finishes the proof of Lemma~\ref{lem:L2-square-fct}.
\end{proof}

\begin{proof}[Proof of Theorem~\ref{thm:Carlseon-embeddings-compact-support}]
We may assume that the superlevel sets $\Set{Mf>\lambda}$, where $M$ is the uncentered Hardy--Littlewood maximal function, have finite measure for all $\lambda>0$, since otherwise the right-hand side of the conclusion is infinite.

Let $\Set{Q_{i}}_{i}$ be a Whitney decomposition of the superlevel set $\Set{Mf>\lambda}$.
Let $x_{i}$ denote the center and $r_{i}$ the diameter of $Q_{i}$.
Let
\begin{equation}
E:=\bigcup_{i} T(x_{i},3\sqrt{d}r_{i})
\end{equation}
and note that
\[
\mu(E) \lesssim \meas{\Set{Mf>\lambda}}.
\]
The claim of the theorem will therefore follow from the more precise results
\begin{align}
\label{eq:Aem-outside-whitney-tents}
\norm{\Aem_{c} f \one_{E^{c}}}_{L^{\infty}(S^{\infty})} &\lesssim \lambda,\\
\label{eq:Dem-outside-whitney-tents}
\norm{\Dem_{c} f \one_{E^{c}}}_{L^{\infty}(S^{2})} &\lesssim \lambda.
\end{align}
Let $(x,t)\in E^{c}$.
Then no ball $B(y,t/\sqrt{d})$ with $\norm{x-y}\leq t$ is contained in a Whitney cube.
It follows that, for some constant $C$ that depends only on the dimension, the ball $B(x,Ct)$ is not contained in $\Set{Mf>\lambda}$.
Hence
\[
\Aem_{c} f(x,t)
\leq
t^{-d} \int_{B(x,Ct)} \abs{f}
\lesssim
\lambda.
\]
This completes the proof of \eqref{eq:Aem-outside-whitney-tents}.
Now we show \eqref{eq:Dem-outside-whitney-tents}.
The Calder\'on--Zygmund decomposition $f=g+b$, $b=\sum_{i}b_{i}$ associated to the Whitney decomposition $\Set{Q_{i}}_{i}$ has the properties
\begin{enumerate}
\item $\norm{g}_{\infty} \lesssim \lambda$,
\item $\supp b_{i} \subset Q_{i}$,
\item $\int b_{i} = 0$,
\item $\meas{Q_{i}}^{-1} \int \abs{b_{i}} \lesssim \lambda$.
\end{enumerate}
Using the bounded support condition on the wave packets and Lemma~\ref{lem:L2-square-fct} we obtain
\begin{align*}
S^2(\Dem_{c}g) (T(x,s))
&=
\left( \frac{1}{s^d} \int_{T(x,s)} \abs{\Dem_{c} g(y,t)}^2 \dif y \frac{\dif t}{t}\right)^{1/2}\\
&=
\left( \frac{1}{s^d} \int_{T(x,s)} \abs{\Dem_{c} (g\one_{B(x,2s)})(y,t)}^2 \dif y \frac{\dif t}{t}\right)^{1/2}\\
&\lesssim
s^{-d/2} \norm{g\one_{B(x,2s)}}_{2}
\lesssim
\norm{g}_{\infty}.
\end{align*}
Hence \eqref{eq:Dem-outside-whitney-tents} holds with $f$ replaced by $g$.
By sublinearity of the embedding map $\Dem$ and subadditivity of the outer $L^{\infty}(S^{2})$ norm it remains to show \eqref{eq:Dem-outside-whitney-tents} holds with $f$ replaced by $b$.
More explicitly, for every tent $T=T(x,r)$ we want to show
\[
S^{2}(\Dem_{c} b \one_{E^{c}})(T)
\lesssim
\lambda.
\]
We know
\[
S^{\infty}(\Dem_{c} b \one_{E^{c}})(T)
\lesssim
S^{\infty}(\Dem_{c} f \one_{E^{c}})(T)
+
S^{\infty}(\Dem_{c} g \one_{E^{c}})(T)
\lesssim
S^{\infty}(\Aem_{c} f \one_{E^{c}})(T)
+
\lambda
\lesssim
\lambda.
\]
By logarithmic convexity of $S^{p}$ sizes it therefore suffices to show
\begin{equation}
\label{eq:bad-fct-L1-embed}
S^{1}(\Dem_{c} b \one_{E^{c}})(T)
\lesssim
\lambda.
\end{equation}

\begin{claim}
\label{claim:bad-atom-L1-embed}
$\int_{t>r_{i}} \Dem_{c} b_{i}(x,t) \dif x \frac{\dif t}{t} \lesssim \lambda r_{i}^{d}$.
\end{claim}
\begin{proof}[Proof of Claim~\ref{claim:bad-atom-L1-embed}.]
Notice that, due to support constraints, $\Dem_{c} b_{i}(x,t)$ can only be non-zero if $\norm{x-x_{i}}\leq r_{i}/2+t$.
Moreover, under this condition and choosing $\phi_{x,t}$ that almost extremizes $\Dem_{c} b_{i}(x,t)$ we obtain
\begin{align*}
\abs[\big]{\int b_{i}(z) \phi_{x,t}(z) \dif z}
&=
\abs[\big]{\int_{\norm{z-x_{i}}\leq r_{i}/2} b_{i}(z) (\phi_{x,t}(z) - \phi_{x,t}(x_{i})) \dif z}\\
&\leq
t^{-d} \omega(r_{i}/(2t)) \int_{\norm{z-x_{i}}\leq r_{i}/2} \abs{b_{i}(z)} \dif z\\
&\lesssim
t^{-d} \omega(r_{i}/(2t)) \lambda r_{i}^{d}.
\end{align*}
Hence
\begin{multline*}
\int_{t>r_{i}} \Dem_{c} b_{i}(x,t) \dif x \frac{\dif t}{t}
\leq
\int_{t>r_{i}, \norm{x-x_{i}}\leq t+r_{i}/2} \Dem_{c} b_{i}(x,t) \dif x \frac{\dif t}{t}\\
\lesssim
\int_{t>r_{i}} \omega(r_{i}/(2t)) \lambda r_{i}^{d} \frac{\dif t}{t}
\lesssim
\lambda r_{i}^{d} \norm{\omega}_{\Dini}.
\end{multline*}
This finishes the proof of Claim~\ref{claim:bad-atom-L1-embed}.
\end{proof}

In order to show \eqref{eq:bad-fct-L1-embed} notice that only the Whitney cubes $Q_{i} \subset B(x,10 r)$ contribute to $\Dem_{c} b \one_{T\setminus E}$.
\begin{align*}
S^{1}(\Dem_{c} b \one_{E^{c}})(T)
&=
r^{-d} \int_{T\setminus E} \Dem_{c} b(z,t) \dif z \frac{\dif t}{t}\\
&\leq
r^{-d} \sum_{i : Q_{i} \subset B(x,10 r)} \int_{T\setminus E} \Dem_{c} b_{i}(z,t) \dif z \frac{\dif t}{t}\\
&\leq
r^{-d} \sum_{i : Q_{i} \subset B(x,10 r)} \int_{t>r_{i}} \Dem_{c} b_{i}(z,t) \dif z \frac{\dif t}{t}\\
\intertext{using Claim~\ref{claim:bad-atom-L1-embed}}
&\lesssim
r^{-d} \lambda \sum_{i : Q_{i} \subset B(x,10 r)} \meas{Q_{i}}\\
\intertext{by disjointness of Whitney cubes}
&\lesssim
r^{-d} \lambda \meas{B(x,10 r)}
\lesssim
\lambda.
\end{align*}
This finishes the proof of Theorem~\ref{thm:Carlseon-embeddings-compact-support}.
\end{proof}

\section{Carleson embeddings with tails}
It is possible to adapt the proofs in Section~\ref{sec:Carlseon-embeddings-compact-support} to embeddings defined using test functions with tails.
Since we do not need testing functions with sharp decay rates for tails, we will instead estimate such embeddings by averaging the results in Section~\ref{sec:Carlseon-embeddings-compact-support}.

In this section we work in dimension $d=1$ and consider the following embedding maps:
\begin{align}
\Aem f(x,t) &:= \int t^{-1} (1+\abs{x-y}/t)^{-5} \abs{f(y)} \dif y,\\
\Dem f(x,t) &:= \sup_{\phi\in\Phi} \abs[\big]{\int t^{-1} \phi((x-y)/t) f(y) \dif y},
\end{align}
where
\[
\Phi = \Set{ \phi:\R\to\C ,\ \int\phi=0,\ \abs{\phi(x)}\leq (1+\abs{x})^{-10},\ \abs{\phi'(x)}\leq (1+\abs{x})^{-10}}.
\]
The smoothness and decay conditions in these embeddings are not optimal, but they suffice for our purposes.
Decomposing the testing functions $(1+\abs{x})^{-5}$ and $\phi\in\Phi$ into series of compactly supported bump functions as in \cite[Lemma 3.1]{MR2320408}, see also Lemma~\ref{lem:split-wave-packet} in this article, we can deduce the embeddings
\begin{align}
\label{eq:A-embedding}
\norm{\Aem f}_{L^{p}(S^{\infty})} &\lesssim \norm{f}_{p},\\
\label{eq:D-embedding}
\norm{\Dem f}_{L^{p}(S^{2})} &\lesssim \norm{f}_{p}
\end{align}
for $1<p\leq\infty$ from Theorem~\ref{thm:Carlseon-embeddings-compact-support}.

\section{Jones beta numbers}
Let $A : \R\to\C$ be a Lipschitz function and let $a$ be its distributional derivative, so that $\norm{a}_{\infty} = \norm{A}_{\Lip}$.
Let $\psi$ be a compactly supported bump function with
\begin{equation}\label{0111e1.6}
\int\psi(x)\dif x = \int x\psi(x) \dif x = 0
\end{equation}
and
\[
\int_{0}^{\infty} \hat\psi(\xi\xi_{0}) \frac{\dif\xi}{\xi} = 1
\quad\text{for}\quad
\xi_{0}\neq 0.
\]
Let $\psi_{t} = t^{-1} \psi(t^{-1}\cdot)$ be an $L^{1}$ normalized mean zero bump function at scale $t$.
Let
\begin{equation}
\label{eq:av-slope}
\alpha(x,t) := \int_{t}^{\infty} a * \psi_{s}(x) \frac{\dif s}{s}
\end{equation}
be the average slope of $A$ near $x$ at scale $t$ and let
\begin{equation}
\label{eq:beta-n-number}
\beta_{n}(x,t) := \sup_{x_{0},x_{1},x_{2} \in B(x,2^{n} \cdot 3t), 2^{-n}t \leq \tilde t \leq 2^{n} t} t^{-1} \abs{A(x_{2})-A(x_{1})-\alpha(x_{0},\tilde t)(x_{2}-x_{1})}.
\end{equation}
This definition includes the supremum over the range of uncertainty around $(x,t)$, which seems convenient.
\begin{lemma}\label{0111lemma1}
With the notation \eqref{eq:beta-n-number} we have
\label{lem:beta-square}
\[
\beta_{n}(x,t)
\lesssim
t^{-1} \int_{0}^{2^{n} t} \Big( \int_{\abs{y-x}\lesssim 2^{n} t} \Dem a(y,s)^{2} s^{-1} \dif y \Big)^{1/2} \dif s
+
2^{2n}\int_{2^{n}t}^{\infty} \Dem a(x,s) \frac{t \dif s}{s^{2}}\\
\]
\end{lemma}
\begin{proof}
Let $x_{0},x_{1},x_{2}\in B(x,2^{n}\cdot 3t)$, $2^{-n}t\leq \tilde t \leq 2^{n}t$.
By the fundamental theorem of calculus and Calder\'on's reproducing formula for $a$ we can write
\begin{align*}
\MoveEqLeft
t^{-1} (A(x_{2})-A(x_{1})-\alpha(x_{0},\tilde t)(x_{2}-x_{1}))
=
t^{-1} \int_{x_{1}}^{x_{2}} a(y) \dif y - t^{-1} \alpha(x_{0},\tilde t)(x_{2}-x_{1})\\
&=
t^{-1} \int_{x_{1}}^{x_{2}} \int_{0}^{\infty} a * \psi_{s}(y) \frac{\dif s}{s} \dif y
-
t^{-1} \int_{x_{1}}^{x_{2}} \int_{\tilde t}^{\infty} a * \psi_{s}(x_{0}) \frac{\dif s}{s} \dif y.
\intertext{Splitting the integral in $s$ in the former term at $\tilde t$ we further obtain }
&=
t^{-1} \int_{x_{1}}^{x_{2}} \int_{0}^{\tilde t} a * \psi_{s}(y) \frac{\dif s}{s} \dif y
+
t^{-1} \int_{x_{1}}^{x_{2}} \int_{\tilde t}^{\infty} (a * \psi_{s}(y) - a * \psi_{s}(x_{0}) ) \frac{\dif s}{s} \dif y\\
&=:
I + II.
\end{align*}
We estimate the two terms on the right-hand side separately.
In the first term we note $\psi_{s} = s (\tilde\psi_{s})'$, where $\tilde\psi_{s}$ is also an $L^{1}$ normalized mean zero bump function at scale $s$, by assumption \eqref{0111e1.6}.
Therefore
\begin{align*}
I
&\leq
t^{-1} \int_{0}^{2^{n} t} \abs[\Big]{ \int_{x_{1}}^{x_{2}}  a * \psi_{s}(y)\ \dif y} \frac{\dif s}{s}\\
&=
t^{-1} \int_{0}^{2^{n} t} \abs{ a * \tilde\psi_{s}(x_{2}) - a * \tilde\psi_{s}(x_{1}) } \dif s\\
&\lesssim
t^{-1} \int_{0}^{2^{n} t} \sup_{\abs{y-x}\lesssim 2^{n}t} \Dem a(y,s) \dif s.
\end{align*}
Since $\Dem a(\cdot,s)$ is almost constant at scale $s$, this can be further estimated by
\[
I
\lesssim
t^{-1} \int_{0}^{2^{n} t} \Big( \int_{\abs{y-x}\lesssim 2^{n}t} \Dem a(y,s)^{2} s^{-1} \dif y \Big)^{1/2} \dif s.
\]
We split the second term $II\leq II_{a}+II_{b}$ via
\begin{equation}\label{0111e1.10}
\int_{\tilde t}^{\infty} \leq \int_{2^{-n}t}^{2^{n}t} + \int_{2^{n}t}^{\infty}.
\end{equation}
Then
\[
II_{a}
\leq
t^{-1} \int_{2^{-n}t}^{2^{n}t} \sup_{\abs{y-x}\lesssim 2^{n} t}\abs{a * \psi_{s}(y)} \frac{\dif s}{s},
\]
and this can be absorbed into the estimate for $I$.
The latter term from \eqref{0111e1.10} is bounded by
\[
II_{b} \lesssim
t^{-1} \int_{x_{1}}^{x_{2}} \int_{2^{n}t}^{\infty} \abs[\big]{a * [\psi_{s}(\cdot-x+y) - \psi_{s}(\cdot-x+x_{0})](x)} \frac{\dif s}{s} \dif y.
\]
Since $\abs{x-y},\abs{x-x_{0}}\lesssim 2^{n} t\lesssim s$, the function in the square brackets is a mean zero $L^{1}$ normalized bump function at scale $s$ with constant $\lesssim 2^{n}t/s$ by the fundamental theorem of calculus, so
\[
II_{b} \lesssim
t^{-1} \int_{x_{1}}^{x_{2}} \int_{2^{n}t}^{\infty} \Dem a(x,s) \frac{2^{n} t \dif s}{s^{2}} \dif y
\lesssim
2^{2n} \int_{2^{n}t}^{\infty} \Dem a(x,s) \frac{t \dif s}{s^{2}}.
\]
This finishes the proof of Lemma~\ref{0111lemma1}.
\end{proof}
\begin{lemma}[{cf.\ \cite[Lemma 3]{MR1013815}}]
$\norm{\beta_{n}}_{L^{\infty}(S^{2})} \lesssim 2^{3n/2} \norm{a}_{\infty}$.
\end{lemma}
\begin{proof}
We have to show
\[
\frac{1}{t_{0}} \int_{t<t_{0}} \int_{\abs{x-x_{0}}<t_{0}} \beta_{n}(x,t)^{2} \dif x \frac{\dif t}{t}
\lesssim
2^{3n} \norm{a}_{\infty}^{2}
\]
with the implicit constant independent of $(x_{0},t_{0})\in\R\times\R_{+}$.

We estimate the $S^{2}$ size on the tent centered at $x_{0}$ with height $t_{0}$ separately for the two terms in the conclusion of Lemma~\ref{lem:beta-square}.
For the first term we consider the square of the $S^{2}$ size:
\begin{align*}
\MoveEqLeft
\frac{1}{t_{0}} \int_{t<t_{0}} \int_{\abs{x-x_{0}}<t_{0}} \Big( t^{-1} \int_{0}^{2^{n} t} \Big( \int_{\abs{y-x}\lesssim 2^{n} t} \Dem a(y,s)^{2} s^{-1} \dif y \Big)^{1/2} \dif s \Big)^{2} \dif x \frac{\dif t}{t}\\
\intertext{Apply H\"older's inequality in the $s$-variable}
&\leq
\frac{1}{t_{0}} \int_{t<t_{0}} \int_{\abs{x-x_{0}}<t_{0}} \int_{0}^{2^{n} t} \Big( \int_{\abs{y-x}\lesssim 2^{n}t} \Dem a(y,s)^{2} \dif y \Big) \frac{\dif s}{s^{1/2}} \cdot \int_{0}^{2^{n} t} \frac{\dif s}{s^{1/2}} \dif x \frac{\dif t}{t^{3}}\\
\intertext{Change the order of integration}
&\lesssim
\frac{2^{n/2}}{t_{0}} \int_{s\leq 2^{n} t_{0}} \int_{2^{-n}s<t<t_{0}} \int_{\abs{y-x_{0}}\lesssim 2^{n}t_{0}} \int_{\abs{x-y}\lesssim 2^{n}t} \dif x \Dem a(y,s)^{2} \dif y \frac{\dif t}{t^{5/2}} \frac{\dif s}{s^{1/2}}\\
&\lesssim
\frac{2^{2n}}{t_{0}} \int_{s\leq 2^{n} t_{0}} \int_{\abs{y-x_{0}}\lesssim 2^{n}t_{0}} \Dem a(y,s)^{2} \dif y \frac{\dif s}{s} \lesssim
2^{3n} \norm{\Dem a}_{L^{\infty}(S^{2})}^{2}.
\end{align*}
For the second term we consider the $S^{2}$ size
\begin{align*}
\MoveEqLeft
2^{2n} \Big(\frac{1}{t_{0}} \int_{t<t_{0}} \int_{\abs{x-x_{0}}<t_{0}} \Big( \int_{2^{n} t}^{\infty} \Dem a(x,s) \frac{t \dif s}{s^{2}} \Big)^{2} \dif x \frac{\dif t}{t} \Big)^{1/2}\\
\intertext{By applying a change of variable $s\to t\tau$ and Minkowski's integral inequality:}
&\leq
2^{2n}\int_{2^{n}}^{\infty} \Big(\frac{1}{t_{0}} \int_{t<t_{0}} \int_{\abs{x-x_{0}}<t_{0}} \Dem a(x,t\tau)^{2} \dif x \frac{\dif t}{t}\Big)^{1/2} \frac{\dif\tau}{\tau^{2}}\\
&=
2^{2n}\int_{2^{n}}^{\infty} \Big(\frac{1}{\tau t_{0}} \int_{s<\tau t_{0}} \int_{\abs{x-x_{0}}<t_{0}} \Dem a(x,s)^{2} \dif x \frac{\dif s}{s}\Big)^{1/2} \frac{\dif\tau}{\tau^{3/2}}\\
&\lesssim
2^{2n}\int_{2^{n}}^{\infty} \norm{\Dem a}_{L^{\infty}(S^{2})} \frac{\dif\tau}{\tau^{3/2}}\lesssim
2^{3n/2} \norm{\Dem a}_{L^{\infty}(S^{2})}.
\end{align*}
The conclusion follows from \eqref{eq:D-embedding}.
\end{proof}

\begin{corollary}[{cf.~\cite[Lemma 4]{MR1013815}}]
\label{cor:jones-beta}
Let $\epsilon>0$ and
\[
\beta(x,t)
=
\sup_{x_{0},x_{1},x_{2} \in \R, \tilde t>0} \big(1 + \frac{\max_{i}(\abs{x_{i}-x})}{t} + \frac{\tilde t}{t} + \frac{t}{\tilde t}\big)^{-3/2-\epsilon} \frac{\abs{A(x_{2})-A(x_{1})-\alpha(x_{0},\tilde t)(x_{2}-x_{1})}}{t}.
\]
Then
\[
\norm{\beta}_{L^{\infty}(S^{2})}
\lesssim
\norm{a}_{\infty}.
\]
\end{corollary}
The difference from the original formulation of Jones's beta number estimate is that we take a supremum over an uncertainty region in all available parameters.

\section{Littlewood--Paley diagonalization of Lipschitz change of variables}
\begin{proof}[Proof of Theorem~\ref{thm:Lip-LP-diag}]
Since the Lipschitz norm of $A$ is strictly smaller than $1$, the change of variable $x\mapsto x+A(x)$ is invertible and bi-Lipschitz.
Denote its inverse function by $b$, so that $z = b(z) + A(b(z))$.

Write
\[
T_{A}(\Psi_{t} * f)(x)
=
T_{A}(\Psi_{t} * P_{t}f)(x)
=
\int_{-\infty}^{\infty} P_{t}f(z) \Psi_{t}(x+A(x)-z) \dif z.
\]
This integral is a linear combination of the functions $x\mapsto \Psi_{t}(x+A(x)-z)$ that we view as non-linear deformations of wave packets centered at $b(z)$.
The main idea is to replace the non-linear change of variable $x\mapsto x+A(x)-z$ in the argument of $\Psi_{t}$ by the linear change of variable $x\mapsto (1+\alpha(b(z),t))(x-b(z))$, where $\alpha$ is the average slope of the function $A$ in the sense of \eqref{eq:av-slope}.
Since $\abs{\alpha} \leq \norm{A}_{\Lip}$, the function
\[
x\mapsto \int_{-\infty}^{\infty} P_{t}f(z) \Psi_{t}((1+\alpha(b(z),t))(x-b(z))) \dif z
\]
has Fourier support inside $t^{-1}[99/100,103/100]$, so it is annihilated by $I-P_{t}$.

It remains to estimate the error that has been made in approximating the non-linear change of coordinates in the argument of $\Psi_{t}$ by a linear one.
To this end we compute the difference of the arguments:
\begin{equation}
\label{eq:lin-lip-arg}
\abs{x+A(x)-z - (1+\alpha(b(z),t))(x-b(z))}
=
\abs{(A(x) - A(b(z)) - \alpha(b(z),t)(x-b(z)))}
\end{equation}
By the Lipschitz property of $A$ and since $\abs{\alpha}\leq\norm{A}_{\Lip}$ we have
\[
\eqref{eq:lin-lip-arg}
\leq
\frac12 \abs{x-b(z)},
\]
and it follows that both $x+A(x)-z$ and $(1+\alpha(b(z),t))(x-b(z))$ have (signed) distance of the order $\approx x-b(z)$ from zero.
Therefore
\begin{align*}
\MoveEqLeft
\abs{\Psi_{t}(x+A(x)-z) - \Psi_{t}((1+\alpha(b(z),t))(x-b(z)))}\\
&\lesssim
t^{-2} (1+\abs{x-b(z)}/t)^{-20} \cdot \eqref{eq:lin-lip-arg}
&&\text{by decay of $\Psi_{t}'$}\\
&\lesssim
t^{-1} \beta(b(z),t) (1+ \abs{x-b(z)}/t)^{-10}
&&\text{by definition of $\beta$ numbers}.
\end{align*}
It follows that
\begin{align*}
\MoveEqLeft
\sum_{t \in 2^{\Z}} \abs{(1-P_{t})T_{A} (\Psi_{t} * f)}\\
&=
\sum_{t \in 2^{\Z}} \abs[\big]{(1-P_{t}) \int P_{t}f(z) (\Psi_{t}(x+A(x)-z) - \Psi_{t}((1+\alpha(b(z),t))(x-b(z)))) \dif z}\\
&\lesssim
\sum_{t \in 2^{\Z}} (\delta_{0} + t^{-1}(1+\abs{\cdot}/t)^{-10}) * \int \abs{P_{t}f(z)} t^{-1} \beta(b(z),t) (1+ \abs{x-b(z)}/t)^{-10} \dif z\\
&\lesssim
\sum_{t \in 2^{\Z}} \int \abs{P_{t}f(z)} t^{-1} \beta(b(z),t) (1+ \abs{x-b(z)}/t)^{-5}.
\end{align*}
Multiplying this with a function $g\in L^{p'}(\R)$ and integrating in $x$ we obtain the estimate
\[
\sum_{t\in 2^{\Z}} \int \Dem f(z,t) \beta(b(z),t) \Aem g(b(z),t) \dif z.
\]
The sum over $t$ can be dominated by $\int_{0}^{\infty} \frac{\dif t}{t}$ since all functions $\Dem,\beta,\Aem$ are almost (up to a multiplicative factor) constant on Carleson boxes $B(x,t) \times [t,2t]$.
By \cite[Proposition 3.6]{MR3312633} and outer H\"older inequality \cite[Proposition 3.4]{MR3312633} this is bounded by
\[
\norm{\Dem f}_{L^{p}(S^{2})} \norm{\beta(b(\cdot),\cdot)}_{L^{\infty}(S^{2})} \norm{\Aem g(b(\cdot),\cdot)}_{L^{p'}(S^{\infty})}.
\]
Since the function $b$ is bi-Lipschitz, it does not affect outer norms up to a multiplicative constant.
To see this note that
\[
\norm{F\one_{(\cup_{i} T(x_{i},s_{i}))^{c}}}_{L^{\infty}(S^{q})} \leq \lambda
\implies
\norm{F(b(\cdot),\cdot)\one_{(\cup_{i} T(b^{-1}(x_{i}),2s_{i}))^{c}}}_{L^{\infty}(S^{q})} \leq C\lambda
\]
for a sufficiently large constant $C$.

Thus we obtain the estimate
\[
\norm{\Dem f}_{L^{p}(S^{2})} \norm{\beta}_{L^{\infty}(S^{2})} \norm{\Aem g}_{L^{p'}(S^{\infty})}.
\]
Estimating the first term using \eqref{eq:D-embedding}, the middle term using Corollary~\ref{cor:jones-beta}, and the last term using \eqref{eq:A-embedding} we obtain the claim.
\end{proof}

\section{Application to truncated directional Hilbert transforms}
\label{guo-subsection2.5}
In this section we prove Corollary~\ref{cor:LL-single-band}.
As an initial reduction observe that it suffices to estimate the restriction of $H_{u}$ to a vertical strip; more precisely we need an estimate of the form
\[
\norm{H_{u}f}_{L^{p}([N-1,N+2]\times \R)} \lesssim \norm{H_{u}f}_{L^{p}([N,N+1]\times \R)}
\]
for functions $f$ supported in the vertical strip $[N,N+1]\times\R$.
This reduction will be important in the case $p_{0}<2$.
Also, it is easy to see that we may replace $H_{u}$ by the smoothly truncated operator
\begin{equation}
\label{tildeHu}
\tilde H_u f(x,y):= \operatorname{p.v.}\int_{\R} f(x+r,y+u(x,y)r) \phi(r) \frac{\dif r}r,
\end{equation}
where $\phi$ is a smooth even function with $\phi(0)=1$, $\int\phi(x)\dif x = \int x\phi(x) \dif x = \dotsb = \int x^{N}\phi(x) = 0$ for some large $N$ and $\supp\phi\subset [-1,1]$.
This is possible because the maps $(x,y)\mapsto (x+r,y+u(x,y)r)$ are uniformly bi-Lipschitz for $r\in [-1,1]$, so $f\mapsto f(x+r,y+u(x,y)r)$ is a bounded operator on $L^{p}$.

We note that the operators $f\mapsto H_{u}(\Psi_{t}*f)$ (as well as the analogous ones obtained with $\tilde H_u$ from \eqref{tildeHu} in place of $H_u$) are also trivially bounded in $L^p$ uniformly in $t\geq t_{0}$.
To see this split
\[
H_{u}f(x,y)
=
\int_{-1}^{1} f(x+r,y) \frac{\dif r}{r}
+
\int_{-1}^{1} (f(x+r,y+u(x,y)r)-f(x+r,y)) \frac{\dif r}{r}.
\]
The first term is a one-dimensional truncated Hilbert transform on each horizontal line, and therefore bounded on any $L^{p}$, $1<p<\infty$.
The second term can be written as
\[
\int_{-1}^{1} \int_{s=0}^{u(x,y)r} \partial_{2} f(x+r,y+s) \dif s \frac{\dif r}{r}
\]
This is in turn bounded by
\[
\int_{-1}^{1} M_{2} \partial_{2} f(x+r,y) \dif r\leq
M_{1} M_{2} \partial_{2} f(x,y),
\]
where $M_{i}$ denotes the Hardy--Littlewood maximal function in the $i$-th variable.
The differential operator $\partial_{2}$ is $L^{p}$ bounded on the subspace of functions with $\hat f(\xi,\eta)=0$ for $\abs{\eta}>2/t_{0}$ and therefore we obtain $L^{p}$ estimates for this term.

\begin{remark}
The same argument can be used to estimate $H_{u}$ on functions with small horizontal frequencies, thus simplifying an argument in \cite[Section 3]{MR3607232}.
\end{remark}

Below, we work with $\tilde H_u$ from \eqref{tildeHu} in place of $H_u$, and omit the tilde for simplicity of notation.
By the argument leading to \eqref{eq:LP-diag-square-fct} and Littlewodd--Paley theory it suffices to show
\[
\norm{\big(\sum_{t\in 2^{\Z}} \abs{H_{u} (\Psi_t *_2 f)}^{2}\big)^{1/2}}_{p}
\lesssim
\norm{\big(\sum_{t\in 2^{\Z}} \abs{P_t *_2 f}^{2}\big)^{1/2}}_{p},
\]
or, more generally,
\begin{equation}
\label{eq:Hu-sq-fct-est}
\norm{\big(\sum_{t\in 2^{\Z}} \abs{H_{u} (\Psi_t *_2 f_{t})}^{2}\big)^{1/2}}_{L^{p}([N-1,N+2]\times\R)}
\lesssim
\norm{\big(\sum_{t\in 2^{\Z}} \abs{f_{t}}^{2}\big)^{1/2}}_{p}
\end{equation}
for arbitrary functions $f_{t}$ supported in the strip $[N,N+1]\times\R$.
In the case $p=p_{0}=2$ this follows immediately from the single band hypothesis \eqref{passumption} and Fubini's theorem.

In order to obtain the larger range of $p$'s in the case $1<p_{0}<2$ we use the technique for proving vector-valued estimates introduced in \cite{MR3148061} (see also \cite{MR3352435} for more applications of this technique).
\begin{theorem}
\label{thm:BT-VV}
Let $1<p,q<\infty$ and let $T_{k} : L^{p,1}(\Omega) \to L^{p,\infty}(\Omega')$ be a sequence of subadditive operators.
Let $0\leq c<1$ and suppose that for every pair of (non-null, finite measure) measurable sets $H\subset\Omega$, $G\subset\Omega'$ with $0<\abs{H},\abs{G}<\infty$ there exist subsets $H'\subset H$, $G'\subset G$ with
\[
\Big(\frac{\abs{G\setminus G'}}{\abs{G}}\Big)^{1-1/p}
+
\Big(\frac{\abs{H\setminus H'}}{\abs{H}}\Big)^{1/p}
\leq
c
\]
for every $k$ and every function $f$ supported on $H'$ we have
\begin{equation}
\label{eq:BT-VV-Lq}
\norm{T_{k}f}_{L^{q}(G')}
\lesssim
(\abs{G}/\abs{H})^{1/q-1/p} \norm{f}_{L^{q}(H')}.
\end{equation}
Then for any functions $f_{k}\in L^{p,1}(\Omega)$ we have
\[
\norm{ (\sum_{k} \abs{T_{k}f_{k}}^{q})^{1/q} }_{L^{p,\infty}(\Omega')}
\lesssim
\norm{ (\sum_{k} \abs{f_{k}}^{q})^{1/q} }_{L^{p,1}(\Omega)}.
\]
\end{theorem}
\begin{proof}
By the monotone convergence theorem it suffices to consider a finite sequence of operators as long as we obtain estimates that do not depend on its length.
The hypothesis \eqref{eq:BT-VV-Lq} continues to hold for the operator $T(\vec f) := (\sum_{k} \abs{T_{k}f_{k}}^{q} )^{1/q}$ defined on $\ell^{q}$-valued functions, and we know
\[
\norm{ Tf }_{L^{p,\infty}(\Omega')}
\lesssim
\norm{ f }_{L^{p,1}(\Omega,\ell^{q})}
\]
with some constant given by the qualitative boundedness assumption on $T_{k}$'s and depending on the length of the sequence of operators.
By duality of Lorentz spaces this is equivalent to
\[
\int_{G} \abs{Tf}
\leq
B \abs{H}^{1/p} \abs{G}^{1-1/p}
\]
for all finite measure sets $H,G$ and all functions $f:\Omega\to\ell^{q}$ with $\abs{f}\leq\one_{H}$.
We have to find a universal upper bound for $B$.

Let $G,H$ be measurable sets with finite measure and $G',H'$ be the major subsets given by the hypothesis.
Then for any function $f : \Omega\to \ell^{q}$ with $\abs{f} \leq \one_{H'}$ we have
\begin{align*}
\int_{G'} \abs{Tf}
&\leq
\norm{ Tf }_{L^{q}(G')} \norm{ \one_{G} }_{L^{q'}}\\
&\lesssim
(\abs{G}/\abs{H})^{1/q-1/p} \norm{f}_{L^{q}(H',\ell^{q})} \abs{G}^{1/q'}\\
&\lesssim
\abs{H}^{1/p} \abs{G}^{1-1/p}
\end{align*}
by H\"older's inequality and the hypothesis.
It follows that for any function $f : \Omega\to\ell^{q}$ with $\abs{f} \leq \one_{H}$ we have
\begin{align*}
\int_{G} \abs{Tf}
&\leq
C \abs{H}^{1/p} \abs{G}^{1-1/p}
+
\int_{G\setminus G'} \abs{Tf}
+
\int_{G'} \abs{T(f\one_{H\setminus H'})}\\
&\leq
C \abs{H}^{1/p} \abs{G}^{1-1/p}
+
B \abs{H}^{1/p} \abs{G\setminus G'}^{1-1/p}
+
B \abs{H\setminus H'}^{1/p} \abs{G}^{1-1/p}\\
&\leq
(C+cB) \abs{H}^{1/p} \abs{G}^{1-1/p}.
\end{align*}
Taking a supremum over $H,G$ we obtain $B\leq C/(1-c)$.
\end{proof}

Corollary~\ref{cor:LL-single-band} will be obtained via an application of Theorem~\ref{thm:BT-VV}  to the operators $T_{k}f = H_{u}(\Psi_{2^{k}}*_{2}f)$, with the  choice  $q=2$.
The corresponding assumption \eqref{eq:BT-VV-Lq} in Theorem~\ref{thm:BT-VV} will follow by interpolation of the estimates
\begin{equation}
\label{eq:local-restricted}
\int (T_{k} (\one_{H'} \one_{F})) \one_{G'} \one_{E}
\lesssim
\abs{E}^{1/2}\abs{F}^{1/2} (\abs{G}/\abs{H})^{\alpha} (\abs{E}/\abs{F})^{\beta},
\end{equation}
where $H\subset [N,N+1]\times\R$, $G\subset [N-1,N+1]\times\R$, $H'\subset H$ and $G'\subset G$ are as in Theorem~\ref{thm:BT-VV}, $F,E\subset\R^{2}$ are arbitrary measurable subsets, $\alpha=1/2-1/p$, and $\beta$ is in a neighborhood of $0$.

\begin{figure}
\begin{tabular}{lp{0.5\textwidth}}
\raisebox{-\totalheight}{\begin{tikzpicture}[
circ/.style={shape=circle, inner sep=1pt, draw},
CFs/.style={shape=rectangle, inner sep=1pt, draw},
LKs/.style={shape=rectangle, inner sep=1pt, draw},
L1s/.style={shape=rectangle, inner sep=1pt, draw},
]
\begin{axis}
[
width=5cm,
height=5cm,
xmin=-0.6,
xmax=0.6,
xtick={-0.4, 0, 0.25, 0.5},
xticklabels={$\frac12-\frac1{p_{0}}$, $0$, $\frac14$, $\frac12$},
ymin=-0.6,
ymax=0.6,
ytick={-0.25, 0, 0.5},
yticklabels={$-\frac14$, $0$, $\frac12$},
,xlabel={$\beta$}
,ylabel={$\alpha$}
]
\draw node (L2) at (axis cs: 0,0) [circ] {}
node (Linf) at (axis cs: 0.5,0) [circ] {}
node (CF) at (axis cs: 0,0.5) [CFs] {}
node (LK) at (axis cs: 0.25,-0.25) [LKs] {}
node (L1) at (axis cs: -0.4,0) [L1s] {};
\draw (L2) to (CF) to (Linf) to (LK) to (L2);
\draw[dashed] (CF) to (L1) to (LK) to (CF);
\end{axis}
\end{tikzpicture}}
&
The estimate \eqref{eq:local-restricted} is known unconditionally in the interior of the solid polygon: the line $\alpha=0$ corresponds to the non-localized estimates in \cite{MR3090145} and the other two endpoints are the localized estimates in \cite{MR3148061}.

In the proof of Corollary~\ref{cor:LL-single-band} we use estimates in the interior of the dashed triangle, whose leftmost vertex is the hypothesis \eqref{passumption}.
\end{tabular}
\caption{Localized estimates for the single band directional Hilbert transform}
\label{fig:loc}
\end{figure}

The set of pairs $(\alpha,\beta)$ for which the estimate \eqref{eq:local-restricted} holds is clearly convex.
Hence it suffices to establish \eqref{eq:local-restricted} near the vertices of the dashed triangle in Figure~\ref{fig:loc}.
The intersection of the line $\beta=0$ with this triangle corresponds to the range of $p$'s claimed in \eqref{pconclusion}.

We will use Estimates 16, 17, 21, and 22 from \cite{MR3148061}, which do not rely on the single parameter assumption on the vector field made in \cite{MR3090145,MR3148061}.
One twist is in the proof of Estimate 21, where we have to use a version of \cite[Theorem 8]{MR3148061} for Lipschitz vector fields.
This result goes back to \cite{MR2219012}; a slightly simplified version of the proof of the required covering lemma in \cite{MR3148061} is presented in Section~\ref{sec:LL}.
The covering lemma for Lipschitz vector fields only holds for parallelograms of bounded length.
This is the reason for restricting the operator $H_{u}$ to a vertical strip: we can apply the covering lemma to the intersection of parallelograms with this vertical strip.
The other difficulty is that we are dealing with a (smooth) truncation of the Hilbert kernel, so the results of \cite{MR3090145} do not directly apply.
The easiest way to work around this seems to be running the argument in \cite{MR3090145} with more general wave packets which can be used to assemble also the truncated Hilbert kernel $\phi(r)/r$.

\subsection{Using the single band estimate below $L^{2}$}
The hypothesis \eqref{passumption} shows in particular that \eqref{eq:local-restricted} holds with $(\alpha,\beta)=(0,1/2-1/p_{0})$.

\subsection{Using the C\'ordoba--Fefferman covering argument}
By Estimates 16, 17, and 22 in \cite{MR3148061} we can estimate the left-hand side of \eqref{eq:local-restricted} by
\[
\sum_{\delta} \sum_{\sigma \lesssim \delta^{-n} (\abs{G}/\abs{H})^{n-1}} \min(\abs{F}\delta\sigma^{-1},\abs{E}\sigma)
\]
for any integer $n\geq 2$, where both sums are over positive dyadic numbers.

The (geometric) sum over $\sigma$ has two critical points: $\sigma \sim \delta^{-n} (\abs{G}/\abs{H})^{n-1}$ and $\sigma \sim (\delta \abs{F}/\abs{E})^{1/2}$.
This  gives the estimate
\[
\sum_{\delta} \min((\delta \abs{F} \abs{E})^{1/2},\abs{E}\delta^{-n} (\abs{G}/\abs{H})^{n-1}).
\]
The sum over $\delta$ has a critical point with $\delta_{0}^{2n+1} \sim (\abs{G}/\abs{H})^{2n-2} (\abs{E}/\abs{F})$, and we obtain the estimate
\[
(\delta_{0} \abs{F} \abs{E})^{1/2}
\sim
(\abs{F} \abs{E})^{1/2} (\abs{G}/\abs{H})^{(n-1)/(2n+1)} (\abs{E}/\abs{F})^{1/(4n+2)}.
\]
This proves the claim with $\alpha=(n-1)/(2n+1)$, $\beta=1/(4n+2)$.
We can make $(\alpha,\beta)$ approach $(1/2,0)$ by choosing $n$ suitably large.

\subsection{Using the Lacey--Li covering argument}
By Estimates 16, 17, and 21 from \cite{MR3148061} we can estimate the left-hand side of \eqref{eq:local-restricted} by
\[
\sum_{\delta} \sum_{\sigma} \min(\abs{F}\delta\sigma^{-1},\abs{E}\sigma,\abs{E} (\abs{H}/\abs{G})^{1/2}\sigma^{-\epsilon} \delta^{-1/2-\epsilon})
\]
The sum over $\sigma$ now has two critical points with $\sigma\sim (\delta \abs{F}/\abs{E})^{1/2}$ and with $\sigma^{1+\epsilon}\sim (\abs{H}/\abs{G})^{1/2} \delta^{-1/2-\epsilon}$ and is dominated by the minimum of the two corresponding terms, so we have the estimate
\[
\sum_{\delta\leq 1} \min(\abs{E} ((\abs{H}/\abs{G})^{1/2} \delta^{-1/2-\epsilon})^{1/(1+\epsilon)}, (\delta \abs{F} \abs{E})^{1/2})
\]
The sum over $\delta$ has a critical point at $\delta_{0}^{2+3\epsilon} \sim (\abs{E}/\abs{F})^{1+\epsilon} (\abs{H}/\abs{G})$.
This gives the estimate
\[
(\delta_{0} \abs{F} \abs{E})^{1/2}
\sim
(\abs{F} \abs{E})^{1/2} ((\abs{E}/\abs{F})^{1+\epsilon} (\abs{H}/\abs{G}))^{1/(4+6\epsilon)}.
\]
Making $\epsilon$ small we can make $(\alpha,\beta)$ approach $(-1/4,1/4)$.
This completes the proof of Corollary~\ref{cor:LL-single-band}.

\begin{remark}
The upper part of the solid polygon in Figure~\ref{fig:loc} yields the hypothesis of Theorem~\ref{thm:BT-VV} for any $2<q<p<\infty$.
This implies that the operator $H_{u}$ maps $L^{p}(\R^{2})$ into a directional Triebel--Lizorkin space of type $F^{0}_{p,q}$ (provided that $u$ is Lipschitz in the vertical direction).
More precisely,
\[
\norm{\big(\sum_{t\in 2^{\Z}} \abs{P_{t} H_{u} f}^{q} \big)^{1/q} }_{L^{p}(\R^{2})}
\lesssim
\norm{f}_{L^{p}(\R^{2})},
\quad
2<p,q<\infty.
\]
Indeed, the left-hand side is monotonically decreasing in $q$, so it suffices to consider $2<q<p<\infty$.
With a suitable choice of $\Psi$ we may write $f=\sum_{t\in 2^{\Z/100}}\Psi_{t} *_{2} f$.
For notational simplicity we consider only the contribution of $t\in 2^{\Z}$.
By the Fefferman--Stein maximal inequality we may replace $P_{t}$ by larger Littlewood--Paley projections such that $\sum_{t\in 2^{\Z}} P_{t} = \operatorname{id}$.

In the diagonal term we use the Fefferman--Stein maximal inequality, the vector-valued estimate provided by Theorem~\ref{thm:BT-VV} with $p>2$, monotonicity of $\ell^{q}$ norms, and Littlewood--Paley theory to estimate
\begin{align*}
\norm{\big(\sum_{t\in 2^{\Z}} \abs{P_{t} H_{u} (\Psi_{t} *_{2} f)}^{q} \big)^{1/q} }_{L^{p}(\R^{2})}
&\lesssim
\norm{\big(\sum_{t\in 2^{\Z}} \abs{H_{u} (\Psi_{t} *_{2} f)}^{q} \big)^{1/q}}_{L^{p}(\R^{2})}\\
&\lesssim
\norm{\big(\sum_{t\in 2^{\Z}} \abs{\Psi_{t} *_{2} f}^{q} \big)^{1/q}}_{L^{p}(\R^{2})}\\
&\leq
\norm{\big(\sum_{t\in 2^{\Z}} \abs{\Psi_{t} *_{2} f}^{2} \big)^{1/2}}_{L^{p}(\R^{2})}\\
&\lesssim
\norm{f}_{L^{p}(\R^{2})}.
\end{align*}
In the off-diagonal term we use monotonicity of $\ell^{q}$ norms, Littlewood--Paley theory, and Corollary~\ref{cor:LP-diag} to estimate
\begin{align*}
\norm{\big(\sum_{t\in 2^{\Z}} \abs{P_{t} H_{u} (\sum_{t'\neq t}\Psi_{t'} *_{2} f)}^{q} \big)^{1/q} }_{L^{p}(\R^{2})}
&\leq
\norm{\big(\sum_{t\in 2^{\Z}} \abs{P_{t} H_{u} (\sum_{t'\neq t}\Psi_{t'} *_{2} f)}^{2} \big)^{1/2} }_{L^{p}(\R^{2})}\\
&\lesssim
\norm{\sum_{t'\in 2^{\Z}} (\sum_{t\neq t'}P_{t}) H_{u} (\Psi_{t'} *_{2} f) }_{L^{p}(\R^{2})}\\
&=
\norm{\sum_{t'\in 2^{\Z}} (1-P_{t'}) H_{u} (\Psi_{t'} *_{2} f) }_{L^{p}(\R^{2})}\\
&\lesssim
\norm{f }_{L^{p}(\R^{2})}.
\end{align*}
\end{remark}

\section{Application to Hilbert transforms along Lipschitz variable parabolas}
For the curved directional Hilbert transform \eqref{eq:Stein-directional-op:curved} we argue similarly as in the case $\alpha=1$.
However, the single band and vector-valued estimates in this case are essentially contained in \cite{MR3669936}, so that we obtain an unconditional result.
\begin{corollary}
\label{cor:Hilbert-curved}
For every $0<\alpha<\infty$, $\alpha\neq 1$, and every $1<p<\infty$, there exits $\epsilon_0>0$ such that for every Lipschitz function $u$ with $\norm{u}_{\Lip}\le \epsilon_0$, we have
\begin{equation}
\norm{H^{(\alpha)}_u f}_p
\lesssim
\norm{f}_p.
\end{equation}
\end{corollary}
\begin{proof}[Proof of Corollary~\ref{cor:Hilbert-curved}]
In the following, we will assume for notational convenience that $0<u\le 1$ almost everywhere.
The region that $-1\le u< 0$ can be handled similarly, while the region $u=0$ is trivial by Fubini as the operator acts only in the first variable.
By the trivial analogue of Corollary~\ref{cor:LP-diag}, it suffices to show
\begin{equation}
\label{main-square}
\norm[\big]{\big(\sum_{t\in2^\Z} \abs{H_{u}^{(\alpha)} (\Psi_t *_2  f)}^2 \big)^{1/2} }_p
\lesssim
\norm{ f }_p.
\end{equation}
We use $P_t(\Psi_t*_2 f)=\Psi_t*_2 f$ where $P_t$ is as defined before acting in the second variable.
We note  that for 
\[\abs{r}^\alpha u(x,y)/t \le 1\]
we have by an application of the fundamental theorem of calculus
\[\abs{P_{t}(\Psi_t *_2 f)(x+r, y+u(x, y)r^\alpha)-P_{t}(\Psi_t *_2 f)(x+r, y)}\le 
u(x,y)\abs{r}^\alpha t^{-1} M_2(\Psi_t *_2 f)(x+r, y).\]
Hence we have for the integral over small values of $r$ 
\[
\norm[\big]{\big(\sum_{t\in2^\Z} \abs{\int_{\abs{r}^\alpha u(x,y)/t\leq 1} 
P_{t}(\Psi_t *_2 f)(x+r, y+u(x, y)r^{\alpha})\frac{\dif r}{r}}^2 \big)^{1/2} }_{L^{p}(x,y)}
\]
\begin{equation}\label{thiele11}
\lesssim \norm[\big]{\big(
\sum_{t\in2^\Z} \abs{\int_{\abs{r}^\alpha u(x,y)/t \leq 1} 
P_t(\Psi_t *_2 f)
(x+r, y) \frac{\dif r}{r}}^2 
\big)^{1/2} }_{L^{p}(x,y)}
\end{equation}
\begin{equation}\label{thiele12}
+
\norm[\big]{\big(
\sum_{t\in2^\Z} \abs{\int_{\abs{r}^\alpha u(x,y)/t \leq 1} 
u(x,y)\abs{r}^\alpha t^{-1} M_2(\Psi_t *_2 f
)(x+r, y)\frac{\dif r}{\abs{r}}}^2 
\big)^{1/2} }_{L^{p}(x,y)}.
\end{equation}
The former term \eqref{thiele11} can be estimated using the vector-valued estimate for the maximally truncated Hilbert transform.
Using integrability of $\abs{r}^{\alpha-1}$ near zero we estimate the latter term \eqref{thiele12} by
\[
\norm[\big]{\big(\sum_{t\in2^\Z} \abs{M_{1} M_{2} 
(\Psi_t *_2 f) 
(x, y)}^2 \big)^{1/2} }_{L^{p}(x,y)}\lesssim
\norm[\big]{\big(\sum_{t\in2^\Z} \abs{ \Psi_t *_2 f(x, y)}^2 \big)^{1/2} }_{L^{p}(x,y)} \lesssim 
\norm{ f }_p.
\]
Here we have used the Fefferman--Stein maximal inequality and Littlewood-Paley theory.

We turn to the remaining part of the kernel with $\abs{r}^\alpha  u(x,y)/t \ge 1$ and $\abs{r}\le 1$.
Note we may restrict the summation over $t$ to $t\le 1$, as for $t>1$ the domain of
integration is empty.
We will break up the integral into lacunary pieces parametrized
by $s\in 2^{\alpha \N}$ and estimate the pieces separately, with suitable power decay 
in $s$ allowing to geometrically sum the estimates.

We introduce Littlewood-Paley projections in the first variable  and write $P_t^{(1)}$
and $P_t^{(2)}$ to distinguish projections in first and second variable.
Consider the averaging operator
\[
E_s^{(1)}=\int_{s}^\infty P_t^{(1)} \frac {dt}t.
\]
We note similarly to above for the averaged part of the integral pieces:
\[
\norm[\big]{\big(\sum_{t\in2^{-\N}} \abs{\int_{s\le \abs{r}^\alpha u(x,y)/t \leq 2^\alpha  s} 
E_{s(\frac{st}{u(x, y)})^{1/\alpha}}^{(1)} P_{t}^{(2)}(\Psi_t *_2 f)(x+r, y+u(x, y)r^{\alpha})\frac{\dif r}{r} }^2 \big)^{1/2} }_{L^{p}(x,y)}
\]
\begin{equation}\label{thiele21}
\lesssim
\norm[\big]{\big(\sum_{t\in2^{-\N}} \abs{\int_{s\le \abs{r}^\alpha u(x,y)/t \leq 2^\alpha s} 
E_{s(\frac{st}{u(x, y)})^{1/\alpha}}^{(1)} P_{t}^{(2)}(\Psi_t *_2 f)(x, y+u(x, y)r^{\alpha})\frac{\dif r}{r} }^2 \big)^{1/2} }_{L^{p}(x,y)}
\end{equation}
\begin{equation}\label{thiele22}
+
\norm[\big]{\big(\sum_{t\in2^{-\N}} (\int_{s\le \abs {r}^\alpha u(x,y)/t\leq 2^{\alpha}s} 
s^{-1} M_1 P_{t}^{(2)}(\Psi_t *_2 f)(x, y+u(x, y)r^{\alpha})\frac{\dif r}{\abs{r}} )^2 \big)^{1/2} }_{L^{p}(x,y)}
\end{equation}
The factor $(st/u)^{1/\alpha}$ in the index of the averaging operator is chosen because it is
roughly $\abs{r}$ in the domain of integration.
In the former term \eqref{thiele21} we change variables, replacing $u(x,y)r^\alpha$ by $r$ on the positive and similarly on the negative axis and do a partial integration in $r$, noting that by the mean zero property the primitive of the kernel of $P_t^{(t)}$ is a bump function again, to estimate this term  by 
\[\lesssim
\norm[\big]{\big(\sum_{t\in2^{-\N}} (\int_{s\leq \abs{r}/t\leq 2^\alpha s} 
t M_1  M_2 (\Psi_t *_2 f)(x, y+r)\frac{\dif r}{\abs{r}^2} )^2 \big)^{1/2} }_{L^{p}(x,y)}
\]
\[\lesssim
s^{-1} \norm[\big]{\big(\sum_{t\in 2^{-\N}}   
(M_2 M_1  M_2 (\Psi_t *_2 f)(x, y) )^2 \big)^{1/2} }_{L^{p}(x,y)}
\]
plus two similar boundary terms, which are all estimated by the Fefferman-Stein maximal inequality
with power decay in $s$.
The latter term \eqref{thiele22} above is estimated by the same change of variables by
\[s^{-1}
\norm[\big]{\big(\sum_{t\in2^{-\N}} (\int_{s\le \abs {r}/t\leq 2^\alpha s} 
 M_1 P_{t}^{(2)}(\Psi_t *_2 f)(x, y+r)\frac{\dif r}{\abs{r}} )^2 \big)^{1/2} }_{L^{p}(x,y)}
\]
\[\lesssim s^{-1}
\norm[\big]{\big(\sum_{t\in2^{-\N}} (
 M_2 M_1 P_{t}^{(2)}(\Psi_t *_2 f)(x, y) )^2 \big)^{1/2} }_{L^{p}(x,y)}
\]
which is again estimated by the Fefferman-Stein maximal inequality with decay in $s$.

A similar estimate can be obtained if instead of the sharp cut-off 
$s\le \abs {r}^\alpha u(x,y)/t\leq 2^{\alpha} s$ we use a smooth cut-off.
More precisely, we will
choose cut-off functions as defined in the following operator:

\begin{equation}
A_s f(x, y)
=
\int_{\R} f(x+r, y+u(x,y) r^\alpha) 
\chi(s^{-1} r^\alpha v(x,y)t^{-1} (u(x,y)v^{-1}(x,y))^{\alpha/(\alpha-1)}) \frac{dr}{r},
\end{equation}
where $\chi$ is smooth and supported on $\pm [2^{-\alpha} ,2^\alpha]$
and $\sum_{s\in 2^{\alpha \N}} \chi(s^{-1}x)=1$ for $x\neq 0$, and where
$v(x,y)$ is the largest integer power of $2$ less than $u(x,y)$.
Note the
auxiliary factor $u/v$ is bounded above and below respectively by $2$ and $1$.

Then, with the above arguments, it suffices to estimate the rough part of each piece
with some $\gamma>0$ that may depend on $p$ as follows:
\begin{equation}\label{guo-e2.40}
\norm[\big]{\big(\sum_{t\in2^{-\N}} \abs{
A_s (1-E_{s(\frac{st}{u(x, y)})^{1/\alpha}}^{(1)}) P_{t}^{(2)}(\Psi_t *_2 f) }^2 \big)^{1/2} }_{L^{p}}\lesssim s^{-\gamma} \norm{f}_p.
\end{equation}

Here we point out that this estimate has essentially been established in \cite{MR3669936}.
First of all, we recognize that the left hand side of \eqref{guo-e2.40} is essentially the term (5.13) in \cite{MR3669936}, there one has a large power of $s$ in the index of $E$ but this makes their bound
only stronger.
By the local smoothing estimates and a certain interpolation argument, the $L^p$ bounds of \eqref{guo-e2.40} for all $1<p\le 2$ have been established in Subsection 5.3 in \cite{MR3669936}.
To prove $L^p$ bounds for all $p>2$, we cite the pointwise estimate (3.19) in \cite{MR3669936}, which implies for these $p$ that 

\[
\norm[\big]{\big(\sum_{t\in2^{-\N}} \abs{
A_s (1-E_{s(\frac{st}{u(x, y)})^{1/\alpha}}^{(1)}) P_{t}^{(2)}(\Psi_t *_2 f) }^2 \big)^{1/2} }_{L^{p}}\lesssim \log(1+s)^4 \norm{f}_p.
\]

A further interpolation gives the desired estimate \eqref{guo-e2.40} for all $1<p<\infty$
for slightly smaller $\gamma$.
This finishes the proof of the square function estimate \eqref{main-square}.
\end{proof}

\chapter{Single scale operator}
\label{chap:single-scale}

Our last result concerns the single scale directional operator
\begin{equation}
\label{eq:single-scale-op}
A_{u,\phi}f(x,y) := \int_{-\infty}^{+\infty} \phi(r) f(x+r,y+u(x,y)r) \dif r
\end{equation}
associated to a Schwartz function $\phi$.
\begin{theorem}
\label{thm:single-scale}
Let $u:\R^{2}\to [-1,1]$ be a measurable function.
Then
\begin{equation}
\label{eq:single-scale-square-fct}
\norm[\big]{\big(\sum_{t \in 2^{\Z}} \abs{ A_{u,\phi} P_{t} f }^{2} \Big)^{1/2}}_{p}
\lesssim_{p,\phi}
\norm{f}_{p}
,
\quad
2<p<\infty.
\end{equation}
\end{theorem}
The operator $A_{u,\phi}$ is in general not bounded on $L^p$ unless $p=\infty$.
Even if we assume $u$ to be Lipschitz in the vertical direction, we cannnot apply Theorem~\ref{thm:Lip-LP-diag} unless $\phi$ has suitable compact support.

Theorem~\ref{thm:single-scale} is intended as a step towards understanding the square function~\eqref{eq:LP-diag-square-fct} in which $\phi$ is replaced by a singular kernel.
As an application of Theorem~\ref{thm:single-scale} we elaborate on a remark made by Demeter in \cite{MR2680067}.
\begin{corollary}\label{cor:single-scale-N-directions}
Assume the measurable function $u:\R^2\to [-1,1]$ takes at most $N$ different values.
Then
\begin{equation} \label{est:single-scale-N-directions}
\norm{A_{u,\phi}f}_p \lesssim_{p,\phi}  \log(N+2)^{1/2} \norm{f}_p,
\quad
2<p<\infty.
\end{equation}
\end{corollary}
Indeed, Demeter proves the sharper endpoint version of this estimate for $p=2$, reproducing an earlier result by Katz \cite{MR1711029}.
Demeter proposes an alternative proof of this result using an inequality by Chang, Wilson, and Wolff \cite{MR800004}, in the same vein as in his proof of \cite[Theorem 2]{MR2680067}.
Theorem~\ref{thm:single-scale} allows to follow through with this proposal, albeit only for $p>2$.
For the operator obtained by replacing $\phi$ in \eqref{eq:single-scale-op} with a one-dimensional singular integral kernel, the same quantitative estimate as \eqref{est:single-scale-N-directions}, up to $\varepsilon$-losses in the power of $\log N$ when $p>2$ is sufficiently close to $2$, holds when the finite range of $u$ is assumed to have additional structure \cite{MR3145928}.
For instance, one may take $u(\R^2)=\{2k/N: k=-N/2,\ldots,N/2\}$.
Thus, it is of interest whether the methods behind Corollary~\ref{cor:single-scale-N-directions} may be applied to the singular integral case, with the aim of lifting the structure restrictions appearing in \cite{MR3145928}.

In this section we prove Theorem~\ref{thm:single-scale}.
The strategy is to use duality and outer H\"older inequality to reduce the estimate to two estimates of Carleson embedding flavor, the ``energy embedding'' in Section~\ref{sec:single-scale:energy-embed} and the ``mass embedding'' in Section~\ref{sec:single-scale:mass-embed}.

\section{Tiles and the outer measure space}
We subdivide the parameter space into \emph{tiles}.
Each tile can be represented in three equivalent ways:
\begin{enumerate}
\item by a \emph{shearing matrix}
\[
A =
\begin{pmatrix}
2^{k_{1}} & 0\\
l2^{k_{1}} & 2^{k_{2}}
\end{pmatrix},
\quad
k_{1},k_{2},l\in\Z
\]
and the spatial location $(2^{-k_{1}}n_{1},2^{-k_{2}}n_{2})$, $n_{1},n_{2}\in\Z$.
\item by the corresponding \emph{spatial parallelogram}
\[
P = A^{-1}([0,1]\times [0,1]) + (2^{-k_{1}}n_{1},2^{-k_{2}}n_{2}),
\]
\item or by the corresponding \emph{frequency parallelogram} $A^{*}([0,1]\times [1,2])$ and the spatial location
\[
(2^{-k_{1}}n_{1},2^{-k_{2}}n_{2}).
\]
\end{enumerate}
Figure~\ref{fig:supp-phiA} shows the spatial and the frequency parallelograms of a tile (with $n_{1}=n_{2}=0$).
The frequency picture also includes the symmetric parallelograms $A^{*}([0,1]\times [-2,-1])$ (in a lighter shade of gray), because the Fourier transforms of the wave packets associated to tiles will concentrate on both these parallelograms.
However, for combinatorial purposes it suffices to consider only the upper parallelogram.
The \emph{slope} of a tile is the number $-l2^{-k_{2}+k_{1}}$.
It is the slope of the lower and the upper side of the corresponding spatial parallelogram.
The spatial parallelogram seems to be the most concise description of a tile, so we denote tiles by the letter $P$ (for ``parallelogram'').

The fact that we are dealing with a single scale operator is reflected in that we define an outer measure on a finite set $X$ of tiles with $k_{1}=0$, that is, tiles with the fixed horizontal scale $1$.
(The restriction to finite sets of tiles avoids technicalities associated with infinite sums.
All estimates will be independent of the specific finite set, so we can pass to the set of all tiles at the end of the argument.)
The outer measure is generated by a function $\sigma$ whose domain $\mathbf{E}$ is the collection of all non-empty subsets of $X$.
We denote by $CP$ the parallelogram with the same slope and center as $P$ but side lengths multiplied by $C$.
For $\calR\in\mathbf{E}$ set
\begin{equation}
\label{eq:def-sigma}
\sigma(\calR)
:=
\sup_{L\geq 1} L^{-C} \meas[\big]{\cup_{R\in\calR} LR},
\end{equation}
where $C$ is a large number to be chosen later.
The three sizes that we need are
\begin{align*}
S^{1}(F)(\calR)
&:=
\sigma(\calR)^{-1} \sum_{R\in\calR} \meas{R} \abs{F(R)},\\
S^{2}(F)(\calR)
&:=
\big( \sigma(\calR)^{-1} \sum_{R\in\calR} \meas{R} \abs{F(R)}^{2} \big)^{1/2}
=
S^{1}(F^{2})(\calR)^{1/2},\\
S^{\infty}(G)(\calR)
&:=
\sup_{R\in\calR} \abs{G(R)}.
\end{align*}

\begin{figure}
\begin{center}
\begin{tabular}{ccc}
\begin{tikzpicture}
\filldraw[gray] (0,0) -- (0,1) -- (1,1) -- (1,0) -- cycle;
\draw[->] (0,0) -- (1.5,0) node[below] {$x_{1}$};
\draw[->] (0,0) -- (0,1.5) node[left] {$x_{2}$};
\draw (0,0) node[below] {$0$};
\draw (1,0) node[below] {$1$};
\draw (0,1) node[left] {$1$};
\end{tikzpicture}
&
{$\displaystyle \xrightarrow{A^{-1} = \begin{pmatrix}
2^{-k_{1}} & 0\\
-l2^{-k_{2}} & 2^{-k_{2}}
\end{pmatrix}}$}
&
\begin{tikzpicture}[yscale=0.5]
\filldraw[gray] (0,0) -- (0,1) -- (1,3) -- (1,2) -- cycle;
\draw[->] (0,0) -- (1.5,0) node[below] {$x_{1}$};
\draw[->] (0,0) -- (0,3.5) node[left] {$x_{2}$};
\draw (0,0) node[below] {$0$};
\draw (1,0) node[below] {$2^{-k_{1}}$};
\draw (0,1) node[left] {$2^{-k_{2}}$};
\draw[dotted] (1,2) -- (0,2) node[left] {$-l2^{-k_{2}}$};
\draw[dotted] (1,3) -- (0,3) node[left] {$(-l+1)2^{-k_{2}}$};
\end{tikzpicture}\\
\begin{tikzpicture}
\filldraw[gray] (0,1) -- (0,2) -- (1,2) -- (1,1) -- cycle;
\filldraw[lightgray] (0,-1) -- (0,-2) -- (1,-2) -- (1,-1) -- cycle;
\draw[->] (0,0) -- (1.5,0) node[below] {$\xi_{1}$};
\draw[->] (0,-2.1) -- (0,2.5) node[left] {$\xi_{2}$};
\draw (0,0) node[left] {$0$};
\draw (1,0) node[below] {$1$};
\draw (0,1) node[left] {$1$};
\draw (0,2) node[left] {$2$};
\end{tikzpicture}
&
{$\displaystyle \xrightarrow{A^{*} = \begin{pmatrix}
2^{k_{1}} & l2^{k_{1}}\\
0 & 2^{k_{2}}
\end{pmatrix}}$}
&
\begin{tikzpicture}[xscale=0.6]
\filldraw[gray] (-2,1) -- (-4,2) -- (-3,2) -- (-1,1) -- cycle;
\filldraw[lightgray] (2,-1) -- (4,-2) -- (5,-2) -- (3,-1) -- cycle;
\draw[->] (-4.5,0) -- (5.5,0) node[below] {$\xi_{1}$};
\draw[->] (0,0) -- (0,2.5) node[left] {$\xi_{2}$};
\draw (0,0) node[below] {$0$};
\draw[dotted] (-2,1) -- (-2,0) node[below] {\scriptsize{$l2^{k_{1}}$}};
\draw[dotted] (-1,1) -- (-1,0) node[below] {\scriptsize{$(l+1)2^{k_{1}}$}};
\draw (0,1) node[left] {$2^{k_{2}}$};
\draw (0,2) node[left] {$2^{k_{2}+1}$};
\end{tikzpicture}
\end{tabular}
\end{center}
\caption{Spatial and frequency parallelograms of a tile}
\label{fig:supp-phiA}
\end{figure}
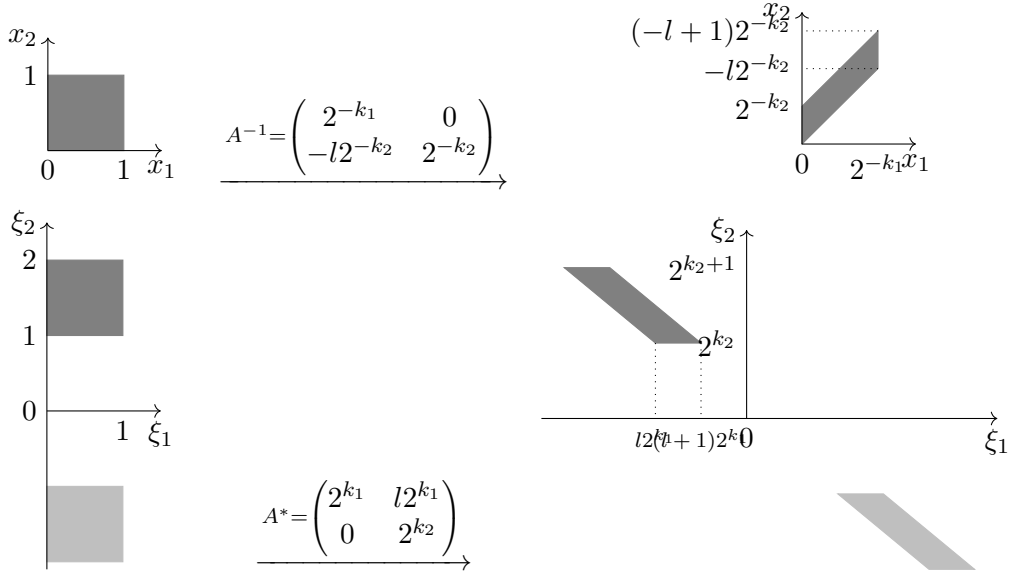

\section{Wave packets and the energy embedding}
\label{sec:single-scale:energy-embed}
Let $\Phi=\Phi_{C}$ be the set of functions on $\R^{2}$ that satisfy
\[
\abs{\partial^{\alpha}\phi(x)} \leq (1+\abs{x})^{-C},
\quad
\norm{\alpha}_{\ell^{1}} \leq C,
\]
for some sufficiently large $C$ that will be chosen later and
\[
\int_{\R} x_{2}^{n} \phi(x_{1},x_{2}) \dif x_{2} = 0,
\quad
x_{1}\in\R,
\quad
n=0,\dotsc,C-2.
\]
We think of $\phi$ as morally supported on $[0,1]^{2}$ and of $\hat\phi$ as morally supported on $[0,1]\times [1,2]$ for $\phi\in\Phi$.

The \emph{$L^{\infty}$ normalized wave packets} associated to a tile $P=(A,n_{1},n_{2})$ are the functions of the form
\[
\phi_{P}^{(\infty)}(x) = \phi(A(x_{1}-2^{-k_{1}}n_{1},x_{2}-2^{-k_{2}}n_{2})),
\quad
\phi\in\Phi.
\]
The \emph{$L^{p}$ normalized wave packets}, $1\leq p<\infty$, are the functions $\phi_{P}^{(p)} = \det(A)^{1/p} \phi_{P}^{(\infty)}$.
Note that $\widehat{\phi(A\cdot)}(\xi) = (\det A)^{-1} \hat\phi(A^{-*}\xi)$.
The spatial and the frequency parallelograms of a tile correspond to the moral space/frequency support of the wave packets associated to this tile.

\subsection{Almost orthogonality}
The fundamental property of the wave packets is their almost orthogonality for tiles with different scales or slopes.
\begin{lemma}
\label{lem:correlation-decay}
\[
\abs{\innerp{\phi_{P}^{(2)}}{\phi_{P'}^{(2)}}}
\lesssim
\min(1,(2^{\max(k_{2},k_{2}')}\abs{2^{-k_{2}}l-2^{-k_{2}'}l'})^{-C},2^{-C\abs{k_{2}-k_{2}'}}),
\]
where $C$ can be made arbitrarily large provided that the order of decay in the definition of $\Phi$ is sufficiently large.
\end{lemma}
\begin{proof}
Without loss of generality suppose $k_{2}\geq k_{2}'$.
We will estimate
\[
\int_{\R^{2}} \abs{\widehat{\phi(A\cdot)}} \abs{\widehat{\phi'(A'\cdot)}}
\]
for $\phi,\phi'\in\Phi$.
This is sufficient because the spatial location of the tiles only affects the phase of the Fourier transforms of the associated wave packets, but not their magnitude.\\

\paragraph{Correlation decay due to shearing}
Let $0<\epsilon\ll 1$ and $S_{N}=\Set{-N,N}\times\R$ be a vertical strip of width $N \geq 1$.
The critical intersection $A^{*}S_{N} \cap (A')^{*}S_{N}$ is a parallelogram centered at zero of width $\sim N$ and height $\sim N/\abs{2^{-k_{2}}l-2^{-k_{2}'}l'}$.
By the vanishing moments assumption we have
\[
\abs{\widehat{\phi(A\cdot)}} \lesssim 2^{-k_{2}} (2^{-k_{2}}N/\abs{2^{-k_{2}}l-2^{-k_{2}'}l'})^{C}
\]
on the critical intersection.
Using the fact that the Fourier transforms $\widehat{\phi(A\cdot)}$ and $\widehat{\phi'(A'\cdot)}$ are $L^{1}$ normalized functions and the decay of these Fourier transforms at infinity we obtain
\begin{align*}
\int_{\R^{2}} \abs{\widehat{\phi(A\cdot)}} \abs{\widehat{\phi'(A'\cdot)}}
&\leq
\int_{\R^{2} \setminus A^{*}S_{N}} + \int_{\R^{2} \setminus (A')^{*}S_{N}} + \int_{A^{*}S_{N} \cap (A')^{*}S_{N}}\\
&\leq
\sup_{\R^{2} \setminus A^{*}S_{N}} \abs{\widehat{\phi(A\cdot)}}
+
\sup_{\R^{2} \setminus (A')^{*}S_{N}} \abs{\widehat{\phi'(A'\cdot)}}
+
\sup_{A^{*}S_{N} \cap (A')^{*}S_{N}} \abs{\widehat{\phi(A\cdot)}}\\
&\lesssim
2^{-k_{2}}N^{-C(1/\epsilon-1)} + 2^{-k_{2}'}N^{-C(1/\epsilon-1)} + 2^{-k_{2}} (2^{-k_{2}}N/\abs{2^{-k_{2}}l-2^{-k_{2}'}l'})^{C}.
\end{align*}
Choosing $N = 2^{\epsilon (k_{2}-k_{2}')/C} (2^{k_{2}}\abs{2^{-k_{2}}l-2^{-k_{2}'}l'})^{\epsilon}$ as we may provided that $\abs{l-2^{k_{2}-k_{2}'}l'} \geq 1$, we obtain
\[
\int_{\R^{2}} \abs{\widehat{\phi(A\cdot)}} \abs{\widehat{\phi'(A'\cdot)}}
\lesssim
2^{-k_{2}+\epsilon (k_{2}-k_{2}')} (2^{k_{2}}\abs{2^{-k_{2}}l-2^{-k_{2}'}l'})^{-(1-\epsilon)C},
\]
and this gives the second estimate in the conclusion of the lemma.\\

\paragraph{Correlation decay for separated scales}
Let $2^{k_{2}'} \ll N \ll 2^{k_{2}}$.
Using again the fact that the Fourier transforms $\widehat{\phi(A\cdot)}$ and $\widehat{\phi'(A'\cdot)}$ are $L^{1}$ normalized functions and the decay of Fourier transforms near $\xi_{2}=0$ and at infinity we obtain
\begin{align*}
\int_{\R^{2}} \abs{\widehat{\phi(A\cdot)}} \abs{\widehat{\phi'(A'\cdot)}}
&\leq
\int_{\abs{\xi_{2}} \leq N} + \int_{\abs{\xi_{2}} \geq N}\\
&\leq
\sup_{\abs{\xi_{2}} \leq N} \abs{\widehat{\phi(A\cdot)}}
+
\sup_{\abs{\xi_{2}} \geq N} \abs{\widehat{\phi'(A'\cdot)}}\\
&\lesssim
2^{-k_{2}} (N/2^{k_{2}})^{C} + 2^{-k_{2}'}(N/2^{k_{2}'})^{-(C+1)/\epsilon}.
\end{align*}
Choosing $N \sim 2^{k_{2}' + \epsilon(k_{2}-k_{2}')}$ we obtain
\[
\int_{\R^{2}} \abs{\widehat{\phi(A\cdot)}} \abs{\widehat{\phi'(A'\cdot)}}
\lesssim
2^{-k_{2}-C(1-\epsilon)(k_{2}-k_{2}')}
=
2^{-k_{2}/2-k_{2}'/2-(C+1/2-\epsilon')(k_{2}-k_{2}')},
\]
and this gives the third estimate in the conclusion of the lemma.
\end{proof}

\subsection{Bessel inequality}
\begin{lemma}
\label{lem:bessel}
For each tile $P$ fix an $L^{2}$ normalized wave packet $\phi_{P}$ adapted to $P$.
Then
\[
\sum_{P} \abs{\innerp{f}{\phi_{P}}}^{2}
\lesssim
\norm{f}_{2}^{2}.
\]
\end{lemma}
\begin{proof}
Schur's test
\begin{align*}
\sum_{P} \abs{\innerp{f}{\phi_{P}}}^{2}
&=
\innerp[\big]{f}{\sum_{P} \phi_{P}\innerp{\phi_{P}}{f}}\\
&\leq
\norm{f}_{2} \norm[\big]{\sum_{P} \phi_{P}\innerp{\phi_{P}}{f}}_{2}\\
&=
\norm{f}_{2} \big(\sum_{P,P'} \innerp{f}{\phi_{P}}\innerp{\phi_{P}}{\phi_{P'}}\innerp{\phi_{P'}}{f} \big)^{1/2}\\
&\leq
\norm{f}_{2} \big(\sum_{P} \abs{\innerp{f}{\phi_{P}}}^{2} \sum_{P'} \abs{\innerp{\phi_{P}}{\phi_{P'}}} \big)^{1/2}
\end{align*}
shows that it suffices to prove
\[
\sup_{P} \sum_{P'} \abs{\innerp{\phi_{P}}{\phi_{P'}}} < \infty.
\]
For a fixed tile $P$ we split the above sum according to the shearing matrix $A'$ of the tile $P'$.
For a given shearing matrix $A'$ we distinguish the cases $k_{2}\leq k_{2}'$ and $k_{2}>k_{2}'$.

In the case $k_{2}\leq k_{2}'$ the tile $P$ has larger scale than $P'$, so the tail of the associated wave packet is more important.
For $L\in 2^{\N}$ let
\[
\tilde\calR_{L} := \Set{ P' \text{ with shearing matrix }A' \text{ such that } LP\cap P'\neq \emptyset }
\]
and let $\calR_{1}:=\tilde\calR_{1}$, $\calR_{L}:=\tilde\calR_{L} \setminus \tilde\calR_{L/2}$ for $L\geq 2$.
Then
\[
\abs{\tilde\calR_{L}}
\lesssim
L(L2^{-k_{2}}+\abs{2^{-k_{2}}l-2^{-k_{2}'}l'})/2^{-k_{2}'},
\]
and
\[
\sum_{L\in 2^{\N}} \sum_{P' \in \calR_{L}} \abs{\innerp{\phi_{P}}{\phi_{P'}}}
\lesssim
\sum_{L\in 2^{\N}} \abs{\tilde\calR_{L}} \min( L^{-C}, 2^{-C(k_{2}'-k_{2})}, (2^{k_{2}'}\abs{2^{-k_{2}}l-2^{-k_{2}'}l'})^{-C})),
\]
where the first estimate inside the minimum is due to spatial separation and the other two estimates come from Lemma~\ref{lem:correlation-decay}.
Summing this over $k_{2}'\geq k_{2}$ and $l'$ we obtain
\begin{multline*}
\sum_{L\in 2^{\N}, k_{2}'\geq k_{2}, l'\in\Z} L(L2^{-k_{2}}+\abs{2^{-k_{2}}l-2^{-k_{2}'}l'})/2^{-k_{2}'} \min( L^{-C}, 2^{-C(k_{2}'-k_{2})}, (2^{k_{2}'}\abs{2^{-k_{2}}l-2^{-k_{2}'}l'})^{-C})\\
\lesssim
\sum_{L\in 2^{\N}, k\geq 0, l'\in\Z} L(L2^{k}+\abs{2^{k}l-l'}) \min( L^{-C}, 2^{-Ck}, \abs{2^{k}l-l'}^{-C})\\
\lesssim
\sum_{L\in 2^{\N}, k\geq 0, l'\in\Z} L(L2^{k}+\abs{2^{k}l-l'}) ( L + 2^{k} + \abs{2^{k}l-l'} )^{-C}
\leq
C.
\end{multline*}

In the region $k_{2}\geq k_{2}'$ we make a similar decomposition with
\[
\tilde\calR_{L} := \Set{ P' \text{ with shearing matrix }A' \text{ such that } P\cap LP'\neq \emptyset }.
\]
The resulting estimate is similar to the above with the roles of $k_{2}$ and $k_{2}'$ reversed.
\end{proof}

\subsection{Splitting into compactly supported wave packets}
In order to obtain a localized Bessel inequality we decompose wave packets into compactly supported parts as in \cite[Lemma 3.1]{MR2320408}.
\begin{lemma}
\label{lem:split-wave-packet}
For every $C$ there exists $C'$ such that if $C'\phi \in \Phi_{C'}$, then there exists a decomposition
\[
\phi = \sum_{k\geq 0} 2^{-Ck} \phi_{k},
\quad
\phi_{k}\in\Phi_{C},
\supp\phi_{k} \subset B(0,2^{k}).
\]
\end{lemma}
\begin{proof}[Sketch of proof]
Let $\psi$ be a smooth function supported on $B(0,1/2)$ and identically equal to $1$ on $B(0,1/4)$.
Write $\psi_{k}(x)=\psi(2^{-k}x)$ for its $L^{\infty}$ dilates.
Let also $\eta^{(0)},\dotsc,\eta^{(C-2)}$ be smooth functions supported on $[-1/2,1/2]$ with
\[
\int x^{n} \eta^{(m)}(x) \dif x = \one_{n=m}.
\]
For $k\in\N$ and $x_{1}\in\R$ let
\[
m_{k}^{(n)}(x_{1}) := \int_{\R} x_{2}^{n} \phi(x_{1},x_{2})\psi_{k}(x_{1},x_{2}) \dif x_{2},
\]
then for $\abs{\alpha}\leq C$ and $n<C$ we have
\[
\abs{\partial^{\alpha} m_{k}^{(n)}(x_{1})}
=
\abs[\big]{\int_{\R} x_{2}^{n} \partial_{1}^{\alpha}\phi(x_{1},x_{2})(\psi_{k}(x_{1},x_{2})-1) \dif x_{2}}
\lesssim
2^{-Ck}(1+\abs{x_{1}})^{-C}
\]
provided that $C'$ is sufficiently large.
The claimed splitting is given by
\[
\phi_{k} :=
\begin{cases}
\phi(\psi_{k}-\psi_{k-1}) - \sum_{n=0}^{C-2}(m_{k}^{(n)}-m_{k-1}^{(n)}) \otimes \eta^{(n)}, & k>0,\\
\phi\psi_{0} - \sum_{n=0}^{C-2} m_{0}^{(n)} \otimes \eta^{(n)}, & k=0.
\end{cases}
\qedhere
\]
\end{proof}

\subsection{Energy embedding}
The energy embedding is defined by
\[
F(R) := \sup_{\phi_{R}^{(1)}} \abs{\innerp{f}{\phi_{R}^{(1)}}},
\quad
R\in X,
\]
where the supremum is taken over all $L^{1}$ normalized wave packets adapted to $R$ with a sufficiently large order of decay $C'$.

\begin{lemma}
\label{lem:energy-embed-L2}
$\norm{F}_{L^{2,\infty}(S^{2})} \lesssim \norm{f}_{2}$.
\end{lemma}
\begin{proof}
Let $\calR$ be a maximal collection of tiles with $S^{2}(F)(\calR) \geq \lambda$.
If $\calR'\subset X\setminus\calR$ also has size $\geq\lambda$, then using subadditivity of $\sigma$ it is easy to see that $\calR\cup\calR'$ also has size $\geq\lambda$, contradicting maximality.
Hence by maximality we have $\operatorname{outsup}_{X\setminus\calR} S^{2}(F)\leq\lambda$.
On the other hand,
\[
\sigma(\calR)
\leq
\lambda^{-2} \sum_{R\in\calR} \meas{R} \abs{F(R)}^{2}
\lesssim
\lambda^{-2} \norm{f}_{2}^{2}
\]
by Lemma~\ref{lem:bessel}.
\end{proof}

\begin{lemma}
\label{lem:energy-embed-Linfty}
$\norm{F}_{L^{\infty}(S^{2})} \lesssim \norm{f}_{\infty}$.
\end{lemma}
\begin{proof}
Let $\calR\in\bfE$ and let $\phi_{R}$, $R\in\calR$, be wave packets that almost extremize $F(R)$.
Splitting the corresponding members of $\Phi_{C'}$ using Lemma~\ref{lem:split-wave-packet} we obtain decompositions $\phi_{R} = \sum_{k\geq 0}2^{-Ck}\phi_{R,k}$, where each $\phi_{R,k}$ is an $L^{1}$ normalized wave packet adapted to $R$ (with a lower order of decay $C$) and supported on $2^{k}R$.

By Lemma~\ref{lem:bessel} and the support condition we have
\begin{align*}
\sum_{R\in\calR} \meas{R} \abs{\innerp{f}{\phi_{R,k}}}^{2}
&\lesssim
2^{-2Ck} \norm[\big]{f \one_{\cup\Set{2^{k}R : R\in\calR}}}_{2}^{2}\\
&\leq
2^{-2Ck} \norm{f}_{\infty}^{2} \meas[\big]{ \cup_{R\in\calR} 2^{k}R}\\
&\leq
2^{(C_{\ref{eq:def-sigma}}-2C)k} \norm{f}_{\infty}^{2} \sigma(R),
\end{align*}
and summing in $k$ we obtain
\[
\sum_{R\in\calR} \meas{R} \abs{\innerp{f}{\phi_{R}}}^{2}
\lesssim
\norm{f}_{\infty}^{2} \sigma(R),
\]
so that $S^{2}(F)(\calR) \lesssim \norm{f}_{\infty}$ as required.
\end{proof}

\section{Covering lemma for parallelograms and the mass embedding}
\label{sec:single-scale:mass-embed}
For completeness we include a slightly streamlined proof of a covering lemma from \cite{MR3148061}.
Covering  lemmas of this type go back to \cite{MR0379785}.
We consider parallelograms with two vertical edges as shown below:
\begin{center}
\begin{tikzpicture}[yscale=0.5]
\draw (0.5,1) node {$R$};
\draw (0,0) node[left] {$A$} -- (0,1) node[left] {$B$} -- (1,2) node[right] {$C$} -- (1,1) node[right] {$D$} -- cycle;
\draw[dotted] (0,0) -- (0,-0.5);
\draw[dotted] (1,1) -- (1,-0.5);
\draw[|-|] (0,-0.5) -- (1,-0.5) node[midway,below] {$I$};
\end{tikzpicture}
\end{center}
The \emph{height} $H(R)$ is the common length of $AB$ and $CD$.
The \emph{shadow} $I(R)$ is the projection of $R$ onto the horizontal axis.
The \emph{slope} $s(R)$ is the common slope of the edges $BC$ and $AD$.
The \emph{uncertainty interval} $U(R) \subset\R$ is the interval between the slopes of $BD$ and $AC$.
It is the interval of length $2H(R)/\meas{I(R)}$ centered at $s(R)$.

\begin{lemma}[{cf.\ \cite[Lemma 7]{MR3148061}}]
\label{lem:bt-cf-covering}
Let $\calR$ a finite collection of parallelograms with vertical edges and dyadic shadow.
Then there exists $\calG\subset\calR$ such that
\begin{equation}
\label{bad}
\meas{ \bigcup _{R \in \calR} R }
\lesssim
\sum_{R\in\calG } \meas{R}
\end{equation}
and for every $n\in\N$ we have
\begin{equation}
\label{eq:good-U}
\sum _{R_{1},\dotsc,R_{n}\in\calG : U(R_{1})\cap\dotsb\cap U(R_{n}) \neq\emptyset } \meas{R_{1}\cap\dotsb\cap R_{n}}
\lesssim_{n}
\sum_{R\in\calG } \meas{R}.
\end{equation}
In particular, for every measurable function $u:\R^{2}\to \R$ the sets
\[
E(R):=\Set{(x,y)\in R: u(x,y)\in U(R)}.
\]
satisfy
\begin{equation}
\label{good}
\int ( \sum _{R\in\calG } \one _{E(R)} )^{q}
\lesssim_{q}
\sum_{R\in\calG } \meas{R},
\quad 0<q<\infty.
\end{equation}
\end{lemma}

In \cite{MR3148061} the conclusion \eqref{good} is stated for one-variable vector fields, but this structural assumption is not used in the proof.

In the proof of Lemma~\ref{lem:bt-cf-covering} we denote by $CR$ the parallelogram with the same center, slope, and shadow as $R$ but height $CH(R)$ (this definition of $CR$ is used only here).
We need the following geometric observation:
\begin{lemma}\label{7rlemma}
Let $R,R'$ be two parallelograms with $I(R)= I(R')$, $U(R)\cap U(R')\neq\emptyset$, and $R\cap R'\neq \emptyset$.
If  $7H(R)\le H(R')$, then $7R\subseteq 7R'$.
\end{lemma}

Let $M_V$ denote the Hardy--Littlewood maximal operator in the vertical direction:
\begin{equation}
\label{eq:MV}
M_Vf(x,y)
=
\sup_{y\in J} \meas{J}^{-1} \int_J \abs{f(x,z)} \dif z,
\end{equation}
where the supremum is taken over all intervals $J$ containing $y$.

\begin{proof}[Proof of Lemma~\ref{lem:bt-cf-covering}]
We select $\calG$ using the following iterative procedure.
Initialize
\begin{align*}
STOCK &:= \calR \\
\calG &:= \emptyset.
\end{align*}
While $STOCK \neq \emptyset$, choose an $R\in STOCK$ with maximal $\meas{I(R)}$.
Update
\begin{align*}
\calG &:= \calG \cup \Set{R},\\
STOCK &:= STOCK \setminus \Set{R\in STOCK : R\subset \Set{M_{V} ( \sum_{R'\in\calG} \one _{7R'}) \geq 10^{-4}}}.
\end{align*}
This procedure terminates after finitely many steps since at each step at least the selected parallelogram $R$ is removed from $STOCK$.

By construction
\begin{align}
\bigcup _{R\in\calR} R
\subset
\Set{ x\colon M_{V} ( \sum_{R\in\calG} \one_{7R}) (x) \geq 10^{-4}},
\end{align}
and \eqref{bad} follows by the weak $(1,1)$ inequality for $M_{V}$.

We prove \eqref{eq:good-U} by induction on $n$.
For $n=1$ the statement clearly holds.
Suppose that \eqref{eq:good-U} holds for a given $n$, we will show that it also holds with $n$ replaced by $n+1$.
For each $R'\in\calG$ let
\[
\calG(R') := \Set{R\in\calG \text{ chosen prior to } R' \text{ with } R\cap R'\neq\emptyset, U(R)\cap U(100 R')\neq \emptyset}.
\]
All terms in \eqref{eq:good-U} in which some $R_{i}$ occurs at least twice are estimated by the inductive hypothesis.
In the remaining terms we may arrange the $R_{i}$'s in the order reverse to the selection order (losing a factor $(n+1)!$), and omitting some vanishing terms we obtain the estimate
\begin{equation}
\label{eq:adm-7r}
\sum_{R_{0}\in\calG, R_{1}\in\calG(R_{0}), \dotsc, R_{n} \in\calG(R_{n-1})}
\meas{R_{0} \cap \dotsb \cap R_{n}}
\leq
\sum_{R_{0}\in\calG, R_{1}\in\calG(R_{0}), \dotsc, R_{n} \in\calG(R_{n-1})}
\meas{I(R_{0})} \cdot \meas{H(R_{n})}.
\end{equation}
We claim that for every $R' \in \calG$ we have
\begin{equation}
\label{eq:adm2}
\sum_{R\in \calG(R')} H(R) \leq H(R').
\end{equation}
To see this let $R\in\calG(R')$, so that in particular $I(R')\subset I(R)$ and $U(R)\cap U(10R') \neq\emptyset$.
If $H(R')\leq H(R)$, then $7H(10R')\leq H(70R)$, and Lemma~\ref{7rlemma} shows that $70R'\subset 490 R$, so that $M_{V}(\one_{R})\geq 490^{-1}$ on $R'$, contradicting $R'\in\calG$.
Therefore $H(R') > H(R)$, so $7H(R) \leq H(10R')$, and Lemma~\ref{7rlemma} shows that
\[
7R \cap (I(R') \times\R) \subset 70 R'.
\]
The  inequality \eqref{eq:adm2} follows, since otherwise $M_{V}(\sum_{R\in\calG(R')}\one_{R}) \geq 70^{-1}$ on $R'$, contradicting $R'\in\calG$.
Hence
\begin{align*}
\eqref{eq:adm-7r}
&\leq
\sum_{R_{0}\in\calG, R_{1}\in\calG(R_{0}), \dotsc, R_{n-1} \in\calG(R_{n-2})}
\meas{I(R_{0})} \cdot \meas{H(R_{n-1})}\\
&\leq \dotsb \leq
\sum_{R_{0}\in\calG}
\meas{I(R_{0})} \cdot \meas{H(R_{0})}
=
\sum_{R_{0}\in\calG} \meas{R_{0}}.
\end{align*}
This completes the proof of \eqref{eq:good-U}.
In order to see \eqref{good} observe that its left-hand side is monotonically increasing in $q$, so it suffices to consider integer values $q=n$, and in this case the left-hand side of \eqref{good} is dominated by the left-hand side of \eqref{eq:good-U}.
\end{proof}

\subsubsection{Mass embedding}
The mass embedding is given by
\[
G(R) := \meas{R}^{-1} \int_{E_{R}} \abs{g},
\quad
R\in X.
\]
\begin{lemma}
\label{lem:mass-embedding}
Let $1<q<\infty$.
If the constant $C$ in the definition of $\sigma$ is sufficiently large depending on $q$, then $\norm{G}_{L^{q,\infty}(S^{\infty})} \lesssim \norm{g}_{q}$.
\end{lemma}
Recall that $CP$ now again denotes the parallelogram $P$ expanded by the factor $C$ both in the horizontal and in the vertical direction.
\begin{proof}
Let $\delta>0$, $g\in L^{q}(\R^{2})$, and let $\calR$ be a collection of tiles such that $G(R)\geq\delta$ for $R\in\calR$ .
We have to show
\begin{equation}
\label{eq:bt-ct:est}
\sup_{L\geq 1} L^{-C}\meas[\big]{\cup_{R\in\calR} LR}
\lesssim_{q}
\delta^{-q} \norm{g}_{q}^{q}.
\end{equation}
Note that the definition of $G(R)$ makes sense for arbitrary parallelograms (not only the dyadic ones that we call tiles).
For the enlarged parallelograms $LR$ we still have $G(LR) \geq \delta/L^{2}$, so it suffices to show \eqref{eq:bt-ct:est} with $L=1$ and a collection of arbitrary parallelograms $\calR$, provided that the constant $C_{\ref{eq:def-sigma}}$ in the definition of $\sigma$ is at least $2q$.

Enlarging the parallelograms in such a way that their shadows become intervals in adjacent dyadic grids and the uncertainty intervals stay the same we preserve the hypothesis $G(R)\gtrsim\delta$ up to a multiplicative constant.
Hence we may assume that the parallelograms have dyadic shadows.

In view of \eqref{bad} it suffices to consider the parallelograms in the subset $\calG\subset\calR$ provided by Lemma~\ref{lem:bt-cf-covering}.
By the density assumption and H\"older's inequality we have
\begin{align*}
\sum_{R\in\calG} \meas{R}
&\leq
\sum_{R\in\calG} \frac{1}{\delta} \int_{E(R)} \abs{g}\\
&=
\frac{1}{\delta}
\norm[\big]{ \sum_{R\in\calG} \one_{E(R)} \abs{g} }_1\\
&\leq
\frac{1}{\delta} \norm{\sum_{R\in\calG} \one_{E(R)}}_{q'} \norm{g}_{q}\\
&\lesssim
\frac{1}{\delta} \left(\sum_{R\in\calG} \meas{R} \right)^{1/q'} \norm{g}_{q},
\end{align*}
where in the last passage we have used the estimate \eqref{good}.
After division by the middle factor of the right hand side we obtain the claim.
\end{proof}

\section{Estimate for the square function}
We finally prove Theorem~\ref{thm:single-scale}.
Note that $A_{u,\phi} P_{2,t} f(x)$ is the integral of $f$ against an $L^{1}$ normalized wave packet associated to a tile that contains $x$ and whose uncertainty interval contains $u(x)$.
Hence the left-hand side of \eqref{eq:single-scale-square-fct} is bounded by
\[
\norm{ \big( \sum_{R\in X} F(R)^{2} \one_{E_{R}} \big)^{1/2} }_{p}
=
\norm{ \sum_{R\in X} F(R)^{2} \one_{E_{R}} }_{p/2}^{1/2}.
\]
Dualizing with a function $g\in L^{(p/2)'}$ we obtain
\[
\int \sum_{R\in X} F(R)^{2} \one_{E_{R}} g
=
\sum_{R\in X} \meas{R} F(R)^{2} G(R).
\]
For every $\calR\in\bfE$ we have $\sum_{R\in\calR} \meas{R} F(R) = \sigma(\calR) S^{1}(F)(\calR)$.
Therefore by \cite[Proposition 3.6]{MR3312633} and outer H\"older inequality \cite[Proposition 3.4]{MR3312633} the above is bounded by
\[
\norm{ F^{2} G }_{L^{1}(S^{1})}
\lesssim
\norm{ F^{2} }_{L^{p/2}(S^{1})} \norm{ G }_{L^{(p/2)'}(S^{\infty})}
=
\norm{ F }_{L^{p}(S^{2})}^{2} \norm{ G }_{L^{(p/2)'}(S^{\infty})}.
\]
The first term is bounded by $\norm{f}_{p}^{2}$ by Lemmas \ref{lem:energy-embed-L2} and \ref{lem:energy-embed-Linfty} and interpolation \cite[Proposition 3.5]{MR3312633}.
The second term is bounded by $\norm{g}_{(p/2)'}$ by Lemma~\ref{lem:mass-embedding} and interpolation \cite[Proposition 3.5]{MR3312633}.

\subsection{Application to a maximal operator with a restricted set of directions}
\label{sec:single-scale-N-directions}

In this section we prove Corollary~\ref{cor:single-scale-N-directions}.

Although the operator \eqref{eq:single-scale-op} is unbounded for general direction fields $u$, it is clearly bounded (on any $L^{p}$, $1\leq p\leq\infty$) with norm $O(N)$ as long as $u$ is allowed to take at most $N$ values.
This trivial estimate has been improved to $O(\sqrt{\log N})$ on $L^{2}$ by Katz \cite{MR1711029}.
Note that we also have the trivial estimate $O(1)$ on $L^{\infty}$, and by interpolation one obtains logarithmic dependence on $N$ of the operator norm of \eqref{eq:single-scale-op} on $L^{p}$ also for all $2<p<\infty$.
Demeter \cite{MR2680067} gives an alternative proof of Katz's result, and furthermore hints at yet another different proof   via reduction to the square function bound Theorem~\ref{thm:single-scale} by means of the good-$\lambda$ inequality with sharp constant due to Chang, Wilson, and Wolff \cite{MR800004}.
The first appearance of a similar reduction to   square function   in the context of maximal multipliers   goes back to Grafakos, Honz\'ik, and Seeger \cite{MR2249617}, and analogous approaches have been since used in Demeter \cite{MR2680067} and Demeter with the first author \cite{MR3145928}.
We have not been able to reproduce the endpoint $p=2$ using this technique.
However, notice that our   square function approach,  after interpolation, recovers the result for $p>2$ up to an arbitrarily small loss in the exponent of the logarithm.

\begin{proof}[Proof of Corollary \ref{cor:single-scale-N-directions}]
For $j\in \Z$, define the dyadic martingale averaging operator
\begin{equation}
E_j f :=\sum 2^{2j}  \innerp{f}{\one_{Q}} \one_{Q} ,
\end{equation}
where the summation runs over all standard dyadic squares $Q$ in $\R^2$ with side length $2^{-j}$.
Further define
\[
\Delta_j =E_{j+1}-E_j ,
\]
\[
\Delta f:=(\sum_{j\in \Z}\abs{\Delta_j f}^2)^{1/2}.
\]
Let $M$ denote the non-dyadic Hardy--Littlewood maximal operator.
Chang, Wilson, and Wolff \cite[Corollary 3.1]{MR800004} prove that there are universal constants $c_1$ and $c_2$ such that for all $\lambda>0$ and $0<\epsilon<1$
\begin{equation}
\label{eq:cww}
\meas{ \Set{z: \abs{f(z)-E_0 f(z)}>2\lambda,\  \Delta f(z)\le \epsilon \lambda} }
\le
c_2 e^{-\frac{c_1}{\epsilon^2}}  \meas{ \Set{z: Mf(z)\ge \lambda} }.
\end{equation}

Denote the finitely many values of $u$ by  $u_i$, $1\le i\le N$, and write $A_{u_i}$ for the operator with the constant direction field $u_i$.
Corollary~\ref{cor:single-scale-N-directions} follows by Marcinkiewicz interpolation from the weak type inequality
\[
\meas{\Set{z:\sup_i \abs{A_{u_i}f(z)}>4 \lambda}}
\le
C \log(N+2)^{p/2} \lambda^{-p}\norm{f}_p^p
\]
for $2<p<\infty$.
Gearing up for Chang, Wilson, and Wolff we estimate
\begin{align}
\nonumber\MoveEqLeft
\meas{\Set{z:\sup_i \abs{A_{u_i}f(z)}>4 \lambda}}\\
&=
\nonumber
\meas{\bigcup_i \Set{z:\abs{A_{u_i}f(z)}>4 \lambda}}\\
\label{firstcww}
&\le
\meas{\bigcup_i \Set{z:\abs{A_{u_i}f(z)-E_0A_{u_i}f(z)}>2 \lambda, \Delta A_{u_i}f(z)\le \epsilon \lambda }}\\
\label{secondcww}
&\quad+
\meas{\bigcup_i \Set{z:\abs{E_0A_{u_i}f(z)}>2 \lambda}}\\
\label{thirdcww}
&\quad+
\meas{\bigcup_i \Set{z:\Delta A_{u_i} f(z)> \epsilon\lambda  }}
\end{align}
Using \eqref{eq:cww} we estimate
\begin{align*}
\eqref{firstcww}
&\le
\sum_i \meas{\Set{z:\abs{A_{u_i}f(z)-E_0A_{u_i}f(z)}>2 \lambda, \Delta A_{u_i}f(z)\le \epsilon \lambda }}\\
&\le
C \sum_{i} e^{-\frac{c_1}{\epsilon^2}}\meas{\Set{z: M(A_{u_i}f)(z)>2 \lambda}}\\
&\le C \sum_{i} e^{-\frac{c_1}{\epsilon^2}} \lambda^{-p}\norm{A_{u_i}f}_p^p\\
&\le C N e^{-\frac{c_1}{\epsilon^2}} \lambda^{-p}\norm{f}_p^p\\
&\le C\lambda^{-p}\norm{f}_p^p
\end{align*}
provided $\epsilon \leq c_{1}^{1/2}\log (N+2)^{1/2}$.

The function $E_0 A_{u_i}f$ in \eqref{secondcww} is pointwise dominated by the standard Hardy--Littlewood maximal operator, because $E_0$ and $A_ {u_i}$ compose to some averaging operator at scale $0$.
Therefore
\[
\eqref{secondcww}
\leq
\meas{\Set{z: Mf(z)>C\lambda}}
\lesssim
\lambda^{-p}\norm{f}_p^p.
\]

To control \eqref{thirdcww} we introduce a suitable Littlewood--Paley decomposition in the second variable, note that $P_{2^k}$ commutes with $A_{u_i}$, and estimate pointwise
\begin{align*}
\sup_i \Delta A_{u_i}f
&=
\sup_i
\Delta (\sum_{k\in \Z} P_{2^k}A_{u_i}P_{2^k} f) \\
&=
\sup_i
(\sum_j \abs{\sum_{k}  \Delta_j P_{2^k}A_{u_i}P_{2^k} f}^2)^{1/2} \\
&\lesssim
\sup_i
(\sum_j (\sum_{k}  2^{-\abs{j-k}/q'} M M_{q,V} A_{u_i}P_{2^k} f)^2)^{1/2} \\
&\lesssim
\sup_i
(\sum_j \sum_{k}  2^{-\abs{j-k}/q'}( M M_{q,V} A_{u_i}P_{2^k} f)^2)^{1/2} \\
&\lesssim
\sup_i
(\sum_{t} (M M_{q,V} A_{u_i}P_{2^k} f)^2)^{1/2} \\
&\le
(\sum_{k} (M M_{q,V} \sup_i \abs{A_{u_i}P_{2^k} f})^2)^{1/2},
\end{align*}
where $M_{q,V}$ is the $q$-maximal operator in the vertical direction $M_{q,V}f=(M_{V}(f^{q}))^{1/q}$ for any fixed $1<q<2$ with $M_{V}$ as in \eqref{eq:MV}, $M$ is the usual two-dimensional Hardy--Littlewood maximal operator, and the pointwise estimate $\abs{\Delta_j P_{2^{k}}f} \lesssim 2^{\abs{j-k}/q'} M M_{q,V} f$ follows from \cite[Sublemma 4.2]{MR2249617} applied in the vertical direction.
The Fefferman--Stein maximal inequalities and Theorem~\ref{thm:single-scale} give
\[
\norm{ (\sum_{t\in 2^{\Z}} (M M_{q,V} \sup_i A_{u_i}P_t f)^2)^{1/2} }_p
\le
C \norm{ (\sum_{t\in 2^{\Z}} (\sup_i A_{u_i}P_t f)^2)^{1/2} }_p
\le
C \norm{ f }_p.
\]
With Tchebysheff we obtain
\[
\eqref{thirdcww}
=
\meas{\Set{ \sup_i \Delta A_{u_i} f(z)> \epsilon\lambda  }}
\le
C(\epsilon\lambda)^{-p}\norm{f}_p^p
\le
C \log(N+2)^{p/2} \lambda^{-p}\norm{f}_p^p,
\]
and this concludes the proof of Corollary \ref{cor:single-scale-N-directions}.
\end{proof}

\section{Lacey--Li covering argument}
\label{sec:LL}
Lacey and Li \cite{MR2654385} have introduced a certain family of  maximal operators associated to a vector field $u$, which they called the ``Lipschitz--Kakeya'' maximal operator:
\[
f\mapsto \sup_{R\in \mathcal R_\delta} \langle f,\mathbf{1}_R\rangle \frac{\mathbf{1}_R}{\abs{R}}
\]
where, using the notation from Section~\ref{sec:single-scale:mass-embed},  $\mathcal R_\delta$ is the collection of those parallelograms  $R$ with $\abs{E(R)}\geq \delta \abs{R}$; that is, the vector field $u$ points within the uncertainty interval of $R$ on (at least a) $\delta$-portion of $R$.
These authors proved that such maximal operators have weak type $(2,2)$ operator norm $O(\delta^{-1/2})$ if the vector field is Lipschitz.
In the same paper, they have further showed that an $L^{p}$ bound for this operator for any $p<2$ implies the $L^{2}$ estimate for the single band version of the directional Hilbert transform.
Bateman and Thiele \cite{MR3148061} gave a streamlined proof of the weak type $(2,2)$ estimate for this maximal operator in the case of a one-variable vector field and used it to obtain square function estimates of the type \eqref{eq:Hu-sq-fct-est} for the directional Hilbert transform.

In this section we further simplify the proof of the weak type $(2,2)$ estimate for this maximal operator, also taking care of Lipschitz vector fields.
We use the notation from Section~\ref{sec:single-scale:mass-embed} and write $L(R)=\meas{I(R)}$.
The main part of the proof is the following covering argument.
\begin{theorem}\label{laceylict}
Let $0< \delta \le 1$ and let $\calR$  be a finite collection of parallelograms with vertical edges and dyadic shadow such that for each $R\in \calR$ we have
\[
\abs{E(R)}\ge \delta \abs{R}
\]
and $L(R)\norm{v}_{\Lip}\leq 1/30$.
Then there is a subset $\calG\subset \calR$ such that
\begin{align} \label{coverct}
\abs{\bigcup_{R\in \calR} R} & \lesssim \sum_{R\in \calG} \abs{R}\ ,
\\
\label{twoonect}
\int (\sum_{R\in \calG}\one_R)^2 & \lesssim \delta^{-1} \sum_{R\in \calG} \abs{R}\ .
\end{align}
\end{theorem}

The set $\calG$ is constructed as in Lemma~\ref{lem:bt-cf-covering}, so that \eqref{coverct} holds by construction.
In the remaining part of this section we will show \eqref{twoonect}.
Expanding the square on the left-hand side of \eqref{twoonect} and using symmetry we obtain the estimate
\[
\sum_{R\in\calG} \meas{R}
+ 2 \sum_{(R,R')\in\calP} \meas{R\cap R'},
\]
where $\calP$ is the set of pairs $(R,R')\in\calG^{2}$ such that $R\cap R'\neq\emptyset$ and $R$ has been chosen before $R'$.
The former term is clearly bounded by the right-hand side of \eqref{twoonect}.
In the latter term we notice first that by \eqref{eq:adm2} we have
\[
\sum_{R'\in\calG} \sum_{R\in\calG(R')} \meas{R\cap R'}
\leq
\sum_{R'\in\calG} \sum_{R\in\calG(R')} L(R') H(R)
\leq
\sum_{R'\in\calG} L(R') H(R'),
\]
and this is also bounded by the right-hand side of \eqref{twoonect}.
Hence it suffices to estimate
\begin{equation}
\label{eq:ll:PR}
\sum_{R\in \calG}\sum_{R'\in\calP(R)} \meas{R\cap R'},
\end{equation}
where
\[
\calP(R) := \Set{R' : (R,R')\in \calP, U(R)\cap 10 U(R') = \emptyset}.
\]

First we clarify the position of $U(R')$ relative to $U(R)$ when $R'\in\calP(R)$.

\begin{lemma}
\label{lem:lk:lac}
Suppose $R'\in\calP(R)$.
Then
\[
\max (\abs{U(R)},\abs{U(R')}) \leq \frac14 \dist(U(R'),U(R)).
\]
\end{lemma}
\begin{proof}
We distinguish two cases: 
\begin{enumerate}
\item $\abs{U(R)}\le \abs{U(R')}$.
In this case we use the definition of $\calP(R)$.
\item $\abs{U(R)}>\abs{U(R')}$.
In this case we have
\[
H(R') = \abs{U(R')} L(R') < \abs{U(R)} L(R) = H(R),
\]
and in particular $7H(R') \leq H(10R)$.
If the conclusion was false, then $10U(R) \cap U(R')\neq\emptyset$, and by Lemma~\ref{7rlemma} we obtain $7R'\subset 70R$.
This contradicts the hypothesis that $R'$ was added to $\calG$ after $R$.
\end{enumerate}
\end{proof}

The next lemma gives a condition for two parallelograms to have comparable slopes.
This is the only place where the Lipschitz hypothesis is used.
Denote the projection onto the first coordinate by $\Pi$.
\begin{lemma}
\label{lem:lk:near}
Assume $L(R) \norm{v}_{\Lip} \leq 1/30$.
Suppose $R',R''\in\calP(R)$ and $\Pi E(R') \cap \Pi E(R'') \neq \emptyset$.
Then
\[
\dist(U(R'),U(R'')) \leq \frac{1}{8} \dist(U(R),U(R')).
\]
\end{lemma}
\begin{proof}
Let $x\in \Pi E(R') \cap \Pi E(R'')$.
The distance of the points $y',y''$ such that $(x,y')\in R'$ and $(x,y'')\in R''$ is bounded above by
\[
H(R) + H(R') + H(R'') + L(R') \dist(U(R),U(R')) + L(R'') \dist(U(R),U(R'')).
\]
Choosing $(x,y')\in E(R')$ and $(x,y'')\in E(R'')$ and using the Lipschitz hypothesis and Lemma~\ref{lem:lk:lac} we obtain
\begin{multline*}
\dist(U(R'),U(R''))
\leq
\norm{v}_{\Lip}\Big(H(R) + H(R') + H(R'')\\
+ L(R') \dist(U(R),U(R')) + L(R'') \dist(U(R),U(R''))\Big)\\
\leq
\frac{1}{30} \Big(\abs{U(R)} + \abs{U(R')} + \abs{U(R'')}+ \dist(U(R),U(R')) + \dist(U(R),U(R''))\Big)\\
\leq
\frac{1}{30} \Big(\frac64 \dist(U(R),U(R')) + \frac54 \dist(U(R),U(R''))\Big)\\
\leq
\frac{1}{30} \Big(\frac{11}{4} \dist(U(R),U(R')) + \frac54 \abs{U(R')} + \frac54 \dist(U(R'),U(R''))\Big)\\
\leq
\frac{1}{30} \Big(\frac{13}{4} \dist(U(R),U(R')) + \frac54 \dist(U(R'),U(R''))\Big).
\end{multline*}
The conclusion follows.
\end{proof}

The basic estimate for the size of the intersection of two parallelograms is the size of the intersection of infinite stripes containing them:
\begin{lemma}
Let $R,R' \in \calR$.
Then
\begin{equation}\label{seconduict}
\abs{R\cap R'} \leq \dist(U(R),U(R'))^{-1} H(R) H(R').
\end{equation}
\end{lemma}
\begin{proof}
By a shearing transformation we may assume that the central line segment of $R$ is horizontal.
Let $u_{0}$ be the central slope of $R'$.
Then $R\cap R'$ is contained in a parallelogram of height $H(R)$ and base $H(R')u_0^{-1}$.
On the other hand, $u_{0}\geq\dist(U(R),U(R'))$.
\end{proof}

We decompose the set $\calP(R)$ dyadically according to the distance between $U(R)$ and $U(R')$.
Specifically, for $k\in\N$ let
\[
\calP_{k}(R) := \Set{ R'\in\calP(R) : 2^{k-3} < \frac{\dist(U(R),U(R'))}{\abs{U(R)}} \leq 2^{k-2}}.
\]
For a fixed $k$ we will estimate the contribution of $\calP_{k}(R)$ to \eqref{eq:ll:PR} using a stopping time argument.
For a dyadic interval $I$ denote $R_{I}  := R \cap (I\times\R)$.
\begin{lemma}\label{nosmallct}
Let $I \subseteq I_{R}$ be a dyadic interval such that there exists $R''\in \calP_{k}(R)$ with $I_{R''}\subseteq I$.
Then
\[
\sum_{R'\in \calP_{k}(R): I\subseteq I_{R'}} \abs{R_I\cap R'} \leq 2 \abs{R_I}.
\]
\end{lemma}

\begin{proof}
Let $R''\in \calP_{k}(R)$ be the parallelogram with $I_{R''}\subseteq I$ that has been chosen last.
Let
\[
\calQ := \Set{R'\in \calP_{k}(R): I\subseteq I_{R'}, R_{I}\cap R'\neq\emptyset} \setminus \Set{R''}.
\]
Since $\abs{R_{I}\cap R''} \leq \abs{R_{I}}$, it suffices to show
\begin{equation}
\label{eq:nosmallct:r}
\sum_{R'\in \calQ} \abs{R_I\cap R'} \leq 10^{-1} \abs{R_I}.
\end{equation}
Assume for contradiction that \eqref{eq:nosmallct:r} fails.
Let $U:= 2^{k} U(R)$.
By Lemma~\ref{lem:lk:lac} we have $U(R'')\subset U$ and thus
\[
H(R'')
\leq
\abs{U} \abs{I_{R''}}
\leq
\abs{U} \abs{I}.
\]
In particular
\[
R''  \subset 50  (1+\abs{U} \abs{I}/H(R)) R_{I} =: \tilde R.
\]
The parallelogram $R''$ has been selected for $\calG$ after the parallelogram $R$ and the parallelograms $R'\in \calQ$.
To obtain a contradiction with the construction of $\calG$ it suffices to show that  
\[
M_V(\one_R + \sum_{R'\in \calQ} \one_{R'})
\] where $M_V$ is the vertical directional maximal function,
is larger than $10^{-3}$ on the parallelogram $\tilde R$.

First assume there exists $R'\in \calQ$ with $H(R')\ge 20\abs{U}\abs{I}$.
Note that
\[
U(R')\subset U \subset U(\tilde{R}).
\]
Applying Lemma \ref{7rlemma} to the rectangles $R'_I$ and $\tilde{R}$ we obtain
\[
M_V (\one_{R'}+\one_R)
\geq
7^{-1}H(\tilde{R})^{-1} \big( \min(H(R'),H(\tilde{R}))+H(R) \big)
>
10^{-3}
\]
on $\tilde{R}$, which proves Lemma \ref{nosmallct} in the given case.

Hence we may assume 
\[
H(R')\le 20 \abs{U}\abs{I}
\]
for every $R'\in \calQ$.
We then have on $\tilde{R}$ that
\begin{align*}
\MoveEqLeft
M_V (\one_R   + \sum_{R'\in \calQ}\one_{R'})\\
&\ge
H(\tilde{R})^{-1} (H(R)+\sum_{R'\in\calQ} H(R'))\\
&\ge
H(\tilde{R})^{-1} (H(R)+\sum_{R'\in\calQ} \abs{R_I\cap R'} \abs{U} H(R)^{-1})
&\text{by \eqref{seconduict}}\\
&\ge
H(\tilde{R})^{-1} (H(R)+\abs{U} H(R)^{-1} 10^{-1} \abs{R_I})
&\text{since \eqref{eq:nosmallct:r} fails}\\
&\ge
500^{-1}.
\end{align*}
This completes the proof of Lemma \ref{nosmallct}.
\end{proof}

\begin{corollary}
\label{cor:lk:1scale}
\[
\sum_{R'\in \calP_{k}(R)} \abs{R\cap R'}
\leq
4 H(R) \cdot \meas[\big]{\bigcup_{R'\in \calP_{k}(R)} \Pi(R')}.
\]
\end{corollary}

\begin{proof}
Let $\calI$ be the set of maximal dyadic intervals contained in $\cup_{R'\in \calP_{k}(R)} \Pi(R')$ that do not contain $I_{R'}$ for any $R'\in\calP_{k}(R)$.
For each $I\in\calI$ let $\tilde I$ denote its dyadic parent.
Then by maximality of $I$ and Lemma \ref{nosmallct} we have
\begin{align*}
\sum_{R'\in \calP_{k}(R)} \abs{R_I\cap R'}
&=
\sum_{R'\in \calP_{k}(R) : I\subsetneq I_{R'}} \abs{R_I\cap R'}\\
&\leq
\sum_{R'\in \calP_{k}(R) : \tilde I\subseteq I_{R'}} \abs{R_{\tilde I}\cap R'}\\
&\leq
2 \abs{R_{\tilde I}}
\leq
4 \abs{R_{I}}.
\end{align*}
The set $\calI$ is a covering of $\cup_{R'\in \calP_{k}(R)} \Pi(R')$, so the conclusion of the lemma follows after summing over all intervals in $\calI$.
\end{proof}

We are now in position to complete the proof of Theorem~\ref{laceylict} by estimating \eqref{eq:ll:PR}:
\begin{align*}
\sum_{R'\in\calP(R)} \abs{R\cap R'}
&=
\sum_{k\in\N} \sum_{R'\in \calP_{k}(R)} \abs{R\cap R'}\\
&\lesssim
H(R) \sum_{k\in\N} \abs{\cup_{R'\in \calP_{k}(R)} \Pi(R')}
&\text{by Corollary~\ref{cor:lk:1scale}}\\
&=
H(R) \sum_{k\in\N} \sum_{R'\in \calP_{k}'(R)} \abs{\Pi(R')}\\
&\lesssim
\delta^{-1} H(R) \sum_{k\in\N} \sum_{R'\in \calP_{k}'(R)} \abs{\Pi E(R')}\\
&\lesssim
\delta^{-1} \abs{R},
\end{align*}
where $\calP_{k}'(R)\subset\calP_{k}(R)$ is a system of representatives for maximal intervals $I_{R'}$, in the penultimate step we have used the density hypothesis in the form $\abs{\Pi(R')} \leq \abs{\Pi E(R')}/\delta$, and in the last step we have used Lemma~\ref{lem:lk:near} to conclude that the projections there have bounded overlap.
